\documentclass[10pt,a4paper,dvipsnames]{article}

\usepackage{fullpage,amssymb,amsmath,amsfonts,amsthm,paralist,graphicx,color,float, mathtools,tikz}
\usepackage[utf8]{inputenc}
\usepackage[english]{babel}
\usepackage[colorlinks=true,linkcolor=BrickRed,citecolor=blue]{hyperref}
\usepackage[normalem]{ulem}
\usepackage{centernot}
\usepackage{ifthen,xkeyval, tikz, calc, graphicx}
\usepackage{xcolor}
\usepackage{framed}
\usepackage[textsize=scriptsize, backgroundcolor=blue!20!white,bordercolor=red]{todonotes}
\usepackage{dsfont}
\usepackage{mathrsfs}
\usepackage{comment}
\usepackage{stmaryrd}
\usepackage{enumitem}
\usepackage{caption}
\usepackage{subcaption}
\usepackage{todonotes}
\usepackage{thmtools}
\usetikzlibrary{plotmarks}
\usetikzlibrary{shapes.misc}
\usepackage{orcidlink}

\newtheorem{theorem}{Theorem}[section]
\newtheorem{lemma}[theorem]{Lemma}
\newtheorem{prop}[theorem]{Proposition}
\newtheorem{corollary}[theorem]{Corollary}
\newtheorem{observation}[theorem]{Observation}

\theoremstyle{definition}
\newtheorem{definition}[theorem]{Definition}

\newtheorem{assumption}{Assumption}
\newtheorem{assumptionalt}{Assumption}[assumption]
\newenvironment{assumptionp}[1]{
  \renewcommand\theassumptionalt{#1}
  \assumptionalt
}{\endassumptionalt}

\theoremstyle{remark}
\newtheorem{remark}[theorem]{Remark}

\definecolor{shadecolor}{named}{GreenYellow}

\makeatletter
\newcommand{\pushright}[1]{\ifmeasuring@#1\else\omit\hfill$\displaystyle#1$\fi\ignorespaces}
\newcommand{\pushleft}[1]{\ifmeasuring@#1\else\omit$\displaystyle#1$\hfill\fi\ignorespaces}
\makeatother

\newcommand{\p}{\mathbb P}
\newcommand{\pla}{\mathbb P_\lambda}
\newcommand{\Pcal}{\mathcal P}

\newcommand{\E}{\mathbb E}
\newcommand{\Ecal}{\mathcal E}
\newcommand{\Efrak}{\mathfrak E}

\newcommand{\R}{\mathbb R}
\newcommand{\X}{\mathbb X}
\newcommand{\Xcal}{\mathcal X}

\newcommand{\Rd}{\mathbb R^d}
\newcommand{\Z}{\mathbb Z}

\newcommand{\N}{\mathbb N}
\newcommand{\B}{\mathcal B}
\newcommand{\dd}{\mathrm{d}} 
\newcommand{\C}{\mathscr {C}}
\newcommand{\Complex}{\mathbb C}

\newcommand{\piv}[1]{\textsf {Piv}(#1)}
\newcommand{\Leb}{{\rm Leb}}
\newcommand{\xbar}{\overline{x}}
\newcommand{\ybar}{\overline{y}}
\newcommand{\zbar}{\overline{z}}
\newcommand{\ubar}{\overline{u}}
\newcommand{\wbar}{\overline{w}}

\newcommand{\sbar}{\overline{s}}
\newcommand{\tbar}{\overline{t}}
\newcommand{\vbar}{\overline{v}}
\newcommand{\zerobar}{\overline{0}}
\newcommand{\Hilbert}{\mathbb H}
\DeclareMathOperator*{\esssup}{ess\,sup}
\newcommand{\Fourier}[1]{\mathcal{F}\left( #1 \right)}
\newcommand{\Rmin}{R^{({\rm min})}_d}
\newcommand{\Rmax}{R^{({\rm max})}_d}
\newcommand{\Vmax}{V^{({\rm max})}}
\newcommand{\LandauBigO}[1]{\mathcal{O}\left(#1\right)}

\newcommand{\conn}[3]{#1 \longleftrightarrow #2\textrm { in } #3}

\newcommand{\nconn}[3]{#1 \centernot\longleftrightarrow #2\textrm { in } #3}
\newcommand{\dconn}[3]{#1 \Longleftrightarrow #2\textrm { in } #3}
\newcommand{\ndconn}[3]{#1 \centernot\Longleftrightarrow #2\textrm { in } #3}
\newcommand{\xconn}[4]{#1 \xleftrightarrow{\,\,#4\,\,} #2\textrm { in } #3}

\newcommand{\thinning}[2]{#1_{\langle #2 \rangle}}

\newcommand{\tlam}{\tau_\lambda}
\newcommand{\tlamo}{\tau_\lambda^\circ}

\newcommand{\tklam}{\tau_{\lambda,k}}

\newcommand{\ftlam}{\widehat\tau_\lambda}

\newcommand{\ftklam}{\widehat\tau_{\lambda,k}}

\newcommand{\trilam}{\triangle_\lambda}
\newcommand{\trilamo}{\triangle^\circ_\lambda}
\newcommand{\trilamoo}{\triangle^{\circ\circ}_\lambda}
\newcommand{\trilamB}{\triangle^{(B)}_\lambda}

\newcommand{\trilamooBar}{\overline{\triangle^{\circ\circ}_\lambda}}

\newcommand{\Wk}{W_k}
\newcommand{\WkBar}{\overline{W_k}}
\newcommand{\HkBar}{\overline{H_k}}

\newcommand{\Ulam}{U_\lambda}
\newcommand{\Vlam}{V_\lambda}

\newcommand{\fgmu}{\widehat G_{\mulam}}
\newcommand{\dispfgmu}{\widehat{\mathcal{G}}_{\mulam}}

\newcommand{\Id}{\mathds 1}
\newcommand{\Optlam}{\mathcal{T}_\lambda}
\newcommand{\OptlamT}{\mathcal{T}_{\lambda_O}}
\newcommand{\fOptlam}{\widehat{\mathcal{T}}_\lambda}
\newcommand{\fOptlamT}{\widehat{\mathcal{T}}_{\lambda_O}}
\newcommand{\Optklam}{\mathcal{T}_{\lambda,k}}
\newcommand{\fOptklam}{\widehat{\mathcal{T}}_{\lambda,k}}
\newcommand{\Opconnf}{\varPhi}
\newcommand{\fOpconnf}{\widehat{\Opconnf}}
\newcommand{\OpLace}{\varPi}
\newcommand{\OpLacelam}{\varPi_\lambda}
\newcommand{\OpLacelamT}{\varPi_{\lambda_O}}
\newcommand{\fOpLace}{\widehat{\OpLace}}
\newcommand{\fOpLacelam}{\fOpLace_{\lambda}}
\newcommand{\fOpLacelamT}{\fOpLace_{\lambda_O}}

\newcommand{\EndBlock}{\overline{\psi}}

\newcommand{\OpEndBlock}{\overline{\Psi}}

\newcommand{\OpNorm}[1]{\left\lVert#1\right\rVert_{\rm op}}
\newcommand{\OneNorm}[1]{\left\lVert#1\right\rVert_{1,\infty}}
\newcommand{\TwoNorm}[1]{\left\lVert#1\right\rVert_{2,\infty}}
\newcommand{\InfNorm}[1]{\left\lVert#1\right\rVert_{\infty,\infty}}

\DeclarePairedDelimiter\abs{\lvert}{\rvert}
\DeclarePairedDelimiter\norm{\lVert}{\rVert}

\DeclarePairedDelimiterX{\inner}[2]{\langle}{\rangle}{#1, #2}
\newcommand{\MinModulus}[1]{\left\lVert#1\right\rVert_{\rm min}}
\newcommand{\SupSpec}[1]{\mathbb{S}\left(#1\right)}

\newcommand{\EssIm}[1]{{\rm ess.Im}\left(#1\right)}
\newcommand{\Ess}{{\rm ess.}}
\DeclareMathOperator*{\essinf}{ess\,inf}

\newcommand{\mulam}{\mu_\lambda}
\newcommand{\orig}{\mathbf{0}}
\newcommand{\e}{\text{e}}

\newcommand{\connf}{\varphi}

\newcommand{\fconnf}{\widehat\connf}

\newcommand{\Bb}{\mathbb B}

\newcommand{\const}{c_f}
\newcommand{\constKappaOne}{c_{\kappa+1}}

\definecolor{darkorange}{RGB}{255,165,0}
\definecolor{altviolet}{RGB}{139,0,139}
\definecolor{turquoise}{RGB}{64,224,208}
\definecolor{lblue}{RGB}{173,216,230}
\definecolor{violet}{RGB}{238,130,238}
\definecolor{darkgreen}{RGB}{0,100,0}
\definecolor{lgreen}{RGB}{144,238,144}

\tikzset{cross/.style={cross out, draw=black, minimum size=2*(#1-\pgflinewidth), inner sep=0pt, outer sep=0pt},
cross/.default={1pt}}

\newcommand{\IntegralDot}{\raisebox{1pt}{\tikz{\filldraw (0,0) circle (2pt)}}}

\newcommand{\SupremumDot}{\raisebox{1pt}{\tikz{\draw (0,0) circle (2pt)}}}

\newcommand{\MarkSupremumDot}{\raisebox{1pt}{\tikz{\node[mark size=2pt] at (0,0) {\pgfuseplotmark{square*}}}}}

\newcommand{\connfline}{\raisebox{1pt}{\tikz{\draw (0,0) -- (1,0); \draw (0.5,0) circle (0pt) node[above]{\scriptsize$\sim$};}}}

\newcommand{\tlamline}{\raisebox{2pt}{\tikz{\draw (0,0) -- (1,0);}}}

\newcommand{\supline}{\raisebox{2pt}{\tikz{\draw[dashed] (0,0) -- (1,0); \filldraw[fill=white] (0,0) circle (2pt)}}}

\newcommand{\tlamoline}{\raisebox{1pt}{\tikz{\draw (0,0) -- (1,0); \draw (0.5,0) circle (0pt) node[above]{\scriptsize$\circ$};}}}

\newcommand{\tlamBline}{\raisebox{1pt}{\tikz{\draw[<->] (0,0) -- (1,0); \draw (0.5,0) circle (0pt) node[above]{\scriptsize$\in B$};}}}

\newcommand{\NOTtlamBline}{\raisebox{1pt}{\tikz{\draw[<->] (0,0) -- (1,0); \draw (0.5,0) circle (0pt) node[above]{\scriptsize$\not\in B$};}}}

\newcommand{\displaceline}{\raisebox{1pt}{\tikz{\draw[<->] (0,0) -- (1,0); \draw (0.5,0) node[cross=3pt]{};}}}

\newcommand{\psioneintegral}{\raisebox{-15pt}{
    \begin{tikzpicture}
        \filldraw (1,0.6) circle (2pt);
        \filldraw (1,-0.6) circle (2pt);
        \filldraw (2,0.6) circle (2pt);
        \filldraw (2,-0.6) circle (2pt);
        \draw (0,0.6) -- (1,0.6) -- (1,-0.6) -- (0,-0.6);
        \draw (0.5,0.6) circle (0pt) node[above]{$\circ$};
        \draw (1,0.6) -- (2,0.6) -- (2,-0.6) -- (1,-0.6);
        \filldraw[fill=white] (0,0.6) circle (2pt);
        \filldraw[fill=white] (0,-0.6) circle (2pt);
    \end{tikzpicture}
}}

\newcommand{\DiagramTriangleO}{\raisebox{-15pt}{
    \begin{tikzpicture}
        \filldraw (1,0.6) circle (2pt);
        \filldraw (1,-0.6) circle (2pt);
        \draw (0,0.6) -- (1,0.6) -- (1,-0.6) -- (0,-0.6);
        \draw (0.5,0.6) circle (0pt) node[above]{$\circ$};
        \filldraw[fill=white] (0,0.6) circle (2pt);
        \filldraw[fill=white] (0,-0.6) circle (2pt);
    \end{tikzpicture}
}}

\newcommand{\DiagramTriangle}{\raisebox{-15pt}{
    \begin{tikzpicture}
        \filldraw (1,0.6) circle (2pt);
        \filldraw (1,-0.6) circle (2pt);
        \draw (0,0.6) -- (1,0.6) -- (1,-0.6) -- (0,-0.6);
        \filldraw[fill=white] (0,0.6) circle (2pt);
        \filldraw[fill=white] (0,-0.6) circle (2pt);
    \end{tikzpicture}
}}

\newcommand{\DiagramTriangleClosed}{\raisebox{-15pt}{
    \begin{tikzpicture}
        \filldraw (1,0.6) circle (2pt);
        \filldraw (1,-0.6) circle (2pt);
        \draw (0,0) -- (1,0.6) -- (1,-0.6) -- (0,0);
        \filldraw[fill=white] (0,0) circle (2pt);
    \end{tikzpicture}
}}

\newcommand{\psitwointegralPtOne}{\raisebox{-15pt}{
    \begin{tikzpicture}
        \filldraw (1.5,0.6) circle (2pt);
        \filldraw (1,-0.6) circle (2pt);
        \filldraw (1.5,-0.1) circle (2pt);
        \filldraw (2,-0.6) circle (2pt);
        \draw (0,0.6) -- (1.5,0.6) -- (1.5,-0.1) -- (1,-0.6);
        \draw (0,-0.6) -- (1,-0.6) -- (2,-0.6) -- (1.5,-0.1);
        \draw (0.75,0.6) circle (0pt) node[above]{$\circ$};
        \draw (1.5,0.25) circle (0pt) node[left]{$\circ$};
        \filldraw[fill=white] (0,0.6) circle (2pt);
        \filldraw[fill=white] (0,-0.6) circle (2pt);
    \end{tikzpicture}
}}

\newcommand{\psitwointegralPtOneTranslate}{\raisebox{-15pt}{
    \begin{tikzpicture}
        \node[mark size=2pt] at (0,-0.6) {\pgfuseplotmark{square*}};
        \draw (0,0.6) -- (1.5,0.6) -- (1.5,-0.1) -- (1,-0.6);
        \draw (0,-0.6) -- (1,-0.6) -- (2,-0.6) -- (1.5,-0.1);
        \draw[dashed] (0,0.6) -- (0,-0.6);
        \draw (0.75,0.6) circle (0pt) node[above]{$\circ$};
        \draw (1.5,0.25) circle (0pt) node[left]{$\circ$};
        \filldraw (1.5,-0.1) circle (2pt);
        \filldraw[fill=white] (2,-0.6) circle (2pt);
        \filldraw (1.5,0.6) circle (2pt);
        \filldraw (1,-0.6) circle (2pt);
        \filldraw[fill=white] (0,0.6) circle (2pt);
        \filldraw (0,-0.6) circle (2pt);
    \end{tikzpicture}
}}

\newcommand{\psitwointegralPtOneTranslatePtOne}{\raisebox{-15pt}{
    \begin{tikzpicture}
        \node[mark size=2pt] at (0,-0.6) {\pgfuseplotmark{square*}};
        \draw (0,0.6) -- (1.5,0.6) -- (1.5,-0.1);
        \draw (0,-0.6) -- (1,-0.6);
        \draw[dashed] (0,0.6) -- (0,-0.6);
        \draw (0.75,0.6) circle (0pt) node[above]{$\circ$};
        \draw (1.5,0.25) circle (0pt) node[left]{$\circ$};
        \filldraw[fill=white] (1.5,-0.1) circle (2pt);
        \filldraw (1.5,0.6) circle (2pt);
        \filldraw[fill=white] (1,-0.6) circle (2pt);
        \filldraw[fill=white] (0,0.6) circle (2pt);
        \filldraw (0,-0.6) circle (2pt);
    \end{tikzpicture}
}}

\newcommand{\psionenoughtintegral}{\raisebox{-15pt}{
    \begin{tikzpicture}
        \draw (0,0) -- (1,0.6) -- (1,-0.6) -- cycle;
        \filldraw[fill=white] (0,0) circle (2pt);
        \filldraw (1,0.6) circle (2pt);
        \filldraw (1,-0.6) circle (2pt);
    \end{tikzpicture}
}}

\newcommand{\psioneNintegral}{\raisebox{-15pt}{
    \begin{tikzpicture}
        \draw (0,0.6) -- (1,0.6) -- (1,-0.6);
        \draw (0.5,0.6) circle (0pt) node[above]{$\circ$};
        \draw (1,0.6) -- (2,0) -- (1,-0.6) -- (0,-0.6);
        \filldraw[fill=white] (0,0.6) circle (2pt);
        \filldraw[fill=white] (0,-0.6) circle (2pt);
        \filldraw (1,0.6) circle (2pt);
        \filldraw (1,-0.6) circle (2pt);
        \filldraw (2,0) circle (2pt);
    \end{tikzpicture}
}}

\newcommand{\psioneNintegralFull}{\raisebox{-15pt}{
    \begin{tikzpicture}
        \draw (0,0.6) -- (1,0.6) -- (1,-0.6);
        \draw (1,0.6) -- (2,0) -- (1,-0.6) -- (0,-0.6);
        \filldraw[fill=white] (0,0.6) circle (2pt);
        \filldraw[fill=white] (0,-0.6) circle (2pt);
        \filldraw (1,0.6) circle (2pt);
        \filldraw (1,-0.6) circle (2pt);
        \filldraw (2,0) circle (2pt);
    \end{tikzpicture}
}}

\newcommand{\psitwoNintegral}{\raisebox{-15pt}{
    \begin{tikzpicture}
        \draw (0,0.6) -- (1,0) -- (0,-0.6);
        \filldraw[fill=white] (0,0.6) circle (2pt);
        \filldraw[fill=white] (0,-0.6) circle (2pt);
        \filldraw (1,0) circle (2pt);
    \end{tikzpicture}
}}

\newcommand{\psitwointegralPtThree}{\raisebox{-15pt}{
    \begin{tikzpicture}
        \filldraw (1,0.6) circle (2pt);
        \filldraw (1,-0.6) circle (2pt);
        \draw (1,-0.6) -- (0,0.6) -- (1,0.6) -- (1,-0.6) -- (0,-0.6);
        \filldraw[fill=white] (0,0.6) circle (2pt);
        \filldraw[fill=white] (0,-0.6) circle (2pt);
    \end{tikzpicture}
}}

\newcommand{\psiZerothreeintegral}{\raisebox{0pt}{
    \begin{tikzpicture}
        \draw (0,0) -- (1,0);
        \draw (0.5,0) circle (0pt) node[above]{$\sim$};
        \filldraw[fill=white] (0,0) circle (2pt);
        \filldraw (1,0) circle (2pt);
    \end{tikzpicture}
}}

\newcommand{\PairThreeFive}{\raisebox{-15pt}{
    \begin{tikzpicture}
        \draw (0,0.6) -- (1,0) -- (0,-0.6) -- (0,0.6);
        \draw (0,0) circle (0pt) node[above, rotate=90]{$\sim$};
        \filldraw[fill=white] (0,0.6) circle (2pt);
        \filldraw (1,0) circle (2pt);
        \filldraw (0,-0.6) circle (2pt);
    \end{tikzpicture}
}}

\newcommand{\PairFiveFive}{\raisebox{-15pt}{
    \begin{tikzpicture}
        \draw (0,0.6) -- (1,0.6) -- (1,-0.6) -- (0,-0.6) -- (1,0.6);
        \filldraw[fill=white] (0,0.6) circle (2pt);
        \filldraw[fill=white] (0,-0.6) circle (2pt);
        \filldraw (1,0.6) circle (2pt);
        \filldraw (1,-0.6) circle (2pt);
    \end{tikzpicture}
}}

\newcommand{\ProblemFiveTwoPtOne}{\raisebox{-15pt}{
    \begin{tikzpicture}
        \draw (0,0.6) -- (1,0) -- (0,-0.6);
        \filldraw[fill=white] (0,0.6) circle (2pt);
        \filldraw[fill=white] (0,-0.6) circle (2pt);
        \filldraw (1,0) circle (2pt);
    \end{tikzpicture}
}}

\newcommand{\ProblemFiveTwo}{\raisebox{0pt}{
    \begin{tikzpicture}
        \draw (0,0) to [out=10,in=170] (1,0) to [out=190,in=350] (0,0);
        \filldraw[fill=white] (0,0) circle (2pt);
        \filldraw (1,0) circle (2pt);
    \end{tikzpicture}
}}

\newcommand{\CaseTwoFirst}{\raisebox{-15pt}{
    \begin{tikzpicture}
    \fill[fill=gray!50] (2,0.6) to [out=260,in=100] (2,-0.6) to [out=170,in=10] (0,-0.6) to [out=140,in=340] (-1,0) to [out=20,in=220] (0,0.6) to [out=350,in=190] (2,0.6);
    \draw (2,0.6) -- (3,0.6) -- (3,-0.6) -- (2,-0.6);
    \draw (3,0.6) -- (4,0) -- (3,-0.6);
    \draw[<->] (2,0.8) -- (3,0.8) -- (4,0.2);
    \draw (2.8,0.8) circle (0pt) node[rotate = 0]{$\times$};
    \draw (2.5,-0.6) circle (0pt) node[above]{$\circ$};
    \node at (1,0) {$\Psi^{n-1}\Psi_0$};
    \filldraw[fill=white] (-1,0) circle (2pt);
    \filldraw (3,0.6) circle (2pt);
    \filldraw (3,-0.6) circle (2pt);
    \filldraw (2,0.6) circle (2pt);
    \filldraw (2,-0.6) circle (2pt);
    \filldraw (4,0) circle (2pt);
    \end{tikzpicture}
}}

\newcommand{\CaseTwoFirstPtTwo}{\raisebox{-15pt}{
    \begin{tikzpicture}
    \fill[fill=gray!50] (2,0.6) to [out=260,in=100] (2,-0.6) to [out=170,in=10] (0,-0.6) to [out=140,in=340] (-1,0) to [out=20,in=220] (0,0.6) to [out=350,in=190] (2,0.6);
    \draw (2,0.6) -- (3,0) -- (2,-0.6);
    \draw[<->] (2,0.8) -- (3,0.2);
    \draw (2.5,0.5) circle (0pt) node[rotate = -30]{$\times$};
    \node at (1,0) {$\Psi^{n-1}\Psi_0$};
    \filldraw[fill=white] (-1,0) circle (2pt);
    \filldraw (2,0.6) circle (2pt);
    \filldraw (2,-0.6) circle (2pt);
    \filldraw (3,0) circle (2pt);
    \end{tikzpicture}
}}

\newcommand{\CaseTwoPtOne}{\raisebox{-15pt}{
    \begin{tikzpicture}
    \draw (0,0.6) -- (1,0.6) -- (1,-0.6) -- (0,-0.6) -- (1,0.6);
    \draw[<->] (0,0.8) -- (1,0.8);
    \draw (0.5,0.8) circle (0pt) node[rotate = 0]{$\times$};
    \filldraw[fill=white] (0,0.6) circle (2pt);
    \filldraw[fill=white] (0,-0.6) circle (2pt);
    \filldraw (1,0.6) circle (2pt);
    \filldraw (1,-0.6) circle (2pt);
    \end{tikzpicture}
}}

\newcommand{\CaseTwoPtTwo}{\raisebox{-15pt}{
    \begin{tikzpicture}
    \draw (2,0.6) -- (3,0.6) -- (3,-0.6) -- (2,-0.6);
    \draw (3,0.6) -- (4,0) -- (3,-0.6);
    \draw[<->] (2,0.8) -- (3,0.8);
    \draw (2.5,0.8) circle (0pt) node[rotate = 0]{$\times$};
    \filldraw (3,0.6) circle (2pt);
    \filldraw (3,-0.6) circle (2pt);
    \filldraw[fill=white] (2,0.6) circle (2pt);
    \filldraw[fill=white] (2,-0.6) circle (2pt);
    \filldraw (4,0) circle (2pt);
    \end{tikzpicture}
}}

\newcommand{\CaseTwoPtThree}{\raisebox{-15pt}{
    \begin{tikzpicture}
    \draw (2,0.6) -- (3,0.6) -- (3,-0.6) -- (2,-0.6);
    \draw (3,0.6) -- (4,0) -- (3,-0.6);
    \draw[<->] (3,0.8) -- (4,0.2);
    \draw (3.5,0.5) circle (0pt) node[rotate = -30]{$\times$};
    \draw (2.5,-0.6) circle (0pt) node[above]{$\circ$};
    \filldraw (3,0.6) circle (2pt);
    \filldraw (3,-0.6) circle (2pt);
    \filldraw[fill=white] (2,0.6) circle (2pt);
    \filldraw[fill=white] (2,-0.6) circle (2pt);
    \filldraw (4,0) circle (2pt);
    \end{tikzpicture}
}}

\newcommand{\CaseOneFirst}{\raisebox{-15pt}{
    \begin{tikzpicture}
    \draw (1,-0.6) -- (0,0) -- (1,0.6);
    \draw (2,0.6) -- (1,0.6) -- (1,-0.6) -- (2,-0.6);
    \draw[<->] (0,0.2) -- (1,0.8);
    \draw (0.5,0.5) circle (0pt) node[rotate = 30]{$\times$};
    \draw (1.5,-0.6) circle (0pt) node[above]{$\circ$};
    \fill[fill=gray!50] (2,0.6) to [out=280,in=80] (2,-0.6) to [out=10,in=170] (4,-0.6) to [out=40,in=200] (5,0) to [out=160,in=320] (4,0.6) to [out=190,in=350] (2,0.6);
    \node at (3,0) {$\OpEndBlock^{n-1}\OpEndBlock_0$};
    \filldraw[fill=white] (0,0) circle (2pt);
    \filldraw (1,0.6) circle (2pt);
    \filldraw (1,-0.6) circle (2pt);
    \filldraw (2,0.6) circle (2pt);
    \filldraw (2,-0.6) circle (2pt);
    \filldraw (5,0) circle (2pt);
    \end{tikzpicture}
}}

\newcommand{\CaseOneFirstShifted}{\raisebox{-15pt}{
    \begin{tikzpicture}
    \draw (1,-0.6) -- (0,0) -- (1,0.6);
    \draw (2,0.6) -- (1,0.6) -- (1,-0.6) -- (2,-0.6);
    \draw[<->] (0,0.2) -- (1,0.8);
    \draw (0.5,0.5) circle (0pt) node[rotate = 30]{$\times$};
    \draw (1.5,-0.6) circle (0pt) node[above]{$\circ$};
    \fill[fill=gray!50] (2,0.6) to [out=280,in=80] (2,-0.6) to [out=10,in=170] (4,-0.6) to [out=40,in=200] (5,0) to [out=160,in=320] (4,0.6) to [out=190,in=350] (2,0.6);
    \node at (3,0) {$\OpEndBlock^{n-1}\OpEndBlock_0$};
    \node[mark size=2pt] at (0,0) {\pgfuseplotmark{square*}};
    \filldraw (1,0.6) circle (2pt);
    \filldraw (1,-0.6) circle (2pt);
    \filldraw (2,0.6) circle (2pt);
    \filldraw (2,-0.6) circle (2pt);
    \filldraw[fill=white] (5,0) circle (2pt);
    \end{tikzpicture}
}}

\newcommand{\CaseOneFirstShiftedPtOne}{\raisebox{-15pt}{
    \begin{tikzpicture}
    \draw (1,-0.6) -- (0,0) -- (1,0.6);
    \draw (2,0.6) -- (1,0.6) -- (1,-0.6) -- (2,-0.6);
    \draw[<->] (0,0.2) -- (1,0.8);
    \draw (0.5,0.5) circle (0pt) node[rotate = 30]{$\times$};
    \draw (1.5,-0.6) circle (0pt) node[above]{$\circ$};
    \fill[fill=gray!50] (2,0.6) to [out=280,in=80] (2,-0.6) to [out=10,in=170] (3,-0.6) to [out=100,in=260] (3,0.6) to [out=190,in=350] (2,0.6);
    \node at (2.5,0) {$\OpEndBlock$};
    \node[mark size=2pt] at (0,0) {\pgfuseplotmark{square*}};
    \filldraw (1,0.6) circle (2pt);
    \filldraw (1,-0.6) circle (2pt);
    \filldraw (2,0.6) circle (2pt);
    \filldraw (2,-0.6) circle (2pt);
    \filldraw[fill=white] (3,0.6) circle (2pt);
    \filldraw[fill=white] (3,-0.6) circle (2pt);
    \end{tikzpicture}
}}

\newcommand{\CaseOneFirstShiftedPtTwo}{\raisebox{-15pt}{
    \begin{tikzpicture}
    \fill[fill=gray!50] (2,0.6) to [out=280,in=80] (2,-0.6) to [out=10,in=170] (4,-0.6) to [out=40,in=200] (5,0) to [out=160,in=320] (4,0.6) to [out=190,in=350] (2,0.6);
    \node at (3,0) {$\OpEndBlock^{n-2}\OpEndBlock_0$};
    \filldraw (2,0.6) circle (2pt);
    \filldraw (2,-0.6) circle (2pt);
    \filldraw[fill=white] (5,0) circle (2pt);
    \end{tikzpicture}
}}

\newcommand{\CaseOneFive}{\raisebox{-15pt}{
    \begin{tikzpicture}
    \draw (1,-0.6) -- (0,0) -- (1,0.6);
    \draw (2,0.6) -- (1,0.6) -- (1,-0.6) -- (2,-0.6);
    \draw (2,0.6) -- (2,-0.6) -- (3,-0.6);
    \draw[<->] (0,0.2) -- (1,0.8);
    \draw (0.5,0.5) circle (0pt) node[rotate = 30]{$\times$};
    \draw (1.5,-0.6) circle (0pt) node[above]{$\circ$};
    \node[mark size=2pt] at (0,0) {\pgfuseplotmark{square*}};
    \filldraw (1,0.6) circle (2pt);
    \filldraw (1,-0.6) circle (2pt);
    \filldraw[fill=white] (2,0.6) circle (2pt);
    \filldraw (2,-0.6) circle (2pt);
    \filldraw[fill=white] (3,-0.6) circle (2pt);
    \end{tikzpicture}
}}

\newcommand{\CaseOneFivePtOne}{\raisebox{-15pt}{
    \begin{tikzpicture}
    \draw (1,-0.6) -- (0,0) -- (1,0.6);
    \draw (2,0.6) -- (1,0.6) -- (1,-0.6) -- (2,-0.6);
    \draw (2,0.6) -- (2,-0.6) -- (3,-0.6);
    \draw[<->] (0,0.2) -- (1,0.8);
    \draw (0.5,0.5) circle (0pt) node[rotate = 30]{$\times$};
    \node[mark size=2pt] at (0,0) {\pgfuseplotmark{square*}};
    \filldraw (1,0.6) circle (2pt);
    \filldraw (1,-0.6) circle (2pt);
    \filldraw[fill=white] (2,0.6) circle (2pt);
    \filldraw (2,-0.6) circle (2pt);
    \filldraw[fill=white] (3,-0.6) circle (2pt);
    \end{tikzpicture}
}}

\newcommand{\CaseOneFivePtTwo}{\raisebox{-15pt}{
    \begin{tikzpicture}
    \draw (1,-0.6) -- (0,0) -- (1,0.6);
    \draw (2,0.6) -- (1,0.6) -- (1,-0.6) -- (2,-0.6);
    \draw (2,0.6) -- (1,-0.6) -- (2,-0.6);
    \draw[<->] (0,0.2) -- (1,0.8);
    \draw (0.5,0.5) circle (0pt) node[rotate = 30]{$\times$};
    \node[mark size=2pt] at (0,0) {\pgfuseplotmark{square*}};
    \filldraw (1,0.6) circle (2pt);
    \filldraw (1,-0.6) circle (2pt);
    \filldraw[fill=white] (2,0.6) circle (2pt);
    \filldraw[fill=white] (2,-0.6) circle (2pt);
    \end{tikzpicture}
}}

\newcommand{\DisplacedFive}{\raisebox{-15pt}{
    \begin{tikzpicture}
        \draw (0,0.6) -- (1,0.6) -- (1,-0.6);
        \filldraw[fill=white] (0,0.6) circle (2pt);
        \draw[<->] (0,0.8) -- (1,0.8);
        \draw (0.5,0.8) circle (0pt) node[rotate = 0]{$\times$};
        \filldraw (1,0.6) circle (2pt);
        \filldraw[fill=white] (1,-0.6) circle (2pt);
    \end{tikzpicture}
}}

\newcommand{\DisplacedTwoPtFivePtOne}{\raisebox{-25pt}{
    \begin{tikzpicture}
        \draw (0,0.6) -- (1.5,0.6) -- (1.5,-0.1) -- (1,-0.6);
        \draw (0,-0.6) -- (1,-0.6) -- (2,-0.6) -- (1.5,-0.1);
        \draw (2,-0.6) -- (3,-0.6);
        \draw (1.5,0.6) -- (3,0.6);
        \draw[dashed] (3,0.6) -- (3,-0.6);
        \draw[<->] (1,-0.8) -- (2,-0.8);
        \draw (1.5,-0.8) circle (0pt) node{$\times$};
        \node[mark size=2pt] at (3,-0.6) {\pgfuseplotmark{square*}};
        \filldraw[fill=white] (0,0.6) circle (2pt);
        \filldraw[fill=white] (0,-0.6) circle (2pt);
        \filldraw[fill=white] (3,0.6) circle (2pt);
        \filldraw (1.5,0.6) circle (2pt);
        \filldraw (1,-0.6) circle (2pt);
        \filldraw (1.5,-0.1) circle (2pt);
        \filldraw (2,-0.6) circle (2pt);
    \end{tikzpicture}
}}

\newcommand{\DisplacedTwoPtFivePtTwo}{\raisebox{-25pt}{
    \begin{tikzpicture}
        \draw (0,0.6) -- (1.5,0.6) -- (1.5,-0.1) -- (1,-0.6);
        \draw (0,-0.6) -- (1,-0.6) -- (2,-0.6) -- (1.5,-0.1);
        \draw (1.5,0.6) -- (3,0.6);
        \draw[dashed] (0,0.6) -- (0,-0.6);
        \draw[<->] (1,-0.8) -- (2,-0.8);
        \draw (1.5,-0.8) circle (0pt) node{$\times$};
        \node[mark size=2pt] at (0,-0.6) {\pgfuseplotmark{square*}};
        \filldraw[fill=white] (0,0.6) circle (2pt);
        \filldraw[fill=white] (3,0.6) circle (2pt);
        \filldraw (1.5,0.6) circle (2pt);
        \filldraw (1,-0.6) circle (2pt);
        \filldraw (1.5,-0.1) circle (2pt);
        \filldraw[fill=white] (2,-0.6) circle (2pt);
    \end{tikzpicture}
}}

\newcommand{\DisplacedTwoPtTwoPlusFive}{\raisebox{-25pt}{
    \begin{tikzpicture}
        \draw (1,0.6) -- (1,-0.6);
        \draw (0,-0.6) -- (1,-0.6) -- (2,-0.6) -- (1,0.6);
        \draw (1,0.6) -- (2,0.6) -- (2,-0.6);
        \draw[<->] (0,-0.8) -- (1,-0.8);
        \draw (0.5,-0.8) circle (0pt) node[rotate = 0]{$\times$};
        \filldraw (1,-0.6) circle (2pt);
        \filldraw (2,-0.6) circle (2pt);
        \filldraw (2,0.6) circle (2pt);
        \filldraw[fill=white] (0,-0.6) circle (2pt);
        \filldraw[fill=white] (1,0.6) circle (2pt);
    \end{tikzpicture}
}}

\newcommand{\DisplacedTwoPtFourPlusFive}{\raisebox{-25pt}{
    \begin{tikzpicture}
        \draw (1,0.6) -- (1,-0.6);
        \draw (0,-0.6) -- (1,-0.6) -- (2,-0.6) -- (1,0.6);
        \draw (1,0.6) -- (2,0.6) -- (2,-0.6);
        \draw[<->] (1,-0.8) -- (2,-0.8);
        \draw (1.5,-0.8) circle (0pt) node[rotate = 0]{$\times$};
        \filldraw (1,-0.6) circle (2pt);
        \filldraw (2,-0.6) circle (2pt);
        \filldraw (2,0.6) circle (2pt);
        \filldraw[fill=white] (0,-0.6) circle (2pt);
        \filldraw[fill=white] (1,0.6) circle (2pt);
    \end{tikzpicture}
}}

\newcommand{\ExplanationCaseOnePtOne}{\raisebox{-15pt}{
    \begin{tikzpicture}
        \fill[fill=gray!50] (2,0.6) to [out=260,in=100] (2,-0.6) to [out=170,in=10] (0,-0.6) to [out=140,in=340] (-1,0) to [out=20,in=220] (0,0.6) to [out=350,in=190] (2,0.6);
        \fill[fill=gray!50] (2,0.6) to [out=350,in=190] (3,0.6) to [out=260,in=100] (3,-0.6) to [out=170,in=10] (2,-0.6) to [out=80,in=280] (2,0.6);
        \fill[fill=gray!50] (3,0.6) to [out=350,in=190] (4,0.6) to [out=260,in=100] (4,-0.6) to [out=170,in=10] (3,-0.6) to [out=80,in=280] (3,0.6);
        \fill[fill=gray!50] (4,0.6) to [out=280,in=80] (4,-0.6) to [out=10,in=170] (6,-0.6) to [out=40,in=200] (7,0) to [out=160,in=320] (6,0.6) to [out=190,in=350] (4,0.6);
        \draw[<->] (2,0.8) -- (3,0.8);
        \draw (2.5,0.8) circle (0pt) node[rotate = 0]{$\times$};
        \node at (1,0) {$\Psi^{i-1}\Psi_0$};
        \node at (2.5,0) {$\Psi^{(j)}$};
        \node at (3.5,0) {$\Psi$};
        \node at (5,0) {$\Psi_n\Psi^{n-i-2}$};
        \filldraw[fill=white] (-1,0) circle (2pt);
        \filldraw (2,0.6) circle (2pt);
        \filldraw (2,-0.6) circle (2pt);
        \filldraw (3,0.6) circle (2pt);
        \filldraw (3,-0.6) circle (2pt);
        \filldraw (4,0.6) circle (2pt);
        \filldraw (4,-0.6) circle (2pt);
        \filldraw (7,0) circle (2pt);
    \end{tikzpicture}
}}

\newcommand{\ExplanationCaseOnePtOnePtOne}{\raisebox{-15pt}{
    \begin{tikzpicture}
        \fill[fill=gray!50] (2,0.6) to [out=350,in=190] (3,0.6) to [out=260,in=100] (3,-0.6) to [out=170,in=10] (2,-0.6) to [out=80,in=280] (2,0.6);
        \fill[fill=gray!50] (3,0.6) to [out=350,in=190] (4,0.6) to [out=260,in=100] (4,-0.6) to [out=170,in=10] (3,-0.6) to [out=80,in=280] (3,0.6);
        \draw[<->] (2,0.8) -- (3,0.8);
        \draw (2.5,0.8) circle (0pt) node[rotate = 0]{$\times$};
        \node at (2.5,0) {$\Psi^{(j)}$};
        \node at (3.5,0) {$\Psi$};
        \filldraw[fill=white] (2,0.6) circle (2pt);
        \filldraw[fill=white] (2,-0.6) circle (2pt);
        \filldraw (3,0.6) circle (2pt);
        \filldraw (3,-0.6) circle (2pt);
        \filldraw (4,0.6) circle (2pt);
        \filldraw (4,-0.6) circle (2pt);
    \end{tikzpicture}
}}

\newcommand{\ExplanationCaseOnePtTwo}{\raisebox{-15pt}{
    \begin{tikzpicture}
        \fill[fill=gray!50] (2,0.6) to [out=260,in=100] (2,-0.6) to [out=170,in=10] (0,-0.6) to [out=140,in=340] (-1,0) to [out=20,in=220] (0,0.6) to [out=350,in=190] (2,0.6);
        \fill[fill=gray!50] (2,0.6) to [out=350,in=190] (3,0.6) to [out=260,in=100] (3,-0.6) to [out=170,in=10] (2,-0.6) to [out=80,in=280] (2,0.6);
        \fill[fill=gray!50] (3,0.6) to [out=280,in=80] (3,-0.6) to [out=10,in=170] (4,-0.6) to [out=40,in=200] (5,0) to [out=160,in=320] (4,0.6) to [out=190,in=350] (3,0.6);
        \draw[<->] (2,0.8) -- (3,0.8);
        \draw (2.5,0.8) circle (0pt) node[rotate = 0]{$\times$};
        \node at (1,0) {$\Psi^{n-2}\Psi_0$};
        \node at (2.5,0) {$\Psi^{(j)}$};
        \node at (3.5,0) {$\Psi_n$};
        \filldraw[fill=white] (-1,0) circle (2pt);
        \filldraw (2,0.6) circle (2pt);
        \filldraw (2,-0.6) circle (2pt);
        \filldraw (3,0.6) circle (2pt);
        \filldraw (3,-0.6) circle (2pt);
        \filldraw (5,0) circle (2pt);
    \end{tikzpicture}
}}

\newcommand{\ExplanationCaseOnePtTwoPtOne}{\raisebox{-15pt}{
    \begin{tikzpicture}
        \fill[fill=gray!50] (2,0.6) to [out=350,in=190] (3,0.6) to [out=260,in=100] (3,-0.6) to [out=170,in=10] (2,-0.6) to [out=80,in=280] (2,0.6);
        \fill[fill=gray!50] (3,0.6) to [out=280,in=80] (3,-0.6) to [out=10,in=170] (4,-0.6) to [out=40,in=200] (5,0) to [out=160,in=320] (4,0.6) to [out=190,in=350] (3,0.6);
        \draw[<->] (2,0.8) -- (3,0.8);
        \draw (2.5,0.8) circle (0pt) node[rotate = 0]{$\times$};
        \node at (2.5,0) {$\Psi^{(j)}$};
        \node at (3.5,0) {$\Psi_n$};
        \filldraw[fill=white] (2,0.6) circle (2pt);
        \filldraw[fill=white] (2,-0.6) circle (2pt);
        \filldraw (3,0.6) circle (2pt);
        \filldraw (3,-0.6) circle (2pt);
        \filldraw (5,0) circle (2pt);
    \end{tikzpicture}
}}

\newcommand{\ExplanationCaseTwoPtOne}{\raisebox{-15pt}{
    \begin{tikzpicture}
        \fill[fill=gray!50] (2,0.6) to [out=260,in=100] (2,-0.6) to [out=170,in=10] (0,-0.6) to [out=140,in=340] (-1,0) to [out=20,in=220] (0,0.6) to [out=350,in=190] (2,0.6);
        \fill[fill=gray!50] (2,0.6) to [out=350,in=190] (3,0.6) to [out=260,in=100] (3,-0.6) to [out=170,in=10] (2,-0.6) to [out=80,in=280] (2,0.6);
        \fill[fill=gray!50] (4,0.6) to [out=350,in=190] (5,0.6) to [out=260,in=100] (5,-0.6) to [out=170,in=10] (4,-0.6) to [out=80,in=280] (4,0.6);
        \fill[fill=gray!50] (5,0.6) to [out=280,in=80] (5,-0.6) to [out=10,in=170] (7,-0.6) to [out=40,in=200] (8,0) to [out=160,in=320] (7,0.6) to [out=190,in=350] (5,0.6);
        \draw (3,0.6) -- (4,0.6);
        \draw (3,-0.6) -- (4,-0.6);
        \draw (3.5,-0.6) circle (0pt) node[above]{$\circ$};
        \draw[<->] (2,0.8) -- (3,0.8);
        \draw (2.5,0.8) circle (0pt) node[rotate = 0]{$\times$};
        \node at (1,0) {$\Psi^{i-1}\Psi_0$};
        \node at (2.5,0) {$\Psi^{(j)}$};
        \node at (4.5,0) {$\OpEndBlock$};
        \node at (6,0) {$\OpEndBlock^{n-i-2}\OpEndBlock_0$};
        \filldraw[fill=white] (-1,0) circle (2pt);
        \filldraw (2,0.6) circle (2pt);
        \filldraw (2,-0.6) circle (2pt);
        \filldraw (3,0.6) circle (2pt);
        \filldraw (3,-0.6) circle (2pt);
        \filldraw (4,0.6) circle (2pt);
        \filldraw (4,-0.6) circle (2pt);
        \filldraw (5,0.6) circle (2pt);
        \filldraw (5,-0.6) circle (2pt);
        \filldraw (8,0) circle (2pt);
    \end{tikzpicture}
}}

\newcommand{\ExplanationCaseTwoPtOnePtOne}{\raisebox{-15pt}{
    \begin{tikzpicture}
        \fill[fill=gray!50] (2,0.6) to [out=350,in=190] (3,0.6) to [out=260,in=100] (3,-0.6) to [out=170,in=10] (2,-0.6) to [out=70,in=290] (2,0.6);
        \fill[fill=gray!50] (4,0.6) to [out=350,in=190] (5,0.6) to [out=260,in=100] (5,-0.6) to [out=170,in=10] (4,-0.6) to [out=80,in=280] (4,0.6);
        \draw (3,0.6) -- (4,0.6);
        \draw (3,-0.6) -- (4,-0.6);
        \draw (3.5,-0.6) circle (0pt) node[above]{$\circ$};
        \draw[<->] (2,0.8) -- (3,0.8);
        \draw (2.5,0.8) circle (0pt) node[rotate = 0]{$\times$};
        \node at (2.5,0) {$\Psi^{(j)}$};
        \node at (4.5,0) {$\OpEndBlock$};
        \draw[dashed] (2,0.6) -- (2,-0.6);
        \filldraw[fill=white] (2,0.6) circle (2pt);
        \node[mark size=2pt] at (2,-0.6) {\pgfuseplotmark{square*}};
        \filldraw (3,0.6) circle (2pt);
        \filldraw (3,-0.6) circle (2pt);
        \filldraw (4,0.6) circle (2pt);
        \filldraw (4,-0.6) circle (2pt);
        \filldraw[fill=white] (5,0.6) circle (2pt);
        \filldraw[fill=white] (5,-0.6) circle (2pt);
    \end{tikzpicture}
}}

\newcommand{\ExplanationCaseTwoPtTwo}{\raisebox{-15pt}{
    \begin{tikzpicture}
        \fill[fill=gray!50] (2,0.6) to [out=260,in=100] (2,-0.6) to [out=170,in=10] (0,-0.6) to [out=140,in=340] (-1,0) to [out=20,in=220] (0,0.6) to [out=350,in=190] (2,0.6);
        \fill[fill=gray!50] (2,0.6) to [out=350,in=190] (3,0.6) to [out=260,in=100] (3,-0.6) to [out=170,in=10] (2,-0.6) to [out=80,in=280] (2,0.6);
        \fill[fill=gray!50] (4,0.6) to [out=280,in=80] (4,-0.6) to [out=10,in=170] (5,-0.6) to [out=40,in=200] (6,0) to [out=160,in=320] (5,0.6) to [out=190,in=350] (4,0.6);
        \draw (3,0.6) -- (4,0.6);
        \draw (3,-0.6) -- (4,-0.6);
        \draw (3.5,-0.6) circle (0pt) node[above]{$\circ$};
        \draw[<->] (2,0.8) -- (3,0.8);
        \draw (2.5,0.8) circle (0pt) node[rotate = 0]{$\times$};
        \node at (1,0) {$\Psi^{n-2}\Psi_0$};
        \node at (2.5,0) {$\Psi^{(j)}$};
        \node at (4.5,0) {$\OpEndBlock_0$};
        \filldraw[fill=white] (-1,0) circle (2pt);
        \filldraw (2,0.6) circle (2pt);
        \filldraw (2,-0.6) circle (2pt);
        \filldraw (3,0.6) circle (2pt);
        \filldraw (3,-0.6) circle (2pt);
        \filldraw (4,0.6) circle (2pt);
        \filldraw (4,-0.6) circle (2pt);
        \filldraw (6,0) circle (2pt);
    \end{tikzpicture}
}}

\newcommand{\ExplanationCaseTwoPtTwoPtOne}{\raisebox{-15pt}{
    \begin{tikzpicture}
        \fill[fill=gray!50] (2,0.6) to [out=350,in=190] (3,0.6) to [out=260,in=100] (3,-0.6) to [out=170,in=10] (2,-0.6) to [out=70,in=290] (2,0.6);
        \fill[fill=gray!50] (4,0.6) to [out=280,in=80] (4,-0.6) to [out=10,in=170] (5,-0.6) to [out=40,in=200] (6,0) to [out=160,in=320] (5,0.6) to [out=190,in=350] (4,0.6);
        \draw (3,0.6) -- (4,0.6);
        \draw (3,-0.6) -- (4,-0.6);
        \draw (3.5,-0.6) circle (0pt) node[above]{$\circ$};
        \draw[<->] (2,0.8) -- (3,0.8);
        \draw (2.5,0.8) circle (0pt) node[rotate = 0]{$\times$};
        \node at (2.5,0) {$\Psi^{(j)}$};
        \node at (4.5,0) {$\OpEndBlock_0$};
        \draw[dashed] (2,0.6) -- (2,-0.6);
        \filldraw[fill=white] (2,0.6) circle (2pt);
        \node[mark size=2pt] at (2,-0.6) {\pgfuseplotmark{square*}};
        \filldraw (2,-0.6) circle (2pt);
        \filldraw (3,0.6) circle (2pt);
        \filldraw (3,-0.6) circle (2pt);
        \filldraw (4,0.6) circle (2pt);
        \filldraw (4,-0.6) circle (2pt);
        \filldraw[fill=white] (6,0) circle (2pt);
    \end{tikzpicture}
}}

\newcommand{\ExtraThreeFiveTwo}{\raisebox{-15pt}{
    \begin{tikzpicture}
        \draw (1,0.6) -- (1,-0.6);
        \draw (1,0.6) -- (2,0.6) --(1,-0.6) -- (2,-0.6) -- (2,0.6);
        \draw[<->] (1,0.8) -- (2,0.8);
        \draw (1.5,0.8) circle (0pt) node{$\times$};
        \draw (1,0) circle (0pt) node[rotate=90, above]{$\sim$};
        \filldraw (0,0) circle (0pt);
        \filldraw[fill=white] (1,0.6) circle (2pt);
        \filldraw (1,-0.6) circle (2pt);
        \filldraw (2,0.6) circle (2pt);
        \filldraw (2,-0.6) circle (2pt);
    \end{tikzpicture}
}}

\newcommand{\ExtraThreeFiveTwoPtTwo}{\raisebox{-15pt}{
    \begin{tikzpicture}
        \draw (1,0.6) -- (1,-0.6);
        \draw (1,0.6) -- (2,0.6) --(1,-0.6) -- (2,-0.6) -- (2,0.6);
        \draw[<->] (1,0.8) -- (2,0.8);
        \draw (1.5,0.8) circle (0pt) node{$\times$};
        \node[mark size=2pt] at (1,0.6) {\pgfuseplotmark{square*}};
        \filldraw (1,-0.6) circle (2pt);
        \filldraw (2,0.6) circle (2pt);
        \filldraw[fill=white] (2,-0.6) circle (2pt);
    \end{tikzpicture}
}}

\newcommand{\ExtraOneOne}{\raisebox{-25pt}{
    \begin{tikzpicture}
        \draw (0,0) -- (1,0.6) -- (1,-0.6) -- (0,0);
        \draw (1,0.6) -- (2,0.6) --(2,-0.6) -- (1,-0.6);
        \draw (2,0.6) -- (3,0) -- (2,-0.6);
        \draw[<->] (0,-0.2) -- (1,-0.8) -- (2,-0.8) -- (3,-0.2);
        \draw (1.5,-0.8) circle (0pt) node{$\times$};
        \draw (1.5,-0.6) circle (0pt) node[above]{$\circ$};
        \filldraw[fill=white] (0,0) circle (2pt);
        \filldraw (1,0.6) circle (2pt);
        \filldraw (1,-0.6) circle (2pt);
        \filldraw (2,0.6) circle (2pt);
        \filldraw (2,-0.6) circle (2pt);
        \filldraw (3,0) circle (2pt);
    \end{tikzpicture}
}}

\newcommand{\ExtraOneTwo}{\raisebox{-15pt}{
    \begin{tikzpicture}
        \draw (0,0) -- (1,0.6) -- (1,-0.6) -- (0,0);
        \draw (1,0.6) -- (2,0) -- (1,-0.6);
        \draw[<->] (0,0.2) -- (1,0.8) -- (2,0.2);
        \draw (0.8,0.68) circle (0pt) node[rotate=30]{$\times$};
        \filldraw[fill=white] (0,0) circle (2pt);
        \filldraw (1,0.6) circle (2pt);
        \filldraw (1,-0.6) circle (2pt);
        \filldraw (2,0) circle (2pt);
    \end{tikzpicture}
}}

\newcommand{\ExtraTwoOne}{\raisebox{-15pt}{
    \begin{tikzpicture}
        \draw (0,0) -- (1,0.6) -- (1,-0.6) -- (0,0);
        \draw (1,0.6) -- (2,0.6) --(2,-0.6) -- (1,-0.6);
        \draw (2,0.6) -- (3,0) -- (2,-0.6);
        \draw[<->] (1,0.8) -- (2,0.8) -- (3,0.2);
        \draw (1.8,0.8) circle (0pt) node{$\times$};
        \draw (1.5,-0.6) circle (0pt) node[above]{$\circ$};
        \filldraw (0,0) circle (2pt);
        \filldraw[fill=white] (1,0.6) circle (2pt);
        \filldraw (1,-0.6) circle (2pt);
        \filldraw (2,0.6) circle (2pt);
        \filldraw (2,-0.6) circle (2pt);
        \filldraw (3,0) circle (2pt);
    \end{tikzpicture}
}}

\newcommand{\ExtraTwoTwo}{\raisebox{-15pt}{
    \begin{tikzpicture}
        \draw (0,0) -- (1,0.6) -- (1,-0.6) -- (0,0);
        \draw (1,0.6) -- (2,0) -- (1,-0.6);
        \draw[<->] (1,0.8) -- (2,0.2);
        \draw (1.5,0.5) circle (0pt) node[rotate=-30]{$\times$};
        \filldraw (0,0) circle (2pt);
        \filldraw[fill=white] (1,0.6) circle (2pt);
        \filldraw (1,-0.6) circle (2pt);
        \filldraw (2,0) circle (2pt);
    \end{tikzpicture}
}}

\newcommand{\ExtraThreeOne}{\raisebox{-15pt}{
    \begin{tikzpicture}
        \draw (0,0.6) -- (1,0.6) -- (2,0) -- (1,-0.6) -- (0,-0.6) -- cycle;
        \draw (1,0.6) -- (1,-0.6);
        \draw[<->] (0,0.8) -- (1,0.8) -- (2,0.2);
        \draw (0.8,0.8) circle (0pt) node{$\times$};
        \draw (0,0) circle (0pt) node[rotate=90,above]{$\sim$};
        \draw (0.5,-0.6) circle (0pt) node[above]{$\circ$};
        \filldraw[fill=white] (0,0.6) circle (2pt);
        \filldraw (1,0.6) circle (2pt);
        \filldraw (2,0) circle (2pt);
        \filldraw (0,-0.6) circle (2pt);
        \filldraw (1,-0.6) circle (2pt);
    \end{tikzpicture}
}}

\newcommand{\CaseThreeTwo}{\raisebox{-15pt}{
    \begin{tikzpicture}
        \draw (0,0.6) -- (1,0) -- (0,-0.6) -- cycle;
        \draw[<->] (0,0.8) -- (1,0.2);
        \draw (0.5,0.5) circle (0pt) node[rotate=-30]{$\times$};
        \draw (0,0) circle (0pt) node[rotate=90,above]{$\sim$};
        \filldraw[fill=white] (0,0.6) circle (2pt);
        \filldraw (1,0) circle (2pt);
        \filldraw (0,-0.6) circle (2pt);
    \end{tikzpicture}
}}

\newcommand{\ThreeTwoDiagramPtOne}{\raisebox{-15pt}{
    \begin{tikzpicture}
        \draw (0,0.6) -- (1,0) -- (0,-0.6) -- cycle;
        \draw[<->] (0,0.8) -- (1,0.2);
        \draw (0.5,0.5) circle (0pt) node[rotate=-30]{$\times$};
        \draw (0,0) circle (0pt) node[rotate=90,above]{$\sim$};
        \draw (0.5,0.3) circle (0pt) node[rotate=-30,below]{$\sim$};
        \filldraw[fill=white] (0,0.6) circle (2pt);
        \filldraw (1,0) circle (2pt);
        \filldraw (0,-0.6) circle (2pt);
    \end{tikzpicture}
}}

\newcommand{\ThreeTwoDiagramPtTwo}{\raisebox{-15pt}{
    \begin{tikzpicture}
        \draw (0,0.6) -- (1,0.6) -- (1,-0.6) -- (0,-0.6) -- cycle;
        \draw[<->] (0,0.8) -- (1,0.8);
        \draw (0.5,0.8) circle (0pt) node{$\times$};
        \draw (0,0) circle (0pt) node[rotate=90,above]{$\sim$};
        \draw (0.5,0.6) circle (0pt) node[below]{$\sim$};
        \filldraw[fill=white] (0,0.6) circle (2pt);
        \filldraw (1,0.6) circle (2pt);
        \filldraw (1,-0.6) circle (2pt);
        \filldraw (0,-0.6) circle (2pt);
    \end{tikzpicture}
}}

\newcommand{\ThreeTwoDiagramPtThree}{\raisebox{-15pt}{
    \begin{tikzpicture}
        \draw (0,0.6) -- (1,0.6) -- (1,-0.6) -- (0,-0.6) -- cycle;
        \draw[<->] (1.2,0.6) -- (1.2,-0.6);
        \draw (1.2,0) circle (0pt) node{$\times$};
        \draw (0,0) circle (0pt) node[rotate=90,above]{$\sim$};
        \draw (0.5,0.6) circle (0pt) node[above]{$\sim$};
        \filldraw[fill=white] (0,0.6) circle (2pt);
        \filldraw (1,0.6) circle (2pt);
        \filldraw (1,-0.6) circle (2pt);
        \filldraw (0,-0.6) circle (2pt);
    \end{tikzpicture}
}}

\newcommand{\ThreeTwoDiagramPtFour}{\raisebox{-15pt}{
    \begin{tikzpicture}
        \draw (0,0.6) -- (1,0) -- (0,-0.6) -- cycle;
        \draw[<->] (0,0.8) -- (1,0.2);
        \draw[<->] (0,1) -- (1,0.4);
        \draw (0.5,0.5) circle (0pt) node[rotate=-30]{$\times$};
        \draw (0,0) circle (0pt) node[rotate=90,below]{$\sim$};
        \draw (0.5,0.3) circle (0pt) node[rotate=-30,below]{$\sim$};
        \draw (0.5,0.7) circle (0pt) node[above, rotate=-30]{\scriptsize$\not\in B$};
        \filldraw[fill=white] (0,0.6) circle (2pt);
        \filldraw (1,0) circle (2pt);
        \filldraw (0,-0.6) circle (2pt);
    \end{tikzpicture}
}}

\newcommand{\ThreeTwoDiagramPtFive}{\raisebox{-15pt}{
    \begin{tikzpicture}
        \draw (0,0.6) -- (1,0) -- (0,-0.6) -- cycle;
        \draw[<->] (0,0.8) -- (1,0.2);
        \draw[<->] (-0.2,0.6) -- (-0.2,-0.7) -- (0,-0.8) -- (1,-0.2);
        \draw (-0.2,0) circle (0pt) node{$\times$};
        \draw (0,0) circle (0pt) node[rotate=90,below]{$\sim$};
        \draw (0.5,0.3) circle (0pt) node[rotate=-30,below]{$\sim$};
        \draw (0.5,0.5) circle (0pt) node[above, rotate=-30]{\scriptsize$\in B$};
        \filldraw[fill=white] (0,0.6) circle (2pt);
        \filldraw (1,0) circle (2pt);
        \filldraw (0,-0.6) circle (2pt);
    \end{tikzpicture}
}}

\newcommand{\ThreeTwoDiagramPtSix}{\raisebox{-15pt}{
    \begin{tikzpicture}
        \draw (0,0.6) -- (1,0) -- (0,-0.6) -- cycle;
        \draw[<->] (0,0.8) -- (1,0.2);
        \draw[<->] (0,1) -- (1,0.4);
        \draw (0.5,0.5) circle (0pt) node[rotate=-30]{$\times$};
        \draw (0.5,0.3) circle (0pt) node[rotate=-30,below]{$\sim$};
        \draw (0.5,0.7) circle (0pt) node[above, rotate=-30]{\scriptsize$\not\in B$};
        \filldraw[fill=white] (0,0.6) circle (2pt);
        \filldraw (1,0) circle (2pt);
        \filldraw (0,-0.6) circle (2pt);
    \end{tikzpicture}
}}

\newcommand{\ThreeTwoDiagramPtSeven}{\raisebox{-15pt}{
    \begin{tikzpicture}
        \draw (0,0.6) -- (1,0) -- (0,-0.6) -- cycle;
        \draw[<->] (0,0.8) -- (1,0.2);
        \draw[<->] (-0.2,0.6) -- (-0.2,-0.6);
        \draw (-0.2,0) circle (0pt) node{$\times$};
        \draw (0.5,0.5) circle (0pt) node[above, rotate=-30]{\scriptsize$\in B$};
        \filldraw[fill=white] (0,0.6) circle (2pt);
        \filldraw (1,0) circle (2pt);
        \filldraw (0,-0.6) circle (2pt);
    \end{tikzpicture}
}}

\newcommand{\ThreeTwoDiagramPtEight}{\raisebox{-20pt}{
    \begin{tikzpicture}
        \draw (0,0.6) -- (1,0) -- (0,-0.6) -- cycle;
        \draw[<->] (0,0.8) -- (1,0.2);
        \draw[<->] (0,-0.8) -- (1,-0.2);
        \draw (0.5,-0.5) circle (0pt) node[rotate=30]{$\times$};
        \draw (0.5,0.5) circle (0pt) node[above, rotate=-30]{\scriptsize$\in B$};
        \filldraw[fill=white] (0,0.6) circle (2pt);
        \filldraw (1,0) circle (2pt);
        \filldraw (0,-0.6) circle (2pt);
    \end{tikzpicture}
}}

\newcommand{\ThreeTwoDiagramPtNine}{\raisebox{-7pt}{
    \begin{tikzpicture}
        \draw (0,0) -- (1,0);
        \draw[<->] (0,0.2) -- (1,0.2);
        \draw (0.5,0.2) circle (0pt) node{$\times$};
        \draw (0.5,0) circle (0pt) node[below]{$\sim$};
        \filldraw[fill=white] (0,0) circle (2pt);
        \filldraw (1,0) circle (2pt);
    \end{tikzpicture}
}}

\newcommand{\psiNrighttwointegral}{\raisebox{-15pt}{
    \begin{tikzpicture}
        \draw (0,0.6) -- (1,0) -- (0,-0.6);
        \draw[<->] (-0.2,0.6) -- (-0.2,-0.6);
        \draw (-0.2,0) circle (0pt) node[above, rotate=90]{\scriptsize$\not\in B$};
        \filldraw[fill=white] (0,0.6) circle (2pt);
        \filldraw[fill=white] (0,-0.6) circle (2pt);
        \filldraw (1,0) circle (2pt);
    \end{tikzpicture}
}}

\newcommand{\DiagramWk}{\raisebox{-15pt}{
    \begin{tikzpicture}
        \draw (0,0.6) -- (1,0) -- (0,-0.6);
        \draw[<->] (0,0.8) -- (1,0.2);
        \draw (0.5,0.5) circle (0pt) node[rotate=-30]{$\mathbf{\times}$};
        \filldraw[fill=white] (0,0.6) circle (2pt);
        \filldraw[fill=white] (0,-0.6) circle (2pt);
        \filldraw (1,0) circle (2pt);
    \end{tikzpicture}
}}

\newcommand{\psiZeroleftthreeintegralPtOne}{\raisebox{-15pt}{
    \begin{tikzpicture}
        \filldraw[fill=white] (0,0) circle (2pt);
        \filldraw (1,0) circle (2pt);
        \draw[<->] (0,-0.2) -- (1,-0.2);
        \draw (0.5,-0.2) circle (0pt) node[below]{\scriptsize$\in B$};   
    \end{tikzpicture}
}}

\newcommand{\DisplacementPsiOneTop}{\raisebox{-15pt}{
    \begin{tikzpicture}
        \filldraw (1,0.6) circle (2pt);
        \filldraw (1,-0.6) circle (2pt);
        \filldraw (2,0.6) circle (2pt);
        \filldraw (2,-0.6) circle (2pt);
        \draw (0,0.6) -- (1,0.6) -- (1,-0.6) -- (0,-0.6);
        \draw (1,0.6) -- (2,0.6) -- (2,-0.6) -- (1,-0.6);
        \filldraw[fill=white] (0,0.6) circle (2pt);
        \draw[<->] (0,0.8) -- (2,0.8);
        \draw (1,0.8) circle (0pt) node[rotate = 0]{$\times$};
        \draw (0.5,0.6) circle (0pt) node[below]{$\circ$};
        \filldraw[fill=white] (0,-0.6) circle (2pt);
    \end{tikzpicture}
}}

\newcommand{\DisplacementPsiOneTopPtOne}{\raisebox{-15pt}{
    \begin{tikzpicture}
        \filldraw (1,0.6) circle (2pt);
        \filldraw (1,-0.6) circle (2pt);
        \filldraw (2,0.6) circle (2pt);
        \filldraw (2,-0.6) circle (2pt);
        \draw (0,0.6) -- (1,0.6) -- (1,-0.6) -- (0,-0.6);
        \draw (1,0.6) -- (2,0.6) -- (2,-0.6) -- (1,-0.6);
        \filldraw[fill=white] (0,0.6) circle (2pt);
        \draw[<->] (0,0.8) -- (1,0.8);
        \draw (0.5,0.8) circle (0pt) node[rotate = 0]{$\times$};
        \filldraw[fill=white] (0,-0.6) circle (2pt);
    \end{tikzpicture}
}}

\newcommand{\DisplacementPsiOneTopPtTwo}{\raisebox{-15pt}{
    \begin{tikzpicture}
        \filldraw (1,0.6) circle (2pt);
        \filldraw (1,-0.6) circle (2pt);
        \filldraw (2,0.6) circle (2pt);
        \filldraw (2,-0.6) circle (2pt);
        \draw (0,0.6) -- (1,0.6) -- (1,-0.6) -- (0,-0.6);
        \draw (1,0.6) -- (2,0.6) -- (2,-0.6) -- (1,-0.6);
        \filldraw[fill=white] (0,0.6) circle (2pt);
        \draw[<->] (1,0.8) -- (2,0.8);
        \draw (1.5,0.8) circle (0pt) node[rotate = 0]{$\times$};
        \draw (0.5,0.6) circle (0pt) node[above]{$\circ$};
        \filldraw[fill=white] (0,-0.6) circle (2pt);
    \end{tikzpicture}
}}

\newcommand{\DisplacementPsiOneBottom}{\raisebox{-25pt}{
    \begin{tikzpicture}
        \filldraw (1,0.6) circle (2pt);
        \filldraw (1,-0.6) circle (2pt);
        \filldraw (2,0.6) circle (2pt);
        \filldraw (2,-0.6) circle (2pt);
        \draw (0,0.6) -- (1,0.6) -- (1,-0.6) -- (0,-0.6);
        \draw (1,0.6) -- (2,0.6) -- (2,-0.6) -- (1,-0.6);
        \filldraw[fill=white] (0,0.6) circle (2pt);
        \draw[<->] (0,-0.8) -- (2,-0.8);
        \draw (1,-0.8) circle (0pt) node[rotate = 0]{$\times$};
        \draw (0.5,0.6) circle (0pt) node[above]{$\circ$};
        \filldraw[fill=white] (0,-0.6) circle (2pt);
    \end{tikzpicture}
}}

\newcommand{\DisplacementPsiOneBottomPtOne}{\raisebox{-25pt}{
    \begin{tikzpicture}
        \filldraw (1,-0.6) circle (2pt);
        \filldraw (2,0.6) circle (2pt);
        \filldraw (2,-0.6) circle (2pt);
        \draw (1,0.6) -- (1,-0.6) -- (0,-0.6);
        \draw (1,0.6) -- (2,0.6) -- (2,-0.6) -- (1,-0.6);
        \draw[<->] (0,-0.8) -- (1,-0.8);
        \draw (0.5,-0.8) circle (0pt) node[rotate = 0]{$\times$};
        \filldraw[fill=white] (0,-0.6) circle (2pt);
        \filldraw[fill=white] (1,0.6) circle (2pt);
    \end{tikzpicture}
}}

\newcommand{\DisplacementPsiOneBottomPtTwo}{\raisebox{-25pt}{
    \begin{tikzpicture}
        \filldraw (1,0.6) circle (2pt);
        \filldraw (1,-0.6) circle (2pt);
        \filldraw (2,0.6) circle (2pt);
        \filldraw (2,-0.6) circle (2pt);
        \draw (0,0.6) -- (1,0.6) -- (1,-0.6) -- (0,-0.6);
        \draw (1,0.6) -- (2,0.6) -- (2,-0.6) -- (1,-0.6);
        \filldraw[fill=white] (0,0.6) circle (2pt);
        \draw[<->] (0,-0.8) -- (1,-0.8);
        \draw (0.5,-0.8) circle (0pt) node[rotate = 0]{$\times$};
        \filldraw[fill=white] (0,-0.6) circle (2pt);
    \end{tikzpicture}
}}

\newcommand{\DisplacementPsiOneBottomPtThree}{\raisebox{-25pt}{
    \begin{tikzpicture}
        \filldraw (1,0.6) circle (2pt);
        \filldraw (1,-0.6) circle (2pt);
        \filldraw (2,0.6) circle (2pt);
        \filldraw (2,-0.6) circle (2pt);
        \draw (0,0.6) -- (1,0.6) -- (1,-0.6) -- (0,-0.6);
        \draw (1,0.6) -- (2,0.6) -- (2,-0.6) -- (1,-0.6);
        \filldraw[fill=white] (0,0.6) circle (2pt);
        \draw[<->] (1,-0.8) -- (2,-0.8);
        \draw (1.5,-0.8) circle (0pt) node[rotate = 0]{$\times$};
        \draw (0.5,0.6) circle (0pt) node[above]{$\circ$};
        \filldraw[fill=white] (0,-0.6) circle (2pt);
    \end{tikzpicture}
}}

\newcommand{\DisplacementPsiTwoTop}{\raisebox{-15pt}{
    \begin{tikzpicture}
        \filldraw (1.5,0.6) circle (2pt);
        \filldraw (1,-0.6) circle (2pt);
        \filldraw (1.5,-0.1) circle (2pt);
        \filldraw (2,-0.6) circle (2pt);
        \draw (0,0.6) -- (1.5,0.6) -- (1.5,-0.1) -- (1,-0.6);
        \draw (0,-0.6) -- (1,-0.6) -- (2,-0.6) -- (1.5,-0.1);
        \draw[<->] (0,0.8) -- (1.5,0.8);
        \draw (0.75,0.8) circle (0pt) node[rotate = 0]{$\times$};
        \draw (1.5,0.25) circle (0pt) node[right]{$\circ$};
        \filldraw[fill=white] (0,0.6) circle (2pt);
        \filldraw[fill=white] (0,-0.6) circle (2pt);
    \end{tikzpicture}
}}

\newcommand{\DisplacementPsiTwoTopPtOne}{\raisebox{-15pt}{
    \begin{tikzpicture}
        \filldraw (1.5,0.6) circle (2pt);
        \filldraw (1,-0.6) circle (2pt);
        \filldraw (1.5,-0.1) circle (2pt);
        \filldraw (2,-0.6) circle (2pt);
        \draw (0,0.6) -- (1.5,0.6) -- (1.5,-0.1) -- (1,-0.6);
        \draw (0,-0.6) -- (1,-0.6) -- (2,-0.6) -- (1.5,-0.1);
        \draw[<->] (0,0.8) -- (1.5,0.8);
        \draw (0.75,0.8) circle (0pt) node[rotate = 0]{$\times$};
        \filldraw[fill=white] (0,0.6) circle (2pt);
        \filldraw[fill=white] (0,-0.6) circle (2pt);
    \end{tikzpicture}
}}

\newcommand{\DisplacementPsiTwoTopPtTwo}{\raisebox{-15pt}{
    \begin{tikzpicture}
        \filldraw (1.5,0.6) circle (2pt);
        \filldraw (1,-0.6) circle (2pt);
        \filldraw (2,-0.6) circle (2pt);
        \draw (0,0.6) -- (1.5,0.6) -- (1,-0.6);
        \draw (0,-0.6) -- (1,-0.6) -- (2,-0.6) -- (1.5,0.6);
        \draw[<->] (0,0.8) -- (1.5,0.8);
        \draw (0.75,0.8) circle (0pt) node[rotate = 0]{$\times$};
        \filldraw[fill=white] (0,0.6) circle (2pt);
        \filldraw[fill=white] (0,-0.6) circle (2pt);
    \end{tikzpicture}
}}

\newcommand{\DisplacementPsiTwoBottom}{\raisebox{-25pt}{
    \begin{tikzpicture}
        \filldraw (1.5,0.6) circle (2pt);
        \filldraw (1,-0.6) circle (2pt);
        \filldraw (1.5,-0.1) circle (2pt);
        \filldraw (2,-0.6) circle (2pt);
        \draw (0,0.6) -- (1.5,0.6) -- (1.5,-0.1) -- (1,-0.6);
        \draw (0,-0.6) -- (1,-0.6) -- (2,-0.6) -- (1.5,-0.1);
        \draw[<->] (0,-0.8) -- (2,-0.8);
        \draw (1,-0.8) circle (0pt) node[rotate = 0]{$\times$};
        \draw (1.5,0.25) circle (0pt) node[right]{$\circ$};
        \draw (0.75,0.6) circle (0pt) node[above]{$\circ$};
        \filldraw[fill=white] (0,0.6) circle (2pt);
        \filldraw[fill=white] (0,-0.6) circle (2pt);
    \end{tikzpicture}
}}

\newcommand{\DisplacementPsiTwoBottomPtOne}{\raisebox{-25pt}{
    \begin{tikzpicture}
        \filldraw (1.5,0.6) circle (2pt);
        \filldraw (1,-0.6) circle (2pt);
        \filldraw (1.5,-0.1) circle (2pt);
        \filldraw (2,-0.6) circle (2pt);
        \draw (0,0.6) -- (1.5,0.6) -- (1.5,-0.1) -- (1,-0.6);
        \draw (0,-0.6) -- (1,-0.6) -- (2,-0.6) -- (1.5,-0.1);
        \draw[<->] (0,-0.8) -- (1,-0.8);
        \draw (0.5,-0.8) circle (0pt) node[rotate = 0]{$\times$};
        \filldraw[fill=white] (0,0.6) circle (2pt);
        \filldraw[fill=white] (0,-0.6) circle (2pt);
    \end{tikzpicture}
}}

\newcommand{\DisplacementPsiTwoBottomPtTwo}{\raisebox{-25pt}{
    \begin{tikzpicture}
        \filldraw (1.5,0.6) circle (2pt);
        \filldraw (1,-0.6) circle (2pt);
        \filldraw (2,-0.6) circle (2pt);
        \draw (0,0.6) -- (1.5,0.6) -- (1,-0.6);
        \draw (0,-0.6) -- (1,-0.6) -- (2,-0.6) -- (1.5,0.6);
        \draw[<->] (0,-0.8) -- (1,-0.8);
        \draw (0.5,-0.8) circle (0pt) node[rotate = 0]{$\times$};
        \filldraw[fill=white] (0,0.6) circle (2pt);
        \filldraw[fill=white] (0,-0.6) circle (2pt);
    \end{tikzpicture}
}}

\newcommand{\DisplacementPsiTwoBottomPtThree}{\raisebox{-25pt}{
    \begin{tikzpicture}
        \filldraw (1,-0.6) circle (2pt);
        \filldraw (2,-0.6) circle (2pt);
        \draw (1.5,0.6) -- (1,-0.6);
        \draw (0,-0.6) -- (1,-0.6) -- (2,-0.6) -- (1.5,0.6);
        \draw[<->] (0,-0.8) -- (1,-0.8);
        \draw (0.5,-0.8) circle (0pt) node[rotate = 0]{$\times$};
        \filldraw[fill=white] (1.5,0.6) circle (2pt);
        \filldraw[fill=white] (0,-0.6) circle (2pt);
    \end{tikzpicture}
}}

\newcommand{\DisplacementPsiTwoBottomPtFour}{\raisebox{-25pt}{
    \begin{tikzpicture}
        \filldraw (1.5,0.6) circle (2pt);
        \filldraw (1,-0.6) circle (2pt);
        \filldraw (1.5,-0.1) circle (2pt);
        \filldraw (2,-0.6) circle (2pt);
        \draw (0,0.6) -- (1.5,0.6) -- (1.5,-0.1) -- (1,-0.6);
        \draw (0,-0.6) -- (1,-0.6) -- (2,-0.6) -- (1.5,-0.1);
        \draw[<->] (1,-0.8) -- (2,-0.8);
        \draw (1.5,-0.8) circle (0pt) node[rotate = 0]{$\times$};
        \filldraw[fill=white] (0,0.6) circle (2pt);
        \filldraw[fill=white] (0,-0.6) circle (2pt);
    \end{tikzpicture}
}}

\newcommand{\DisplacementPsiTwoBottomPtFive}{\raisebox{-25pt}{
    \begin{tikzpicture}
        \filldraw (1.5,0.6) circle (2pt);
        \filldraw (1,-0.6) circle (2pt);
        \filldraw (2,-0.6) circle (2pt);
        \draw (0,0.6) -- (1.5,0.6) -- (1,-0.6);
        \draw (0,-0.6) -- (1,-0.6) -- (2,-0.6) -- (1.5,0.6);
        \draw[<->] (1,-0.8) -- (2,-0.8);
        \draw (1.5,-0.8) circle (0pt) node[rotate = 0]{$\times$};
        \filldraw[fill=white] (0,0.6) circle (2pt);
        \filldraw[fill=white] (0,-0.6) circle (2pt);
    \end{tikzpicture}
}}

\newcommand{\DisplacementPsiTwoBottomPtSix}{\raisebox{-25pt}{
    \begin{tikzpicture}
        \filldraw (1,-0.6) circle (2pt);
        \filldraw (2,-0.6) circle (2pt);
        \draw (1.5,0.6) -- (1,-0.6);
        \draw (0,-0.6) -- (1,-0.6) -- (2,-0.6) -- (1.5,0.6);
        \draw[<->] (1,-0.8) -- (2,-0.8);
        \draw (1.5,-0.8) circle (0pt) node[rotate = 0]{$\times$};
        \filldraw[fill=white] (1.5,0.6) circle (2pt);
        \filldraw[fill=white] (0,-0.6) circle (2pt);
    \end{tikzpicture}
}}

\newcommand{\DisplacementPsiFourTop}{\raisebox{-15pt}{
    \begin{tikzpicture}
        \filldraw (1,0.6) circle (2pt);
        \filldraw (1,-0.6) circle (2pt);
        \draw (0,0.6) -- (1,0.6) -- (1,-0.6) -- (0,-0.6);
        \filldraw[fill=white] (0,0.6) circle (2pt);
        \draw[<->] (0,0.8) -- (1,0.8);
        \draw (0.5,0.8) circle (0pt) node[rotate = 0]{$\times$};
        \filldraw[fill=white] (0,-0.6) circle (2pt);
    \end{tikzpicture}
}}

\newcommand{\DisplacementPsiOneTopPtOneAug}{\raisebox{-15pt}{
    \begin{tikzpicture}
        \filldraw (1,0.6) circle (2pt);
        \filldraw (1,-0.6) circle (2pt);
        \filldraw (2,0.6) circle (2pt);
        \filldraw (2,-0.6) circle (2pt);
        \draw (0,0.6) -- (1,0.6) -- (1,-0.6) -- (0,-0.6);
        \draw (1,0.6) -- (2,0.6) -- (2,-0.6) -- (1,-0.6);
        \draw (3,0.6) -- (2,0.6) -- (2,-0.6) -- (3,-0.6);
        \draw[dashed] (0,0.6) -- (0,-0.6);
        \filldraw[fill=white] (0,0.6) circle (2pt);
        \node[mark size=2pt] at (0,-0.6) {\pgfuseplotmark{square*}};
        \draw[<->] (0,0.8) -- (1,0.8);
        \draw (0.5,0.8) circle (0pt) node[rotate = 0]{$\times$};
        \draw (2.5,-0.6) circle (0pt) node[above]{$\circ$};
        \filldraw[fill=white] (3,0.6) circle (2pt);
        \filldraw[fill=white] (3,-0.6) circle (2pt);
    \end{tikzpicture}
}}

\newcommand{\DisplacementPsiOneTopPtTwoAug}{\raisebox{-15pt}{
    \begin{tikzpicture}
        \filldraw (1,0.6) circle (2pt);
        \filldraw (1,-0.6) circle (2pt);
        \filldraw (2,0.6) circle (2pt);
        \filldraw (2,-0.6) circle (2pt);
        \draw (0,0.6) -- (1,0.6) -- (1,-0.6) -- (0,-0.6);
        \draw (1,0.6) -- (2,0.6) -- (2,-0.6) -- (1,-0.6);
        \draw[<->] (1,0.8) -- (2,0.8);
        \draw (1.5,0.8) circle (0pt) node[rotate = 0]{$\times$};
        \draw (0.5,0.6) circle (0pt) node[above]{$\circ$};
        \draw (2,0.6) -- (3,0.6);
        \draw (2,-0.6) -- (3,-0.6);
        \draw[dashed] (0,0.6) -- (0,-0.6);
        \draw (2.5,-0.6) circle (0pt) node[above]{$\circ$};
        \node[mark size=2pt] at (0,-0.6) {\pgfuseplotmark{square*}};
        \filldraw[fill=white] (0,0.6) circle (2pt);
        \filldraw[fill=white] (3,-0.6) circle (2pt);
        \filldraw[fill=white] (3,0.6) circle (2pt);
    \end{tikzpicture}
}}

\newcommand{\DisplacementPsiOneBottomPtTwoAug}{\raisebox{-25pt}{
    \begin{tikzpicture}
        \filldraw (1,0.6) circle (2pt);
        \filldraw (1,-0.6) circle (2pt);
        \filldraw (2,0.6) circle (2pt);
        \filldraw (2,-0.6) circle (2pt);
        \draw (0,0.6) -- (1,0.6) -- (1,-0.6) -- (0,-0.6);
        \draw (1,0.6) -- (2,0.6) -- (2,-0.6) -- (1,-0.6);
        \draw[<->] (0,-0.8) -- (1,-0.8);
        \draw (0.5,-0.8) circle (0pt) node[rotate = 0]{$\times$};
        \draw (2,0.6) -- (3,0.6);
        \draw (2,-0.6) -- (3,-0.6);
        \draw[dashed] (0,0.6) -- (0,-0.6);
        \draw (2.5,-0.6) circle (0pt) node[above]{$\circ$};
        \node[mark size=2pt] at (0,-0.6) {\pgfuseplotmark{square*}};
        \filldraw[fill=white] (0,0.6) circle (2pt);
        \filldraw[fill=white] (3,-0.6) circle (2pt);
        \filldraw[fill=white] (3,0.6) circle (2pt);
    \end{tikzpicture}
}}

\newcommand{\DisplacementPsiOneBottomPtThreeAug}{\raisebox{-25pt}{
    \begin{tikzpicture}
        \filldraw (1,0.6) circle (2pt);
        \filldraw (1,-0.6) circle (2pt);
        \filldraw (2,0.6) circle (2pt);
        \filldraw (2,-0.6) circle (2pt);
        \draw (0,0.6) -- (1,0.6) -- (1,-0.6) -- (0,-0.6);
        \draw (1,0.6) -- (2,0.6) -- (2,-0.6) -- (1,-0.6);
        \draw[<->] (1,-0.8) -- (2,-0.8);
        \draw (1.5,-0.8) circle (0pt) node[rotate = 0]{$\times$};
        \draw (0.5,0.6) circle (0pt) node[above]{$\circ$};
        \draw (2,0.6) -- (3,0.6);
        \draw (2,-0.6) -- (3,-0.6);
        \draw[dashed] (0,0.6) -- (0,-0.6);
        \draw (2.5,-0.6) circle (0pt) node[above]{$\circ$};
        \node[mark size=2pt] at (0,-0.6) {\pgfuseplotmark{square*}};
        \filldraw[fill=white] (0,0.6) circle (2pt);
        \filldraw[fill=white] (3,-0.6) circle (2pt);
        \filldraw[fill=white] (3,0.6) circle (2pt);
    \end{tikzpicture}
}}

\newcommand{\DisplacementPsiTwoTopPtOneAug}{\raisebox{-15pt}{
    \begin{tikzpicture}
        \filldraw (1.5,0.6) circle (2pt);
        \filldraw (1,-0.6) circle (2pt);
        \filldraw (1.5,-0.1) circle (2pt);
        \filldraw (2,-0.6) circle (2pt);
        \draw (0,0.6) -- (1.5,0.6) -- (1.5,-0.1) -- (1,-0.6);
        \draw (0,-0.6) -- (1,-0.6) -- (2,-0.6) -- (1.5,-0.1);
        \draw[<->] (0,0.8) -- (1.5,0.8);
        \draw (0.75,0.8) circle (0pt) node[rotate = 0]{$\times$};
        \draw (1.5,0.6) -- (3,0.6);
        \draw (2,-0.6) -- (3,-0.6);
        \draw[dashed] (0,0.6) -- (0,-0.6);
        \draw (2.5,-0.6) circle (0pt) node[above]{$\circ$};
        \node[mark size=2pt] at (0,-0.6) {\pgfuseplotmark{square*}};
        \filldraw[fill=white] (0,0.6) circle (2pt);
        \filldraw[fill=white] (3,-0.6) circle (2pt);
        \filldraw[fill=white] (3,0.6) circle (2pt);
    \end{tikzpicture}
}}

\newcommand{\DisplacementPsiTwoBottomPtOneAug}{\raisebox{-25pt}{
    \begin{tikzpicture}
        \filldraw (1.5,0.6) circle (2pt);
        \filldraw (1,-0.6) circle (2pt);
        \filldraw (1.5,-0.1) circle (2pt);
        \filldraw (2,-0.6) circle (2pt);
        \draw (0,0.6) -- (1.5,0.6) -- (1.5,-0.1) -- (1,-0.6);
        \draw (0,-0.6) -- (1,-0.6) -- (2,-0.6) -- (1.5,-0.1);
        \draw[<->] (0,-0.8) -- (1,-0.8);
        \draw (0.5,-0.8) circle (0pt) node[rotate = 0]{$\times$};
        \draw (1.5,0.6) -- (3,0.6);
        \draw (2,-0.6) -- (3,-0.6);
        \draw[dashed] (0,0.6) -- (0,-0.6);
        \draw (2.5,-0.6) circle (0pt) node[above]{$\circ$};
        \node[mark size=2pt] at (0,-0.6) {\pgfuseplotmark{square*}};
        \filldraw[fill=white] (0,0.6) circle (2pt);
        \filldraw[fill=white] (3,-0.6) circle (2pt);
        \filldraw[fill=white] (3,0.6) circle (2pt);
    \end{tikzpicture}
}}

\newcommand{\DisplacementPsiTwoBottomPtFourAug}{\raisebox{-25pt}{
    \begin{tikzpicture}
        \filldraw (1.5,0.6) circle (2pt);
        \filldraw (1,-0.6) circle (2pt);
        \filldraw (1.5,-0.1) circle (2pt);
        \filldraw (2,-0.6) circle (2pt);
        \draw (0,0.6) -- (1.5,0.6) -- (1.5,-0.1) -- (1,-0.6);
        \draw (0,-0.6) -- (1,-0.6) -- (2,-0.6) -- (1.5,-0.1);
        \draw[<->] (1,-0.8) -- (2,-0.8);
        \draw (1.5,-0.8) circle (0pt) node[rotate = 0]{$\times$};
        \draw (1.5,0.6) -- (3,0.6);
        \draw (2,-0.6) -- (3,-0.6);
        \draw[dashed] (0,0.6) -- (0,-0.6);
        \draw (2.5,-0.6) circle (0pt) node[above]{$\circ$};
        \node[mark size=2pt] at (0,-0.6) {\pgfuseplotmark{square*}};
        \filldraw[fill=white] (0,0.6) circle (2pt);
        \filldraw[fill=white] (3,-0.6) circle (2pt);
        \filldraw[fill=white] (3,0.6) circle (2pt);
    \end{tikzpicture}
}}

\newcommand{\DisplacementPsiFourTopAugVOne}{\raisebox{-15pt}{
    \begin{tikzpicture}
        \filldraw (1,0.6) circle (2pt);
        \filldraw (1,-0.6) circle (2pt);
        \draw (0,0.6) -- (1,0.6) -- (1,-0.6) -- (0,-0.6);
        \draw[<->] (0,0.8) -- (1,0.8);
        \draw (0.5,0.8) circle (0pt) node[rotate = 0]{$\times$};
        \draw (1,0.6) -- (2,0.6);
        \draw (1,-0.6) -- (2,-0.6);
        \draw[dashed] (0,0.6) -- (0,-0.6);
        \draw (1.5,-0.6) circle (0pt) node[above]{$\circ$};
        \node[mark size=2pt] at (0,-0.6) {\pgfuseplotmark{square*}};
        \filldraw[fill=white] (0,0.6) circle (2pt);
        \filldraw[fill=white] (2,-0.6) circle (2pt);
        \filldraw[fill=white] (2,0.6) circle (2pt);
    \end{tikzpicture}
}}

\newcommand{\DisplacementPsiFourTopAugVTwo}{\raisebox{-25pt}{
    \begin{tikzpicture}
        \filldraw (1,0.6) circle (2pt);
        \filldraw (1,-0.6) circle (2pt);
        \draw (0,0.6) -- (1,0.6) -- (1,-0.6) -- (0,-0.6);
        \draw[<->] (0,-0.8) -- (1,-0.8);
        \draw (0.5,-0.8) circle (0pt) node[rotate = 0]{$\times$};
        \draw (1,0.6) -- (2,0.6);
        \draw (1,-0.6) -- (2,-0.6);
        \draw[dashed] (0,0.6) -- (0,-0.6);
        \draw (1.5,-0.6) circle (0pt) node[above]{$\circ$};
        \node[mark size=2pt] at (0,-0.6) {\pgfuseplotmark{square*}};
        \filldraw[fill=white] (0,0.6) circle (2pt);
        \filldraw[fill=white] (2,-0.6) circle (2pt);
        \filldraw[fill=white] (2,0.6) circle (2pt);
    \end{tikzpicture}
}}

\newcommand{\DisplacementPsiOneTopPtTwoAugPlusFive}{\raisebox{-15pt}{
    \begin{tikzpicture}
        \filldraw (1,0.6) circle (2pt);
        \filldraw (1,-0.6) circle (2pt);
        \filldraw (2,0.6) circle (2pt);
        \filldraw (2,-0.6) circle (2pt);
        \draw (0,0.6) -- (1,0.6) -- (1,-0.6) -- (0,-0.6);
        \draw (1,0.6) -- (2,0.6) -- (2,-0.6) -- (1,-0.6);
        \draw[<->] (1,0.8) -- (2,0.8);
        \draw (1.5,0.8) circle (0pt) node[rotate = 0]{$\times$};
        \draw (0.5,0.6) circle (0pt) node[above]{$\circ$};
        \draw (2,0.6) -- (3,0.6);
        \draw (2,-0.6) -- (3,-0.6);
        \draw[dashed] (0,0.6) -- (0,-0.6);
        \draw (2.5,-0.6) circle (0pt) node[above]{$\circ$};
        \draw (3,0.6) -- (3,-0.6) -- (4,-0.6);
        \node[mark size=2pt] at (0,-0.6) {\pgfuseplotmark{square*}};
        \filldraw[fill=white] (0,0.6) circle (2pt);
        \filldraw (3,-0.6) circle (2pt);
        \filldraw[fill=white] (3,0.6) circle (2pt);
        \filldraw[fill=white] (4,-0.6) circle (2pt);
    \end{tikzpicture}
}}

\newcommand{\DisplacementPsiOneBottomPtThreeAugPlusFive}{\raisebox{-25pt}{
    \begin{tikzpicture}
        \filldraw (1,0.6) circle (2pt);
        \filldraw (1,-0.6) circle (2pt);
        \filldraw (2,0.6) circle (2pt);
        \filldraw (2,-0.6) circle (2pt);
        \draw (0,0.6) -- (1,0.6) -- (1,-0.6) -- (0,-0.6);
        \draw (1,0.6) -- (2,0.6) -- (2,-0.6) -- (1,-0.6);
        \draw[<->] (1,-0.8) -- (2,-0.8);
        \draw (1.5,-0.8) circle (0pt) node[rotate = 0]{$\times$};
        \draw (0.5,0.6) circle (0pt) node[above]{$\circ$};
        \draw (2,0.6) -- (3,0.6);
        \draw (2,-0.6) -- (3,-0.6);
        \draw[dashed] (0,0.6) -- (0,-0.6);
        \draw (2.5,-0.6) circle (0pt) node[above]{$\circ$};
        \draw (3,0.6) -- (3,-0.6) -- (4,-0.6);
        \node[mark size=2pt] at (0,-0.6) {\pgfuseplotmark{square*}};
        \filldraw[fill=white] (0,0.6) circle (2pt);
        \filldraw (3,-0.6) circle (2pt);
        \filldraw[fill=white] (3,0.6) circle (2pt);
        \filldraw[fill=white] (4,-0.6) circle (2pt);
    \end{tikzpicture}
}}

\newcommand{\DisplacementPsiFourTopAugPlusFiveVOne}{\raisebox{-15pt}{
    \begin{tikzpicture}
        \filldraw (1,0.6) circle (2pt);
        \filldraw (1,-0.6) circle (2pt);
        \draw (0,0.6) -- (1,0.6) -- (1,-0.6) -- (0,-0.6);
        \draw[<->] (0,0.8) -- (1,0.8);
        \draw (0.5,0.8) circle (0pt) node[rotate = 0]{$\times$};
        \draw (1,0.6) -- (2,0.6);
        \draw (1,-0.6) -- (2,-0.6);
        \draw[dashed] (0,0.6) -- (0,-0.6);
        \draw (1.5,-0.6) circle (0pt) node[above]{$\circ$};
        \draw (2,0.6)-- (2,-0.6) -- (3,-0.6);
        \node[mark size=2pt] at (0,-0.6) {\pgfuseplotmark{square*}};
        \filldraw[fill=white] (0,0.6) circle (2pt);
        \filldraw (2,-0.6) circle (2pt);
        \filldraw[fill=white] (2,0.6) circle (2pt);
        \filldraw[fill=white] (3,-0.6) circle (2pt);
    \end{tikzpicture}
}}

\newcommand{\DisplacementPsiFourTopAugPlusFiveVTwo}{\raisebox{-25pt}{
    \begin{tikzpicture}
        \filldraw (1,0.6) circle (2pt);
        \filldraw (1,-0.6) circle (2pt);
        \draw (0,0.6) -- (1,0.6) -- (1,-0.6) -- (0,-0.6);
        \draw[<->] (0,-0.8) -- (1,-0.8);
        \draw (0.5,-0.8) circle (0pt) node[rotate = 0]{$\times$};
        \draw (1,0.6) -- (2,0.6);
        \draw (1,-0.6) -- (2,-0.6);
        \draw[dashed] (0,0.6) -- (0,-0.6);
        \draw (1.5,-0.6) circle (0pt) node[above]{$\circ$};
        \draw (2,0.6)-- (2,-0.6) -- (3,-0.6);
        \node[mark size=2pt] at (0,-0.6) {\pgfuseplotmark{square*}};
        \filldraw[fill=white] (0,0.6) circle (2pt);
        \filldraw (2,-0.6) circle (2pt);
        \filldraw[fill=white] (3,-0.6) circle (2pt);
        \filldraw[fill=white] (2,0.6) circle (2pt);
    \end{tikzpicture}
}}

\newcommand{\DisplacementPsiOneTopPtTwoAugPlusTwo}{\raisebox{-15pt}{
    \begin{tikzpicture}
        \filldraw (1,0.6) circle (2pt);
        \filldraw (1,-0.6) circle (2pt);
        \filldraw (2,0.6) circle (2pt);
        \filldraw (2,-0.6) circle (2pt);
        \draw (0,0.6) -- (1,0.6) -- (1,-0.6) -- (0,-0.6);
        \draw (1,0.6) -- (2,0.6) -- (2,-0.6) -- (1,-0.6);
        \draw[<->] (1,0.8) -- (2,0.8);
        \draw (1.5,0.8) circle (0pt) node[rotate = 0]{$\times$};
        \draw (0.5,0.6) circle (0pt) node[above]{$\circ$};
        \draw (2,0.6) -- (3,0);
        \draw (2,-0.6) -- (3,0);
        \draw[dashed] (0,0.6) -- (0,-0.6);
        \node[mark size=2pt] at (0,-0.6) {\pgfuseplotmark{square*}};
        \filldraw[fill=white] (0,0.6) circle (2pt);
        \filldraw[fill=white] (3,0) circle (2pt);
    \end{tikzpicture}
}}

\newcommand{\DisplacementPsiOneBottomPtThreeAugPlusTwo}{\raisebox{-25pt}{
    \begin{tikzpicture}
        \filldraw (1,0.6) circle (2pt);
        \filldraw (1,-0.6) circle (2pt);
        \filldraw (2,0.6) circle (2pt);
        \filldraw (2,-0.6) circle (2pt);
        \draw (0,0.6) -- (1,0.6) -- (1,-0.6) -- (0,-0.6);
        \draw (1,0.6) -- (2,0.6) -- (2,-0.6) -- (1,-0.6);
        \draw[<->] (1,-0.8) -- (2,-0.8);
        \draw (1.5,-0.8) circle (0pt) node[rotate = 0]{$\times$};
        \draw (0.5,0.6) circle (0pt) node[above]{$\circ$};
        \draw (2,0.6) -- (3,0);
        \draw (2,-0.6) -- (3,0);
        \draw[dashed] (0,0.6) -- (0,-0.6);
        \node[mark size=2pt] at (0,-0.6) {\pgfuseplotmark{square*}};
        \filldraw[fill=white] (0,0.6) circle (2pt);
        \filldraw[fill=white] (3,0) circle (2pt);
    \end{tikzpicture}
}}

\newcommand{\DisplacementPsiFourTopAugPlusTwoVOne}{\raisebox{-15pt}{
    \begin{tikzpicture}
        \filldraw (1,0.6) circle (2pt);
        \filldraw (1,-0.6) circle (2pt);
        \draw (0,0.6) -- (1,0.6) -- (1,-0.6) -- (0,-0.6);
        \draw[<->] (0,0.8) -- (1,0.8);
        \draw (0.5,0.8) circle (0pt) node[rotate = 0]{$\times$};
        \draw (1,0.6) -- (2,0);
        \draw (1,-0.6) -- (2,0);
        \draw[dashed] (0,0.6) -- (0,-0.6);
        \node[mark size=2pt] at (0,-0.6) {\pgfuseplotmark{square*}};
        \filldraw[fill=white] (0,0.6) circle (2pt);
        \filldraw[fill=white] (2,0) circle (2pt);
    \end{tikzpicture}
}}

\newcommand{\DisplacementPsiFourTopAugPlusTwoVTwo}{\raisebox{-25pt}{
    \begin{tikzpicture}
        \filldraw (1,0.6) circle (2pt);
        \filldraw (1,-0.6) circle (2pt);
        \draw (0,0.6) -- (1,0.6) -- (1,-0.6) -- (0,-0.6);
        \draw[<->] (0,-0.8) -- (1,-0.8);
        \draw (0.5,-0.8) circle (0pt) node[rotate = 0]{$\times$};
        \draw (1,0.6) -- (2,0);
        \draw (1,-0.6) -- (2,0);
        \draw[dashed] (0,0.6) -- (0,-0.6);
        \node[mark size=2pt] at (0,-0.6) {\pgfuseplotmark{square*}};
        \filldraw[fill=white] (0,0.6) circle (2pt);
        \filldraw[fill=white] (2,0) circle (2pt);
    \end{tikzpicture}
}}

\newcommand{\ExampleOne}{\raisebox{-15pt}{
    \begin{tikzpicture}
        \draw (0,0) -- (1,0.6) -- (1,-0.6) -- (0,0);
        \draw (1,0.6) -- (2,0.6) -- (2,-0.6) -- (1,-0.6);
        \draw (2,0.6) -- (3,0.6) -- (3,-0.6) -- (2,-0.6);
        \draw (3,0.6) -- (4,0) -- (3,-0.6);
        \draw (1.5,-0.6) circle (0pt) node[above]{$\circ$};
        \draw[<->] (0,0.2) -- (1,0.8) -- (3,0.8) -- (4,0.2);
        \draw (2,0.8) circle (0pt) node[rotate = 0]{$\times$};
        \filldraw[fill=white] (0,0) circle (2pt);
        \filldraw (1,0.6) circle (2pt);
        \filldraw (1,-0.6) circle (2pt);
        \filldraw (2,0.6) circle (2pt);
        \filldraw (2,-0.6) circle (2pt);
        \filldraw (3,0.6) circle (2pt);
        \filldraw (3,-0.6) circle (2pt);
        \filldraw (4,0) circle (2pt);
    \end{tikzpicture}
}}

\newcommand{\ExampleOnePtOne}{\raisebox{-15pt}{
    \begin{tikzpicture}
        \draw (0,0) -- (1,0.6) -- (1,-0.6) -- (0,0);
        \draw (1,0.6) -- (2,0.6) -- (2,-0.6) -- (1,-0.6);
        \draw (2,0.6) -- (3,0.6) -- (3,-0.6) -- (2,-0.6);
        \draw (3,0.6) -- (4,0) -- (3,-0.6);
        \draw (1.5,-0.6) circle (0pt) node[above]{$\circ$};
        \draw[<->] (0,0.2) -- (1,0.8);
        \draw (0.5,0.5) circle (0pt) node[rotate = 30]{$\times$};
        \filldraw[fill=white] (0,0) circle (2pt);
        \filldraw (1,0.6) circle (2pt);
        \filldraw (1,-0.6) circle (2pt);
        \filldraw (2,0.6) circle (2pt);
        \filldraw (2,-0.6) circle (2pt);
        \filldraw (3,0.6) circle (2pt);
        \filldraw (3,-0.6) circle (2pt);
        \filldraw (4,0) circle (2pt);
    \end{tikzpicture}
}}

\newcommand{\ExampleOnePtTwo}{\raisebox{-15pt}{
    \begin{tikzpicture}
        \draw (0,0) -- (1,0.6) -- (1,-0.6) -- (0,0);
        \draw (1,0.6) -- (2,0.6) -- (2,-0.6) -- (1,-0.6);
        \draw (2,0.6) -- (3,0.6) -- (3,-0.6) -- (2,-0.6);
        \draw (3,0.6) -- (4,0) -- (3,-0.6);
        \draw (1.5,-0.6) circle (0pt) node[above]{$\circ$};
        \draw[<->] (1,0.8) -- (3,0.8);
        \draw (2,0.8) circle (0pt) node[rotate = 0]{$\times$};
        \filldraw[fill=white] (0,0) circle (2pt);
        \filldraw (1,0.6) circle (2pt);
        \filldraw (1,-0.6) circle (2pt);
        \filldraw (2,0.6) circle (2pt);
        \filldraw (2,-0.6) circle (2pt);
        \filldraw (3,0.6) circle (2pt);
        \filldraw (3,-0.6) circle (2pt);
        \filldraw (4,0) circle (2pt);
    \end{tikzpicture}
}}

\newcommand{\ExampleOnePtThree}{\raisebox{-15pt}{
    \begin{tikzpicture}
        \draw (0,0) -- (1,0.6) -- (1,-0.6) -- (0,0);
        \draw (1,0.6) -- (2,0.6) -- (2,-0.6) -- (1,-0.6);
        \draw (2,0.6) -- (3,0.6) -- (3,-0.6) -- (2,-0.6);
        \draw (3,0.6) -- (4,0) -- (3,-0.6);
        \draw (1.5,-0.6) circle (0pt) node[above]{$\circ$};
        \draw[<->] (3,0.8) -- (4,0.2);
        \draw (3.5,0.5) circle (0pt) node[rotate = -30]{$\times$};
        \filldraw[fill=white] (0,0) circle (2pt);
        \filldraw (1,0.6) circle (2pt);
        \filldraw (1,-0.6) circle (2pt);
        \filldraw (2,0.6) circle (2pt);
        \filldraw (2,-0.6) circle (2pt);
        \filldraw (3,0.6) circle (2pt);
        \filldraw (3,-0.6) circle (2pt);
        \filldraw (4,0) circle (2pt);
    \end{tikzpicture}
}}

\newcommand{\SubCaseOne}{\mathrm{\left(I\right)}}
\newcommand{\SubCaseTwo}{\mathrm{\left(II\right)}}
\newcommand{\SubCaseThree}{\mathrm{\left(III\right)}}

\tikzset{cross/.style={cross out, draw=black, minimum size=2*(#1-\pgflinewidth), inner sep=0pt, outer sep=0pt},
cross/.default={1pt}}

\numberwithin{equation}{section}
\allowdisplaybreaks

\title{The Triangle Condition for the\\ Marked Random Connection Model}
\author{Matthew Dickson\footnote{University of British Columbia, Department of Mathematics, Vancouver, BC, Canada, V6T 1Z2; Email: dickson@math.ubc.ca; \orcidlink{0000-0002-8629-4796}~https://orcid.org/0000-0002-8629-4796} \and Markus Heydenreich\footnote{Universität Augsburg, Institut für Mathematik, 86135 Augsburg, Germany; Email: markus.heydenreich@uni-a.de;\orcidlink{0000-0002-3749-7431}~https://orcid.org/0000-0002-3749-7431}
}
 
\date{}

\begin{document}
\maketitle

\vspace{-1em}

{\centering{ \today}\par}

\vskip-3em

\begin{abstract}
We investigate a spatial random graph model whose vertices are given as a marked Poisson process on $\Rd$. Edges are inserted between any pair of points independently with probability depending on the spatial displacement of the two endpoints and on their marks. Upon variation of the Poisson density, a percolation phase transition occurs under mild conditions: for low density there are finite connected components only, whilst for large density there is an infinite component almost surely. 
Our focus is on the transition between the low- and high-density phase, where the system is critical. We prove that if the dimension is high enough and the edge probability function satisfies certain conditions, then an infrared bound for the critical connection function is valid. This implies the triangle condition, and thus mean-field behaviour. 
We achieve this result through combining the recently established lace expansion for Poisson processes with spectral estimates.
\end{abstract}



{\footnotesize
\tableofcontents
}

\section{Introduction}

\subsection{Phase transition in percolation} 
For percolation, each edge of the hypercubic lattice is retained independently with fixed probability $p\in[0,1]$. It is well-known that the model undergoes a phase transition: in any dimension $d\ge2$ there is a $p_c\in(0,1)$ such that for $p<p_c$ there is no infinite cluster and if $p>p_c$ there is a unique infinite cluster. Behaviour of the model at (and near) the critical threshold $p_c$ is a major theme in probability theory. Indeed, we have good control on the critical behaviour in two regimes: 
\begin{itemize}
    \item if $d=2$, then planarity and the link with complex analysis allows a refined description of the critical behaviour (to some extent only on the triangular lattice), see \cite{SmirnWerne01,Werne09,BolloRiord06} and references therein;
    \item if $d$ is large (at least $d>6$), then Hara and Slade \cite{HarSla90} established the famous triangle condition, which implies that various critical exponents exist and take on their mean-field values \cite{AizNew84,BarAiz91,Ngu87,KozNac09,KozNac11,HeyHof17,hutchcroft2021derivation}, see also \cite{HeyMat20} for a corresponding result for site percolation. While it is generally believed that $d>6$ is sufficient, current proof techniques only allow the proof for $d\ge11$ \cite{FitHof17}. Already \cite{HarSla90} show it for $d>6$ under an additional spread-out assumption. 
\end{itemize}
%
%
The focus of the present paper is on the second case, which we coin the ``mean-field regime''. 
Indeed, it has been shown that certain mean-field critical exponents exist for a wide range of (discrete and continuous) models provided that the dimension is large enough \cite{HeyHofLasMat19,HeyHofSak08}. This therefore substantiates the postulated \emph{universality} of critical exponents. 
The principle workhorse in showing the mean-field critical exponents in Euclidean spaces is the establishment of the triangle condition, which has also been verified for non-amenable graphs under varying assumptions \cite{Schon01,Schon02,Hutch20b}. 
However, all of the above papers crucially rely on the independence of edge occupation, and the question we put forward is: 
\begin{center}\emph{To what extent can dependency of the edge occupation change critical exponents?} \end{center}
The present work establishes the triangle condition for a class of \emph{dependent} percolation models, and thus shows that the critical exponents do not change. Mean-field bounds on some critical exponents were proven recently in the context of  Boolean percolation \cite{DewMui}, but a proof of the triangle condition in such situations is new. 

We prove an infrared bound for the marked random connection model whenever $d\geq d^*$ for some $d^*>6$, and thus (with the results of Caicedo-Dickson \cite{CaiDic24}) establish mean-field critical exponents. Mean-field theory predicts that mean-field criticality is true for all dimensions $d>6$, and thus $6$ is the upper critical dimension. 
On the other hand, the triangle condition, and hence the mean-field value of various critical exponents, has also been established for \emph{long-range percolation} in lower dimensions provided that the edge occupation probability is decaying slowly enough \cite{HeyHofSak08,  CheSak15,CheSak19}. 
A further question, which is open, is thus: Can strongly correlated (but short-range) percolation models have an upper critical dimension different from 6?   
A special instance (with negative answer to this question) is the \emph{cable system for the Gaussian Free field} on lattices. 
The very particular dependency structure of that model admits a rigorous analysis via  loop soups. Based on an explicit calculation of the two-point function of this model by Lupu \cite{Lup16}, Werner \cite{Wer21} pointed out that the upper critical dimension is indeed $6$ for this model, and recently Cai and Ding \cite{CaiDin24} even verified that the \emph{one-arm} critical exponent takes its critical value whenever $d>6$. 
The fact that this strongly correlated model shows mean-field behaviour as soon as $d>6$ is nevertheless most remarkable, as Drewitz, Pr{\'e}vost, and Rodriguez \cite{DrePreRod23} used links with potential theory to identify various critical exponents in dimension $d=3$, which differ from the (supected) exponents for independent edge occupation. 

\subsection{Complex networks} 
Random geometric graphs are versatile stochastic models for real networks. As many real-world networks exhibit certain stylised facts such as scale-free degrees, short distances and geometric clustering, it is clear that mathematical models of complex networks must share such features. While numerous different mathematical models for complex networks exist, it has been pointed out by \cite{GraHeyMonMor} that the \emph{weight-dependent random connection model} provides a unified approach that generalises many models for (infinite) geometric random graphs. 
This is a continuum random graph model whose vertices are given as a Poisson point process on $\R^d\times(0,1)$. If $x,y\in\Rd$ and $s,t\in(0,1)$, where edges $\big\{(x,s),(y,t)\}$ are inserted with probability given by a connection function $\connf((x,s),(y,t))$. A standard assumption is that $\connf$ is decreasing in the spatial distance $|x-y|$ as well as in both marks $s$ and $t$. 

Relevance of this model comes from the fact that different choices for the connection function $\connf$ generalise various well-known random graph models such as Boolean percolation, scale-free Gilbert graph, scale-free percolation, soft random geometric graphs, long-range percolation, and spatial preferential attachment graphs. Results proven in the generalised setup are thus valid in great generality.  

For some choices of $\connf$, these models are known to be \emph{robust}, i.e., infinite components exist for all edge densities; see \cite{GraLucMor,GraLucMon}. 
Our interest is in regimes where these graphs are non-robust and thus a proper percolation phase transition exist: if $\connf$ is small (scaled by a certain parameter) then there is no infinite cluster, but if $\connf$ is large, then an infinite cluster does appear. Investigation of such transitions is classical in percolation theory, and of great importance in modern probability. 

The present paper is the first one to prove an infrared bound for an inhomogeneous spatial percolation model. While such results are well-known for the homogeneous case (i.e., without marks), there are hardly any results for the percolation phase transition for inhomogeneous models. A notable exception is Boolean percolation, where sharp phase transition is known \cite{DumRaoTas20} and the mean-field bounds on three critical exponents is identified \cite{DewMui}. 

\bigskip
The key challenge is therefore, to overcome the dependencies imposed by the vertex weights. We prove the infrared bound under specific conditions, and verify these conditions for various concrete examples.

\subsection{The Marked Random Connection Model}
In this paper we consider a generalisation of the weight-dependent random connection model, which we first describe informally and refer to Section \ref{sec:formalConstruction} for a formal construction. It is for our purposes more convenient to generalise the setup by considering a general measurable mark space $\Ecal$ and thus define the model on the space $\mathbb{X} = \Rd\times \Ecal$. 
Let $\nu = \Leb \times \Pcal$ be a $\sigma$-finite measure on $\X$ where $\Leb$ denotes the Lebesgue measure on $\R^d$ and $\Pcal$ is a probability measure on $\Ecal$. Let $\lambda>0$, and then $\lambda \nu$ is an intensity measure on $\X$. The connection function is given by the measurable and symmetric $\connf\colon\X^2\to\left[0,1\right]$. We furthermore require that $\connf$ is translation and reflection invariant in the position components. That is, if we denote $x=\left(\xbar,a\right)$ and $y=\left(\ybar,b\right)$ for $\xbar,\ybar\in\R^d$ and $a,b\in\Ecal$, we require that $\connf\left(x,y\right) = \connf(\xbar-\ybar;a,b) = \connf(-\xbar+\ybar;a,b)$. 
We denote by $\eta$ the Poisson point process with intensity measure $\lambda\nu$ (which represents the vertex set of the random graph) and by $\xi$ the full random graph including additional randomness for the formation of edges.

Once we have the concept of two vertices being adjacent, we can consider the concept of two vertices being \emph{connected}. We say $x,y\in\X$ are connected, and write $x\longleftrightarrow y$, if they are equal or there exists a path in $\xi$ connecting the two vertices. By this we mean that there exists a sequence of distinct vertices $x = u_0, u_1, \ldots, u_k, u_{k+1} = y \in \eta$ (with $k \in\N_0$) such that $u_i \sim u_{i+1}$ for all $0 \leq i \leq k$. This definition can also be extended to connectedness on the RCM augmented by the vertex $x\in\X$ (resp. $x,y\in\X$), where $\left(\eta,\xi\right)$ is replaced by $\left(\eta^x,\xi^x\right)$ (resp. $\left(\eta^{x,y},\xi^{x,y}\right)$).
Therefore given a pair of points $x,y\in\X$ we can consider the event $\left\{\conn{x}{y}{\xi^{x,y}}\right\}$, and in turn the probability of this event. We define the pair-connectedness (or two-point) function $\tlam \colon \X^2 \to \left[0,1\right]$ to be
\begin{equation}
    \tlam\left(x,y\right) := \pla\left(\conn{x}{y}{\xi^{x,y}}\right).
\end{equation}

Using the connection and pair-connectedness functions $\connf,\tlam\colon\X^2 \to \left[0,1\right]$ as kernel functions, we construct associated linear integral operators. Where $L^2\left(\X\right)$ is the Hilbert space of square integrable functions on $\X$, we define the connection operator, $\Opconnf\colon L^2\left(\X\right) \to L^2\left(\X\right)$, and the pair-connectedness operator, $\Optlam\colon L^2\left(\X\right)\to L^2\left(\X\right)$. For $f\in L^2\left(\X\right)$, we have
\begin{equation}
    \Opconnf f \left(x\right) = \int \connf(x,y)f(y)\nu\left(\dd y\right), \qquad \Optlam f \left(x\right) = \int \tlam(x,y)f(y)\nu\left(\dd y\right),
\end{equation}
for $\nu$-almost every $x\in \X$. Note that the composition of two of these operators can be seen as a ``convolution over $\X$" of the kernel functions. Where we rarely use it, we reserve the $\star$ notation for straight-forward convolution over $\Rd$ only. For example, given $\xbar\in\Rd$ and $a_1,a_2,a_3,a_4\in \Ecal$,
\begin{equation}
    \left(\connf\left(\cdot;a_1,a_2\right)\star \connf\left(\cdot;a_3,a_4\right)\right)\left(\xbar\right) = \int_{\Rd}\connf\left(\xbar-\ybar;a_1,a_2\right)\connf\left(\ybar;a_3,a_4\right) \dd \ybar.
\end{equation}

We will also consider Fourier transformed versions of these operators. For a Lebesgue-integrable function $g\colon\Rd\to\R$, we define the Fourier transform of $g$ to be
\begin{equation}
    \widehat{g}(k) = \int \e^{ik\cdot\xbar}g\left(\xbar\right) \dd \xbar
\end{equation}
for $k\in\Rd$ where $k\cdot\xbar = \sum^d_{j=1}k_j\xbar_j$ denotes the standard inner product on $\Rd$. For given marks $a,b\in\Ecal$, we define $\fconnf(k;a,b)$ and $\ftlam(k;a,b)$ as the Fourier transforms of $\connf(\xbar;a,b)$ and $\tlam(\xbar;a,b)$ respectively. We use these $\Ecal^2\to\R$ functions to define linear operators on $L^2(\Ecal)$, the space of square integrable functions on $\Ecal$. For each $k\in\Rd$ we define the operators $\fOpconnf(k),\fOptlam(k)\colon L^2\left(\Ecal\right) \to L^2\left(\Ecal\right)$ by their action on $f\in L^2(\Ecal)$:
\begin{equation}
    \left[\fOpconnf(k) f\right] \left(a\right) = \int \fconnf(k;a,b)f(b)\Pcal\left(\dd b\right), \qquad \left[\fOptlam(k) f\right] \left(a\right) = \int \tlam(k;a,b)f(b)\Pcal\left(\dd b\right),
\end{equation}
for $\Pcal$-almost every $a$.

Our main aims of this paper are to derive an infrared bound for the two-point function (via its integral operators), and to prove that the so-called \emph{triangle condition} is satisfied by a large family of models in sufficiently high dimension. For $\lambda>0$ we define
\begin{equation}
\label{eqn:triangledef}
    \trilam := \lambda^2\esssup_{x,y\in\X}\int\tlam(x,u)\tlam(u,v)\tlam(v,y)\nu^{\otimes 2}\left(\dd u, \dd v\right).
\end{equation}
It is relatively simple to show that $\trilam$ is finite for sub-critical $\lambda$ and infinite for super-critical $\lambda$, but a model is said to satisfy the triangle condition if $\trilam$ is finite at criticality. 

An aside: the essential supremum taken in \eqref{eqn:triangledef} is with respect to the measure $\nu$ on $\X$. We will be taking the essential supremum many times in this paper over various spaces, and each space will naturally have a different measure implicitly associated with it. It should be clear from the context what that measure is.

\subsection{Overview}

In Section~\ref{sec:results} we provide the main results of the paper. We begin by making precise the idea of a critical intensity in Section~\ref{sec:Results_criticalintensities}. These propositions are largely independent of the rest of the paper and their proofs can be found in Appendix~\ref{sec:Critical_intensities}. The main concern of the paper is in proving the results of Section~\ref{sec:Results_twopointoperator}. Given the set of assumption \ref{Assump:2ndMoment}, \ref{Assump:Bound}, and \ref{Assump:BallDecay},  we prove that the triangle condition is satisfied and we provide an infrared bound on the two-point operator in terms of the adjacency operator. Some examples of models that satisfy these assumptions are provided in Section~\ref{subsection:Examples}, with the proofs that they satisfy them given in Appendix~\ref{sec:Model_properties}.

Before we begin with the proof proper, we give some preliminaries in Section~\ref{sec:preliminaries}. First Section~\ref{sec:formalConstruction} gives a formal construction of the RCM. Then we have Sections~\ref{sec:Prelim:Probability} and \ref{sec:Prelim:LinearOperators}, which give standard but important results in probability and linear algebra respectively that we will frequently use in our arguments. The probabilistic lemmas will be familiar to students of point process theory and percolation theory, with the caveat that we require them to apply with the augmented space $\X=\Rd\times\Ecal$ taking the place of the more usual $\Rd$. This adds no complication. The linear algebra results may well be foreign to a probabilistic audience, but will nevertheless be fundamental to our proof. Here we provide a sufficient technical description of the linear operators we construct and the Hilbert spaces upon which they act. Perhaps the most important result we call upon from this section is the Spectral Theorem (Theorem~\ref{thm:spectraltheorem}), which allows us to relate bounded self-adjoint linear operators to a multiplication operator by some unitary map. Multiplication operators are far simpler to work with, and we use the theorem many times towards the end of our proof. For operators on finite dimensional vector spaces, the spectral theorem is equivalent to saying that a Hermitian matrix can be diagonalized over $\R$. The proofs of the remaining lemmas in this subsection can be found in Appendix~\ref{appendix:LinearOperatorProofs}.

In Section~\ref{sec:expansion} we describe the finite-term lace expansion. This describes the two-point function in terms of the adjacency function, the lace expansion coefficient functions, and a remainder function. The presence of marks makes little difference here and this step is largely the same as for \cite{HeyHofLasMat19}. The novelty is that we formulate the resulting convolution equation as an operator equation.

Next we show that the finite-term lace expansion converges in the sub-critical regime to an Ornstein-Zernike equation (OZE). The first step (in Section~\ref{sec:diagrammaticbounds}) is to bound the lace expansion coefficient operators with long implicitly defined integral expressions which can be decomposed into simpler integral expressions like the triangle diagram (which is also implicitly defined). While the bounds derived here may not be surprising to those familiar with lace expansion arguments, the arguments in \cite{HeyHofLasMat19} make use of translation symmetry which does not apply if marks are present. The details of the proofs of these bounds can be found in Appendix~\ref{appendix:DiagrammaticBoundsProofs}. In Section~\ref{sec:bootstrapanalysis} these integral expressions are in turn bounded in terms of a bootstrap function. Once again, the marks introduce complications that mean that this is not a simple modification of the arguments in \cite{HeyHofLasMat19}. The proofs of the results in this section can be found in Appendix~\ref{appendix:BoundingwithBootstrap}. These bounds on the the lace expansion coefficient operators culminate in Proposition~\ref{thm:convergenceoflaceexpansion}, which proves that the OZE holds for $\lambda<\lambda_O$.

Section~\ref{sec:Bootstrap:ForbiddenRegion} takes this further and shows that the convergence in the sub-critical regime is in fact uniform in $\lambda$. This is shown by deriving a uniform upper bound on the bootstrap function by means of a forbidden-region argument in Proposition~\ref{thm:bootstrapargument}. A vital part of this argument is the differentiability of $\lambda\mapsto\Optlam$ and related functions. The proofs of these differentiability properties can be found in Appendix~\ref{sec:Prelim:differentiating}.

This uniform convergence allows for many results to be taken from the sub-critical regime into the critical regime in Section~\ref{sec:ConcludeProofs}. Proposition~\ref{lem:Triangle_Critical} uses this uniformity to prove extend the bound on the triangle diagram to $\lambda=\lambda_O$, while Proposition~\ref{lem:convergence_of_lace_expansion_corollary_lambda_T} shows that the bounds on the lace expansion coefficient operator and the OZE also extend to the critical threshold.

\section{Results}
\label{sec:results}

As we will be working with operators, we will require some concept of size for these operators. This will take a number of forms, as some will have advantages in various scenarios.

Let $\left(\Xcal,\mu\right)\in\left\{\left(\X,\nu\right),\left(\Ecal,\Pcal\right)\right\}$, and let $L^2\left(\Xcal\right)\in\left\{L^2\left(\X\right),L^2\left(\Ecal\right)\right\}$ denote the corresponding space of square integrable functions. In practice it will be clear which is being used. We will require five various ideas of size. First suppose $H\colon L^2\left(\Xcal\right) \to L^2\left(\Xcal\right)$ is a bounded linear operator. Then the \emph{operator norm} is given by
\begin{equation}
    \OpNorm{H} = \sup_{f\in L^2(\Xcal):f\ne 0}\frac{\norm*{Hf}_2}{\norm*{f}_2},
\end{equation}
where $\norm*{\cdot}_2$ is the standard $L^2$-norm on $L^2\left(\Xcal\right)$. In particular, we have $\OpNorm{H} = \sup \left\{\abs*{z}:z\in\sigma(H)\right\}$, where $\sigma\left(H\right)\subset \Complex$ denotes the spectrum of $H$. Let $\inner{\cdot}{\cdot}$ be the inner product on $L^2\left(\Xcal\right)$ defined by \eqref{eqn:innerproduct}. If $H$ is also self-adjoint, then $\inner{f}{Hf}\in\R$ and we define the \emph{spectral supremum}
\begin{equation}
    \label{eqn:supspecDef}
    \SupSpec{H} := \sup_{f\in L^2(\Xcal):f\ne 0}\frac{\inner{f}{Hf}}{\inner{f}{f}}.
\end{equation}
In particular, the self-adjointness of $H$ implies that $\sigma(H)\subset \R$ and $\SupSpec{H} = \sup \left\{z:z\in\sigma(H)\right\}$. Note that $\SupSpec{\cdot}$ is not a norm: it is not necessarily non-negative.

Now suppose $H\colon L^2(\Xcal)\to L^2(\Xcal)$ is a bounded linear integral operator with measurable kernel function $h\colon\Xcal^2 \to \Complex$. Then we have the explicit norms:
\begin{align}
    \OneNorm{H} &:= \esssup_{y\in\Xcal}\norm*{h\left(\cdot,y\right)}_1 =  \esssup_{y\in\Xcal}\int \abs*{h\left(x,y\right)}\mu\left(\dd x\right), \label{eqn:One_Infty_Norm}\\
    \TwoNorm{H} &:= \esssup_{y\in\Xcal}\norm*{h\left(\cdot,y\right)}_2 =  \esssup_{y\in\Xcal}\left(\int \abs*{h\left(x,y\right)}^2\mu\left(\dd x\right)\right)^\frac{1}{2},\label{eqn:Two_Infty_Norm}\\
    \InfNorm{H} &:= \esssup_{y\in\Xcal}\norm*{h\left(\cdot,y\right)}_\infty =  \esssup_{x,y\in\Xcal}\abs*{h\left(x,y\right)} \label{eqn:Infty_Infty_Norm}.
\end{align}

\subsection{Critical Intensities}
\label{sec:Results_criticalintensities}

Before we move on to our results describing behaviour at and around criticality, we will first explore how we are defining criticality. In our main results in Section~\ref{sec:Results_twopointoperator} we use the \emph{operator critical intensity} $\lambda_O$, defined by
\begin{equation}
    \lambda_O := \inf\left\{\lambda>0:\OpNorm{\fOptlam(0)}=\infty\right\}.
\end{equation}
Whilst the divergence of a two-point operator has been used previously to define a critical threshold in random graphs before (for example in \cite{bollobas2007phase, hutchcroft2019percolation, Hutch20}), at first glance it may seem to be quite far removed from the standard probabilistic interpretations of a percolation transition. A variety of transition thresholds for the continuum Boolean disc model are discussed in \cite{roy1990russo,Gou08,DumRaoTas20}. However, we are able to show that under certain conditions the operator critical intensity coincides with more familiar interpretations.

For $\lambda\geq 0$, we define the two functions $\theta_\lambda\colon\Ecal\mapsto \left[0,1\right]$ and $\chi_\lambda\colon\Ecal\mapsto \left[0,+\infty\right]$ as
\begin{align}
    \theta_\lambda(a) &:= \pla\left(\abs*{\C(\zerobar,a)}=\infty\right),\\
    \chi_\lambda(a) &:= \E_\lambda\left[\abs*{\C(\zerobar,a)}\right].
\end{align}
For each $p\in\left[1,+\infty\right]$, we define corresponding functions $\theta^{(p)}\colon\lambda\mapsto \norm{\theta_\lambda}_p$ and $\chi^{(p)}\colon\lambda\mapsto \norm{\chi_\lambda}_p$. These suggest the critical intensities
\begin{align}
    \lambda_c^{(p)} &:= \inf\left\{\lambda\geq 0: \norm{\theta_\lambda}_p >0\right\},\\
    \lambda_T^{(p)} &:= \inf\left\{\lambda\geq 0: \norm{\chi_\lambda}_p =\infty\right\}.
\end{align}
Note that the critical intensity $\lambda_T^{(\infty)}$ corresponds to the critical intensity $\lambda_{1\to 1}$ used in \cite{dickson2024nonuniqueness} and analogous to $p_{1\to 1}$ used in \cite{hutchcroft2019percolation,Hutch20}. On the other hand, $\lambda_O$ corresponds to to the critical intensity $\lambda_{2\to 2}$ and is analogous to $p_{2\to 2}$ used in these references.

Some relations between these intensities can be swiftly deduced without requiring the RCM to satisfy any special conditions.
\begin{itemize}
    \item Their definitions clearly show that $\theta_\lambda(a)>0$ implies $\chi_\lambda(a)=\infty$. Therefore $\lambda^{(p)}_c \geq \lambda^{(p)}_T$ for all $p\in\left[1,\infty\right]$.
    \item If $\norm*{\theta_\lambda}_p>0$ for some $p\in\left[1,\infty\right]$, then it is non-zero on some $\Pcal$-positive set, and therefore $\norm*{\theta_\lambda}_q>0$ for all $q\in\left[1,\infty\right]$. Therefore $\lambda_c^{(p)} = \lambda_c^{(1)}$ for all $p\in\left[1,+\infty\right]$, and we denote this value simply by $\lambda_c$.
    \item By Jensen's inequality, $1\leq p_1 \leq p_2 \leq \infty$ implies the ordering $\norm*{\chi_\lambda}_1 \leq \norm*{\chi_\lambda}_{p_1} \leq \norm*{\chi_\lambda}_{p_2} \leq \norm*{\chi_\lambda}_\infty$. Therefore we have $\lambda_T^{(1)} \geq \lambda^{(p_1)}_T \geq \lambda^{(p_2)}_T \geq \lambda^{(\infty)}_T$.
\end{itemize}

With some of the techniques developed in this paper it is possible to relate $\lambda_O$ to the intensities $\lambda^{(1)}_T$, $\lambda^{(\infty)}_T$, and $\lambda_c$ described here. 

\begin{prop}
\label{prop:lambda_T=lambda_0}
The critical intensities satisfy $\lambda^{(\infty)}_T \leq \lambda_O \leq \lambda^{(1)}_T \leq \lambda_c$. If
\begin{equation}
\label{eqn:supDegree}
    \esssup_{a,b\in\Ecal}\int\connf(\xbar;a,b)\dd \xbar <\infty,
\end{equation}
then we also have $\lambda^{(p)}_T = \lambda_O$ for all $p\in\left[1,\infty\right]$.
\end{prop}

This proof of Proposition~\ref{prop:lambda_T=lambda_0} uses Mecke's equation, Lemma~\ref{lem:BoundsonOperatorNorm}, Lemma~\ref{thm:Spectrum_Subset}, and some simple probabilistic ideas. The details can be found in Appendix~\ref{sec:Critical_intensities}.

The coincidence of critical intensities for percolation in marked random connection models has also been considered by \cite{men1988percolation,Gou08,DumRaoTas20} in the specific case of a family of Boolean defect models, including Boolean disc models with random bounded radii. In our terminology they introduced a new mark, $\dagger$, which forms connections as if it had radius $0$, and proved that $\theta_\lambda\left(\dagger\right) =0$ if and only if $\chi_\lambda\left(\dagger\right) < \infty$. Therefore if we define
\begin{equation}
    \lambda_c\left(\dagger\right) := \inf\left\{\lambda\geq 0: \theta_\lambda\left(\dagger\right) >0\right\}, \qquad \lambda_T\left(\dagger\right) := \inf\left\{\lambda\geq 0: \chi_\lambda\left(\dagger\right) = \infty\right\},
\end{equation}
this gives $\lambda_c\left(\dagger\right) = \lambda_T\left(\dagger\right)$. The following Proposition uses this to prove that for the Boolean disc model with bounded radii, all the critical intensities we have considered are equal. 

\begin{prop}
\label{prop:Boolean_Critical_Intensities}
In the case of the Boolean disc model with bounded radii,
\begin{equation}
    \lambda_c\left(\dagger\right)= \lambda_T\left(\dagger\right) = \lambda_c = \lambda_O = \lambda_T^{(p)},
\end{equation}
for all $p\in\left[1,\infty\right]$.
\end{prop}

To prove Proposition~\ref{prop:Boolean_Critical_Intensities} we show $\lambda_T\left(\dagger\right) \leq \lambda^{(p)}_T$ by considering points in $\eta$ adjacent to the root point with mark $\dagger$. The details can be found in Appendix~\ref{sec:Critical_intensities}.

\subsection{Two-Point Operator}
\label{sec:Results_twopointoperator}

The majority of this paper is concerned with describing the two-point function and the associated two-point operator. Our results will require Assumption \ref{Model_Assumption} described below to hold. This can be considered to be a generalisation of the ``finite variance" models described by \cite{HeyHofLasMat19} - extending to cases with multiple marks.

\begin{assumptionp}{\textbf{H}}
\label{Model_Assumption}
We require $\connf$ to satisfy the following properties: 
\begin{enumerate}[label=\textbf{(\ref{Model_Assumption}.\arabic*)}]
\item \label{Assump:2ndMoment} For all dimensions $d$ we have $\SupSpec{\fOpconnf(0)} < \infty$, and there exists a $d$-independent constant $C>0$ such that
\begin{align}
    \label{eqn:fconnfIsL2}
    \InfNorm{\fOpconnf(0)} &\leq C \SupSpec{\fOpconnf(0)},\\
\label{eqn:directional2ndMoment}
    \InfNorm{\fOpconnf(0) - \fOpconnf(k)} &\leq C \left(\SupSpec{\fOpconnf(0)} - \SupSpec{\fOpconnf(k)}\right).
\end{align}

\item \label{Assump:Bound}There exist $d$-independent constants $C_1\in\left(0,1\right)$ and $C_2>0$ such that
\begin{equation}
    \frac{\SupSpec{\fOpconnf\left(k\right)}}{\SupSpec{\fOpconnf\left(0\right)}} \leq C_1 \vee \left(1 - C_2\abs{k}^2\right).
\end{equation}
\item \label{Assump:BallDecay} There exists a function $g\colon\mathbb N \to \R_{\geq 0}$ with the following three properties. Firstly, that $g(d)\to 0$ as $d\to\infty$. Secondly, that
\begin{equation}
\label{eqn:ThreeAdjStep}
    \SupSpec{\fOpconnf(0)}^{-3}\esssup_{\xbar\in\Rd,a_1,\ldots,a_6\in\Ecal}\left(\connf(\cdot;a_1,a_2)\star \connf(\cdot;a_3,a_4) \star \connf(\cdot;a_5,a_6)\right)\left(\xbar\right) \leq g(d).
\end{equation}
Thirdly, that the family of sets $\left\{B\left(x\right)\right\}_{x\in\X}$ given by
\begin{equation}
\label{eqn:BallDecay}
    B\left(x\right) := \left\{y\in\X: \SupSpec{\fOpconnf(0)}^{-2}\int \connf\left(x,u\right)\connf\left(u,y\right)\nu\left(\dd u\right)  > g\left(d\right)\right\}
\end{equation}
satisfy $\esssup_{x\in\X}\nu\left(B\left(x\right)\right) \leq g\left(d\right)$.
\end{enumerate}
\end{assumptionp}

\begin{remark}
We have some remarks regarding these assumptions.
\begin{itemize}
    \item By Jensen's inequality, the condition  \eqref{eqn:fconnfIsL2} in \ref{Assump:2ndMoment} also implies that
\begin{equation}
    \label{eqn:Sup-Exp_Ratio}
    \OneNorm{\fOpconnf(0)} \leq \TwoNorm{\fOpconnf(0)} \leq C \SupSpec{\fOpconnf(0)}.
\end{equation}
    We will often use these bounds (the $\OneNorm{\cdot}$ bound in particular) more directly than \eqref{eqn:fconnfIsL2} itself.

    \item From \eqref{eqn:ThreeAdjStep} and \eqref{eqn:Sup-Exp_Ratio} we can use a supremum bound to deduce that analogous bounds hold for the convolution of $m$ adjacency functions (for $m\geq 3$) with corresponding bound $C^{m-3}g(d)$. 

    \item The sets $\left\{B\left(x\right)\right\}_{x\in\X}$ appearing in \ref{Assump:BallDecay} are all measurable - this follows from the measurability of $\connf$. Furthermore, the symmetry of $\connf$ implies that $y\in B\left(x\right)\iff x\in B\left(y\right)$.
    
    \item In stating our theorems, we will use a parameter $\beta=\beta(\connf, d)$ that depends upon $g(d)$. The definition depends upon the asymptotic properties of $g(d)$ in that
    \begin{equation}
        \beta(d) := \begin{cases}
            g(d)^{\frac{1}{4} -\frac{3}{2d}}d^{-\frac{3}{2}} &: \lim_{d\to\infty}g(d)\rho^{-d}\Gamma\left(\frac{d}{2}+1\right)^2 = 0 \qquad\forall\rho>0,\\
            g(d)^\frac{1}{4}&: \text{otherwise}.
        \end{cases}
    \end{equation}
    That is, $\beta(d) = g(d)^\frac{1}{4}$ unless $g(d)$ approaches zero particularly quickly. The parameter $\beta$ will reappear many times in this paper - often in the case where $\beta$ is `small.' This corresponds to `large' dimension $d$, and in particular, where the Landau notation $\LandauBigO{\beta}$ is used the asymptotics are as $d\to\infty$.
\end{itemize}
\end{remark}

\begin{theorem}\label{thm:OZEtheorem}
Suppose $\connf$ and $\nu$ satisfy the assumptions \ref{Assump:2ndMoment}, \ref{Assump:Bound}, and \ref{Assump:BallDecay}. Then there exists a $d^*>6$ and $C>0$ such that for $\lambda\in\left[0,\lambda_O\right]$ there exists an operator $\OpLacelam$ with $\OpNorm{\OpLacelam}\leq C\beta$ and a family of operators $\left\{\fOpLacelam(k)\right\}_{k\in\Rd}$ with $\OpNorm{\fOpLacelam(k)}\leq C\beta$, such that for all $d\geq d^*$ the following statements are satisfied.

\begin{itemize}
    \item For $\lambda\in\left[0,\lambda_O\right)$ the bounded operator $\Optlam$ satisfies the Ornstein–Zernike equation:
\begin{equation}
    \Optlam = \Opconnf + \OpLacelam + \lambda \Optlam \left( \Opconnf + \OpLacelam \right).
\end{equation}
    \item For $\lambda\in\left[0,\lambda_O\right)$ and $k\in\Rd$, or $\lambda = \lambda_O$ and $k\in\Rd\setminus\left\{0\right\}$, the bounded operator $\fOptlam(k)$ satisfies the Ornstein–Zernike equation in Fourier form:
\begin{equation}
    \fOptlam(k) = \fOpconnf(k) + \fOpLacelam(k) + \lambda \fOptlam(k) \left( \fOpconnf(k) + \fOpLacelam(k) \right).
\end{equation}
\end{itemize}
\end{theorem}

\begin{theorem}
\label{thm:InfraRed}
If $\connf$ and $\nu$ satisfy the assumptions \ref{Assump:2ndMoment}, \ref{Assump:Bound}, and \ref{Assump:BallDecay}, then there exist $d^*>6$ and $C>0$ such that for all $\lambda\in\left[0,\lambda_O\right]$ and $d\geq d^*$
\begin{align}
    \lambda\OpNorm{\fOptlam\left(k\right)} \leq \frac{\SupSpec{\fOpconnf\left(k\right)}+ C \beta}{\SupSpec{\fOpconnf\left(0\right)} - \SupSpec{\fOpconnf\left(k\right)}}\vee 1 \qquad\forall k\in\Rd\setminus\left\{0\right\}.
\end{align}
and $\trilam \leq C\beta^2$.
\end{theorem}

For the following two results, we call upon a strong form of irreducibility:

\begin{assumptionp}{\textbf{I}}
\label{assumption:StrongIrreducible}
``Some mark can be connected to every other mark in exactly $k$ steps for some $k$.'' That is,
    \begin{equation}
        \esssup_{a_0\in\Ecal}\,
        \sup_{k\geq 1}\,
        \essinf_{a_k\in\Ecal}
        \int\prod^k_{i=1}\fconnf\left(0;a_i,a_{i-1}\right)\Pcal^{\otimes (k-1)}\!\left(\dd \vec{a}_{\left[1,k-1\right]}\right)>0.
    \end{equation}
\end{assumptionp}

\begin{restatable}{prop}{NoCriticalPercolation}
Suppose assumptions \ref{Model_Assumption} and \ref{assumption:StrongIrreducible} hold. Then for $d\geq d^*$ the critical percolation function $\theta_{\lambda_O}(a)=0$ for $\Pcal$-almost every $a\in\Ecal$.    
\end{restatable}

This infrared bound and particularly the resulting triangle condition, have immediate consequences on various critical exponents. 
Such implications have been proven in the context of independent lattice percolations, see e.g.\  \cite{AizBar87,AizNew84,BarAiz91}; recently Caicedo and the first author \cite{CaiDic24} extended it to continuum percolation models with vertex weights: 

\begin{corollary}
If $\connf$ and $\nu$ satisfy the assumptions \ref{Model_Assumption} and \ref{assumption:StrongIrreducible}, then there exist $d^*>6$ such that for all $d\geq d^*$ there exist constants $C,C'>0$ such that 
    \begin{align}
    C\left(\lambda_c-\lambda\right)^{-1} \leq &\;\norm*{\chi_\lambda}_p \leq C'\left(\lambda_c-\lambda\right)^{-1}
    \quad \text{for $\lambda<\lambda_c$,}
    \\
    C\left(\lambda - \lambda_c\right) \leq &\;\norm*{\theta_\lambda}_p \leq C'\left(\lambda - \lambda_c\right)
    \qquad\quad \text{for $\lambda>\lambda_c$,}
    \\
    Cn^{\frac{1}{2}} \leq &\;\mathbb{P}_{\lambda_c}\left(\abs*{\C\left(0,{a}\right)}\geq n\right) \leq C'n^{\frac{1}{2}}
    \quad \text{for all $n\in\mathbb N$}
    \end{align}
for all $p\in\left[1,\infty\right]$ and $\mathcal P$-almost all $a\in\mathcal E$. 
In other words, the critical exponents $\gamma$, $\beta$ and $\delta$ (as defined in \cite{CaiDic24}) exist and take on their mean-field values 1, 1, and 2, respectively. 
\end{corollary}

The assertions are readily implied by Theorems 1.8, 1.11, 1.12 of \cite{CaiDic24}. Indeed, these theorems have three assumptions: Assumption (D1) (in that paper) is verified through our Hypothesis \ref{Assump:2ndMoment}, Assumption (D2) is precisely our Assumption \ref{assumption:StrongIrreducible}, and their Assumption (T) is precisely guaranteed by the triangle condition in Theorem \ref{thm:InfraRed}.

This infrared bound can also be related directly to the two-point function by Lemma~\ref{lem:BoundsonOperatorNorm} via
\begin{equation}
    \int\ftlam\left(k;a,b\right)\Pcal^{\otimes 2}\left(\dd a, \dd b\right) \leq \OpNorm{\fOptlam\left(k\right)}.
\end{equation}

\begin{remark}
Note that in Theorem~\ref{thm:OZEtheorem} we say there exists some $d^*>6$ such that the results hold for all $d\geq d^*$. If we were willing to say that there exists some $d^*>10$ instead (which would not change the strict logical content of the result), then the proof would be far simpler and Assumptions~\ref{Assump:2ndMoment} and~\ref{Assump:BallDecay} could be weakened. The $\InfNorm{\cdot}$ in \eqref{eqn:fconnfIsL2} and \eqref{eqn:directional2ndMoment} could be replaced with $\TwoNorm{\cdot}$ and $\OneNorm{\cdot}$ respectively, and \eqref{eqn:ThreeAdjStep} could be replaced with $\InfNorm{\Opconnf^3}\leq g(d)\SupSpec{\fOpconnf(0)}^3$. We do not pursue this approach because it is expected that a very similar argument to that presented here (with context appropriate changes to the assumptions) would prove the corresponding result for the spread-out model in $d>6$. For the simpler argument and weaker conditions, we would only get the spread-out result in $d>10$, and this \emph{would} be a logically weaker statement.
\end{remark}

\subsection{Examples of Models}
\label{subsection:Examples}
It is natural to ask if there are any models for which $\connf$ and $\nu$ do indeed satisfy \ref{Assump:2ndMoment}, \ref{Assump:Bound}, and \ref{Assump:BallDecay}. All the models we describe here also satisfy the condition \eqref{eqn:supDegree} for every dimension $d$.
We prove in Appendix \ref{sec:Model_properties} that all the models below do satisfy our assumptions. 

\paragraph{Single Mark Finite Variance Model.} Suppose $\Ecal$ is a singleton, so essentially $\X=\Rd$ and $\nu= \Leb$. Then the finite variance models described in \cite{HeyHofLasMat19} all satisfy our conditions. These include the Poisson blob model (or the spherical Boolean model) in which $\connf(\xbar) = \mathds{1}_{\{\abs*{\xbar} \leq r_d\}}$ where $r_d$ is the radius of the unit volume $d$-ball, and the Gaussian connection model with $\smash\connf(\xbar) = \left(2\pi\right)^{-d/2}\exp\big(-\tfrac{1}{2}\abs*{\xbar}^2\big)$.

\paragraph{Space-Mark Factorisation Model} Suppose now that the connection function admits the factorisation $\connf(\xbar;a,b) = \overline{\connf}(\xbar)K(a,b)$. We require that $\overline{\connf}$ satisfies the conditions of a single mark finite variance model, and that the ($d$-independent) mark-kernel $K\colon \Ecal^2 \to \left[0,1\right]$ is symmetric and measurable. Note that this means that the integral operator $\mathcal{K}\colon L^2(\Ecal) \to L^2(\Ecal)$ with kernel function $K$ is a bounded self-adjoint operator with operator norm $\OpNorm{\mathcal{K}}\leq 1$. We further suppose that $\SupSpec{\mathcal{K}} = \OpNorm{\mathcal{K}}$. Then the connection operator is a bounded operator on the mark-space multiplied by a scalar function of spatial position. This formulation places no constraint on the number of marks, permitting even models with uncountably many marks.

\paragraph{Marked Multivariate Gaussian Model.}
Let $\Sigma\colon\Ecal^2\to\R^{d\times d}$ be a measurable map where for every $a,b\in\Ecal$, $\Sigma\left(a,b\right)$ is itself a symmetric positive definite covariance matrix. We further require that there exists $\Sigma_{\max}<\infty$ and $\Sigma_{\min} = \Sigma_{\min}(d) > 0$ such that the spectrum $\sigma(\Sigma(a,b)) \subset \left[\Sigma_{\min},\Sigma_{\max}\right]$ for all $a,b\in\Ecal$ and dimension $d$. Then for $\mathcal{A} = \mathcal{A}(d)>0$ such that $\mathcal{A}^2 \leq \left(2\pi\right)^d\det \Sigma(a,b)$ for all $a,b\in\Ecal$ and $\limsup_{d\to\infty} \mathcal{A}^2\left(4\pi \Sigma_{\min}\right)^{-d/2} = 0$, let the connection function be given by
\begin{equation}
    \connf(\xbar;a,b) = \mathcal{A}\left(2\pi\right)^{-d/2}\left(\det \Sigma(a,b)\right)^{-1/2}\exp\left(-\frac{1}{2}\xbar^{\intercal}\Sigma(a,b)^{-1}\xbar\right).
\end{equation}
This class of model includes examples of models with (uncountably) infinitely many marks, and for which the mark and spatial behaviours are truly coupled. For a concrete example, consider $\Ecal=\left[\tfrac{1}{4\pi},1\right]$ with $\Sigma(a,b) = (a+b)\Id$ and $\mathcal{A} = 1$.

\paragraph{Bounded-Volume Boolean Disc Model.}
Let $\left(\left[0,1\right], \Sigma_{\left[0,1\right]}, \Pcal\right)$ be a separable probability space and $\left\{R_d\right\}_{d\geq 1}$ be a sequence of measurable functions $R_d\colon\left[0,1\right]\to\R_+$ such that there exist $ 0<c_1<1/\sqrt{8\pi\e}$ and $\Vmax>0$ such that
\begin{equation}\label{eqn:RadiusBounds}
    \Rmin := c_1 d^{\frac{1}{2}} \leq R_d\left(a\right) \leq \frac{1}{2\sqrt{\pi}}\left(\Vmax\right)^{\frac{1}{d}} \Gamma\left(\frac{d}{2}+1\right)^\frac{1}{d} =: \Rmax,
\end{equation}
for all $a\in[0,1]$. Note that as $d\to\infty$, Stirling's formula gives that the upper bound $\Rmax\sim \sqrt{d/8\pi\e}$. It is also required that there exist $\epsilon>0$ and $c_2>0$ such that 
\begin{equation}
\label{eqn:MeanRadiusBound}
    \Pcal\left(R_d > \left(1-\frac{c_2}{d}\right)\Rmax\right)>\epsilon
\end{equation}
for all $d$. Then let the connection function be defined by
\begin{equation}
    \connf_d(\xbar;a,b) = \mathds{1}\left\{\abs*{\xbar} \leq R_d(a) + R_d(b)\right\}.
\end{equation}

\vspace{5mm}

It is also perhaps worth mentioning a model that \emph{does not} satisfy condition \ref{Model_Assumption}. Such an example would be a Boolean disk model with random radii for which $\Pcal$ gives positive measure to arbitrarily large radii. No matter the decay rate of $\Pcal$ for large radii, the presence of the suprema in \eqref{eqn:fconnfIsL2} and \eqref{eqn:directional2ndMoment} ensure that \ref{Assump:2ndMoment} cannot be satisfied. Similarly the condition \eqref{eqn:supDegree} is also not satisfied for such a model.

\section{Preliminaries} \label{sec:preliminaries}

\subsection{Formal Construction of the RCM}\label{sec:formalConstruction}

We choose to construct our RCM in largely the same way as \cite{HeyHofLasMat19}, with appropriate replacements of $\Rd$ with the marked space $\X$. As they did, we follow \cite{LasZie17} in defining the RCM, $\xi$, as an independent edge-marking of a Poisson point process $\eta$.

Recall $\X=\R^d\times\Ecal$. We denote by $\mathbf N\left(\X\right)$ the set of all at most countably infinite subsets of $\X$ and accompany it with the $\sigma$-algebra $\mathcal N(\mathbb{X})$ generated by the
sets $\{A\in\mathbf{N}\left(\X\right):|A \cap B| = k, B \in\Sigma_\X, k \in \N_0\}$, 
where $\abs*{A\cap B}$ denotes the cardinality of $A\cap B$. Then a point process on $\X$ is a measurable mapping $\zeta\colon\Omega\to \mathbf N\left(\X\right)$ for some underlying probability space $\left(\Omega,\Sigma_\Omega,\mathbb{P}\right)$. The intensity measure of such a point process is then the measure on $\X$ given by $B\mapsto \E\left[\abs*{\zeta\cap B}\right]$ for $B\in\Sigma_\X$. Let $\zeta$ be a point process on $\X$ with a $\sigma$-finite intensity measure. Then there exist measurable $\pi_i\colon\mathbf N \left(\X\right) \to \X$, $i\in\N$ such that almost surely
\begin{equation}\label{eq:orderPPP}
    \zeta = \left\{\pi_i\left(\zeta\right): i\leq \abs*{\zeta\cap\X}\right\}.
\end{equation}

For a $\sigma$-finite non-atomic measure on $\X$, say $\nu$, a Poisson
point process (PPP) on $\mathbb{X}$ with intensity measure $\nu$ is a point
process $\zeta$ such that the number of points $\abs*{\zeta \cap B}$ is
$\text{Poi}(\nu(B))$-distributed for each $B \in \Sigma_\mathbb{X}$, and the
random variables $\abs*{\zeta \cap B_1}, \ldots, \abs*{\zeta \cap B_m}$ are independent
whenever $B_1, \ldots, B_m\in\Sigma_\X$ are pairwise disjoint. In this paper we will be considering a PPP $\eta$ in $\X$ with intensity measure $\lambda \nu$, where $\nu=\Leb\otimes \Pcal$ and $\lambda>0$. We write $\eta=\left\{X_i:i\in\N\right\}$, where $X_i=\pi_i\left(\eta\right)$ for $i\in\N$.

We now construct the RCM as a (non-Poissionian) point process on the space $\X^{[2]}\times\left[0,1\right]$. We denote by $\X^{[2]}$ the set of all subsets of $\X$ containing precisely two elements, and each $e\in\X^{[2]}$ is then a potential edge of the RCM. Also let $\left\{U_{i,j}\right\}_{i,j\in\N}$ be a double sequence of independent identically distributed random variables uniformly distributed on $\left[0,1\right]$, such that the sequence is independent of the PPP $\eta$. We define
\begin{equation}
    \xi = \left\{\left(\left\{X_i,X_j\right\},U_{i,j}\right):1\leq i\leq j\right\},
\end{equation}
where we recall $X_i=\pi_i(\eta)$. Thus $\xi$ is a random point process on $\X^{[2]}\times\left[0,1\right]$, which we interpret as a random marked graph and say that $\xi$ is an {\em independent edge-marking} of $\eta$.

The independent edge-marking $\xi$ of $\eta$ can then be related by a deterministic functional, $\Gamma_\connf$, to a random graph on $\X$ by defining its vertex and edge sets:
\begin{equation}
    V(\Gamma_\connf(\xi)) = \eta = \{X_i:i\in\N\},\qquad
	E(\Gamma_\connf(\xi)) = \{ \{X_i,X_j\}:U_{i,j} \leq \connf(X_i,X_j), 1\leq i\leq j\}.
\end{equation}
Since the sequence $U_{i,j}$ is independent of the Poisson process $\eta$, the ordering in \eqref{eq:orderPPP} does not change the distribution of the resulting random graph. 

This interpretation of the RCM also allows for the introduction of additional points. This will be required many times in this paper, not least in our definition of the two-point function itself. For additional points $x_1,x_2\in\X$, the point process $\eta$ is augmented to make
\begin{equation}
    \eta^{x_1}:=\eta\cup\{x_1\},\quad \eta^{x_1,x_2}:=\eta\cup\{x_1,x_2\}.
\end{equation}
To extend these point processes on $\X$ to a RCM, we extend the sequence $\left\{U_{m,n}\right\}_{m,n \geq 1}$ to a sequence $\left\{U_{m,n}\right\}_{m,n \geq -1}$ of independent random variables uniformly distributed on $[0,1]$, independent of the Poisson process $\eta$. We then define $\xi^{x_1,x_2}$ on $\X^{[2]}\times[0,1]$ as 
\begin{equation}
    \xi^{x_1,x_2}:=\{(\{X_i,X_j\},U_{i,j}):-1\leq i\leq j\},
\end{equation}
where $X_0 = x_1$ and $X_{-1} = x_2$. It is straightforward to define $\xi^{x_1, \ldots, x_m}$ for arbitrary $m \geq 3$ by the same idea, and to define $\xi^{x_1}$ by removing all (marked) edges incident to $x_2$ from $\xi^{x_1,x_2}$. 
We assume implicitly that the augmented points are properly coupled such that, e.g., $\xi^{x_1}$ obtained in this way from $\xi^{x_1,x_2}$ (by removing all edges adjacent to $x_2$) is the same as if obtained from $\xi^{x_1,x_3}$ (by removing all edges adjacent to $x_3$); see \cite[Section 3.2]{HeyHofLasMat19} for further details. 

\subsection{Probabilistic Lemmas}
\label{sec:Prelim:Probability}

We introduce here a number of important lemmas that are standard in point process theory or percolation theory.

\subsubsection{The Mecke Equation}
An important part of our analysis will be a version of the Mecke equation - for a discussion see \cite[Chapter~4]{LasPen17}. This result allows us to make sums over the random sum of points in $\eta$ more tractable. Given $m\in\N$ and a measurable function $f\colon\mathbf{N}\left(\X^{[2]}\times[0,1]\right) \times \X^m \to \R_{\geq 0}$, the Mecke equation for $\xi$ states that 
\begin{equation}
    \E_\lambda \left[ \sum_{\vec x \in \eta^{(m)}} f(\xi, \vec{x})\right] = \lambda^m \int
				\E_\lambda\left[ f\left(\xi^{x_1, \ldots, x_m}, \vec x\right)\right] \nu^{\otimes m}\left(\dd \vec{x}\right),  \label{eq:prelim:mecke_n}
\end{equation}
where $\vec x=(x_1,\ldots,x_m)$, $\eta^{(m)}=\{(x_1,\ldots,x_m): x_i \in \eta, x_i \neq x_j \text{ for } i \neq j\}$, and $\nu^{\otimes m}$ is the $m$-product measure of $\nu$ on $\X^m$. We will only need \eqref{eq:prelim:mecke_n} for $m\leq 3$, and largely only use it for $m=1$.

\subsubsection{The Margulis-Russo Formula}
An often used and very important tool in (discrete) percolation theory, the Margulis-Russo formula will be widely used in our analysis as well. The version we use follows from a slight adjustment to the more general result given in \cite[Theorem~3.2]{LasZie17}. Writing $\mathbf{N}:=\mathbf{N}\left(\X^{[2]}\times[0,1]\right)$, let $\Lambda\subset\X$ be $\nu$-finite, $\zeta\in\mathbf{N}$, and define
\begin{equation}
    \zeta_\Lambda := \{(\{x,y\},u)\in \zeta :\{x,y\}\subseteq\Lambda \}.  \label{eq:prelim:def:PP_restriction}
\end{equation}
We call $\zeta_\Lambda$ the restriction of $\zeta$ to $\Lambda$. We say that $f\colon \mathbf N\to\R$ \emph{lives} on $\Lambda$ if $f(\zeta) = f(\zeta_\Lambda)$ for every $\zeta \in\mathbf N$. Assume that there exists a $\nu$-finite set $\Lambda \subset \X$ such that $f$ lives on $\Lambda$. Moreover, assume that there exists  $\lambda_0>0$ such that $\E_{\lambda_0} \left[\abs*{f(\xi)}\right] < \infty$. Then the Margulis-Russo formula states that, for all $\lambda\leq \lambda_0$,
\begin{equation}
    \frac{\partial}{\partial \lambda} \E_\lambda[f(\xi)] = \int_\Lambda \E_\lambda\left[f(\xi^x) - f(\xi)\right] \nu\left(\dd x\right). \label{eq:prelim:russo}
\end{equation}
The result in \cite{LasZie17} requires that $\X=\Rd$ and that $\Lambda\subset\R^d$ is compact, but \cite[Theorem~3.1]{last2014perturbation} allows us to make the necessary adjustments to permit $\X=\Rd\times\Ecal$.

\subsubsection{The BK Inequality}

The idea of the BK inequality is that we acquire a simple upper bound for the probability that two increasing events happen, and that they do so on disjoint subsets of the space. For example, we may want to bound the probability that there exist two disjoint paths (up to the start and end points) in $\xi$ between two given points. The BK inequality we present here is slightly more general than the usual BK inequality in that the point process $\xi$ is augmented with independent random variables. It differs only slightly from the version appearing in \cite{HeyHofLasMat19}.

For notational clarity, we write $\mathbf{N}:=\mathbf{N}(\X^{[2]}\times[0,1])$ and denote the $\sigma$-algebra on $\mathbf{N}$ by $\mathcal{N}$. We call a set $E\subset \mathbf{N}$ \emph{increasing} if $\mu \in E$ implies $\nu\in E$ for each $\nu\in\mathbf{N}$ with $\mu\subseteq\nu$. Let $(\mathbb{Y}_1,\mathcal{Y}_1)$, $(\mathbb{Y}_2,\mathcal{Y}_2)$ be two measurable spaces.  We say that a set $E_i\subset \mathbf{N}\times \mathbb{Y}_i$ is {\em increasing} if $\E^z_i:=\{\mu\in\mathbf{N}:(\mu,z)\in E_i\}$ is increasing for each $z\in \mathbb{Y}_i$.

For Borel $\Lambda\in \B\left(\Rd\right)$ and $\mu\in \mathbf{N}$, we define $\mu_\Lambda$, the restriction of $\mu$ to all edges completely contained in $\Lambda\times \Ecal$, analogously to~\eqref{eq:prelim:def:PP_restriction}. Also let $\mathcal{R}$ denote the ring of all finite unions of  half-open $d$-dimensional rectangles in $\Rd$ with rational coordinates. Then for increasing $E_i\in\mathcal{N}\otimes\mathcal{Y}_i$, we define
\begin{multline}
    E_1 \circ E_2 :=  \left\{(\mu,z_1,z_2)\in\mathbf{N}\times\mathbb{Y}_1\times\mathbb{Y}_2: \exists K_1,K_2 \in \mathcal{R} \text{ s.t.~} \right.\\\left.
			K_1\cap K_2 = \varnothing, (\mu_{K_1},z_1)\in E_1,(\mu_{K_2},z_2)\in E_2   \right\}. \label{eq:prelim:circ2}
\end{multline}
A set $E_i\in\mathcal{N}\otimes \mathcal{Y}_i$ {\em lives on} $\Lambda$ if $\mathds 1_{E_i}(\mu,z)=\mathds 1_{E_i}(\mu_\Lambda,z)$ for each $(\mu,z)\in\mathbf{N}\times\mathbb{Y}_i$. We consider random elements $W_1,W_2$ of $\mathbb{Y}_1$ and $\mathbb{Y}_2$, respectively, and assume that $\xi$, $W_1$, and $W_2$ are independent.

\begin{theorem}[BK inequality] \label{lem:prelimbk_inequality}
Let $E_1\in\mathcal{N}\otimes\mathcal{Y}_1$  and $E_2\in\mathcal{N}\otimes\mathcal{Y}_2$ be increasing events that live on $\Lambda\times\Ecal$ for some bounded set $\Lambda\in\B\left(\Rd\right)$. Then 
\begin{equation}
    \pla((\xi,W_1,W_2)\in E_1 \circ E_2) \leq \pla((\xi,W_1)\in E_1)\pla((\xi,W_2)\in E_2).
\end{equation}
\end{theorem}

A proof for this theorem can be found in \cite{HeyHofLasMat19}. The presence of the marked space $\X=\Rd\times\Ecal$ rather than just $\Rd$ makes no difference to the validity of their proof.
It is worth noting that the operation $\circ$ is commutative and associative, which allows repeated application of the inequality.

\subsubsection{The FKG Inequality}

In contrast to the BK inequality, the FKG inequality gives us a simple lower bound on the probability of two increasing events occurring.

Given two increasing and integrable functions $f,g$ on $\mathbf{N}\left(\X^{[2]}\times[0,1]\right)$, we have
\begin{equation}
    \E_\lambda\left[f(\xi) g(\xi)\right] = \E_\lambda \left[ \E_\lambda[f(\xi)g(\xi) \mid \eta] \right] \geq \E_\lambda \left[ \E_\lambda\left[f(\xi)\mid \eta\right] \; \E_\lambda\left[g(\xi)\mid \eta\right] \right]
				\geq \E_\lambda\left[f(\xi)\right] \; \E_\lambda\left[g(\xi)\right]. \label{eq:prelim:FKG}
\end{equation}
The first inequality was obtained by applying FKG to the random graph conditioned to have $\eta$ as its vertex set, the second inequality by applying FKG for point processes (see \cite{LasPen17} for details).

\subsection{Linear Operator Lemmas}
\label{sec:Prelim:LinearOperators}

We shall formulate our problem as one of linear operators acting on Hilbert spaces. We will want to define analogous objects on the full marked space $\X=\Rd\times\Ecal$ and on the mark space $\Ecal$ itself. Let $\left(\Xcal,\Sigma_\Xcal,\mu\right)\in\left\{\left(\X,\Sigma_\X,\nu\right),\left(\Ecal,\Sigma_\Ecal,\Pcal\right)\right\}$, and $L^2\left(\Xcal\right)\in\left\{L^2\left(\X\right),L^2\left(\Ecal\right)\right\}$ correspondingly.

Given the measure space $\left(\Xcal,\Sigma_\Xcal,\mu\right)$, we consider the Banach space of square-integrable functions $L^2(\Xcal) = L^2\left(\Xcal,\Sigma_{\Xcal},\mu\right)$. We augment this with the inner product defined by
\begin{equation}
\label{eqn:innerproduct}
    \inner{f}{g} := \int \overline{f\left(x\right)}g\left(x\right) \mu\left(\dd x \right),
\end{equation}
for all $f,g\in L^2(\Xcal)$, where $\overline{f(x)}$ is the complex conjugate of $f(x)$. With this inner product, $L^2(\Xcal)$ is a Hilbert space.

The measure space $\left(\Xcal,\Sigma_\Xcal,\mu\right)$ is separable if there exists a countable family $\left\{E_k\right\}_{k=1}^\infty\subset \Sigma_\Xcal$ such that given $E\in\Sigma_\Xcal$ with $\mu(E)<\infty$ there exists an $E$-dependent subsequence $\{n_k\}_{k=1}^\infty$ such that
\begin{equation}
    \mu\left(E\Delta E_{n_k}\right) \to 0, \qquad\text{ as }k\to\infty,
\end{equation}
where $\Delta$ here is the symmetric difference. A well-known example of a separable measure space is $\left(\R^d,\mathfrak{B},\Leb\right)$ - where $\mathfrak B$ is the Borel $\sigma$-algebra - and therefore if the probability measure space $\left(\Ecal,\Sigma_{\Ecal},\Pcal\right)$ is separable then so is the product measure space $\left(\X,\Sigma_\X,\nu\right)$. A topological space is called separable if there exists a countable dense subset. If the measure space $\left(\Xcal,\Sigma_\Xcal,\mu\right)$ is separable, then the function spaces $L^p\left(\Xcal\right) = L^p\left(\Xcal,\Sigma_\Xcal,\mu\right)$ with $p\in\left[1,\infty\right)$ are all separable (see, for example, \cite{stein2011functional}). Since we have assumed that $\left(\Ecal,\Sigma_\Ecal,\Pcal\right)$ is separable, we know that both $L^2(\X)$ and $L^2(\Ecal)$ are separable Hilbert spaces.

The inner product for $L^2(\Xcal)$ allows us to define the $L^2$-norm for $f\in L^2(\Xcal)$:
\begin{equation}
    \norm*{f}_2 = \inner{f}{f}^{\frac{1}{2}}.
\end{equation}
This allows us to define a norm on the space of bounded linear operators on $L^2(\Xcal)$. Suppose we have a linear operator $H\colon \mathcal{D}\left(H\right) \to L^2\left(\Xcal\right)$ for some domain $\mathcal{D}\left(H\right)\subset L^2\left(\Xcal\right)$. Then we can define the \emph{operator norm} by
\begin{equation}
    \OpNorm{H} := \begin{cases}
        \sup_{f\in L^2(\Xcal):f\ne 0}\frac{\norm*{Hf}_2}{\norm*{f}_2} &\text{if }\mathcal{D}\left(H\right) = L^2\left(\Xcal\right),\\
        +\infty &\text{otherwise},
    \end{cases}
\end{equation}
and $H$ is bounded if $\OpNorm{H}<\infty$. The space of bounded linear operators on a Banach space is itself a Banach space when augmented with this operator norm. Since $H$ is linear, the supremum can equivalently be taken over $f\in L^2(\Xcal)$ such that $\norm*{f}_2=1$. 

If $H$ is defined as the integral operator with kernel function $h\colon \Xcal^2\to \R$, then we can take
\begin{equation}
    \mathcal{D}\left(H\right) = \left\{f\in L^2\left(\Xcal\right)\colon \esssup_{x\in\Xcal}\int_\Xcal\abs*{h\left(x,y\right)}\abs*{f(y)}\mu\left(\dd y\right)<\infty\right\},
\end{equation}
since this ensures that $(Hf)(x)$ is defined for $\mu$-almost every $x\in\Xcal$ by dominated convergence.

Note that $\OpNorm{\cdot}$ is sub-multiplicative, in the sense that for linear operators $H\colon \mathcal{D}(H)\to \mathcal{D}(G)$ and $G\colon \mathcal{D}(G)\to L^2(\Xcal)$ we have
\begin{equation}
    \OpNorm{GH} \leq \OpNorm{G}\OpNorm{H}.
\end{equation}

Given a linear operator $H\colon L^2(\Xcal)\to L^2(\Xcal)$ on Hilbert spaces, the \emph{adjoint} of $H$ is the operator $H^\dagger\colon L^2(\Xcal)\to L^2(\Xcal)$ such that
\begin{equation}
    \inner{f}{Hg} = \inner{H^\dagger f}{g}
\end{equation}
for all $f,g\in L^2(\Xcal)$. An operator is \emph{self-adjoint} if $H=H^\dagger$. It is worth noting that a bounded integral operator $H$ with kernel function $h$ is self-adjoint if and only if $h(x,y) = \overline{h(y,x)}$ for $\mu$-a.e. $x,y\in\Xcal$. If $h$ is real valued, this is equivalent to $h$ being symmetric on such a set of $x,y$. A simple calculation shows that if $H$ is bounded and self-adjoint, then $\inner{f}{Hf}\in\R$ for all $f\in L^2(\Xcal)$. This allows us to define the following object. Given a bounded and self-adjoint linear $H\colon L^2(\Xcal)\to L^2(\Xcal)$, define
\begin{equation}
\label{eqn:SupSpecDef}
    \SupSpec{H} := \sup_{f\in L^2(\Xcal):f\ne 0}\frac{\inner{f}{Hf}}{\inner{f}{f}}.
\end{equation}
Note that, in contrast to $\OpNorm{\cdot}$, the spectral supremum $\SupSpec{\cdot}$ is not a norm - it is possible for it to be negative. 

The objects $\OpNorm{H}$ and $\SupSpec{H}$ can also be understood via the spectrum of $H$. The spectrum, denoted $\sigma(H)$, is the set of $z\in\Complex$ such that $z\Id - H$ does not have a bounded linear inverse. For linear operators on finite dimensional vector spaces this coincides with the set of eigenvalues of $H$. If $H$ is self-adjoint, then $\sigma(H)\subset\R$. We can then express $\OpNorm{H}$ and $\SupSpec{H}$ in terms of the spectrum:
\begin{equation}
\label{eqn:Opnorm_Supspec_Spectrum}
    \OpNorm{H} = \sup \left\{\abs*{z}:z\in\sigma(H)\right\},\qquad \SupSpec{H} = \sup \left\{z:z\in\sigma(H)\right\}.
\end{equation}
Note that for self-adjoint operators the operator norm is equal to the spectral radius of the operator.

While $\OpNorm{\cdot}$ and $\SupSpec{\cdot}$ are very important properties of the operators we will be considering, they are not necessarily simple to compute for an integral operator given its kernel function. Recall that in \eqref{eqn:One_Infty_Norm}-\eqref{eqn:Infty_Infty_Norm} we defined the three integral operator norms $\OneNorm{\cdot}$, $\TwoNorm{\cdot}$, and $\InfNorm{\cdot}$. These have the advantage that they are explicitly defined in terms of the kernel function. We now provide some results regarding the four norms (plus $\SupSpec{\cdot}$), including relations between some of them.

\begin{lemma}[Bounds on the Operator Norm]
\label{lem:BoundsonOperatorNorm}
Let $H\colon L^2\left(\Xcal\right)\to L^2\left(\Xcal\right)$ be a bounded self-adjoint integral operator. Then we can bound the operator norm:
\begin{equation}
\label{eqn:OpnormBounds}
    \abs*{\SupSpec{H}} \leq \OpNorm{H} \leq \OneNorm{H}.
\end{equation}
Furthermore, if $\mathcal{X}=\Ecal$ and $H$ has kernel function $h$, then
\begin{equation}
    \int h\left(a,b\right)\Pcal^{\otimes 2}\left(\dd a,\dd b\right) \leq \SupSpec{H}
\end{equation}
\end{lemma}

\begin{lemma}
\label{lem:NormBounds}
Suppose $h,g\colon \Xcal^2\to\Complex$ are kernel functions for the self-adjoint linear operators $H,G\colon L^2\left(\Xcal\right) \to L^2\left(\Xcal\right)$, such that $h\left(x,y\right)\in\R_{\geq0}$ and $h\left(x,y\right) \geq \abs*{g\left(x,y\right)}$ for  $\mu$-a.e. $x,y\in\Xcal$. Then
\begin{equation}
    \SupSpec{H}\geq\SupSpec{G},\quad \OpNorm{H}\geq \OpNorm{G}, \quad \OneNorm{H}\geq \OneNorm{G}, \quad \TwoNorm{H}\geq \TwoNorm{G}, \quad \InfNorm{H}\geq \InfNorm{G}.
\end{equation}
\end{lemma}

\begin{lemma}
\label{lem:SupSpecTriangle}
Let $G,H$ be bounded self-adjoint linear operators on a separable Hilbert space. Then
\begin{equation}
    \abs*{\SupSpec{G+H} - \SupSpec{G}} \leq \OpNorm{H}.
\end{equation}
\end{lemma}

The main advantage of the Hilbert space formulation for the operators is that it allows us to use the following theorem to better understand the operators. Although the ``Multiplication Operator" version presented here may not provide as much information as the ``Projection-Valued Measure" or ``Direct Integral" versions (see \cite{hall2013quantum}), it has the advantage of being simpler to state and is sufficient for our purposes.

\begin{theorem}[Spectral Theorem - Multiplication Operator Version]\label{thm:spectraltheorem}
Let $\left\{H_i\right\}_{i\in I}$ be a commutative family of bounded, self-adjoint operators on a separable Hilbert space $\Hilbert$. Then there exists a $\sigma$-finite measure
space $\left(\Omega,\Sigma_\Omega, \mu\right)$, a family of bounded, measurable, real-valued functions $\left\{h_i\right\}_{i\in I}$ on $\Omega$, and a unitary map $U \colon  \Hilbert \to L^2\left(\Omega,\Sigma_\Omega, \mu\right)$ such that for $\mu$-a.e. $\omega$,
\begin{equation}
    \left(UH_iU^{-1}f\right)\left(\omega\right) = h_i\left(\omega\right)f\left(\omega\right)
\end{equation}
for all $f\in L^2\left(\Omega,\Sigma_\Omega. \mu\right)$ and $i\in I$.
\end{theorem}

\begin{proof}
For a proof in more generality, see \cite[Theorem~1.47]{folland2016course}.
\end{proof}

In the context of a finite dimensional Hilbert space $\Hilbert$, the Spectral Theorem is effectively just the statement that Hermitian matrices can be diagonalized over $\R$. Here $U$ is represented by a unitary matrix composed from the orthonormal eigenbasis of $\left\{H_i\right\}$, and the entries of the resulting diagonal matrix corresponding to $H_i$ are the values of the function $h_i(\omega)$. For convolution operators on functions on $\Rd$, the Fourier transform plays the role of $U$ and turns the complicated convolution operation into simple pointwise multiplication.

We also present here a lemma that relates the operator norm and spectral supremum for operators on $L^2(\X)$ to operators on $L^2(\Ecal)$. Given a family $\left\{A_k\right\}_{k\in\Rd}$ of subsets of $\Complex$, we define the \emph{essential union} of these sets to be
\begin{equation}
    \Ess \bigcup_{k\in\Rd}A_k := \left\{z\in\Complex:\forall \epsilon>0, \Leb\left(k\in\Rd:\exists x\in A_k \text{ s.t. }\abs*{x-z}<\epsilon\right)>0\right\}.
\end{equation}
This can be understood as the union of $A_k$ over $k$ neglecting null sets of $k$.

\begin{lemma}
\label{thm:Spectrum_Subset}
We have
\begin{equation}
\label{eqn:spectrum_equality}
    \sigma\left(\Opconnf\right) = \Ess \bigcup_{k\in\Rd}\sigma\left(\fOpconnf(k)\right) , \qquad \sigma\left(\Optlam\right) = \Ess\bigcup_{k\in\Rd}\sigma\left(\fOptlam(k)\right).
\end{equation}
It then follows that
\begin{equation}
    \OpNorm{\Opconnf} = \esssup_{k\in\Rd}\OpNorm{\fOpconnf\left(k\right)}, \qquad \OpNorm{\Optlam} = \esssup_{k\in\Rd}\OpNorm{\fOptlam(k)},
\end{equation}
\begin{equation}
    \SupSpec{\Opconnf} = \esssup_{k\in\Rd}\SupSpec{\fOpconnf(k)}, \qquad \SupSpec{\Optlam} = \esssup_{k\in\Rd}\SupSpec{\fOptlam(k)}.
\end{equation}
\end{lemma}

\section{The Expansion}\label{sec:expansion}

There are two parts to this section. First we recount an expansion of the two-point function. This argument is essentially identical to the corresponding step in \cite{HeyHofLasMat19}, so we refer to this reference for most of the details. The novelty in this section is rather in the formulation of this expansion in terms of operators, contained in the second part.

\subsection{Function-level expansion}
\label{sec:expansion:Functions}

The expansion of the two-point function $\tlam$ proceeds in essentially the same way as in \cite{HeyHofLasMat19}. The similarity holds because the argument uses quite general properties of Poisson point processes and connection models - the change from $\Rd$ to $\X=\Rd\times\Ecal$ adds no further complication.  We give here a very brief overview of the derivation - the details can be found in that reference.

We first introduce some notation. It will be important to consider thinning events and pivotal points. 
\begin{definition} \label{def:LE:thinning_events} Let $x,y \in \X$, and let $A \subset \X$ be locally finite and of cardinality $\abs{A}$.
\begin{compactitem}
\item[(1)] Set 
\begin{equation}
    \bar\connf(A,x) :=\prod_{y\in A}(1-\connf(x,y))  \label{eq:LE:def:thinning_probability}
\end{equation}
and define $\thinning{\eta}{A}$ as a \emph{$\bar\connf(A, \cdot)$-thinning of $\eta$} (or simply \emph{$A$-thinning of $\eta$}) as follows. We keep a point $w \in \eta$ as a point of $\thinning{\eta}{A}$ with probability $\bar\connf(A,w)$ independently of all other points of $\eta$. We similarly define $\thinning{\eta}{A}^x$ as a $\bar\connf(A,\cdot)$-thinning of $\eta^x$ using the marks in $\xi^x$.

\item[(2)] We write $\xconn{x}{y}{\xi}{A}$ if both $\{\conn{x}{y}{\xi}\}$ and $\{\nconn{x}{y}{\xi[\thinning{\eta}{A} \cup\{x\}]} \}$ take place. In words, $\{\xconn{x}{y}{\xi}{A}\}$ is the event that $x,y\in\eta$ and $x$ is connected to $y$ in $\xi$, but that this connection is broken by an $A$-thinning of $\eta\setminus\{x\}$. In particular, the connection does not survive if $y$ is thinned out.

\item[(3)] We define
\begin{equation}
    \tlam^{A}(x,y) = \pla \left(\conn{y}{x}{\xi^{x,y}[\thinning{\eta^y}{A}\cup\{x\}]} \right). \label{eq:def:LE:offconn}
\end{equation}
In words, $\tlam^A(x,y)$ is the probability of the event that there exists an open path between $x$ and $y$ in an RCM driven by an $A$-thinning of $\eta^y$, where the point $x$ is fixed to be present (but $y$ is not).

\item[(4)] Given $x,y\in\X$ and edge-marking $\xi$, we say $u\in\X$ is \emph{pivotal} and $u\in\piv{y,x,\xi}$ if $\left\{\conn{y}{x}{\xi^{x,y}}\right\}$ and yet $\left\{\nconn{y}{x}{\xi\left[\eta\setminus\left\{u\right\}\right]}\right\}$. That is,  every path on $\xi^{x,y}$ connecting $x$ and $y$ uses the vertex $u$. Note that the end points $x$ and $y$ are never pivotal.

\item[(5)] We will also use the set
\begin{equation}
    E(x,y;A,\xi) := \{\xconn{x}{y}{\xi}{A} \} \cap \{\nexists w \in \piv{x,y;\xi}: \xconn{x}{w}{\xi}{A} \}, \label{eq:LE:def:E_event}
\end{equation}
for a locally finite set $A \subset\X$, and $x,y\in \X$. If we consider the pivotal points of the $x$ to $y$ connection in $\xi$ in sequence, then this is the event that an $A$-thinning breaks the connection after the last pivotal point, but not before.
\end{compactitem}
\end{definition}

Now we can state a continuum version of a standard lemma often used in discrete models - see, for example, \cite[Lemma 2.1]{HarSla90}. In bond percolation, it has the name ``Cutting-bond lemma''. The proof of this continuum version can be found in \cite{HeyHofLasMat19}. The notation $\dconn{x}{y}{\xi}$ (for $x,y\in\X$ and configuration $\xi$) denotes the event that $x,y\in\eta$ and $x\sim y$, or that $x,y\in\eta$ and that there are two paths in $\xi$ that connect $x$ and $y$ and that these paths are disjoint in their interior vertices.
 
\begin{lemma}[Cutting-point lemma] \label{lem:LE:cutting_point}
Let $\lambda \geq 0$ and let $v,u,x \in \X$ with $u \neq x$ and let $A\subset \X$ be locally finite. Then
\begin{equation}
    \E_\lambda\big[\mathds 1_{E(v,u;A,\xi^{v,u,x})} \mathds 1_{\{ u \in \textsf{\textup{Piv}}(v,x;\xi^{v,u,x})\}} \big] 
					= \E_\lambda\left[\mathds 1_{E(v,u;A,\xi^{v,u})} \cdot \tlam^{\C(v,\xi^v) }(x,u) \right].
\end{equation}
Moreover,
\begin{equation}
    \pla\left(\dconn{y}{u}{\xi^{y,u,x}}, u \in \textsf{\textup{Piv}}(y,x; \xi^{y,u,x})\right) = \E_\lambda\left[\mathds 1_{\{\dconn{y}{u}{\xi^{y,u}}\}} \cdot \tlam^{\C(y,\xi^{y}) }(x,u) \right].
\end{equation}
\end{lemma}

We now define the lace expansion coefficient function.
\begin{definition}[Lace-expansion function coefficients] \label{def:LE:lace_expansion_coefficients}
For $n\in\N$ and $x,y\in\X$, we define
\begin{align}
    \pi_\lambda^{(0)}(x,y) &:= \pla \left(\dconn{y}{x}{\xi^{y,x}}\right) - \connf(x,y), \label{eq:LE:Pi0_def} \\
			\pi_\lambda^{(n)}(x,y) &:= \lambda^n \int \pla \left( \left\{\dconn{y}{u_0}{\xi^{y, u_0}_{0}}\right\}
					\cap \bigcap_{i=1}^{n} E\left(u_{i-1},u_i; \C_{i-1}, \xi^{u_{i-1}, u_i}_{i}\right) \right) \nu^{\otimes n}\left(\dd \vec u_{[0,n-1]}\right) , \label{eq:LE:Pin_def}
\end{align}
where $u_n=x$ and $\C_{i} = \C(u_{i-1}, \xi^{u_{i-1}}_{i})$ is the cluster of $u_{i-1}$ in $\xi^{u_{i-1}}_i$. Further define the remainder functions
\begin{align}
    r_{\lambda, 0} (x,y) &:= - \lambda \int \pla \left( \left\{\dconn{y}{u_0}{\xi^{y, u_0}_0}\right\} \cap \left\{\xconn{u_0}{x}{\xi^{u_0,x}_1}{\C_0}\right\}  \right) \nu\left(\dd u_0\right), \label{eq:LE:R0_def}\\
		r_{\lambda, n}(x,y) &:= (-\lambda)^{n+1} \int \pla \left( \left\{\dconn{y}{u_0}{\xi^{y, u_0}_0}\right\} \cap \bigcap_{i=1}^{n} E\left(u_{i-1},u_i; \C_{i-1}, \xi^{u_{i-1}, u_i}_{i}\right)\right. \nonumber \\
				& \hspace{6cm}\left.\cap \left\{ \xconn{u_n}{x}{\xi^{u_n,x}_{n+1}}{\C_n} \right\} \right) \nu^{\otimes\left(n+1\right)}\left(\dd \vec u_{[0,n]}\right)
							 \label{eq:LE:Rn_def}.
\end{align}
Additionally, define $\pi_{\lambda, n}$ as the alternating partial sum
\begin{equation}
    \pi_{\lambda, n}(x,y) := \sum_{m=0}^{n} (-1)^m \pi_\lambda^{(m)}(x,y). \label{eq:LE:PiN_def}
\end{equation}
We will also be requiring Fourier transforms of these functions. Since the whole model is spatially translation invariant, so are the functions defined above. Therefore for each pair of marks $(a,b)$ we perform the Fourier transform on the spatial displacement to get the functions $\widehat{\pi}^{(n)}_\lambda(k;a,b)$, $\widehat{\pi}_{\lambda,n}(k;a,b)$, and $\widehat{r}_{\lambda,n}(k;a,b)$ for $n\geq 0$. 
\end{definition}

\begin{prop}[Lace expansion: function-level]
\label{thm:laceexpansionFunction}
Let $\lambda \in [0,\lambda_O)$ and $n\geq0$. Then for $x,y\in\X$, and for $k\in\Rd$ and $a,b\in\Ecal$,
\begin{align}
    \tlam\left(x,y\right) &= \connf\left(x,y\right) + \pi_{\lambda, n}\left(x,y\right) + \lambda \int\tlam\left(x,u\right)\left(\connf + \pi_{\lambda, n}\right)\left(u,y\right)\nu\left(\dd u\right) + r_{\lambda, n}\left(x,y\right), \label{eqn:LaceexpansionFunction1}\\
    \ftlam\left(k;a,b\right) &= \fconnf\left(k;a,b\right) + \widehat{\pi}_{\lambda, n}\left(k;a,b\right) + \lambda \int\ftlam\left(k;a,c\right)\left(\fconnf + \widehat{\pi}_{\lambda, n}\right)\left(k;c,b\right)\Pcal\left(\dd c\right) + \widehat{r}_{\lambda, n}\left(k;a,b\right). \label{eqn:LaceexpansionFunction2}
\end{align}
\end{prop}

We outline the proof of \eqref{eqn:LaceexpansionFunction1} here. The equation \eqref{eqn:LaceexpansionFunction2} then follows by applying the Fourier transform. The details of the proof of \eqref{eqn:LaceexpansionFunction1} - with appropriate contextual changes - can be found in \cite{HeyHofLasMat19}. Firstly, the definitions give
\begin{equation}
    \tlam(x,y) = \connf(x,y) + \pi_\lambda^{(0)}(x,y) + \pla\left(\conn{y}{x}{\xi^{y,x}}, \ndconn{y}{x}{\xi^{y,x}}\right). \label{eq:lace_expansion_first_step_1}
\end{equation}
With an appropriate event partition, Mecke's equation and the Cutting-point lemma allow us to re-write the probability term as
\begin{equation}
    \pla\left(\conn{y}{x}{\xi^{y,x}}, \ndconn{y}{x}{\xi^{y,x}}\right) = \lambda \int \E_\lambda\left[ \mathds 1_{\{\dconn{y}{u}{\xi^{y,u}}\}} \tlam^{\C_0}(x,u) \right] \dd u
\end{equation}
Now we use the identity
\begin{equation}
    \tlam^{A}(x,u) = \tlam(x,u) - \pla \left( \xconn{u}{x}{\xi^{u,x}}{A} \right), \label{eq:LE:incl_excl_probabilities}
\end{equation}
to extract the deterministic function $\tlam(x,u)$ from the expectation at the cost of an extra term. The $\tlam$ term then gives us the second $\pi^{(0)}_\lambda$ term and the extra term becomes the remainder $r_{\lambda,0}$. Repeated use of \eqref{eq:LE:incl_excl_probabilities} with  Mecke's equation and the Cutting-point lemma gives an alternating series that produces the various $\pi_\lambda^{(n)}$ and $r_{\lambda,n}$ terms.

\subsection{Operator-level expansion}
\label{sec:expansion:Operators}

\begin{definition}[Lace-expansion operator coefficients] \label{def:LE:lace_expansion_coefficients-Operators}
For $n\in\N$, we define $\OpLacelam^{(n)}\colon \mathcal{D}\left(\OpLacelam^{(n)}\right) \to L^2\left(\X\right)$ as the linear operator with kernel function $\pi^{(n)}_\lambda$, $R_{\lambda,n}\colon \mathcal{D}\left(R_{\lambda,n}\right) \to L^2\left(\X\right)$ as the linear operator with kernel function $r_{\lambda,n}$, and $\OpLace_{\lambda,n}\colon \mathcal{D}\left(\OpLace_{\lambda,n}\right) \to L^2\left(\X\right)$ as the linear operator with kernel function $\pi_{\lambda, n}$. It follows that
\begin{equation}
    \OpLace_{\lambda, n} = \sum_{m=0}^{n} (-1)^m \OpLacelam^{(m)},
\end{equation}
and $\mathcal{D}\left(\OpLace_{\lambda,n}\right) \supset\bigcap_{m=0}^n\mathcal{D}\left(\OpLacelam^{(n)}\right)$.

We also use the kernel functions $\widehat{\pi}^{(n)}_\lambda(k;a,b)$, $\widehat{\pi}_{\lambda,n}(k;a,b)$, and $\widehat{r}_{\lambda,n}(k;a,b)$ to define the integral operators $\fOpLace^{(n)}_{\lambda}(k),\fOpLace_{\lambda,n}(k),\widehat{R}_{\lambda,n}(k)$ respectively for each $k\in\Rd$.  
\end{definition}

\begin{prop}[Lace expansion: operator-level]
\label{thm:OperatorLaceExpansion}
Let $\lambda \in [0,\lambda_O)$. Then for $n \geq 0$ and $k\in\Rd$,
\begin{align}
\label{eqn:OperatorLaceExpansion}
    \Optlam &= \Opconnf + \OpLace_{\lambda, n} +  \lambda\Optlam\left(\Opconnf + \OpLace_{\lambda, n}\right) + R_{\lambda, n},\\
    \fOptlam(k) &= \left(\fOpconnf + \fOpLace_{\lambda, n}\right)(k) +  \lambda\fOptlam(k)\left(\fOpconnf + \fOpLace_{\lambda, n}\right)(k) + \widehat{R}_{\lambda, n}(k).\label{eqn:OperatorLaceExpansion2}
\end{align}
\end{prop}

\begin{proof}
Let $f\in L^2\left(\X\right)$ be a test function. Then Proposition~\ref{thm:laceexpansionFunction} gives for $x\in\X$,
\begin{align}
    \Optlam f\left(x\right) &= \int \tlam\left(x,y\right)f\left(y\right)\nu\left(\dd y\right)\nonumber\\
    & = \int \left(\connf\left(x,y\right) + \pi_{\lambda, n}\left(x,y\right) + \lambda \int\tlam\left(x,u\right)\left(\connf + \pi_{\lambda, n}\right)\left(u,y\right)\nu\left(\dd u\right) + r_{\lambda, n}\left(x,y\right)\right)f\left(y\right)\nu\left(\dd y\right)\nonumber\\
    & = \left(\Opconnf + \OpLace_{\lambda, n} +  \lambda\Optlam\left(\Opconnf + \OpLace_{\lambda, n}\right) + R_{\lambda, n}\right)f\left(x\right).
\end{align}
Note that in this last equality we used 
\begin{align}
    \int\int\tlam\left(x,u\right)\left(\connf + \pi_{\lambda, n}\right)\left(u,y\right)\nu\left(\dd u\right)f\left(y\right)\nu\left(\dd y\right) &= \int\tlam\left(x,u\right)\int\left(\connf + \pi_{\lambda, n}\right)\left(u,y\right)f\left(y\right)\nu\left(\dd y\right)\nu\left(\dd u\right)\nonumber \\
    & = \Optlam\left(\Opconnf + \OpLace_{\lambda, n}\right)f\left(x\right).
\end{align}
This exchange of integrals is valid by Fubini's theorem because $\lambda<\lambda_O$ ensures the integrals are finite.

The same argument applies for the Fourier transformed operators to complete the proof.
\end{proof}


\section{Sub-Critical Convergence}
\label{sec:SubCritConvergence}

In this section, we outline the lemmas and propositions that show that the operator expansion in Proposition~\ref{thm:OperatorLaceExpansion} converges for $\lambda<\lambda_O$. We present the lemmas for this argument here, but the proof for them are contained in Appendix~\ref{appendix:DiagrammaticBoundsProofs}. The bounds are largely analogous to those proven in \cite{HeyHofLasMat19}, but the argument here makes much reduced use of Assumption~\ref{Assump:BallDecay}. The proofs are also significantly different in parts as we don't have the universal applicability of tools like translation symmetry and the Fourier transform (due to the introduction of marks).

\subsection{Diagrammatic bounds} \label{sec:diagrammaticbounds}

Our argument will require bounds on $\OpNorm{\OpLacelam^{(n)}}$, $\OpNorm{\fOpLacelam^{(n)}(k)}$, $\OpNorm{\fOpLacelam^{(n)}\left(0\right) - \fOpLacelam^{(n)}\left(k\right)}$, $\OpNorm{R_{\lambda,n}}$ and $\OpNorm{\widehat{R}_{\lambda,n}(k)}$ for all $n\in\N$ and $k\in\Rd$. However, it turns out that we can bound the last two remainder terms using others.

\begin{restatable}{lemma}{remainderBound}
\label{lem:remainderBound}
For all $n\in\N$, $k\in\Rd$, and $\lambda>0$,
\begin{align}
    \OpNorm{R_{\lambda,n}} &\leq \lambda \OpNorm{\Optlam}\OpNorm{\OpLacelam^{(n)}},\\
    \OpNorm{\widehat{R}_{\lambda,n}(k)} &\leq \lambda \OpNorm{\fOptlam(0)}\OpNorm{\fOpLacelam^{(n)}(0)}.
\end{align}
\end{restatable}

We first present the bounds for $\OpNorm{\OpLacelam^{(0)}}$, $\OpNorm{\fOpLacelam^{(0)}(k)}$, and $\OpNorm{\fOpLacelam^{(0)}\left(0\right) - \fOpLacelam^{(0)}\left(k\right)}$. The bounds and proofs for these with $n\geq 1$ will be qualitatively different.

Let us introduce the notation $\tklam\left(\xbar;a,b\right) = \left(1-\cos\left(k\cdot\xbar\right)\right)\tlam\left(\xbar;a,b\right)$ and similarly $\connf_k\left(\xbar;a,b\right) = \left(1-\cos\left(k\cdot\xbar\right)\right)\connf\left(\xbar;a,b\right)$ for $k\in\Rd$. $\Optklam$ and $\Opconnf_k$ then denote the integral operators constructed using these kernels. We can also define the Fourier transforms of these functions: $\ftklam\left(l;a,b\right)$ and $\fconnf_k\left(l;a,b\right)$ for $k,l\in\Rd$, and use these as kernels to define the integral operators $\fOptklam(l)$ and $\fOpconnf_k(l)$. In particular note that this results in 
\begin{align}
    \fOptklam\left(l\right) &= \fOptlam\left(k+l\right) - \fOptlam\left(l\right),\\
    \fOpconnf_k\left(l\right) &= \fOpconnf\left(k+l\right) - \fOpconnf\left(l\right),
\end{align}
for $k,l\in\R^d$.

\begin{restatable}[Bounds for $n=0$]{prop}{DBPiObounds} \label{thm:DB:Pi0_bounds}
Let $\lambda\in[0,\lambda_O)$. Then
\begin{equation}
    \OpNorm{\OpLacelam^{(0)}} \leq \tfrac{1}{2}\lambda^2\InfNorm{\Opconnf\Optlam^2\Opconnf} \leq \tfrac{1}{2}\lambda^2\OneNorm{\Opconnf}\InfNorm{\Optlam^3}.
\end{equation}
For $k\in\Rd$,
\begin{align}
    \OpNorm{\fOpLacelam^{(0)}(k)} &\leq \tfrac{1}{2}\lambda^2\OneNorm{\Opconnf}\InfNorm{\Optlam^3},\\
    \OpNorm{\fOpLacelam^{(0)}\left(0\right) - \fOpLacelam^{(0)}\left(k\right)} &\leq \lambda^2\left(\InfNorm{\Optlam^3}\OneNorm{\Opconnf_k} + \InfNorm{\Opconnf\Optlam\Optklam\Opconnf}\right).
\end{align}
\end{restatable}

\subsubsection{Diagrammatic Bounds for Lace Expansion Coefficients}

Now we aim to bound $\OpNorm{\OpLacelam^{(n)}}$ and $\OpNorm{\fOpLacelam^{(n)}(k)}$ for $n\geq 1$. Lemmas~\ref{lem:BoundsonOperatorNorm} and \ref{lem:NormBounds} along with the positivity of $\pi^{(n)}_\lambda(x,y)$ prove that it is sufficient to bound $\OneNorm{\OpLacelam^{(n)}}$ to get bounds on both of these. Much of the hard work for this has been done previously - for example in \cite{HeyHofLasMat19}. The novelty here is to see how the argument can be naturally written in terms of operators.

We inherit the following notation from \cite{HeyHofLasMat19}. Note that we make use of the Dirac delta function - specifically the one that holds with respect to the measure $\nu$. They are used here primarily for convenience and to make the argument more readable. In particular, they appear when applying the Mecke equation~\eqref{eq:prelim:mecke_n} to obtain
\begin{equation}
	\E \left[ \sum_{y\in\eta^u} f(y, \xi^u) \right] = \int \left(\lambda+\delta_{y,u}\right) \E_\lambda \left[ f\left(y,\xi^{u,y}\right)\right] \nu\left(\dd y\right), 
\end{equation} 
and so the factor $\lambda+\delta_{y,u}$ encodes a case distinction of whether point $y$ coincides with $u$ or not. 

\begin{definition}[The $\psi$ functions] \label{def:DB:psi_functions}
Let $r,s,u,w,x,y\in\X$. We first set $\tlamo(x,y) := \lambda^{-1}\delta_{x,y} + \tlam(x,y)$. Also define
\begin{align*}
    \psi_0^{(1)}(w,u,y) &:= \lambda^2\tlam(y,u)\tlam(u,w)\tlam(w,y),\\
    \psi_0^{(2)}(w,u,y) &:=\lambda^2 \delta_{w,y} \tlam(y,u)\int \tlam(u,t)\tlam(t,y) \nu\left(\dd t\right),\\
    \psi_0^{(3)}(w,u,y) &:= \lambda^2\connf(u,y) \left(\lambda^{-1}\delta_{w,y}\right),\\
    \psi^{(1)}(w,u,r,s) &:= \lambda^4\tlam(w,u)\int \tlamo(t,s) \tlam(t,w)\tlam(u,z)\tlam(z,t)\tlam(z,r)\nu^{\otimes2}\left(\dd z, \dd t\right), \\
    \psi^{(2)}(w,u,r,s) &:= \lambda^4\tlamo(w,s)\int \tlam(t,z)\tlam(z,u)\tlam(u,t)\tlamo(t,w)\tlam(z,r)\nu^{\otimes2}\left(\dd z, \dd t\right), \\
	\psi^{(3)}(w,u,r,s) &:= \lambda^2\tlam(u,w)\tlam(w,s)\tlam(u,r),\\
	\psi^{(4)}(w,u,r,s) &:= \lambda^2\tlam(u,w)\left(\lambda^{-1}\delta_{w,s}\right)\tlam(u,r),\\
	\psi_n^{(1)} (x,r,s) &:= \lambda^3\int \tlamo(t,s)\tlam(z,r)\tlam(t,z)\tlam(z,x)\tlam(x,t)\nu^{\otimes2}\left(\dd z, \dd t\right),\\
    \psi_n^{(2)}(x,r,s) &:=\lambda \tlam(x,s)\tlam(x,r),
\end{align*}
and set $\psi_0 := \psi_0^{(1)}+\psi_0^{(2)}+\psi_0^{(3)}$, $\psi_n := \psi_n^{(1)} + \psi_n^{(2)}$, and $\psi:= \psi^{(1)}+\psi^{(2)}+\psi^{(3)} + \psi^{(4)}$. Note that by using $\lambda^{-1}\delta_{\left(\cdot,\cdot\right)}$ in $\tlamo$ and in $\psi^{(3)}_0$, we ensure that each integral that will be truly performed (that is, not an integration over a Dirac delta function accounting for a coincidence of points) has a $\lambda$ factor associated with it when the bound $\pi^{(n)}_\lambda$ is calculated.

Diagrammatic representations of these functions can be found in Figure~\ref{fig:psiFunctions}. In these diagrams, the declared variables are represented by \SupremumDot~vertices whereas the variables that are integrated over are represented by \IntegralDot~vertices. If $\tlam$ connects two variables then a standard edge \tlamline~connects their vertices, if $\tlamo$ connects two variables then an edge \tlamoline~connects their vertices, and if $\connf$ connects two variables then an edge \connfline~connects their vertices.
\end{definition}

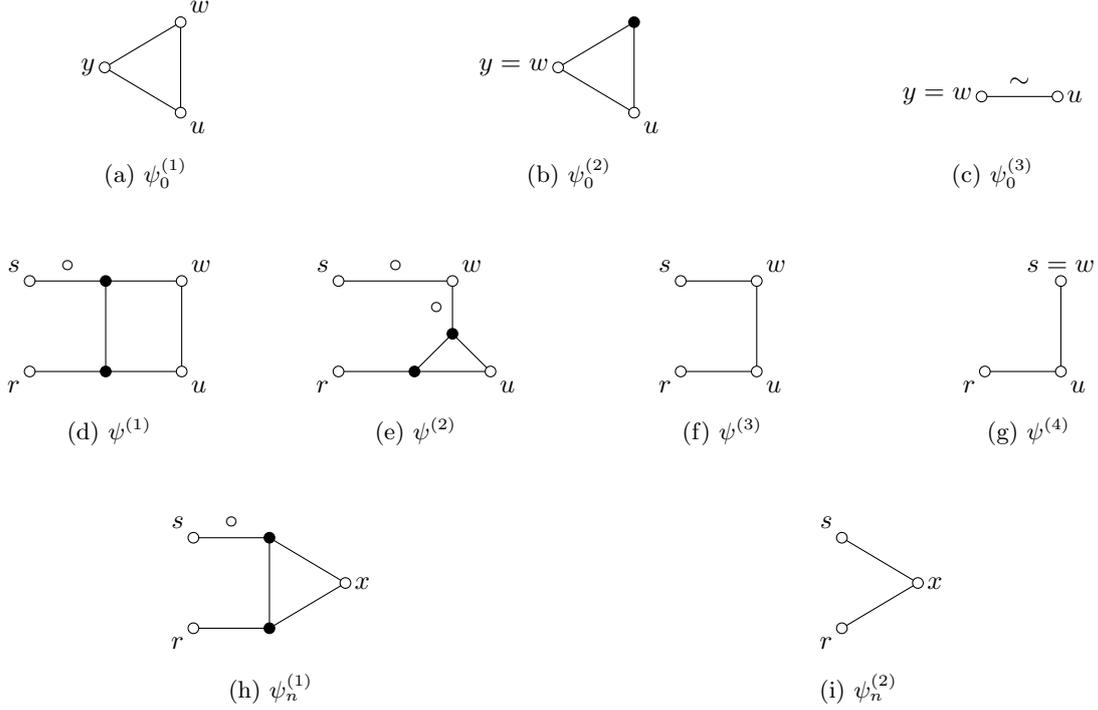
\begin{figure}
    \centering
    \begin{subfigure}[b]{0.3\textwidth}
    \centering
        \begin{tikzpicture}
        \draw (0,0) -- (1,0.6) -- (1,-0.6) -- cycle;
        \filldraw[fill=white] (0,0) circle (2pt) node[left]{$y$};
        \filldraw[fill=white] (1,0.6) circle (2pt) node[above right]{$w$};
        \filldraw[fill=white] (1,-0.6) circle (2pt) node[below right]{$u$};
        \end{tikzpicture}
    \caption{$\psi^{(1)}_0$}
    \end{subfigure}
    \hfill
    \begin{subfigure}[b]{0.3\textwidth}
    \centering
        \begin{tikzpicture}
        \draw (0,0) -- (1,0.6) -- (1,-0.6) -- cycle;
        \filldraw[fill=white] (0,0) circle (2pt) node[left]{$y=w$};
        \filldraw (1,0.6) circle (2pt);
        \filldraw[fill=white] (1,-0.6) circle (2pt) node[below right]{$u$};
       \end{tikzpicture}
    \caption{$\psi^{(2)}_0$}
    \end{subfigure}
    \hfill
    \begin{subfigure}[b]{0.3\textwidth}
    \centering
        \begin{tikzpicture}
        \draw (0,0) -- (1,0);
        \draw (0.5,0) circle (0pt) node[above]{$\sim$};
        \draw (0,-0.6) circle (0pt);
        \filldraw[fill=white] (0,0) circle (2pt) node[left]{$y=w$};
        \filldraw[fill=white] (1,0) circle (2pt) node[right]{$u$};
        \end{tikzpicture}
    \caption{$\psi^{(3)}_0$}
    \end{subfigure}
    \begin{subfigure}[b]{0.24\textwidth}
    \centering
        \begin{tikzpicture}
        \draw (0,0.6) -- (1,0.6) -- (1,-0.6) -- (0,-0.6);
        \draw (0.5,0.6) circle (0pt) node[above]{$\circ$};
        \draw (1,0.6) -- (2,0.6) -- (2,-0.6) -- (1,-0.6);
        \draw (0,1.75) circle (0pt);
        \filldraw[fill=white] (0,0.6) circle (2pt) node[above left]{$s$};
        \filldraw[fill=white] (0,-0.6) circle (2pt) node[below left]{$r$};
        \filldraw (1,0.6) circle (2pt);
        \filldraw (1,-0.6) circle (2pt);
        \filldraw[fill=white] (2,0.6) circle (2pt) node[above right]{$w$};
        \filldraw[fill=white] (2,-0.6) circle (2pt) node[below right]{$u$};
        \end{tikzpicture}
    \caption{$\psi^{(1)}$}
    \end{subfigure}
    \hfill
    \begin{subfigure}[b]{0.24\textwidth}
    \centering
        \begin{tikzpicture}
        \draw (0,0.6) -- (1.5,0.6) -- (1.5,-0.1) -- (1,-0.6);
        \draw (0.75,0.6) circle (0pt) node[above]{$\circ$};
        \draw (1.5,0.25) circle (0pt) node[left]{$\circ$};
        \draw (0,-0.6) -- (1,-0.6) -- (2,-0.6) -- (1.5,-0.1);
        \filldraw[fill=white] (0,0.6) circle (2pt) node[above left]{$s$};
        \filldraw[fill=white] (0,-0.6) circle (2pt) node[below left]{$r$};
        \filldraw[fill=white] (1.5,0.6) circle (2pt) node[above right]{$w$};
        \filldraw (1,-0.6) circle (2pt);
        \filldraw (1.5,-0.1) circle (2pt);
        \filldraw[fill=white] (2,-0.6) circle (2pt) node[below right]{$u$};
        \end{tikzpicture}
    \caption{$\psi^{(2)}$}
    \end{subfigure}
    \hfill
    \begin{subfigure}[b]{0.24\textwidth}
    \centering
        \begin{tikzpicture}
        \draw (0,0.6) -- (1,0.6) -- (1,-0.6) -- (0,-0.6);
        \filldraw[fill=white] (0,0.6) circle (2pt) node[above left]{$s$};
        \filldraw[fill=white] (0,-0.6) circle (2pt) node[below left]{$r$};
        \filldraw[fill=white] (1,0.6) circle (2pt) node[above right]{$w$};
        \filldraw[fill=white] (1,-0.6) circle (2pt) node[below right]{$u$};
        \end{tikzpicture}
    \caption{$\psi^{(3)}$}
    \end{subfigure}
    \hfill
    \begin{subfigure}[b]{0.24\textwidth}
    \centering
        \begin{tikzpicture}
        \draw (1,0.6) -- (1,-0.6) -- (0,-0.6);
        \filldraw[fill=white] (0,-0.6) circle (2pt) node[below left]{$r$};
        \filldraw[fill=white] (1,0.6) circle (2pt) node[above]{$s=w$};
        \filldraw[fill=white] (1,-0.6) circle (2pt) node[below right]{$u$};
        \end{tikzpicture}
    \caption{$\psi^{(4)}$}
    \end{subfigure}
    \hspace*{\fill}%
    \begin{subfigure}[b]{0.45\textwidth}
    \centering
        \begin{tikzpicture}
        \draw (0,0.6) -- (1,0.6) -- (1,-0.6);
        \draw (0.5,0.6) circle (0pt) node[above]{$\circ$};
        \draw (1,0.6) -- (2,0) -- (1,-0.6) -- (0,-0.6);
        \filldraw[fill=white] (0,0.6) circle (2pt) node[above left]{$s$};
        \filldraw[fill=white] (0,-0.6) circle (2pt) node[below left]{$r$};
        \filldraw (1,0.6) circle (2pt);
        \filldraw (1,-0.6) circle (2pt);
        \filldraw[fill=white] (2,0) circle (2pt) node[right]{$x$};
        \end{tikzpicture}
    \caption{$\psi^{(1)}_n$}
    \end{subfigure}
    \hfill
    \begin{subfigure}[b]{0.45\textwidth}
    \centering
        \begin{tikzpicture}
        \draw (1,0.6) -- (2,0) -- (1,-0.6);
        \draw (0,1.75) circle (0pt);
        \filldraw[fill=white] (1,0.6) circle (2pt) node[above left]{$s$};
        \filldraw[fill=white] (1,-0.6) circle (2pt) node[below left]{$r$};
        \filldraw[fill=white] (2,0) circle (2pt) node[right]{$x$};
        \end{tikzpicture}
    \caption{$\psi^{(2)}_n$}
    \end{subfigure}
    \hspace*{\fill}%
    \caption{Diagrams of the $\psi_0$, $\psi$, and $\psi_n$ functions.}
    \label{fig:psiFunctions}
\end{figure}

\begin{definition}[The $\Psi$ operators]
\label{def:DB:Psi_operators}
Here we turn the functions of Definition~\ref{def:DB:psi_functions} into linear operators, with the $\psi$ functions acting as their kernels. For $j\in\left\{1,2,3\right\}$, we define $\Psi^{(j)}_0\colon L^2\left(\X\right) \to L^2\left(\X^2\right)$ as the linear operator acting on $f\in L^2\left(\X\right)$ as
\begin{equation}
    \Psi^{(j)}_0 f\left(w,u\right) = \int \psi^{(j)}_0(w,u,y) f(y)\nu\left(\dd y\right).
\end{equation}
For $j\in\left\{1,2,3,4\right\}$, we define $\Psi^{(j)}\colon L^2\left(\X^2\right) \to L^2\left(\X^2\right)$ as the linear operator acting on $f\in L^2\left(\X^2\right)$ as
\begin{equation}
    \Psi^{(j)} f\left(w,u\right) = \int \psi^{(j)}(w,u,r,s) f(r,s)\nu^{\otimes2}\left(\dd r, \dd s\right).
\end{equation}
For $j\in\left\{1,2\right\}$, we define $\Psi^{(j)}_n\colon L^2\left(\X^2\right) \to L^2\left(\X\right)$ as the linear operator acting on $f\in L^2\left(\X^2\right)$ as
\begin{equation}
    \Psi^{(j)}_n f\left(x\right) = \int \psi^{(j)}_n(x,r,s) f(r,s)\nu^{\otimes2}\left(\dd r, \dd s\right).
\end{equation}
We can then naturally define $\Psi_0\colon L^2(\X)\to L^2(\X^2)$, $\Psi\colon L^2(\X^2)\to L^2(\X^2)$, and $\Psi_n\colon L^2(\X^2)\to L^2(\X)$ as the sums of $\Psi^{(j)}_0$, $\Psi^{(j)}$, and $\Psi^{(j)}_n$ respectively, or equivalently as those operators having kernel functions $\psi_0$, $\psi$, and $\psi_n$ respectively.
\end{definition}

\begin{prop}[Bound in terms of $\psi$  and $\Psi$] \label{thm:DB:Pi_bound_Psi}
Let $n \geq 1$, $x,y\in\X$, and $\lambda\in [0,\lambda_O)$. Then
\begin{equation}
    \label{eqn:Lacefunction_bound}
    \lambda\pi_\lambda^{(n)}(x,y) \leq \int  \psi_n(x,w_{n-1},u_{n-1}) \left( \prod_{i=1}^{n-1} \psi(\vec v_i) \right) \psi_0(w_0,u_0,y) \nu^{\otimes\left(2n\right)}\left(\dd\left( \left(\vec w, \vec u\right)_{[0,n-1]} \right)\right),
\end{equation}
where $\vec v_i = (w_i,u_i,w_{i-1},u_{i-1})$. In operator terms, this means
\begin{equation}
\label{eqn:Laceoperator_bound}
    \lambda\OneNorm{\OpLacelam^{(n)}} \leq \OneNorm{\Psi_n \Psi^{n-1} \Psi_0}.
\end{equation}
\end{prop}

\begin{proof}
The argument for \eqref{eqn:Lacefunction_bound} is essentially identical to \cite[Proposition~4.7]{HeyHofLasMat19}. It uses only general properties of Poisson point process and connection models - for example thinning edges and Mecke's equation.
The inequality \eqref{eqn:Laceoperator_bound} then follows from the clear sequential structure of \eqref{eqn:Lacefunction_bound}.
\end{proof}

The inequality \eqref{eqn:Laceoperator_bound} suggests that we will bound the lace coefficient operators with some product of $\OneNorm{\Psi_0}$, $\OneNorm{\Psi}$, and $\OneNorm{\Psi_n}$ terms. This is \emph{nearly} correct. Unfortunately the norms $\OneNorm{\Psi_n}$, $\OneNorm{\Psi}$ and $\OneNorm{\Psi_0}$ will not have the decay in $d$ that we will require. We will have to consider the norms of \emph{pairs} of operators.

\begin{definition}
\label{def:triangles}
Define
\begin{align*}
    \trilam &:= \lambda^2\esssup_{r,s\in\X}\int\tlam\left(r,u\right)\tlam\left(u,v\right)\tlam\left(v,s\right)\nu^{\otimes2}\left(\dd u, \dd v\right) = \lambda^2\InfNorm{\Optlam^3},\\
    \trilamo &:= \lambda^2 \esssup_{r,s\in\X}\int\tlamo\left(r,u\right)\tlam\left(u,v\right)\tlam\left(v,s\right)\nu^{\otimes2}\left(\dd u, \dd v\right) = \lambda^2\InfNorm{\Optlam^3} + \lambda\InfNorm{\Optlam^2},\\
    \trilamoo &:= \lambda^2 \esssup_{r,s\in\X}\int\tlamo\left(r,u\right)\tlamo\left(u,v\right)\tlam\left(v,s\right)\nu^{\otimes2}\left(\dd u, \dd v\right) = \lambda^2\InfNorm{\Optlam^3} + \lambda\InfNorm{\Optlam^2} + 1,\\
    \trilamooBar&:= \lambda^2\esssup_{\xbar\in\Rd,a_1,\ldots,a_6\in\Ecal}\int\tlamo\left(\xbar-\ubar;a_1,a_2\right)\tlamo\left(\ubar-\vbar;a_3,a_4\right)\tlam\left(\vbar;a_5,a_6\right)\dd \ubar \dd \vbar.
\end{align*}
We can think of $\trilam$ as an $\X$-convolution of three $\tlam$ functions with suprema taken over the end vertices. Then $\trilamo$ and $\trilamoo$ are produced by adding on the $\X$-convolution of two $\tlam$ functions and one $\tlam$ function (which will trivially take value $1$ once the suprema are taken). The object $\trilamooBar$ differs from $\trilamoo$ in that now the marks at adjacent $\tlam$ functions need not be equal, and then we take the supremum over all the marks - not just the end ones. We then use these elementary diagrams to produce composite objects:
\begin{equation}
    \Ulam := \trilam \trilamooBar + \trilamooBar^2 + \lambda\OneNorm{\Opconnf}, \qquad
    \Vlam := \left(\trilam\trilamooBar\Ulam\right)^\frac{1}{2}.
\end{equation}
We note a few relations. Firstly, since all the terms in $\Ulam$ are non-negative we have $\trilam\trilamooBar\leq \Ulam$ and thus $\Vlam \leq \Ulam$. Also note that $1\leq \trilamooBar$ and thus $\Ulam \geq \trilamooBar$. In particular this means that $\Ulam\geq 1$ and $\Ulam \geq \trilamo$.
\end{definition}

The following lemma tells us that we can use $\Ulam$ to bound single operators, and that we can use $\Vlam^2$ to bound pairs of operators. To aid the reader's understanding, bear in mind that later (in Section~\ref{sec:bootstrapanalysis}) we will prove that $\Ulam = \LandauBigO{1}$, whilst $\Vlam$ is much smaller and is $\LandauBigO{\beta}$.

\begin{restatable}{lemma}{singleandpairbounds}
\label{lem:singleandpairbounds}
For all $j_0\in\left\{1,2,3\right\}$, $j,j'\in\left\{1,2,3,4\right\}$, and $j_n\in\left\{1,2\right\}$,
\begin{align}
    \Ulam &\geq \OneNorm{\Psi^{(j_0)}_0} \vee \OneNorm{\Psi^{(j)}} \vee \OneNorm{\Psi^{(j_n)}_n},\\
    \Vlam^2 &\geq \OneNorm{\Psi^{(j)}\Psi^{(j_0)}_0} \vee \OneNorm{\Psi^{(j)}\Psi^{(j')}} \vee \OneNorm{\Psi^{(j_n)}_n\Psi^{(j)}} \vee \OneNorm{\Psi^{(j_n)}_n\Psi^{(j_0)}_0}.
\end{align}
\end{restatable}

These bounds with Proposition~\ref{thm:DB:Pi_bound_Psi} lead to the following proposition.

\begin{restatable}{prop}{LaceCoefficientBound}
\label{prop:LaceCoefficientBound}
For $n\geq 1$,
\begin{equation}
    \lambda\OneNorm{\OpLacelam^{(n)}} \leq 6\times 4^{n-1}\Ulam \Vlam^n.
\end{equation}
\end{restatable}

\subsubsection{Diagrammatic Bounds for Displaced Lace Expansion Coefficients}

Now let us aim to bound $\OpNorm{\fOpLacelam^{(n)}\left(0\right) - \fOpLacelam^{(n)}\left(k\right)}$ for $n\geq 1$. Our strategy draws inspiration from the corresponding step in \cite{HeyHofLasMat19}, whilst being adapted to account for inhomogeneous marks. 

The central idea is to bound
\begin{equation}
    \OpNorm{\fOpLacelam^{(n)}\left(0\right) - \fOpLacelam^{(n)}\left(k\right)} \leq \OneNorm{\fOpLacelam^{(n)}\left(0\right) - \fOpLacelam^{(n)}\left(k\right)} = \esssup_{b\in\Ecal}\int\left(1- \cos\left(k\cdot\xbar\right)\right)\pi^{(n)}_\lambda\left(\xbar;a,b\right)\dd \xbar \Pcal\left(\dd a\right),
\end{equation}
and bound $\pi^{(n)}_\lambda$ using the function expression from Proposition~\ref{thm:DB:Pi_bound_Psi}. This produces an integral (or diagram) composed from simpler segments with the displacement factor, $\left(1- \cos\left(k\cdot\xbar\right)\right)$, spanning the whole length of the diagram. The following Cosine-Splitting result allows us to get a sum of diagrams where the displacement factor only spans a single segment of each diagram. 

\begin{lemma}[{Split of cosines, \cite[Lemma 2.13]{FitHof16}}] \label{lem:cosinesplitlemma}
Let $t_i \in \R$ for $i=1, \ldots, m$ and $t = \sum_{i=1}^{m} t_i$. Then
	\[1-\cos (t) \leq m \sum_{i=1}^{m} [1- \cos(t_i)]. \]
\end{lemma}

To see this strategy in practice, consider the diagram corresponding to $\Psi^{(2)}_n\Psi^{(1)}\Psi^{(1)}_0$ with the displacement factor:
\begin{multline}
    \ExampleOne \quad \leq 3\left(\quad\ExampleOnePtOne \quad + \quad \right. \\ \left.\ExampleOnePtTwo \quad + \quad \ExampleOnePtThree\quad\right).
\end{multline}
In these diagrams and hereafter we use \displaceline~to denote that a displacement factor is connecting the two indicated vertices. As above, the path the \displaceline~takes will also suggest the route along which we will intend to use cosine-splitting. For the sake of simplicity, we will be taking the displacement path across the `top' of the diagram from the perspective of the usual orientation of the $\Psi_0$ segment. However - as can be seen in the above example - the way the operators compose with each other means that the orientation of each segment inverts when compared to its neighbours. Therefore the displacement may be across the `top' or the `bottom' of the $\Psi$ and $\Psi_n$ segments (in their usual orientation) depending upon the parity of its place in the sequence of segments.

Our general strategy will be to isolate the displaced segment (and perhaps a neighbouring segment) by splitting off `earlier' and `later' segments. When we split off earlier segments (and sometimes later segments), we will use the $\Psi$ structure we have been using so far. However, we will sometimes want to group vertices and edges slightly differently when we split off `later' segments. We use an observation from \cite{HeyHofLasMat19} to conveniently group these. Let us define the kernel functions
\begin{align*}
    \EndBlock^{(1)}_0\left(w,u,x\right) &:= \lambda^2\tlam(y,r)\tlam(r,s)\tlam(s,x),\\
    \EndBlock^{(2)}_0\left(w,u,x\right) &:= \lambda^2\left(\lambda^{-1}\delta_{y,r}\right)\left(\lambda^{-1}\delta_{y,s}\right),\\
    \EndBlock^{(1)}(w,u,r,s) &:= \lambda^4\tlam(w,u)\int \tlam(t,s) \tlam(t,w)\tlam(u,z)\tlam(z,t)\tlamo(z,r)\nu^{\otimes2}\left(\dd z, \dd t\right), \\
    \EndBlock^{(2)}(w,u,r,s) &:= \lambda^4\tlam(w,s)\int \tlam(t,z)\tlam(z,u)\tlam(u,t)\tlamo(t,w)\tlamo(z,r)\nu^{\otimes2}\left(\dd z, \dd t\right), \\
	\EndBlock^{(3)}(w,u,r,s) &:= \lambda^2\tlam(u,w)\tlam(w,s)\tlam(u,r),\\
	\EndBlock^{(4)}(w,u,r,s) &:= \lambda^2\tlam(u,w)\tlam(w,s)\left(\lambda^{-1}\delta_{u,r}\right),
\end{align*}
and set $\EndBlock_0 := \EndBlock^{(1)}_0+\EndBlock^{(2)}_0$ and $\EndBlock := \EndBlock^{(1)}+\EndBlock^{(2)}+\EndBlock^{(3)} + \EndBlock^{(4)}$. Diagrammatic representations of these functions can be found in Figure~\ref{fig:EndBlockFunctions}. Note the similarity to the $\psi$ functions, with the $\tlamo$ edge appearing elsewhere in the integral. 

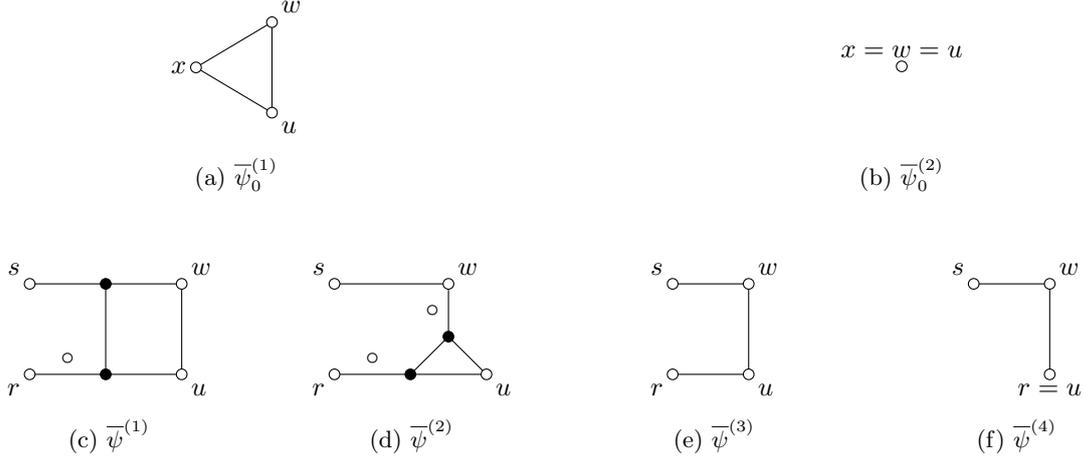
\begin{figure}
    \centering
    \begin{subfigure}[b]{0.45\textwidth}
    \centering
        \begin{tikzpicture}
        \draw (0,0) -- (1,0.6) -- (1,-0.6) -- cycle;
        \filldraw[fill=white] (0,0) circle (2pt) node[left]{$x$};
        \filldraw[fill=white] (1,0.6) circle (2pt) node[above right]{$w$};
        \filldraw[fill=white] (1,-0.6) circle (2pt) node[below right]{$u$};
        \end{tikzpicture}
    \caption{$\EndBlock^{(1)}_0$}
    \end{subfigure}
    \hfill
    \begin{subfigure}[b]{0.45\textwidth}
    \centering
        \begin{tikzpicture}
        \filldraw[fill=white] (0,0) circle (2pt) node[above]{$x=w=u$};
        \draw[white] (0,-1) circle (0pt);
       \end{tikzpicture}
    \caption{$\EndBlock^{(2)}_0$}
    \end{subfigure}
    \begin{subfigure}[b]{0.24\textwidth}
    \centering
        \begin{tikzpicture}
        \draw (0,0.6) -- (1,0.6) -- (1,-0.6) -- (0,-0.6);
        \draw (0.5,-0.6) circle (0pt) node[above]{$\circ$};
        \draw (1,0.6) -- (2,0.6) -- (2,-0.6) -- (1,-0.6);
        \draw (0,1.75) circle (0pt);
        \filldraw[fill=white] (0,0.6) circle (2pt) node[above left]{$s$};
        \filldraw[fill=white] (0,-0.6) circle (2pt) node[below left]{$r$};
        \filldraw (1,0.6) circle (2pt);
        \filldraw (1,-0.6) circle (2pt);
        \filldraw[fill=white] (2,0.6) circle (2pt) node[above right]{$w$};
        \filldraw[fill=white] (2,-0.6) circle (2pt) node[below right]{$u$};
        \end{tikzpicture}
    \caption{$\EndBlock^{(1)}$}
    \end{subfigure}
    \hfill
    \begin{subfigure}[b]{0.24\textwidth}
    \centering
        \begin{tikzpicture}
        \draw (0,0.6) -- (1.5,0.6) -- (1.5,-0.1) -- (1,-0.6);
        \draw (0.5,-0.6) circle (0pt) node[above]{$\circ$};
        \draw (1.5,0.25) circle (0pt) node[left]{$\circ$};
        \draw (0,-0.6) -- (1,-0.6) -- (2,-0.6) -- (1.5,-0.1);
        \filldraw[fill=white] (0,0.6) circle (2pt) node[above left]{$s$};
        \filldraw[fill=white] (0,-0.6) circle (2pt) node[below left]{$r$};
        \filldraw[fill=white] (1.5,0.6) circle (2pt) node[above right]{$w$};
        \filldraw (1,-0.6) circle (2pt);
        \filldraw (1.5,-0.1) circle (2pt);
        \filldraw[fill=white] (2,-0.6) circle (2pt) node[below right]{$u$};
        \end{tikzpicture}
    \caption{$\EndBlock^{(2)}$}
    \end{subfigure}
    \hfill
    \begin{subfigure}[b]{0.24\textwidth}
    \centering
        \begin{tikzpicture}
        \draw (0,0.6) -- (1,0.6) -- (1,-0.6) -- (0,-0.6);
        \filldraw[fill=white] (0,0.6) circle (2pt) node[above left]{$s$};
        \filldraw[fill=white] (0,-0.6) circle (2pt) node[below left]{$r$};
        \filldraw[fill=white] (1,0.6) circle (2pt) node[above right]{$w$};
        \filldraw[fill=white] (1,-0.6) circle (2pt) node[below right]{$u$};
        \end{tikzpicture}
    \caption{$\EndBlock^{(3)}$}
    \end{subfigure}
    \hfill
    \begin{subfigure}[b]{0.24\textwidth}
    \centering
        \begin{tikzpicture}
        \draw (0,0.6) -- (1,0.6) -- (1,-0.6);
        \filldraw[fill=white] (0,0.6) circle (2pt) node[above left]{$s$};
        \filldraw[fill=white] (1,0.6) circle (2pt) node[above right]{$w$};
        \filldraw[fill=white] (1,-0.6) circle (2pt) node[below]{$r=u$};
        \end{tikzpicture}
    \caption{$\EndBlock^{(4)}$}
    \end{subfigure}
    \hspace*{\fill}%
    \caption{Diagrams of the $\EndBlock_0$, and $\EndBlock$ functions.}
    \label{fig:EndBlockFunctions}
\end{figure}

For $j\in\left\{1,2\right\}$ we define $\OpEndBlock^{(j)}_0\colon \mathcal{D}\left(\OpEndBlock^{(j)}_0\right) \to L^2\left(\X^2\right)$ to be the linear operators acting on $f\in L^2\left(\X\right)$ as
\begin{equation}
    \OpEndBlock^{(j)}_0 f\left(w,u\right) = \int \EndBlock^{(j)}_0(w,u,y) f(y)\nu\left(\dd y\right),
\end{equation}
and for $j\in\left\{1,2,3,4\right\}$ we define $\OpEndBlock^{(j)}\colon \mathcal{D}\left(\OpEndBlock^{(j)}\right) \to L^2\left(\X^2\right)$ to be the linear operators acting on $f\in L^2\left(\X^2\right)$ as
\begin{equation}
    \OpEndBlock^{(j)} f\left(w,u\right) = \int \EndBlock^{(j)}(w,u,r,s) f(r,s)\nu^{\otimes2}\left(\dd r, \dd s\right).
\end{equation}

In our terminology, the important observation of \cite{HeyHofLasMat19} was that for some $m\geq 0$ the `later' segments of each diagram can be bounded using $\big(\OpEndBlock^m\OpEndBlock_0\big)^\dagger$, the adjoint of $\OpEndBlock^m\OpEndBlock_0$. In terms of the kernel functions, taking the adjoint in this case amounts to reflecting the `input' and `output' arguments, and for the diagrams this amounts to reflecting in the vertical plane.

\begin{restatable}{lemma}{StartEndBounds}
\label{lem:StartEndBounds}
For $n\geq 1$, $m\in\left\{0,\ldots,n\right\}$, $j_0\in\left\{1,2,3\right\}$, $j_1,\ldots,j_{n-1}\in\left\{1,2,3,4\right\}$, and $j_n\in\left\{1,2\right\}$,
\begin{align}
    \OneNorm{\Psi^{(j_m)}\ldots\Psi^{(j_1)}\Psi^{(j_0)}_0} &\leq \Ulam \Vlam^m,\\
    \OneNorm{\Psi^{(j_n)}_n\Psi^{(j_{n-1})}\ldots\Psi^{(j_{m})}} &\leq \Ulam \Vlam^{n-m},\\
    \OneNorm{\OpEndBlock^{(j_m)}\ldots\OpEndBlock^{(j_1)}\OpEndBlock^{(j_0)}_0} &\leq \begin{cases}
    \Ulam &: m=0\\
    \Ulam^2\Vlam^{m-1} &: m\geq 1. 
    \end{cases}
\end{align}
\end{restatable}

\begin{definition}
In addition to the terms defined in Definition~\ref{def:triangles}, we will also use
\begin{align*}
    \Wk &:= \lambda\esssup_{x,y\in\X}\int \tklam(x,u)\tlam(u,y)\nu\left(\dd u\right) = \lambda\InfNorm{\Optklam\Optlam},\\
    \WkBar &:= \lambda\esssup_{\xbar\in\Rd,a_1,\ldots,a_4\in\Ecal}\int \tklam(\xbar-\ubar;a_1,a_2)\tlam(\ubar;a_3,a_4)\dd u,\\
    \HkBar  &:= \lambda^5 \esssup_{\xbar_1,\xbar_2\in\Rd,a_1,\ldots,a_{16}\in\Ecal}\int \tlam(\sbar-\xbar_1;a_1,a_2)\tlam(\ubar;a_3,a_4) \tlam(\vbar-\sbar;a_5,a_6) \tlam(\vbar+\xbar_2-\tbar;a_7,a_8)\nonumber\\ 
    & \hspace{1.5cm}\times \tlam(\sbar-\wbar;a_9,a_{10})\tlam(\wbar-\ubar;a_{11},a_{12}) \tlam(\tbar-\wbar;a_{13},a_{14}) \tklam(\tbar-\ubar;a_{15},a_{16}) \dd\tbar\dd\wbar\dd\zbar\dd\ubar.
\end{align*}
Note that we only use $\WkBar$ in Propositions~\ref{thm:DisplacementNgeq2} and \ref{thm:DisplacementNeq1} because $\WkBar\geq \Wk$, but in the proof we will sometimes use $\Wk$ to make the derivation easier to follow. The diagrams for $\WkBar$ and $\HkBar$ are in Figure~\ref{fig:WkandHk}. 

Also recall the sets $\left\{B\left(x\right)\right\}_{x\in\X}$ assumed to exist in Assumption~\ref{Assump:BallDecay}. Then define 
\begin{align*}
    \mathbb{B} &:= \esssup_{x\in\X}\nu\left(B\left(x\right)\right),\\
    \trilamB &:= \lambda^2 \esssup_{r,s\in\X: r\notin{B\left(s\right)}}\int\tlamo\left(r,u\right)\tlam\left(u,v\right)\tlam\left(v,s\right)\nu^{\otimes2}\left(\dd u, \dd v\right).
\end{align*}
Note that these are only required to deal with one specific diagram in the following $n=1$ case.
\end{definition}

\begin{figure}
    \centering
    \begin{subfigure}[b]{0.3\textwidth}
    \centering
    \begin{tikzpicture}[scale=2]
        \draw (0,0.6) -- (1,0) -- (0,-0.6);
        \draw[<->] (0,0.8) -- (1,0.2);
        \draw (0.5,0.5) node[cross=4pt, rotate=-30]{};
        \filldraw[fill=white] (0,0.6) circle (2pt);
        \filldraw[fill=white] (0,-0.6) circle (2pt);
        \node[mark size=4pt] at (1,0) {\pgfuseplotmark{square*}};
    \end{tikzpicture}
    \caption{$\WkBar$}
    \end{subfigure}
    \hfill
    \begin{subfigure}[b]{0.6\textwidth}
    \centering
    \begin{tikzpicture}[scale=2]
        \draw (0,0.6) -- (1.5,0.6) -- (1.5,-0.1) -- (1,-0.6);
        \draw (0,-0.6) -- (1,-0.6) -- (2,-0.6) -- (1.5,-0.1);
        \draw (1.5,0.6) -- (3,0.6);
        \draw (2,-0.6) -- (3,-0.6);
        \draw[dashed] (3,0.6) -- (3,-0.6);
        \filldraw[fill=white] (0,0.6) circle (2pt);
        \filldraw[fill=white] (0,-0.6) circle (2pt);
        \filldraw (1.5,0.6) circle (2pt);
        \filldraw[fill=white] (3,0.6) circle (2pt);
        \node[mark size=4pt] at (3,-0.6) {\pgfuseplotmark{square*}};
        \filldraw (1,-0.6) circle (2pt);
        \filldraw (1.5,-0.1) circle (2pt);
        \filldraw (2,-0.6) circle (2pt);
        \draw[<->, thick] (1,-0.75) -- (2,-0.75);
        \draw (1.5,-0.75) node[cross=4pt]{};
    \end{tikzpicture}
    \caption{$\HkBar$}
    \end{subfigure}
    \caption{Diagrams representing $\WkBar$ and $\HkBar$. In this figure, the marks on edges incident to filled square vertices need not be equal.\label{fig:WkandHk}}
\end{figure}
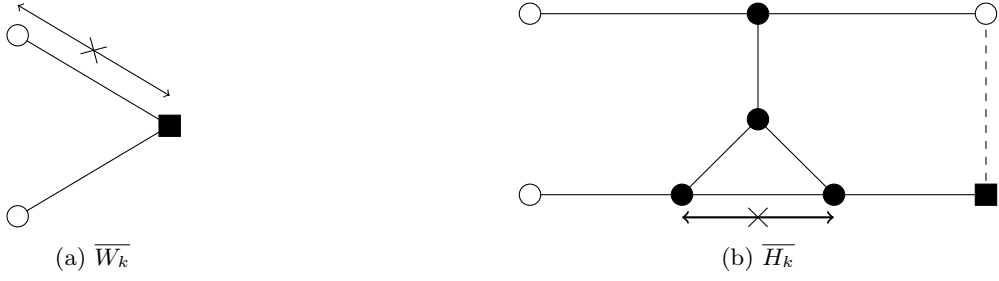

We already have a bound for the $n=0$ version of the displaced operators from Proposition~\ref{thm:DB:Pi0_bounds}. The following proposition provides the bound for the $n=1$ version. 
\begin{restatable}{prop}{DisplacementNeqOne}
\label{thm:DisplacementNeq1}
For $k\in\Rd$,
\begin{equation}
    \lambda\OpNorm{\fOpLacelam^{(1)}\left(0\right) - \fOpLacelam^{(1)}\left(k\right)} \leq 13\WkBar\Vlam\Ulam + \lambda \OneNorm{\Opconnf_k}\left(\Vlam + \trilamB\right) + \lambda^3\InfNorm{\Opconnf\Optlam\Optklam\Opconnf} + 4\lambda\mathbb{B}\WkBar.
\end{equation}
\end{restatable}

The following proposition then provides the bounds for the $n\geq 2$ versions, for which the proof will be qualitatively different.

\begin{restatable}{prop}{DisplacementNgeqTwo}
\label{thm:DisplacementNgeq2}
For $k\in\Rd$ and $n\geq 2$,
\begin{equation}
    \lambda\OpNorm{\fOpLacelam^{(n)}\left(0\right) - \fOpLacelam^{(n)}\left(k\right)} \leq
    \begin{cases}
    \left(n+1\right)^24^{n-1}\left(50\WkBar \Vlam + 6\HkBar\right) \Ulam^4 \Vlam^{n-4} &: n\geq4\\
    16\times16\left(50\WkBar \Vlam^2 + 6\HkBar\right) \Ulam^2 &: n=3\\
    9\times4\left(50\WkBar \Vlam + 6\HkBar\right) \Ulam^2 &: n=2.
    \end{cases}
\end{equation}
\end{restatable}

\subsection{The Bootstrap Function}\label{sec:bootstrapanalysis}

Now we have some bounds for the terms in our expansion, we will in turn bound these using a bootstrap function. To aid readability, in our following arguments we will assume that
\begin{equation}
    \SupSpec{\fOpconnf(0)} = 1.
\end{equation}
This does not reduce the generality of our result. First note that since $\fconnf(0;a,b)>0$ for a positive $\Pcal$-measure set of $a,b\in\Ecal$, we have $\SupSpec{\fOpconnf(0)}>0$. Now suppose we scale $\Rd$ by a factor of $q^{-1/d}$ (that is, the new unit radius ball is the previous ball of radius $q^{1/d}$), and leave the mark space $\Ecal$ unchanged. Then we find that the system we are left with has  the distribution of an RCM model with the new parameters
\begin{equation}
    \lambda^* = \E_\lambda\left[\abs*{\eta \cap \left[0, q^{\frac{1}{d}}\right]^d }\right] = \lambda q, \qquad \connf^*\left(\xbar;a,b\right) = \connf\left(q^{\frac{1}{d}}\xbar;a,b \right).
\end{equation}
Therefore for all $a,b\in\Ecal$ we have
\begin{equation}
    \fconnf^*(0;a,b) = \int \connf^*(\xbar;a,b) \dd \xbar = q^{-1}\int \connf^*(\xbar;a,b) \dd \xbar = q^{-1}\fconnf(\xbar;a,b),
\end{equation}
and $\SupSpec{\fOpconnf^*(0)} = q^{-1} \SupSpec{\fOpconnf(0)}$. We can therefore achieve our assumption by choosing $q=\SupSpec{\fOpconnf(0)}$. Furthermore, under this scaling we find that the old event $\left\{\conn{\left(\xbar,a\right)}{\left(\ybar,b\right)}{\xi^{x,y}}\right\}$ becomes the new event $\left\{ \conn{\left(q^{-1/d}\xbar,a\right)}{\left(q^{-1/d}\ybar,b\right)}{\xi^{\left(q^{-1/d}\xbar,a\right),\left(q^{-1/d}\ybar,b\right)}} \right\}$, and we get
\begin{equation}
    \tau^*_{\lambda^*}(\xbar;a,b) =  \tlam\left(q^\frac{1}{d}\xbar;a,b\right)
\end{equation}
where $\tau^*_{\lambda^*}$ is the two-point function in the RCM governed by the connection function $\connf^*$ and intensity $\lambda^*$. 

It is also worth noting that the scaling choice to have $\SupSpec{\fOpconnf(0)}=1$ means that \ref{Assump:2ndMoment} reduces to the condition that there exists a $d$-independent constant $C>0$ such that
\begin{align}
    \label{eqn:Sup-Exp_Ratio_2}
    \OneNorm{\fOpconnf(0)} \leq \TwoNorm{\fOpconnf(0)} \leq \InfNorm{\fOpconnf(0)} &\leq C,\\
    \label{eqn:directional2ndMoment_2}
    \OneNorm{\fOpconnf(0) - \fOpconnf(k)} \leq \InfNorm{\fOpconnf(0) - \fOpconnf(k)} &\leq C \left(1 - \SupSpec{\fOpconnf(k)}\right).
\end{align}
Note that the inequalities relating $\OneNorm{\cdot}$, $\TwoNorm{\cdot}$, and $\InfNorm{\cdot}$ follow from Jensen's inequality and a supremum bound on a probability space.

\begin{definition}[Bootstrap Function]
For $\lambda<\lambda_O$ and $k,l\in\Rd$, define
\begin{align}
    \mulam & := 1- \frac{1}{\SupSpec{\fOptlam(0)}},\\
    \fgmu\left(k\right) &:= \frac{1}{1 - \mulam \SupSpec{\fOpconnf\left(k\right)}},\\
    \widehat{J}_{\mulam}\left(k,l\right) &:=
    \left(1 - \SupSpec{\fOpconnf\left(k\right)}\right)\left(\fgmu\left(l-k\right)\fgmu\left(l\right) + \fgmu\left(l\right)\fgmu\left(l+k\right) + \fgmu\left(l-k\right)\fgmu\left(l+k\right)\right).
\end{align}
From these we can then define the bootstrap function $f:= f_1\vee f_2\vee f_3$, where
\begin{equation}\label{eq:def-f}
    f_1\left(\lambda\right) := \lambda, \qquad
    f_2\left(\lambda\right) := \esssup_{k\in\Rd} \frac{\OpNorm{\fOptlam\left(k\right)}}{\fgmu\left(k\right)}, \qquad
    f_3\left(\lambda\right) := \esssup_{k,l\in\Rd}\frac{\OpNorm{\fOptklam\left(l\right)}}{\widehat{J}_{\mulam}\left(k,l\right)}.
\end{equation}    
\end{definition}

\subsubsection{Bounding the Lace Expansion Coefficient Operator}
\label{sec:BootstrapBounds1}
We have managed to bound the lace expansion coefficient operator and its displacement in terms of a variety of objects. For the lace expansion coefficient operator itself (and its associated Fourier transformed operators), these are:
\begin{equation}
    \lambda, \trilam, \trilamooBar, \OneNorm{\Opconnf}.
\end{equation}
As a guide, we expect $\lambda$, $\trilamooBar$, and $\OneNorm{\Opconnf}$ to be bounded uniformly in $d$ and $\lambda<\lambda_O$, whereas we expect $\trilam$ to decay uniformly in $\lambda<\lambda_O$ as $d\to\infty$. For some, these properties are easy to prove. The boundedness of $\OneNorm{\Opconnf}$ follows from
\begin{equation}
    \OneNorm{\Opconnf} = \OneNorm{\fOpconnf(0)} \leq C,
\end{equation}
where we have used \ref{Assump:2ndMoment} via \eqref{eqn:Sup-Exp_Ratio_2}. Regarding $\lambda$, since we are only interested in $\lambda<\lambda_O$ the question is perhaps more clearly written as \emph{``Is $\lambda_O$ uniformly bounded in $d$?"}. While there may be more elementary ways of proving this directly, we will get it for free from the bootstrap argument. Since we have the bound $\lambda \leq f(\lambda)$, proving the uniform boundedness of $f(\lambda)$ with respect to $d$ and $\lambda$ for $\lambda<\lambda_O$ will prove the same for $\lambda_O$.

This will leave us with only $\trilam$ and $\trilamooBar$. Recall that we can write
\begin{equation}
    \trilam = \lambda^2\InfNorm{\Optlam^3}, \quad \trilamoo = \lambda^2\InfNorm{\Optlam^3} + \lambda\InfNorm{\Optlam^2} + 1.
\end{equation}
Then $\trilamooBar$ differs from $\trilamoo$ in that now the supremum is taken over all the marks, and adjacent $\tlam$ can have different marks. To get the desired decay and boundedness behaviour, we will therefore only need to prove that $\InfNorm{\Optlam^3}$ decays and to bound
\begin{align}
    &\esssup_{\xbar\in\Rd,a_1,\ldots,a_6\in\Ecal}\int\tlam\left(\xbar-\ubar;a_1,a_2\right)\tlam\left(\ubar - \vbar;a_3,a_4\right)\tlam\left(\vbar;a_5,a_6\right)\dd \ubar \dd \vbar,\\
    &\esssup_{\xbar\in\Rd,a_1,\ldots,a_4\in\Ecal}\int\tlam\left(\xbar-\ubar;a_1,a_2\right)\tlam\left(\ubar;a_3,a_4\right)\dd \ubar.
\end{align}
We begin by proving the decay of $\InfNorm{\Optlam^3}$.

Let us define the function $\overline{\epsilon}\colon \N \to \R_+$, where
\begin{equation}
    \overline{\epsilon}(d) := \begin{cases}
    \frac{1}{d} + \frac{\log d}{\log g(d)} &: \lim_{d\to\infty}g(d)\rho^{-d}\Gamma\left(\frac{d}{2}+1\right)^2 = 0 \qquad\forall\rho>0,\\
    0&: \text{otherwise}.
    \end{cases}
\end{equation}
Recall $g(d)$ is the function defined in \ref{Assump:BallDecay}. In particular, $\overline{\epsilon}(d) = 0$ unless $g(d)$ approaches zero particularly quickly, and $\beta(d) = g(d)^{\frac{1}{4}-\frac{3}{2}\overline{\epsilon}(d)}$.

\begin{restatable}[Bound for the $n$-gon diagrams]{prop}{TauNbound}
\label{thm:TauNbound}
Let $\lambda < \lambda_O$. Then for each $n$ such that $d>2n$, there exists finite $\const= c(n,f(\lambda))$ (where $x\mapsto c(n,x)$ is increasing and independent of $d$) such that for $d>6$
\begin{equation}
    \InfNorm{\Optlam^n} \leq \begin{cases}
    \const &: n=1,2,\\
    \const g(d)^{\frac{1}{2}-n\overline{\epsilon}(d)}&: n\geq 3.
    \end{cases}
\end{equation}
\end{restatable}

The following is not a direct corollary of the previous proposition, because $\trilamooBar$ includes suprema over the intermediate marks. Nevertheless, the proof in Appendix~\ref{appendix:BoundingwithBootstrap} follows a similar strategy. 

\begin{restatable}{lemma}{BoundTrilamBar}
\label{lem:BoundTrilamBar}
For $\lambda<\lambda_O$, there exists $\const=c\left(f\left(\lambda\right)\right)$ (where $x\mapsto c\left(x\right)$ is increasing and independent of $d$) such that for $d>6$,
\begin{align}
    \esssup_{\xbar\in\Rd,a_1,\ldots,a_6\in\Ecal}&\int\tlam\left(\xbar-\ubar;a_1,a_2\right)\tlam\left(\ubar - \vbar;a_3,a_4\right)\tlam\left(\vbar;a_5,a_6\right)\dd \ubar \dd \vbar \leq \const,\\
    \esssup_{\xbar\in\Rd,a_1,\ldots,a_4\in\Ecal}&\int\tlam\left(\xbar-\ubar;a_1,a_2\right)\tlam\left(\ubar;a_3,a_4\right)\dd \ubar \leq \const,
\end{align}
and therefore
\begin{equation}
    \trilamooBar \leq 1 + 2\const.
\end{equation}
\end{restatable}

\subsubsection{Bounding the Lace Expansion Coefficient Operator Displacement}
\label{sec:BootstrapBounds2}

We now consider the objects required to bound the displacement of the lace expansion coefficient operator. In addition to the objects used to bound the lace expansion coefficient operator itself, we require bounds on
\begin{equation}
    \OneNorm{\Opconnf_k}, \InfNorm{\Opconnf\Optklam\Optlam\Opconnf}, \WkBar, \HkBar, \trilamB, \mathbb{B}.
\end{equation}
The decay $\mathbb{B}\leq g(d)$ follows from Assumption~\ref{Assump:BallDecay}, and the following Observation~\ref{obs:trilamB_bound} demonstrates the decay of $\trilamB$.
\begin{observation}
\label{obs:trilamB_bound}
    In proving Proposition~\ref{thm:TauNbound} (in Appendix~\ref{appendix:BoundingwithBootstrap}), it is proven that for all $\lambda\geq 0$ and $x,y\in\X$
    \begin{equation}
        \label{eqn:tau-phi-ExpansionfunctionINTEXT}
        \tlam(x,y) \leq \connf(x,y) + \lambda \int \connf(x,u)\tlam(u,y)\nu\left(\dd u\right).
    \end{equation}
    This inequality allows us to bound $\trilamB$ in terms of other terms. By using the inequality twice and appropriately bounding $\connf$ with $\tlam$, we get
\begin{equation}
    \int\tlam(r,u)\tlam(u,s)\nu(\dd u) \leq \int\connf(r,u)\connf(u,s)\nu(\dd u) + 2\lambda\int\tlam(r,v)\tlam(v,u)\tlam(u,s)\nu^{\otimes 2}(\dd u, \dd v)
\end{equation}
for any $r,s\in\X$. Therefore
\begin{equation}
    \trilamB \leq 3\trilam + \lambda\esssup_{r\not\in B(s)}\int\connf(r,u)\connf(u,s)\nu(\dd u) \leq 3\trilam + \lambda g(d),
\end{equation}
where we have used \ref{Assump:BallDecay}. Therefore Proposition~\ref{thm:TauNbound} implies that for $d>6$ there exists $\const$ such that 
\begin{equation}
    \trilamB \leq \const\beta^2.
\end{equation}
\end{observation}

We are then left with $\OneNorm{\Opconnf_k}$, $\InfNorm{\Opconnf\Optklam\Optlam\Opconnf}$, $\WkBar$, and $\HkBar$, and we will deal them in this order. First note that
\begin{multline}
    \OneNorm{\Opconnf_k} = \esssup_{y\in\X}\int\connf_k(x,y)\nu\left(\dd x\right)  = \esssup_{b\in\Ecal}\int(1-\cos\left(k\cdot \xbar\right))\connf(\xbar;a,b)\dd \xbar\Pcal(\dd a) \\ = \OneNorm{\fOpconnf_k(0)} = \OneNorm{\fOpconnf(0) - \fOpconnf(k)}.
\end{multline}
The following lemma therefore also allows us to bound $\OneNorm{\Opconnf_k}$. Its extra generality will be required later.

\begin{restatable}{lemma}{connfComparison}
\label{thm:connfComparison}
For all $k,l\in\Rd$,
\begin{equation}
\label{eqn:connfComparison}
    \OpNorm{\fOpconnf\left(l\right) - \fOpconnf\left(l+k\right)} \leq \OneNorm{\fOpconnf\left(l\right) - \fOpconnf\left(l+k\right)} \leq \InfNorm{\fOpconnf\left(l\right) - \fOpconnf\left(l+k\right)} \leq C\left[1 - \SupSpec{\fOpconnf\left(k\right)}\right],
\end{equation}
where $C>0$ is the $d$-independent constant given in \ref{Assump:2ndMoment}.
\end{restatable}

We now address $\InfNorm{\Opconnf\Optklam\Optlam\Opconnf}$.

\begin{restatable}{lemma}{PTkTPBound}
\label{lem:PTkTPBound}
For $\lambda<\lambda_O$, there exists $\const=c\left(f\left(\lambda\right)\right)$ (where $x\mapsto c\left(x\right)$ is increasing and independent of $d$) such that for $d>6$,
\begin{equation}
    \InfNorm{\Opconnf\Optklam\Optlam\Opconnf}\leq \const\beta^2\left(1-\SupSpec{\fOpconnf(k)}\right).
\end{equation}
\end{restatable}

The following lemma allows us to deal with occurrences of $\WkBar$.

\begin{restatable}{lemma}{WkBound}
\label{lem:W_kBound}
Let $x,y \in \X$. Then
\begin{equation}
    \tklam(x,y) \leq \connf_k(x,y) + 2\lambda \int \left(\tlam(x,u)\connf_k(u,y) + \tklam(x,u)\connf(u,y)\right)\nu\left(\dd u\right). \label{eqn:tau-phiExpansionDisplacement}
\end{equation}
For $\lambda<\lambda_O$, there exists $\const=c\left(f\left(\lambda\right)\right)$ (where $x\mapsto c\left(x\right)$ is increasing and independent of $d$) such that for $d>6$,
\begin{equation}
    \WkBar \leq \const\left(1-\SupSpec{\fOpconnf(k)}\right).
\end{equation}
\end{restatable}

We are now only left with bounding $\HkBar$.

\begin{restatable}{lemma}{MartiniBound}
\label{lem:Martini_Bound}
For $\lambda<\lambda_O$, there exists $\const=c\left(f\left(\lambda\right)\right)$ (where $x\mapsto c\left(x\right)$ is increasing and independent of $d$) such that for $d>6$,
\begin{equation}
    \HkBar \leq \const\beta^2\left(1-\SupSpec{\fOpconnf(k)}\right).
\end{equation}
\end{restatable}

Now we can bring together these bounds to show that the finite lace expansion of Proposition~\ref{thm:OperatorLaceExpansion} converges to an Ornstein-Zernike equation (OZE) for both the raw operators on $L^2\left(\X\right)$ and the Fourier-transformed operators on $L^2\left(\Ecal\right)$. Proposition~\ref{thm:convergenceoflaceexpansion} does this for $\lambda\in\left[0,\lambda_O\right)$, with the caveat that there is a constant $\const$ that is increasing in the bootstrap function $f$ (but independent of $d$) appearing in the bounds. 


\begin{prop}[Convergence of the operator lace expansion and OZE] \label{thm:convergenceoflaceexpansion} \
Let $\lambda \in [0, \lambda_O)$ and $d>6$ be sufficiently large. Then there exists $\const=c(f(\lambda))$ (which is increasing in $f$ and independent of $d$) such that
\begin{align}
    \sum_{n\geq 0}\OpNorm{\OpLacelam^{(n)}}\leq \const \beta, &\qquad \sum_{n\geq 0}\OpNorm{\fOpLacelam^{(n)}(k)}\leq \const \beta, \label{eq:BA:convLE_intPi_bds1}\\ \sum_{n\geq 0}\OpNorm{\fOpLacelam^{(n)}\left(0\right) - \fOpLacelam^{(n)}\left(k\right)} &\leq \const \left(\SupSpec{\fOpconnf(0)}-\SupSpec{\fOpconnf(k)}\right) \beta, \label{eq:BA:convLE_intPi_bds2}\\
	\OpNorm{R_{\lambda, n}}  &\leq \lambda \OpNorm{\Optlam} \const\beta^{n}. \label{eq:BA:convLE_R_bd}
\end{align}
Furthermore, the limit $\OpLacelam := \lim_{M\to\infty} \sum^M_{n=0}\OpLacelam^{(n)}$ exists and is a bounded operator with Fourier transform $\fOpLacelam\left(k\right) = \lim_{M\to\infty} \sum^M_{n=0}\fOpLacelam^{(n)}\left(k\right)$ for $k\in\Rd$. The operators $\Optlam$ and $\fOptlam(k)$ satisfy operator Ornstein-Zernike equations, taking the form
\begin{equation}
    \Optlam = \Opconnf + \OpLacelam + \lambda \Optlam(\Opconnf+\OpLacelam), \label{eq:LE_identity_OZE}
\end{equation}
and 
\begin{equation}
    \fOptlam\left(k\right) = \fOpconnf\left(k\right) + \fOpLacelam\left(k\right) + \lambda \fOptlam\left(k\right)\left(\fOpconnf\left(k\right)+\fOpLacelam\left(k\right)\right), \label{eq:LE_identity_OZE_Fourier}
\end{equation}
for all $\lambda < \lambda_O$ and $k\in\Rd$.
\end{prop}

\begin{proof}
So far in this section we have provided bounds for the building blocks of the bounds appearing in Propositions~\ref{thm:DB:Pi0_bounds},~\ref{prop:LaceCoefficientBound},~\ref{thm:DisplacementNgeq2}, and~\ref{thm:DisplacementNeq1}. We still have to convert these into bounds for the composite terms $\Ulam$ and $\Vlam$. From their definitions and the bounds on the building blocks, it is simple to see that there exists $\const'$ such that for $d>6$
\begin{equation}
    \Ulam \leq \const', \qquad\Vlam \leq \const'\beta.
\end{equation}
This then implies that there exists $\const''$ such that for $d>6$
\begin{equation}
    \label{eqn:BoundonPiN}
    \OpNorm{\OpLacelam^{(n)}} \leq \begin{cases}
    \const''\beta^2 &: n=0,\\
    \const''\left(\const''\beta\right)^n &: n\geq 1.
    \end{cases}
\end{equation}
If $d$ is sufficiently large we have $\const''\beta<1/2$ (for example), and therefore there exists $\const$ such that
\begin{equation}
    \sum_{n\geq0}\OpNorm{\OpLacelam^{(n)}} \leq \const''\left(\beta^2 + \frac{\const''\beta}{1-\const''\beta}\right) \leq \const\beta.
\end{equation}
When we want to consider the displaced terms, we also make use of Observation~\ref{obs:trilamB_bound}, and Lemmas~\ref{thm:connfComparison},~\ref{lem:PTkTPBound},~\ref{lem:W_kBound}, and~\ref{lem:Martini_Bound} to give us the existence of $\const'''$ such that for $d>6$
\begin{equation}
    \OpNorm{\fOpLacelam^{(n)}\left(0\right) - \fOpLacelam^{(n)}\left(k\right)} \leq \begin{cases}
    \const'''\beta^2\left(1 - \SupSpec{\fOpconnf(k)}\right) &: n=0,3,\\
    \const'''\beta\left(1 - \SupSpec{\fOpconnf(k)}\right) &: n=1,2,\\
    \left(n+1\right)^2\const'''\left(\const'''\beta\right)^{n-3}\left(1 - \SupSpec{\fOpconnf(k)}\right) &: n\geq 4.
    \end{cases}
\end{equation}
If $d$ is sufficiently large we have $\const'''\beta<1/2$ (for example), and therefore there exists $\const$ such that
\begin{multline}
    \sum_{n\geq0}\OpNorm{\fOpLacelam^{(n)}\left(0\right) - \fOpLacelam^{(n)}\left(k\right)} \leq \const'''\left(2\beta^2 + 2\beta + \frac{16\left(\const'''\beta\right)^3-39\left(\const'''\beta\right)^2 + 25\const'''\beta}{\left(1-\const'''\beta\right)^3}\right)\left(1 - \SupSpec{\fOpconnf(k)}\right) \\\leq \const\beta\left(1 - \SupSpec{\fOpconnf(k)}\right).
\end{multline}
Lemma~\ref{lem:remainderBound} then allows us to use \eqref{eqn:BoundonPiN} to bound $\OpNorm{R_{\lambda,n}}$. This immediately gives the required result for $n\geq 1$. For $n=0$, we simply additionally require $d$ to be large enough that $\beta<1$ and therefore $\beta^2<\beta^0$.

Note that the dual space of a Banach space (endowed with the operator norm) is also a Banach space. In particular, it is complete. Since $\OpNorm{\OpLacelam^{(n)}}\leq \const''\left(\const''\beta\right)^n$ for $n\geq1$, the sequence $\OpLace_{\lambda,n}$ is a Cauchy sequence in the dual space of $L^2(\X)$ for sufficiently large $d$. Therefore the limit $\OpLacelam$ exists and is a bounded linear operator on $L^2(\X)$. Since the Fourier transform is a unitary linear operator it is bounded and continuous, and therefore the same argument says that the limit $\fOpLacelam(k)$ exists and is indeed the Fourier transform of $\OpLacelam$.

Since $\OpNorm{\Optlam}$ is finite for $\lambda<\lambda_O$, we have $\OpNorm{R_{\lambda,n}}\to0$ as $n\to\infty$ if $d$ is sufficiently large. As a a consequence of this and equation \eqref{eqn:OperatorLaceExpansion} from Proposition~\ref{thm:OperatorLaceExpansion}, the $\Optlam$ operator satisfies the operator Ornstein-Zernike equation.
\end{proof}

The following lemma has two main uses in what follows. Firstly, it takes the OZEs from Proposition~\ref{thm:convergenceoflaceexpansion} and uses it to write the two-point operators in terms of the sum of the adjacency operators and the lace expansion coefficient operators. Secondly, it uses this and an intermediate value theorem argument to bound the spectral supremum of the sum of the adjacency operators and the lace expansion coefficient operators.

\begin{lemma}
\label{lem:Optlam-a}
Let $\lambda\in\left[0,\lambda_O\right)$ and $d>6$ be sufficiently large. Then the operators $\left(\Id - \lambda\left(\Opconnf + \OpLacelam\right)\right)$ and $\left(\Id - \lambda\left(\fOpconnf(k) + \fOpLacelam(k)\right)\right)$ all have bounded linear inverses and
\begin{align}
    \Optlam &= \left(\Opconnf + \OpLacelam\right)\left(\Id - \lambda\left(\Opconnf + \OpLacelam\right)\right)^{-1}, \label{eqn:Optlam_Inverted}\\
    \fOptlam(k) &= \left(\fOpconnf(k) + \fOpLacelam(k)\right)\left(\Id - \lambda\left(\fOpconnf(k) + \fOpLacelam(k)\right)\right)^{-1}, \qquad\forall k\in\Rd. \label{eqn:Optlam_Inverted_Fourier}
\end{align}
Furthermore, $\lambda\mapsto \OpLacelam$ and $\lambda\mapsto \fOpLacelam(k)$ are all continuous in the operator norm topology and
\begin{align}
    \lambda\SupSpec{\Opconnf + \OpLacelam} &< 1,\\
    \lambda \SupSpec{\fOpconnf(k) + \fOpLacelam(k)} &< 1, \qquad\forall k\in\Rd.
\end{align}
\end{lemma}

\begin{proof}
We first prove that $\Id - \lambda\left(\Opconnf + \OpLacelam\right)$ has a bounded linear inverse. This is clear for $\lambda = 0$, so we only need to consider $\lambda\in\left(0,\lambda_O\right)$. Suppose for contradiction that $\Id - \lambda\left(\Opconnf + \OpLacelam\right)$ does not have a bounded linear inverse. Then there exists a sequence $f_n\in L^2(\X)$ such that $\norm*{f_n}_2 = 1$ and $\norm*{\left(\Id - \lambda\left(\Opconnf + \OpLacelam\right)\right)f_n}_2\to0$. This also means that $\lambda\norm*{\left(\Opconnf + \OpLacelam\right)f_n}\to 1$. Then from Lemma~\ref{thm:Spectrum_Subset} and sub-criticality we have $\OpNorm{\Optlam} \leq \OpNorm{\fOptlam(0)} < \infty$, and therefore the OZE equation \eqref{eq:LE_identity_OZE} implies that $\norm*{\left(\Opconnf + \OpLacelam\right)f_n}\to 0$, a contradiction. A similar argument proves the corresponding statement for the Fourier transformed operators.

Once we know that the inverses exist, it is only a matter of rearranging the OZE equations to get the expressions for $\Optlam$ and $\fOptlam(k)$ in terms of the other operators.

Since $\Opconnf + \OpLacelam$ is a bounded operator, a similar argument proves that $\Id + \lambda \Optlam$ has a bounded inverse, and therefore
\begin{equation}
    \Opconnf + \OpLacelam = \left(\Id + \lambda \Optlam\right)^{-1}\Optlam.
\end{equation}
From Corollary~\ref{thm:differentiabilityofOpT} we have the continuity of $\lambda\mapsto\Optlam$, and therefore the continuity $\Opconnf + \OpLacelam$ (and therefore $\OpLacelam$). We also have the continuity of $\lambda\mapsto\fOptlam(k)$ from Corollary~\ref{thm:differentiabilityofOpT}, and we similarly get the continuity of $\fOpLacelam(k)$.

The continuity of $\Opconnf + \OpLacelam$ implies the continuity of $\lambda\SupSpec{\Opconnf + \OpLacelam}$ (via Lemma~\ref{lem:SupSpecTriangle}). At $\lambda = 0$ it is clear that $\lambda\SupSpec{\Opconnf + \OpLacelam} = 0$, and our above argument implies that $\lambda\SupSpec{\Opconnf + \OpLacelam} \ne 1$. Therefore an intermediate value theorem argument proves that $\lambda\SupSpec{\Opconnf + \OpLacelam} < 1$. This argument also works for $\lambda \SupSpec{\fOpconnf(k) + \fOpLacelam(k)}$.
\end{proof}

\section{Uniform Convergence}
\label{sec:Bootstrap:ForbiddenRegion}

In Proposition~\ref{thm:convergenceoflaceexpansion} we proved that the lace expansion converges in the sub-critical regime. To show the convergence at criticality, we will need a stronger result. We will need uniform convergence.

\begin{observation}[Uniform convergence of the lace expansion] \label{obs:convergenceoflaceexpansion_uniform}
Suppose that there exists finite $C>0$ such that $f \leq C$ on $[0,\lambda_O)$. Then for $d>6$ sufficiently large there exits $c>0$ (independent of $\lambda$ and $d$) such that the bounds~\eqref{eq:BA:convLE_intPi_bds1},~\eqref{eq:BA:convLE_intPi_bds2},~\eqref{eq:BA:convLE_R_bd} hold with $\const$ replaced by $c$.
\end{observation}

To show that $f$ is indeed uniformly bounded for $\lambda\in\left[0,\lambda_O\right)$. We do this by performing a forbidden-region argument. In Proposition~\ref{thm:bootstrapargument} we prove that $f(0)$ is bounded and $f$ is continuous on $\left[0,\lambda_O\right)$. However we also prove that $f$ is never in the region $\left(\kappa,\kappa+1\right]$, where 
\begin{equation}
\label{eqn:DefineKappa}
    \kappa:= 10 C \left(1+C\right)^2,
\end{equation}
$C$ being the constant appearing in \ref{Assump:2ndMoment}. Note that \eqref{eqn:fconnfIsL2} in \ref{Assump:2ndMoment} requires that $C\geq 1$. The intermediate value theorem then implies that $\kappa$ acts as an upper bound on the whole domain $\left[0,\lambda_O\right)$. One should not read too much into the value of $\kappa$ here. It is sufficient for our purposes, but not remotely optimal.

\begin{prop}[The forbidden-region argument] \label{thm:bootstrapargument}
The following three statements are all true:
\begin{enumerate}[label=\arabic*)]
\item \label{thm:bootstrapargument-1} $f$ satisfies $f(0) \leq \kappa$.
\item \label{thm:bootstrapargument-3}$f(\lambda) \notin (\kappa,\kappa+1]$ for all $\lambda \in [0,\lambda_O)$ provided that $d$ is sufficiently large.
\item \label{thm:bootstrapargument-2}$f$ is continuous on $[0, \lambda_O)$.
\end{enumerate}
Therefore $f(\lambda) \leq \kappa$ holds uniformly in $\lambda<\lambda_O$ for all $d$ sufficiently large.
\end{prop}

Before we embark on our proof of Proposition~\ref{thm:bootstrapargument}, we prove a lemma which will give us some properties of the components that went into the bootstrap function.

\begin{lemma}\label{lem:mulam_bounds}
The map $\lambda\mapsto \SupSpec{\fOptlam(0)}$ is continuous and bounded below by $\SupSpec{\fOpconnf(0)}$ on $\left[0,\lambda_O\right)$. Therefore the map $\lambda\mapsto \mulam$ is a continuous map from $\left[0,\lambda_O\right)\to \left[0,1\right)$. Furthermore, $\fgmu(k)$ is bounded below by $\left(1+\OpNorm{\fOpconnf(0)}\right)^{-1}$ for all $k\in\Rd$.
\end{lemma}

\begin{proof}
First note that it is proven in Appendix~\ref{sec:Prelim:differentiating} that for $\lambda<\lambda_O$ the map $\lambda \mapsto \fOptlam(0)$ is differentiable with respect to the operator norm (Corollary~\ref{thm:differentiabilityofOpT}), and is therefore continuous. Lemma~\ref{lem:SupSpecTriangle} then implies that $\lambda\mapsto \SupSpec{\fOptlam(0)}$ is continuous. For the bounded below property, we use Lemma~\ref{lem:NormBounds} to get $\SupSpec{\fOptlam(0)} \geq \SupSpec{\fOpconnf(0)} = 1$.

The continuity of $\mulam$ follows directly from the continuity and positivity of $\SupSpec{\fOptlam(0)}$. The bound $\SupSpec{\fOptlam(0)} \geq 1$ bounds $\mulam \geq 0$, and $\SupSpec{\fOptlam(0)} \leq \OpNorm{\fOptlam(0)} < \infty$ shows that $\mulam <1$ for all $\lambda \in \left(0,\lambda_O\right)$. The case $\mu_0 = 0$ follows immediately from $\mathcal{T}_0 = \Opconnf$.

Because $\mulam \in\left[0,1\right)$, to minimise $\fgmu(k)$ we want to find a lower bound for $\SupSpec{\fOpconnf(k)}$. The $k$-uniform bound on $\fgmu(k)$ then follows from $\SupSpec{\fOpconnf(k)} \geq -\OpNorm{\fOpconnf(k)} \geq -\OpNorm{\fOpconnf(0)}$.
\end{proof}

We now return to Proposition~\ref{thm:bootstrapargument}. We will deal with each of the three statements in Proposition~\ref{thm:bootstrapargument} in turn.

\begin{proof}[Proof of Proposition \ref{thm:bootstrapargument}~-\ref{thm:bootstrapargument-1}]
Firstly, it is clear $f_1\left(0\right) = 0$. Then note that $\mathcal{T}_0 = \Opconnf$, and hence 
\begin{equation}
    \OpNorm{\widehat{\mathcal{T}_0}\left(k\right)} = \OpNorm{\fOpconnf\left(k\right)} \leq \OpNorm{\fOpconnf(0)} \leq \OneNorm{\fOpconnf(0)} \leq C,
\end{equation}
where we use \ref{Assump:2ndMoment} (via \eqref{eqn:Sup-Exp_Ratio_2}) in the last inequality. Furthermore $\SupSpec{\widehat{\mathcal{T}}_0(0)} = \SupSpec{\fOpconnf(0)} =1$, and therefore $\widehat{G}_{\mu_{0}}\left(k\right) = \widehat{G}_0\left(k\right) = 1$ and $f_2\left(0\right) \leq C$. Similarly we find that
\begin{equation}
\label{eqn:f_3(0)}
    f_3\left(0\right) = \esssup_{k,l\in\Rd}\frac{\OpNorm{\fOpconnf_k\left(l\right)}}{3\left[1-\SupSpec{\fOpconnf\left(k\right)}\right]}.
\end{equation}
We use then Lemma~\ref{thm:connfComparison} to get the bound
\begin{equation}
    \OpNorm{\fOpconnf_k\left(l\right)} \leq C\left[1 - \SupSpec{\fOpconnf\left(k\right)}\right].
\end{equation}
Hence $f_3\left(0\right) \leq \tfrac{2}{3}C$.
\end{proof}

\begin{proof}[Proof of Proposition \ref{thm:bootstrapargument}~-\ref{thm:bootstrapargument-3}]
We will assume that $f\leq \kappa+1$ and show that this implies that $f\leq \kappa$. Most crucially, the assumption $f\leq \kappa+1$ allows us to apply Observation~\ref{obs:convergenceoflaceexpansion_uniform}. To highlight the $\kappa$-dependence, we write $\constKappaOne$ as the constant arising from this observation.

It will be convenient to introduce some temporary notation for this section. Define $\widehat{a}\left(k\right)\colon L^2\left(\Ecal\right)\to L^2\left(\Ecal\right)$ as $\widehat{a}\left(k\right) = \fOpconnf\left(k\right) + \fOpLacelam\left(k\right)$ for all $k\in\Rd$. We will also make use of the functions $\delta_\lambda\colon \Rd\to\R$ defined by
\begin{equation}
\label{eqn:delta_def}
    \delta_\lambda\left(k\right) := \SupSpec{\fOpconnf\left(k\right) + \fOpLacelam\left(k\right)} - \SupSpec{\fOpconnf\left(k\right)}
\end{equation}
It is clear from the triangle inequality and Lemma~\ref{lem:SupSpecTriangle} that
\begin{equation}
    \abs{\delta_\lambda\left(k\right)}\leq \OpNorm{\fOpLacelam\left(k\right)}\leq \constKappaOne\beta,
\end{equation}
but we will also need the following inequality. We will prove it at the end of this section.

\begin{lemma}
\label{thm:deltatoLacecoefficients}
Let $k_1,k_2\in\Rd$. Then
\begin{equation}
\label{eqn:deltatoLacecoefficients}
    \abs*{\delta_\lambda(k_1)-\delta_\lambda(k_2)} \leq 2 \OpNorm{\fOpLacelam\left(k_1\right) - \fOpLacelam\left(k_2\right)}
\end{equation}
\end{lemma}

\begin{figure}
    \centering
    \begin{tikzpicture}[scale = 1.5]
        \draw[->] (-3, 0) -- (2, 0) node[above] {$\widetilde a(k)(e)$};
        \draw[->] (0, -2) -- (0, 3) node[left] {$\widetilde{\tau}_\lambda\left(k\right)(e)$};
        \draw[dashed] (-3,-1) -- (2,-1)node[below]{$-1/\lambda$};
        \draw[dashed] (1,-2) node[right]{$1/\lambda$} -- (1,3);
        \draw[scale=1, domain=-3:0.75, smooth, variable=\x, thick] plot ({\x}, {\x/(1-\x)});
    \end{tikzpicture}
    \caption{If $\fOptlam(k)$ and $\widehat a (k)$ commute, then they are simultaneously diagonalizable in the sense of Theorem~\ref{thm:spectraltheorem}. The arguments of their diagonal functions ($\widetilde{\tau}_\lambda(k)$ and $\widetilde a(k)$ respectively) are related by the monotone increasing function $x\mapsto \tfrac{x}{1-\lambda x}$ depicted here.}
    \label{fig:fOptlamVSa}
\end{figure}
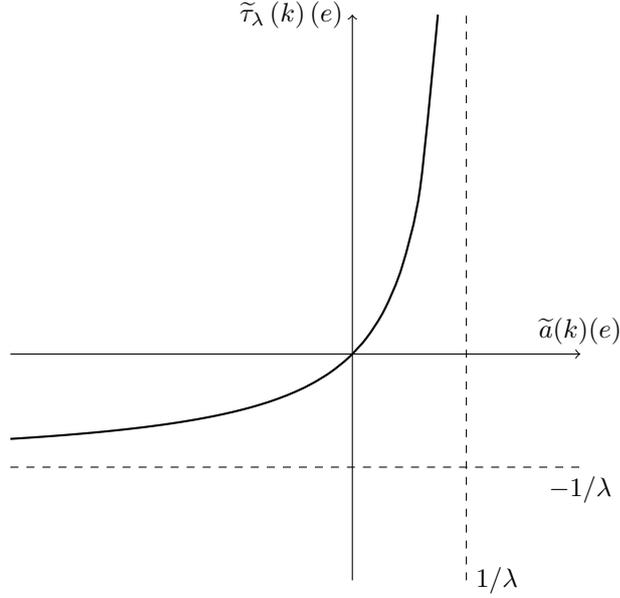

It will be important to note how $\widehat{a}(k)$ relates to $\fOptlam(k)$. From Observation~\ref{obs:convergenceoflaceexpansion_uniform}, the OZE gives us the following expression involving the commutator $\left[\fOptlam(k),\widehat{a}(k)\right] = \fOptlam(k) \widehat{a}(k) - \widehat{a}(k)\fOptlam(k)$:
\begin{equation}
\label{eqn:Commutator}
    \left[\fOptlam(k),\widehat{a}(k)\right]\left(\Id - \lambda \widehat{a}(k)\right) = 0.
\end{equation}
Lemma~\ref{lem:Optlam-a} implies that $\left(\Id - \lambda \widehat{a}(k)\right)$ has a bounded linear inverse and with \eqref{eqn:Commutator} implies that $\fOptlam(k)$ and $\widehat{a}(k)$ commute. Therefore there exists a single unitary map that `diagonalizes' them both in the sense of Theorem~\ref{thm:spectraltheorem}. By considering the diagonal functions, the OZE then implies that $\SupSpec{\fOptlam(k)} = \SupSpec{\widehat{a}(k)}/\left(1-\lambda\SupSpec{\widehat{a}(k)}\right)$. In particular, this relation allows us to write an alternative expression for $\mu_\lambda$:
\begin{equation}
\label{eq:mulamBSidentity}
    \mu_\lambda := 1 - \frac{1}{\SupSpec{\fOptlam(0)}} = \lambda + \frac{\SupSpec{\widehat{a}(0)}-1}{\SupSpec{\widehat{a}(0)}} = \lambda + \frac{\delta_\lambda\left(0\right)}{1+\delta_\lambda\left(0\right)}.
\end{equation}

\begin{itemize}
    \item We first address $f_1$. From the above discussion relating $\fOptlam(k)$ and $\widehat{a}(k)$, we have $\OpNorm{\fOptlam(0)} \geq \SupSpec{\fOptlam(0)} = \SupSpec{\widehat a(0)}/\left(1-\lambda \SupSpec{\widehat{a}(0)}\right)$. From Lemma~\ref{lem:SupSpecTriangle} and our choice of scaling we have $\SupSpec{\widehat a(0)} \geq 1 - \OpNorm{\fOpLacelam(0)} \geq 1 - \constKappaOne\beta$, and therefore we have $\SupSpec{\widehat a(0)}>0$ for sufficiently large $d$ and
\begin{equation}
\label{eqn:lambda_c_asymptotics}
    \lambda \leq \frac{1}{\SupSpec{\widehat a(k)}}\left(1 - \frac{\SupSpec{\widehat a(0)}}{\OpNorm{\fOptlam(0)}}\right) \leq \frac{1}{1-\constKappaOne\beta} \leq 1 + 2\constKappaOne\beta.
\end{equation}
Here we have used an idea that we will use very often in this part of the proof. For $d$ large enough that $\constKappaOne\beta <1$ we have $\left(1-\constKappaOne\beta\right)^{-1} = \sum^\infty_{n=0}\left(\constKappaOne\beta\right)^n$, and for $d$ sufficiently large enough ($\constKappaOne\beta \leq 1/2$) we have $\left(1-\constKappaOne\beta\right)^{-1}\leq 1 + 2\constKappaOne\beta$. We also use a similar idea of neglecting higher order terms by increasing an earlier term's coefficient in other cases.

\item For $f_2$ we first replace $\OpNorm{\fOptlam(k)}$. Since $\fOptlam(k)$ and $\widehat a (k)$ commute, they are simultaneously diagonalizable in the sense of Theorem~\ref{thm:spectraltheorem}, and therefore their spectra are related by the function $x\mapsto \tfrac{x}{1-\lambda x}$ (monotone increasing on $\left(-\infty,1/\lambda\right)$). Therefore we have $\OpNorm{\fOptlam(k)} = \sup_{x\in \sigma\left(\widehat a (k)\right)}\abs*{\frac{x}{1-\lambda x}}$. Now (perhaps with the aid of Figure~\ref{fig:fOptlamVSa}) it is easy to see that
\begin{equation}
    \OpNorm{\fOptlam\left(k\right)} \leq \frac{\SupSpec{\widehat a(k)}}{1- \lambda \SupSpec{\widehat a(k)}} \vee \SupSpec{-\widehat{a}(k)} = \SupSpec{\fOptlam(k)} \vee \SupSpec{-\widehat a (k)}.
\end{equation}
Now since $\SupSpec{-\widehat{a}(k)} \leq \OpNorm{\widehat{a}(k)}\leq 1 + \constKappaOne\beta$, and $\fgmu(k)^{-1} \leq  1+C$ (recall Lemma~\ref{lem:mulam_bounds} and $\OpNorm{\fOpconnf(0)} \leq \OneNorm{\fOpconnf(0)} \leq C$ from \ref{Assump:2ndMoment}), we only need to bound
\begin{equation}
    \esssup_{k\in\Rd}\frac{\SupSpec{\fOptlam(k)}}{\fgmu(k)}.
\end{equation}

We now partition $k\in\Rd$ into two. Define the subset of the Fourier space
\begin{equation}
    A_{\frac{1}{2}} := \left\{k\in\Rd:\SupSpec{\fOpconnf\left(k\right)}\leq \frac{1}{2}\right\}.
\end{equation}
Then we only need to bound the two functions $f_4,f_5$, where
\begin{equation}
    f_4\left(\lambda\right) := \esssup_{k\in A_{\frac{1}{2}}} \frac{\SupSpec{\fOptlam(k)}}{\fgmu(k)},\qquad \qquad f_5\left(\lambda\right):= \esssup_{k\in A_{\frac{1}{2}}^c} \frac{\SupSpec{\fOptlam(k)}}{\fgmu(k)}.
\end{equation}
We first deal with the simpler $f_4$. We consider $k\in A_{\frac{1}{2}}$ and choose $\beta$ small enough that 
\begin{equation}
    1-\lambda\SupSpec{\widehat{a}(k)} \geq 1 - \frac{\frac{1}{2}+\constKappaOne\beta}{1-\constKappaOne\beta} \geq \frac{1}{2} - 2\constKappaOne\beta > 0.
\end{equation}
Then
\begin{equation}
    \frac{1-\mulam \SupSpec{\fOpconnf\left(k\right)}}{1-\lambda \SupSpec{\widehat a(k)}} \leq \frac{1+\OpNorm{\fOpconnf(0)}}{\frac{1}{2} - 2\constKappaOne\beta} \leq 2\left(1+C\right)\left(1+5\constKappaOne\beta\right). \label{eq:bootstrap_f_6_first_bound}
\end{equation}
Since we have $\SupSpec{\widehat a(k)} \leq \SupSpec{\widehat{a}(0)} \leq 1 + \constKappaOne\beta$, we then get
\begin{equation}
    f_4(\lambda) \leq 2\left(1+C\right)\left(1+5\constKappaOne\beta\right)\left(1 + \constKappaOne\beta\right) \leq 2\left(1+C\right)\left(1+7\constKappaOne\beta\right).
\end{equation}
To address $f_5$ we will have $k\in A^c_\rho$. We define:
\begin{equation}
    \widehat N(k) = \frac{\SupSpec{\widehat a(k)}}{\SupSpec{\widehat{a}(0)}}, \qquad \widehat F(k) = \frac{1-\lambda \SupSpec{\widehat a(k)}}{\SupSpec{\widehat a(0)}}, \qquad \widehat Q(k) = \frac{1+\delta_\lambda(k)}{\SupSpec{\widehat{a}(0)}}.
\end{equation}
Note that $\SupSpec{\fOptlam(k)} = \widehat N(k) / \widehat F(k)$. Rearranging gets us to
\begin{align}
    \frac{\SupSpec{\fOptlam(k)}}{\SupSpec{\fOpconnf\left(k\right)}\fgmu(k)} &= \widehat N(k) \frac{1-\mulam\SupSpec{\fOpconnf\left(k\right)}}{\SupSpec{\fOpconnf\left(k\right)}\widehat F(k)}\nonumber \\&= \widehat Q(k) + \frac{\widehat N(k)}{\SupSpec{\fOpconnf\left(k\right)}\widehat F(k)} \left(1-\mulam\SupSpec{\fOpconnf\left(k\right)} - \frac{\widehat Q(k)}{\widehat N(k)}\SupSpec{\fOpconnf\left(k\right)}\widehat F(k)\right) \label{eq:f5rearranging}.
\end{align}
Since $\SupSpec{\fOpconnf(0)}=1$ (by our scaling choice), the extracted term $\widehat Q(k)$ satisfies $\abs*{\widehat Q(k)} \leq 1 + 3\constKappaOne\beta$. We further observe that
\begin{equation}
    \frac{\SupSpec{\fOpconnf\left(k\right)} \widehat Q(k)}{\widehat N(k)} = \frac{\SupSpec{\fOpconnf\left(k\right)} \left(1+ \delta_\lambda(k)\right)}{\SupSpec{\fOpconnf\left(k\right)} +\delta_\lambda(k)} = 1 - \frac{\left(1-\SupSpec{\fOpconnf\left(k\right)}\right) \delta_\lambda(k)}{\SupSpec{\fOpconnf\left(k\right)} + \delta_\lambda(k)} =: 1- \widehat b(k).
\end{equation}
Recalling identity~\eqref{eq:mulamBSidentity} for $\mulam$, we can rewrite the quantity $\left[1- \mulam \SupSpec{\fOpconnf\left(k\right)}- (1-\widehat b(k)) \widehat F(k)\right]$, appearing in \eqref{eq:f5rearranging}, as
\begin{multline}
    \frac{1+\delta_\lambda(0) - \left[\lambda + \delta_\lambda(0) + \lambda\delta_\lambda(0)\right] \SupSpec{\fOpconnf\left(k\right)} - 1 + \lambda\left(\SupSpec{\fOpconnf\left(k\right)}+\delta_\lambda(k)\right) + \widehat b(k)\left(1-\lambda\widehat a(k)\right)}{1+\delta_\lambda(0)} \\
		= \frac{ \left[1-\SupSpec{\fOpconnf\left(k\right)}\right]\left(\delta_\lambda(0) + \lambda \delta_\lambda(0)\right) + \lambda\left[\delta_\lambda(k)-\delta_\lambda(0)\right]}{1+\delta_\lambda(0)}+ \frac{\widehat b(k)(1-\lambda\widehat a(k))}{1+\delta_\lambda(0)}.
\end{multline}
Noting that $\abs{\delta_{\lambda}(0)-\delta_{\lambda}(k)} \leq 2\constKappaOne \left[1-\SupSpec{\fOpconnf\left(k\right)}\right] \beta$ by Lemma~\ref{thm:deltatoLacecoefficients} and Observation~\ref{obs:convergenceoflaceexpansion_uniform}, the first term can be bounded
\begin{align}
    &\abs*{\frac{ \left[1-\SupSpec{\fOpconnf\left(k\right)}\right]\left(\delta_\lambda(0) + \lambda \delta_\lambda(0)\right) + \lambda\left[\delta_\lambda(k)-\delta_\lambda(0)\right]}{1+\delta_\lambda(0)}} \nonumber\\
    &\qquad\leq \frac{ \constKappaOne\beta + \left(1 + 2\constKappaOne\beta\right)\constKappaOne\beta + 2\left(1+2\constKappaOne\beta\right)\constKappaOne\beta}{1 - \constKappaOne\beta}\left[1-\SupSpec{\fOpconnf\left(k\right)}\right] \nonumber\\
    &\qquad \leq 5\constKappaOne\beta\left(1+2\constKappaOne\beta\right)\left[1-\SupSpec{\fOpconnf\left(k\right)}\right] \nonumber\\
    &\qquad \leq 6\constKappaOne\beta\left[1-\SupSpec{\fOpconnf\left(k\right)}\right].
\end{align}
Using~\eqref{eqn:lambda_c_asymptotics}, the last term is
\begin{align}
    \abs*{\frac{\widehat b(k)(1-\lambda \SupSpec{\widehat a(k)})}{1+\delta_\lambda(0)}} &= \abs*{ \frac{\left(1-\SupSpec{\fOpconnf\left(k\right)}\right)\delta_\lambda(k)}{\left(1+\delta_\lambda(0)\right)\left(\SupSpec{\fOpconnf\left(k\right)} + \delta_\lambda(k)\right)}
				- \frac{\lambda\left[1-\SupSpec{\fOpconnf\left(k\right)}\right] \delta_\lambda(k)}{1+\delta_\lambda(0)}} \nonumber\\
		& = \frac{\abs*{\delta_{\lambda}(k)}}{\abs*{1+\delta_\lambda(0)}}\abs*{ \frac{1}{\SupSpec{\fOpconnf\left(k\right)} + \delta_\lambda(k)} -\lambda }\left[1-\SupSpec{\fOpconnf\left(k\right)}\right] \nonumber\\
		& \leq \frac{\constKappaOne\beta}{1-\constKappaOne\beta}\left( \frac{1}{\frac{1}{2}-\constKappaOne\beta} + \frac{1}{1-\constKappaOne\beta}\right)\left[1-\SupSpec{\fOpconnf\left(k\right)}\right] \nonumber\\
		&\leq \constKappaOne\beta\left(1+2\constKappaOne\beta\right)\left(3+8\constKappaOne\beta\right)\left[1-\SupSpec{\fOpconnf\left(k\right)}\right] \nonumber\\
		& \leq 4\constKappaOne\beta \left[1-\SupSpec{\fOpconnf\left(k\right)}\right].
\end{align}
Putting these bounds back into \eqref{eq:f5rearranging}, we can find
\begin{align}
    \frac{\SupSpec{\fOptlam(k)}}{\SupSpec{\fOpconnf\left(k\right)}\fgmu(k)} &\leq \widehat Q(k) + 10\constKappaOne\beta \frac{1}{\SupSpec{\fOpconnf\left(k\right)}}\frac{\widehat N(k)}{\widehat F(k)} \left[1-\SupSpec{\fOpconnf\left(k\right)}\right]\  \notag \\
		& \leq 1+ 3\constKappaOne\beta + 20\constKappaOne\beta \frac{\SupSpec{\fOptlam\left(k\right)}}{\fgmu(k)} \notag \\
		& \leq 1+ \left(20\kappa + 23\right)\constKappaOne\beta. \label{eq:bootstrap:f_5_improvement} 
\end{align}
Note that we have used the bound $\left[1-\SupSpec{\fOpconnf\left(k\right)}\right] \leq \left[1-\mulam \SupSpec{\fOpconnf\left(k\right)}\right]$ (since $\mulam\leq 1$ and $\SupSpec{\fOpconnf(k)}>\frac{1}{2} >0$) to get from $\widehat G_1(k)^{-1}$ to $\fgmu(k)^{-1}$. As $\SupSpec{\fOpconnf\left(k\right)} \in\left(\frac{1}{2}, 1\right]$, this concludes the improvement of $f_5$ and hence of $f_2$. 
\end{itemize}

Before we treat $f_3$, we introduce $f_6$ given by
\begin{equation}
    f_6(\lambda) := \esssup_{k\in\Rd} \frac{1-\mulam \SupSpec{\fOpconnf\left(k\right)}}{1- \lambda \SupSpec{\widehat{a}\left(k\right)}}.
\end{equation}
We show that $f(\lambda) \leq \kappa+1$ implies $f_6(\lambda) \leq M$ uniformly for some $\kappa$-independent constant $M$. For $k\in A_\frac{1}{2}$ we showed in \eqref{eq:bootstrap_f_6_first_bound} that we can bound uniformly with $3\left(1+C\right)$ for sufficiently large small $\beta$. For $k \in A_{\frac{1}{2}}^c$, we have
\begin{multline}
    \frac{1-\mulam \SupSpec{\fOpconnf\left(k\right)}}{1-\lambda \SupSpec{\widehat a(k)}} =  \frac{\SupSpec{\fOptlam\left(k\right)}}{\SupSpec{\fOpconnf\left(k\right)} \fgmu(k)} \cdot \frac{\SupSpec{\fOpconnf\left(k\right)}}{\SupSpec{\widehat a(k)}} \leq \left(1+ \left(20\kappa + 23\right)\constKappaOne\beta\right) \abs*{ 1- \frac{\delta_\lambda(k)}{\SupSpec{\widehat a(k)}}} \\ \leq \left(1+ \left(20\kappa + 23\right)\constKappaOne\beta\right) \left(1+\frac{\constKappaOne\beta}{\frac{1}{2}- \constKappaOne \beta} \right) \leq 1+ \left(20\kappa + 27\right)\constKappaOne\beta. \label{eq:bootstrap_f_6_second_bound}
\end{multline}
Note that for the first bound in~\eqref{eq:bootstrap_f_6_second_bound}, we used the estimate established in~\eqref{eq:bootstrap:f_5_improvement}, which is stronger than a bound on $f_5$. We therefore have the uniform bound $f_6(\lambda) \leq 3 + 3C$ for $\beta$ sufficiently small.

\begin{itemize}

    \item Let us now improve the bound on $f_3$. It will be notationally convenient to denote $A\left(k\right) := \left[\Id - \lambda \widehat{a}\left(k\right)\right]^{-1}$. Since $\widehat{a}\left(k\right)$ is self-adjoint, it can be ``diagonalized" in the sense of Theorem~\ref{thm:spectraltheorem}. Therefore $A\left(k\right)$ can be ``diagonalized" over the same space and with the same unitary operator - it will have the diagonal function $\widetilde{A}(k)(e) = 1/\left(1-\widetilde{a}\left(k\right)(e)\right)$. It is then clear that $A\left(k\right)$ and $\widehat{a}\left(k\right)$ commute (for each $k\in\Rd$). Using Observation~\ref{obs:convergenceoflaceexpansion_uniform}, we have
    \begin{align}
        \OpNorm{\fOptklam\left(l\right)} &= \OpNorm{\widehat{a}\left(l\right)A\left(l\right) - \widehat{a}\left(l+k\right)A\left(l+k\right)}\nonumber\\
        &= \OpNorm{A\left(l\right)\widehat{a}\left(l\right)\left[\Id - \lambda \widehat{a}\left(l+k\right)\right]A\left(l+k\right) - A\left(l\right)\left[\Id - \lambda \widehat{a}\left(l\right)\right]\widehat{a}\left(l+k\right)A\left(l+k\right)} \nonumber \\
        &= \OpNorm{A(l)\left(\widehat{a}(l)-\widehat{a}(l+k)\right)A(l+k)} \nonumber\\
        &\leq \OpNorm{\widehat{a}(l)-\widehat{a}(l+k)}\OpNorm{A(l)}\OpNorm{A(l+k)}. \label{eqn:taudisplacementBound}
    \end{align}
    For the $\OpNorm{\widehat{a}(l)-\widehat{a}(l+k)}$ factor, we use the triangle inequality to separate this into a term of $\fOpconnf$ and $\fOpLacelam$. Lemma~\ref{thm:connfComparison} immediately gives
    \begin{equation}
        \OpNorm{\fOpconnf\left(l\right) - \fOpconnf\left(l+k\right)} \leq C\left[1 - \SupSpec{\fOpconnf\left(k\right)}\right],
    \end{equation}
    for the $C>0$ arising from \ref{Assump:2ndMoment}. For the $\fOpLacelam$ term, a similar argument to \eqref{eqn:l-independence} gives
    \begin{equation}
        \OpNorm{\fOpLacelam\left(l\right) - \fOpLacelam\left(1+k\right)} \leq \OneNorm{\fOpLacelam\left(0\right) - \fOpLacelam\left( k\right)} \leq \constKappaOne\beta\left[1 - \SupSpec{\fOpconnf\left(k\right)}\right],
    \end{equation}
    where the last inequality follows from Observation~\ref{obs:convergenceoflaceexpansion_uniform}. We then deal with the remaining factors of \eqref{eqn:taudisplacementBound} by noting that for all $k\in\Rd$,
    \begin{equation}
        \OpNorm{A(k)} = \frac{1}{1-\lambda \SupSpec{\widehat a(k)}} = \frac{1-\mulam \SupSpec{\fOpconnf(k)}}{1-\lambda \SupSpec{\widehat a(k)}} \fgmu(k) \leq 3\left(1+C\right) \fgmu(k), \label{eq:f6application}
    \end{equation}
    employing our improved bound on $f_6$. The result of this is that we have the bound
    \begin{equation}
        \label{eqn:tauComparison}
        \OpNorm{\fOptlam\left(l\right) - \fOptlam\left(l+k\right)} \leq 9\left(1+C\right)^2\left(C + \constKappaOne\beta \right)\left[1 - \SupSpec{\fOpconnf\left(k\right)}\right]\fgmu\left(l\right)\fgmu\left(l+k\right).
    \end{equation}
    Then since $\fgmu(k)\geq0$ we have
    \begin{equation}
        \OpNorm{\fOptklam\left(l\right)} \leq 9\left(1+C\right)^2\left(C+\constKappaOne\beta\right)\widehat{J}_{\mulam}\left(k,l\right).
    \end{equation}

\end{itemize}
\end{proof}

\begin{proof}[Proof of Proposition \ref{thm:bootstrapargument}~-\ref{thm:bootstrapargument-2}]
The continuity of $f_1$ is obvious. For the other two functions we prove equicontinuity of a family of functions and use \cite[Lemma~5.13]{Sla06} to show the continuity of the desired functions.

The procedure is outlined here. Suppose we wish to show that $H\left(\lambda\right) := \sup_{\alpha\in B}\OpNorm{h_\alpha\left(\lambda\right)}$ is continuous on $\left[0,\lambda_O\right)$. For our purposes the parameter $\alpha$ will be either $k$ or $\left(k,l\right)$ and thus $B=\R^d$ or $B=\R^{2d}$. 
\begin{itemize}
    \item The continuity on the half-open interval is implied by having continuity on the closed interval $\left[0,\lambda_O-\rho\right]$ for any $\rho>0$.
    \item For closed intervals, \cite[Lemma~5.13]{Sla06} gives continuity if the family $\left\{\OpNorm{h_\alpha}\right\}_{\alpha\in B}$ is equicontinuous and $H\left(\lambda\right)<+\infty$ for all $\lambda$ in the closed interval.
    \item The family $\left\{\OpNorm{h_\alpha}\right\}_{\alpha\in B}$ is equicontinuous if for all $\varepsilon>0$ there exists $\delta>0$ such that $|s-t|<\delta$ implies $\abs*{\OpNorm{h_\alpha(s)} - \OpNorm{h_\alpha(t)}} \leq \varepsilon$ uniformly in $\alpha\in B$. In fact, the reverse triangle inequality implies that
    \begin{equation}
        \abs*{\OpNorm{h_\alpha(s)} - \OpNorm{h_\alpha(t)}} \leq \OpNorm{h_\alpha(s)-h_\alpha(t)},
    \end{equation}
    and therefore we only need to prove equicontinuity for the un-normed family $\left\{h_\alpha\right\}_{\alpha\in B}$.
    \item We prove equicontinuity by bounding the `near-derivative'
    \begin{equation}
        \limsup_{\varepsilon\to0}\frac{1}{\varepsilon}\OpNorm{h_\alpha(\lambda+\varepsilon)-h_\alpha(\lambda)}
    \end{equation}
    uniformly in $\alpha\in B$ for $\lambda\in\left[0,\lambda_O-\rho\right]$ and arbitrary $\rho>0$.
\end{itemize}

For $f_2$ the operator-valued functions $h_\alpha$ are $\lambda\mapsto \fOptlam\left(k\right)/\fgmu\left(k\right)$. By using a variation on the chain rule, we get
\begin{multline}
\label{eqn:Near-derivative_bound}
    \limsup_{\varepsilon\to0}\frac{1}{\varepsilon}\OpNorm{\frac{\widehat{\mathcal{T}}_{\lambda+\varepsilon}\left(k\right)}{\widehat{G}_{\mu_{\lambda+\varepsilon}}\left(k\right)} - \frac{\fOptlam\left(k\right)}{\fgmu\left(k\right)}} \leq \frac{1}{\fgmu\left(k\right)}\OpNorm{\frac{\dd}{\dd \lambda}\fOptlam\left(k\right)} \\+ \frac{\OpNorm{\fOptlam\left(k\right)}}{\fgmu\left(k\right)^2}\abs*{\left.\frac{\dd}{\dd \mu}\widehat{G}_\mu\left(k\right)\right|_{\mu=\mu_\lambda}}\limsup_{\varepsilon\to0}\frac{1}{\varepsilon}\abs*{\mu_{\lambda+\varepsilon} - \mu_\lambda}.
\end{multline}
Recall from Corollary~\ref{thm:differentiabilityofOpT} that $\fOptlam\left(k\right)$ is differentiable with the bound $\OpNorm{\frac{\dd}{\dd \lambda}\fOptlam\left(k\right)} \leq \OpNorm{\fOptlam\left(0\right)}^2$. We also have $\fgmu\left(k\right) \geq 1/\left(1 + \OpNorm{\fOpconnf(0)}\right)$ from Lemma~\ref{lem:mulam_bounds} and $\left.\frac{\dd}{\dd \mu}\widehat{G}_\mu\left(k\right)\right|_{\mu=\mu_\lambda} = \fgmu\left(k\right)^2\SupSpec{\fOpconnf\left(k\right)}$ immediately from the definition. It remains to deal with the $\mu_\lambda$ term. Recall that $\mu_\lambda := 1 - \SupSpec{\fOptlam(0)}^{-1}$. Therefore, using the reverse triangle inequality,
\begin{multline}
    \limsup_{\varepsilon\to0}\frac{1}{\varepsilon}\abs*{\mu_{\lambda+\varepsilon} - \mu_\lambda} = \frac{1}{\SupSpec{\fOptlam(0)}^2}\limsup_{\varepsilon\to0}\frac{1}{\varepsilon}\abs*{\SupSpec{\widehat{\mathcal{T}}_{\lambda+\varepsilon}(0)}-\SupSpec{\fOptlam(0)}} \\\leq \frac{1}{\SupSpec{\fOptlam(0)}^2}\limsup_{\varepsilon\to0}\frac{1}{\varepsilon}\OpNorm{\widehat{\mathcal{T}}_{\lambda+\varepsilon}(0)-\fOptlam(0)} = \frac{1}{\SupSpec{\fOptlam(0)}^2}\OpNorm{\frac{\dd}{\dd \lambda}\fOptlam(0)}\leq \OpNorm{\fOptlam(0)}^2,
\end{multline}
where we have used $\SupSpec{\fOptlam(0)} \geq 1$ from Lemma~\ref{lem:mulam_bounds} in the last inequality. Therefore
\begin{multline}
    \limsup_{\varepsilon\to0}\frac{1}{\varepsilon}\OpNorm{\frac{\widehat{\mathcal{T}}_{\lambda+\varepsilon}\left(k\right)}{\widehat{G}_{\mu_{\lambda+\varepsilon}}\left(k\right)} - \frac{\fOptlam\left(k\right)}{\fgmu\left(k\right)}} \leq \left(1 + \OpNorm{\fOpconnf(0)}\right)\OpNorm{\fOptlam\left(0\right)}^2 + \abs*{\SupSpec{\fOpconnf\left(k\right)}}\OpNorm{\fOptlam\left(k\right)}\OpNorm{\fOptlam(0)}^2\\
    \leq \left(1 + C\right)\OpNorm{\fOptlam\left(0\right)}^2 + C\OpNorm{\fOptlam(0)}^3.
\end{multline}
We therefore have a finite $k$-independent bound.

This has proven the equicontinuity. For the uniform boundedness, we note that for $\lambda\in\left[0,\lambda_O-\rho\right]$ we have $\OpNorm{\fOptlam\left(k\right)} \leq \OpNorm{\widehat{\mathcal{T}}_{\lambda_O-\rho}(0)}<\infty$. In conjunction with $\fgmu\left(k\right) \geq 1/\left(1 + \OpNorm{\fOpconnf(0)}\right)$ we have the required uniform boundedness and therefore the continuity of $f_2$.

We repeat this approach for $f_3$. The corresponding step to \eqref{eqn:Near-derivative_bound} now reads
\begin{multline}
\label{eqn:Near-derivative_boundf3}
    \limsup_{\varepsilon\to0}\frac{1}{\varepsilon}\OpNorm{\frac{\widehat{\mathcal{T}}_{\lambda+\varepsilon,k}\left(l\right)}{\widehat{J}_{\mu_{\lambda+\varepsilon}}\left(k,l\right)} - \frac{\fOptklam\left(l\right)}{\widehat{J}_{\mu_{\lambda}}\left(k,l\right)}} \leq \frac{1}{\widehat{J}_{\mu_{\lambda}}\left(k,l\right)}\OpNorm{\frac{\dd}{\dd \lambda}\fOptklam\left(l\right)} \\+ \frac{\OpNorm{\fOptklam\left(l\right)}}{\widehat{J}_{\mu_{\lambda}}\left(k,l\right)^2}\abs*{\left.\frac{\dd}{\dd \mu}\widehat{J}_{\mu}\left(k,l\right)\right|_{\mu=\mu_\lambda}}\limsup_{\varepsilon\to0}\frac{1}{\varepsilon}\abs*{\mu_{\lambda+\varepsilon} - \mu_\lambda}.
\end{multline}
Recall from Corollary~\ref{thm:differentiabilityofOpT} that $\fOptklam\left(l\right)$ is differentiable with the operator norm bound $\OpNorm{\frac{\dd}{\dd \lambda}\fOptklam\left(l\right)}\leq 4\OpNorm{\fOptlam\left(0\right) - \fOptlam\left(k\right)}\OpNorm{\fOptlam\left(0\right)}$. The operator itself has the similar bound $\OpNorm{\fOptklam\left(l\right)}\leq \OpNorm{\fOptlam\left(0\right) - \fOptlam\left(k\right)}$. We would like to have the bound \eqref{eqn:tauComparison}, but this was proven under the assumption $f\leq \kappa+1$, which we no longer assume. Fortunately, we only require our bound for $\lambda < \lambda_O -\rho$. Therefore we can use Lemma~\ref{lem:tau-connf subcritical} and Lemma~\ref{thm:connfComparison} to get
\begin{multline}
    \OpNorm{\fOptlam\left(l\right) - \fOptlam\left(l+k\right)} \leq \e^{4\lambda\OpNorm{\fOptlam(0)}}\OpNorm{\fOpconnf\left(0\right) - \fOpconnf\left(k\right)} \\\leq C\e^{4\left(\lambda_O-\rho\right)\OpNorm{\widehat{\mathcal T}_{\lambda_O-\rho}(0)}}\left[1 - \SupSpec{\fOpconnf\left(k\right)}\right].
\end{multline}
From this and $\fgmu\left(k\right) \geq \left(1+C\right)^{-1}$, there exists a constant $\widetilde M>0$ such that both
\begin{equation}
    \frac{1}{\widehat{J}_{\mu_{\lambda}}\left(k,l\right)}\OpNorm{\frac{\dd}{\dd \lambda}\fOptklam\left(l\right)} \leq \widetilde M \OpNorm{\fOptlam\left(0\right)}, \qquad \frac{1}{\widehat{J}_{\mu_{\lambda}}\left(k,l\right)}\OpNorm{\fOptklam\left(l\right)} \leq \widetilde M,
\end{equation}
for sufficiently large $d$.

Now recall $\left.\frac{\dd}{\dd \mu}\widehat{G}_\mu\left(k\right)\right|_{\mu=\mu_\lambda} = \fgmu\left(k\right)^2\SupSpec{\fOpconnf\left(k\right)}$ and $\left(1+C\right)^{-1}\leq \fgmu\left(k\right) \leq \left(1-\mulam\right)^{-1} = \SupSpec{\fOptlam(0)} \leq \OpNorm{\fOptlam(0)}$. Therefore an application of the chain rule gives the bound
\begin{equation}
    \abs*{\frac{1}{\widehat{J}_{\mu_{\lambda}}\left(k,l\right)}\left.\frac{\dd}{\dd \mu}\widehat{J}_{\mu}\left(k,l\right)\right|_{\mu=\mu_\lambda}} \leq 6\OpNorm{\fOptlam(0)}^4\OpNorm{\fOpconnf(0)},
\end{equation}
uniformly in $k$ and $l$. These properties, along with the bound on $\limsup_{\varepsilon\to0}\frac{1}{\varepsilon}\abs*{\mu_{\lambda+\varepsilon} - \mu_\lambda}$ from above, are sufficient to prove equicontinuity and the uniform boundedness for the required $\lambda$. Therefore we have proved that $f_3\left(\lambda\right)$ is continuous.

\end{proof}

It just remains to prove that we can indeed bound the difference of $\delta_\lambda(k)$ with the difference of $\fOpLacelam(k)$ as we claimed.

\begin{proof}[Proof of Lemma~\ref{thm:deltatoLacecoefficients}]
First note that there is nothing to prove if $\OpNorm{\fOpLacelam\left(k_1\right)}=\OpNorm{\fOpLacelam\left(k_2\right)}=0$, because the triangle inequality then forces $\delta_\lambda(k_1) = \delta_\lambda(k_2)=0$.

Without loss of generality, assume $\OpNorm{\fOpLacelam\left(k_1\right)} \geq \OpNorm{\fOpLacelam\left(k_2\right)}$ with $\OpNorm{\fOpLacelam\left(k_1\right)}>0$. Then
\begin{equation}
\label{eqn:deltaDifference}
    \delta_\lambda(k_1)-\delta_\lambda(k_2) = \frac{\delta_\lambda(k_1)}{\OpNorm{\fOpLacelam(k_1)}}\left(\OpNorm{\fOpLacelam\left(k_1\right)} - \OpNorm{\fOpLacelam\left(k_2\right)}\right) + \left[\delta_\lambda(k_1)\frac{\OpNorm{\fOpLacelam\left(k_2\right)}}{\OpNorm{\fOpLacelam\left(k_1\right)}} - \delta_\lambda(k_2)\right]
\end{equation}
The second factor is the more troublesome:
\begin{multline}
    \abs*{\delta_\lambda(k_1)\frac{\OpNorm{\fOpLacelam\left(k_2\right)}}{\OpNorm{\fOpLacelam\left(k_1\right)}} - \delta_\lambda(k_2)}  = \frac{1}{\OpNorm{\fOpLacelam\left(k_1\right)}}\abs*{\delta_\lambda(k_1)\OpNorm{\fOpLacelam\left(k_2\right)} - \delta_\lambda(k_2)\OpNorm{\fOpLacelam\left(k_1\right)}} \\
    \leq \frac{\abs*{\delta_\lambda(k_1)} \vee \abs*{\delta_\lambda(k_2)}}{\OpNorm{\fOpLacelam\left(k_1\right)}}\abs*{\OpNorm{\fOpLacelam\left(k_1\right)} - \OpNorm{\fOpLacelam\left(k_2\right)}}.
\end{multline}
Since $\OpNorm{\fOpLacelam\left(k_1\right)} \geq \OpNorm{\fOpLacelam\left(k_2\right)}$, we have $\abs*{\delta_\lambda(k_1)} \vee \abs*{\delta_\lambda(k_2)} \leq \OpNorm{\fOpLacelam\left(k_1\right)}$. Therefore
\begin{equation}
    \abs*{\delta_\lambda(k_1)-\delta_\lambda(k_2)} \leq 2\abs*{\OpNorm{\fOpLacelam\left(k_1\right)} - \OpNorm{\fOpLacelam\left(k_2\right)}} \leq 2\OpNorm{\fOpLacelam\left(k_1\right)-\fOpLacelam\left(k_1\right)},
\end{equation}
where the last inequality holds because of the reverse triangle inequality.
\end{proof}


\section{Proof of Main Theorems}
\label{sec:ConcludeProofs}

\begin{prop}
\label{lem:Triangle_Critical}
Let $d>6$ be sufficiently large. Then for $\lambda\leq \lambda_O$ there exists $c>0$ such that $\trilam < c\beta^2$.
\end{prop}

\begin{proof}
Proposition~\ref{thm:TauNbound} gives us that there exists $c_f$ (increasing in $f$) such that $\trilam \leq \const\beta^2$ for all $\lambda<\lambda_O$. Proposition~\ref{thm:bootstrapargument} then implies that there exists $c$ such that $\const\leq c$ uniformly for all $\lambda < \lambda_O$. 

It remains to prove the assertion for $\lambda=\lambda_O$. 
Recall from Section~\ref{sec:Prelim:differentiating} the definition of $\tlam^n(x,y)$ and $\Lambda_n(x)$. Furthermore, recall that if $\conn{x}{y}{\xi^{x,y}}$ then such a connection is achieved in finitely many steps and the bound in $\ref{Assump:2ndMoment}$ then implies that
\begin{equation}
    \left\{\conn{x}{y}{\xi^{x,y}}\right\} = \bigcup^\infty_{n=1}\left\{\conn{x}{y}{\xi^{x,y}_{\Lambda_n(y)}}\right\}.
\end{equation}
Therefore by monotone convergence we have the pointwise limit $\tlam^n(x,y) \to \tlam(x,y)$ as $n\to\infty$ for all $\lambda>0$ and $x,y\in\X$. Since $\tlam^n(x,y)$ only depends upon the finite region $\Lambda_n(y)$, the functions $\lambda \mapsto \tlam^n(x,y)$ are continuous for all $x,y\in\X$. This, with the monotonicity $\tlam^n(x,y) \leq \tlam^{n+1}(x,y)$, implies that the function $\lambda \mapsto \tlam(x,y)$ is lower-semicontinuous. Since $\tlam(x,y)$ is non-decreasing in $\lambda$, this lower semi-continuity implies that the pointwise limit $\tlam(x,y)\uparrow\tau_{\lambda_O}(x,y)$ as $\lambda\uparrow\lambda_O$ holds, and the limit holds monotonically.

This monotonic pointwise convergence implies that the integral
\begin{equation}
    \int \tlam(x,u)\tlam(u,v)\tlam(v,y)\nu^{\otimes2}\left(\dd u, \dd v\right) \to \int \tau_{\lambda_O}(x,u)\tau_{\lambda_O}(u,v)\tau_{\lambda_O}(v,y)\nu^{\otimes2}\left(\dd u, \dd v\right)
\end{equation}
for all $x,y\in\X$ as $\lambda\uparrow\lambda_O$, and that this convergence is monotone increasing. If we take the supremum over $x,y\in\X$, we get that $\lim_{\lambda\uparrow\lambda_O}\InfNorm{\Optlam^3} = \InfNorm{\OptlamT^3}$. The uniform bound for $\lambda<\lambda_O$ then implies that this bound also holds at $\lambda=\lambda_O$.
\end{proof}

The following proposition uses the bound $\triangle_{\lambda_O} < c\beta^2$ to show that percolation does not occur at criticality. Another element of the proof is that there is almost surely at most one infinite cluster. There are general considerations that show that this is the case. Indeed \cite{gandolfi1992uniqueness} establishes that on the discrete space $\Z^d$ there is at most one infinite cluster if the edge occupation measure is stationary and obeys the ‘finite energy property’, and an analogous result for Poisson processes in the continuum applies in our case (see \cite{burton1993long,MeeRoy96}).

\NoCriticalPercolation*

\begin{proof}
Note that the strong form of irreducibility in Assumption~\ref{assumption:StrongIrreducible} implies the more general form of irreducibility in \cite{chebunin2024uniqueness}. We can therefore use their result that the infinite cluster is almost surely unique when it exists.

Suppose for contradiction that $\theta_{\lambda_O}(a)>0$ for some positive measure of marks. Then there almost surely exists a unique infinite cluster. Assume that $(\zerobar,a),(\xbar,b)\in\X$ are in this almost surely unique infinite cluster, and are therefore connected. Then the FKG inequality implies that
\begin{equation}
    \theta_{\lambda_O}(a)\theta_{\lambda_O}(b) 
    \leq
    \p_{\lambda_O}\big(|\C(\zerobar,a)|=\infty,|\C(\xbar,b)|=\infty\big)
    \le
    \tau_{\lambda_O}((\zerobar,a),(\xbar,b)).
\end{equation}
Since $\triangle_{\lambda_O} \leq c\beta^2$, for $\Pcal$-a.e.\ $a,b\in\Ecal$ there exists a sequence $\xbar^{(a,b)}_n$ such that $\tau_{\lambda_O}((\zerobar,a),(\xbar^{(a,b)}_n,b))\to0$. Therefore for $\Pcal$-a.e. $a\in\Ecal$ we have $\theta_{\lambda_O}(a)=0$.
\end{proof}

We can define the linear operator $\OptlamT$ by its action on $f\in L^2(\X)$:
\begin{equation}
    \left(\OptlamT f\right)(x) = \int \tau_{\lambda_O}(x,y)f(y)\nu\left(\dd y\right).
\end{equation}
Contrary to $\Optlam$ for $\lambda<\lambda_O$, the operator $\OptlamT$ may be an unbounded linear operator.

\begin{prop}[The operator OZE at the critical point]\label{lem:convergence_of_lace_expansion_corollary_lambda_T}
Let $d>6$ be sufficiently large. The Ornstein-Zernike equation then extends to $\lambda_O$ in the sense that the sum of (unbounded) linear operators vanishes:
\begin{equation}
\label{eqn:OZE_operator_critical}
    \OptlamT - \Opconnf - \OpLacelamT - \lambda_O\OptlamT\left(\Opconnf + \OpLacelamT\right) = 0.
\end{equation}
Furthermore, for $k\in\R^d\setminus\left\{0\right\}$ the linear operator $\fOptlamT(k)$ is a bounded linear operator and the following equality holds:
\begin{equation}
    \fOptlamT\left(k\right) - \fOpconnf\left(k\right) - \fOpLacelamT\left(k\right) - \lambda_O\fOptlamT\left(k\right)\left(\fOpconnf\left(k\right) + \fOpLacelamT\left(k\right)\right) = 0.
\end{equation}
Finally, there exists $c>0$ such that for all $k\in\Rd$
\begin{equation}
    \OpNorm{\OpLacelamT} \leq c\beta, \qquad \OpNorm{\fOpLacelamT(k)}\leq c\beta.
\end{equation}
\end{prop}

\begin{proof}
Let $f\in L^2(\X)$. Then by the triangle inequality, for $\lambda<\lambda_O$
\begin{multline}
    \norm{\left(\OptlamT - \Opconnf - \OpLacelamT - \lambda_O\OptlamT\left(\Opconnf + \OpLacelamT\right)\right)f}_2 \leq \norm{\left(\Optlam - \Opconnf - \OpLacelam - \lambda\Optlam\left(\Opconnf + \OpLacelam\right)\right)f}_2 \\+ \norm{\left(\OptlamT - \Optlam - \OpLacelamT + \OpLacelam - \lambda_O\OptlamT\left(\Opconnf + \OpLacelamT\right) + \lambda\Optlam\left(\Opconnf +\OpLacelam\right)\right)f}_2.
\end{multline}
The first norm on the right hand side vanishes, since the OZE holds for $\lambda<\lambda_O$. To prove that the left hand side vanishes, we aim to show that the second norm on the right hand side vanishes as $\lambda\uparrow\lambda_O$. To do this, the triangle inequality implies that we only need to show that the following limits hold as $\lambda\uparrow\lambda_O$:
\begin{align}
    \norm{\left(\OptlamT - \Optlam\right)f}_2 &\to 0, \label{eqn:taulimit}\\
    \norm{\left(\OpLacelamT - \OpLacelam\right)f}_2 &\to 0,\label{eqn:pilimit}\\
    \norm{\left(\OptlamT\OpLacelamT - \Optlam\OpLacelam\right)f}_2 &\to 0\label{eqn:taupilimit}.
\end{align}

We begin with \eqref{eqn:taulimit}. Recall from the proof of Proposition~\ref{lem:Triangle_Critical} that $\left(\tau_{\lambda_O} - \tlam\right)(x,y) := \tau_{\lambda_O}(x,y) - \tlam(x,y)$ converges monotonically and pointwise to zero. From the definition of the $\norm{\cdot}_2$ norm,
\begin{align}
     \norm{\left(\OptlamT - \Optlam\right)f}_2^2 &= \int\left(\int \left(\tau_{\lambda_O} - \tlam\right)(x,y)f(y)\nu\left(\dd y\right)\right)^2\nu\left(\dd x\right)\\
     &\leq \int\left(\tau_{\lambda_O} - \tlam\right)(x,y)\left(\tau_{\lambda_O} - \tlam\right)(x,z)\abs*{f(y)}\abs*{f(z)}\nu^{\otimes(3)}\left(\dd x,\dd y, \dd z\right).
\end{align}
The monotone convergence of $\left(\tau_{\lambda_O}-\tlam\right)(x,y)\downarrow0$ then implies that the limit \eqref{eqn:taulimit} holds.

For \eqref{eqn:pilimit}, we use the functions $h^{(n)}_\lambda\colon \X^2\to\R_+$ defined by
\begin{equation}
    h^{(n)}_\lambda(x,y) := \begin{cases}
    \lambda^{-1}\int  \psi_n(x,w_{n-1},u_{n-1}) \left( \prod_{i=1}^{n-1} \psi(\vec v_i) \right) \psi_0(w_0,u_0,y) \nu^{\otimes(2n)}\left(\dd\left( \left(\vec w, \vec u\right)_{[0,n-1]} \right)\right) &:n\geq 1\\
    \tfrac 12 \lambda^2 \left(\int\tlam\left(x,w\right)\connf\left(w,y\right)\nu\left(\dd w\right)\right)^2 &: n=0.
    \end{cases}
\end{equation}
From Proposition~\ref{thm:DB:Pi_bound_Psi}
and the proof of Proposition~\ref{thm:DB:Pi0_bounds}, we know that $0\leq \pi^{(n)}_\lambda(x,y) \leq h^{(n)}_\lambda(x,y)$ for all $n\geq 0$, $x,y\in\X$, and $\lambda\in\left(0,\lambda_O\right)$ - the non-negativity follows from the definition of $\pi^{(n)}_\lambda$. The functions $ h^{(n)}_\lambda$ can also be defined for $\lambda=\lambda_O$, and since they are monotone increasing in $\lambda$, we have the $\lambda$-independent bound $\pi^{(n)}_\lambda(x,y) \leq h^{(n)}_{\lambda_O}(x,y)$ for all $\lambda<\lambda_O$ and $x,y\in\X$. For the single-mark version of the model it is proven in \cite[Corollary~6.1]{HeyHofLasMat19} that $\pi_{\lambda_O}-\pi_{\lambda}$ converges pointwise to zero. This same argument works for our multi-mark version, and therefore $\pi^{(n)}_\lambda(x,y) \leq h^{(n)}_{\lambda_O}(x,y)$ for all $\lambda\leq\lambda_O$. The triangle inequality then implies that
\begin{equation}
    \abs*{\left(\pi_{\lambda_O} - \pi_{\lambda}\right)(x,y)} = \abs*{\sum^\infty_{n=0}\left(-1\right)^n\left( \pi^{(n)}_{\lambda_O}(x,y) - \pi^{(n)}_{\lambda}(x,y)\right)} \leq 2\sum^\infty_{n=0}h^{(n)}_{\lambda_O}(x,y).
\end{equation}
It will be convenient to define $h_{\lambda}\colon \X^2\to\R_+$ as $h_\lambda(x,y) := \sum^\infty_{n=0}h^{(n)}_{\lambda}(x,y)$. Then
\begin{align}
     \norm{\left(\OpLacelamT - \OpLacelam\right)f}_2^2 &= \int\left(\int \left(\pi_{\lambda_O} - \pi_{\lambda}\right)(x,y)f(y)\nu\left(\dd y\right)\right)^2\nu\left(\dd x\right) \nonumber\\
     &= \int\left(\pi_{\lambda_O} - \pi_{\lambda}\right)(x,y)\left(\pi_{\lambda_O} - \pi_{\lambda}\right)(x,z)f(y)f(z)\nu^{\otimes3}\left(\dd x,\dd y, \dd z\right) \label{eqn:diff-pi_integral}\\
     & \leq 4\int h_{\lambda_O}(x,y)h_{\lambda_O}(x,z)\abs*{f(y)}\abs*{f(z)}\nu^{\otimes3}\left(\dd x,\dd y, \dd z\right)\nonumber\\
     & = 4\int\left(\int h_{\lambda_O}(x,y)\abs*{f(y)}\nu\left(\dd y\right)\right)^2\nu\left(\dd x\right). \label{eqn:diff-pi_integral_2}
\end{align}
By Schur's test, this integral is finite for all $f\in L^2(\X)$ if both the values
\begin{equation}
    \esssup_{y\in\X}\int h_{\lambda_O}(x,y) \nu\left(\dd x\right) \qquad \text{ and } \qquad \esssup_{x\in\X}\int h_{\lambda_O}(x,y) \nu\left(\dd y\right)
\end{equation}
are finite. To prove this for the former, we repeat the arguments of Proposition~\ref{thm:DB:Pi0_bounds} and Section~\ref{sec:diagrammaticbounds} to bound it in terms of $\lambda_O$, the triangle diagrams, and other $\lambda$-independent terms. The argument of Proposition~\ref{lem:Triangle_Critical} with the uniform bounds of Observation~\ref{obs:convergenceoflaceexpansion_uniform} then proves that $\esssup_{y}\int h_{\lambda_O}(x,y) \nu\left(\dd x\right)$ is indeed finite. A similar argument also holds for $\esssup_{x}\int h_{\lambda_O}(x,y) \nu\left(\dd y\right)$.

We now have an integrable function $h_{\lambda_O}(x,y)h_{\lambda_O}(x,z)\abs*{f(y)}\abs*{f(z)}$ that dominates the integrand of \eqref{eqn:diff-pi_integral}. Recall that $\pi_{\lambda_O}-\pi_{\lambda}$ converges pointwise to zero. Thus the dominated convergence theorem implies that the limit \eqref{eqn:pilimit} holds.

For \eqref{eqn:taupilimit}, we can use $h_{\lambda_O}$ to get the bound
\begin{multline}
    \norm{\left(\OptlamT\OpLacelamT - \Optlam\OpLacelam\right)f}_2^2 \\\leq \int\left(\tau_{\lambda_O} - \tlam\right)(x,u)h_{\lambda_O}(u,y)\left(\tau_{\lambda_O} - \tlam\right)(x,v)h_{\lambda_O}(v,z)\abs*{f(y)}\abs*{f(z)} \nu^{\otimes5}\left(\dd u,\dd v, \dd x, \dd y, \dd z\right).
\end{multline}
The integrand converges monotonically to zero, and therefore the integral vanishes in the limit.

We now consider the Fourier version of the OZE. We use the notation that, given a bounded linear operator $A\colon L^2(\X) \to L^2(\X)$, the Fourier transform is given by $\Fourier{A}(k)\colon L^2(\Ecal) \to L^2(\Ecal)$. By hypothesis, $\Opconnf$ is a bounded operator, and the finiteness of $\esssup_{y}\int h_{\lambda_O}(x,y) \nu\left(\dd x\right)$ implies that $\OpLacelamT$ is also bounded. Then \eqref{eqn:OZE_operator_critical} implies that
\begin{equation}
    \Fourier{\OptlamT\left(\Id - \lambda_O\left(\Opconnf + \OpLacelamT\right)\right)}(k) = \Fourier{\Opconnf + \OpLacelamT}(k)
\end{equation}
for all $k\in\Rd$.

We now make the claim that $\Fourier{\Id - \lambda_O\left(\Opconnf + \OpLacelamT\right)}(k)$ has a bounded linear inverse for all $k\in\Rd\setminus\left\{0\right\}$. This is proven in Lemma~\ref{lem:Bounded_Linear_Inverse} below. This claim then implies that for $k\ne0$ we have
\begin{align}
    &\Fourier{\Opconnf + \OpLacelamT}(k)\Fourier{\Id - \lambda_O\left(\Opconnf + \OpLacelamT\right)}(k)^{-1} \nonumber\\& \hspace{3cm} = \Fourier{\OptlamT\left(\Id - \lambda_O\left(\Opconnf + \OpLacelamT\right)\right)}(k) \Fourier{\Id - \lambda_O\left(\Opconnf + \OpLacelamT\right)}(k)^{-1} \nonumber\\& \hspace{3cm} = \Fourier{\OptlamT\left(\Id - \lambda_O\left(\Opconnf + \OpLacelamT\right)\right)\left(\Id - \lambda_O\left(\Opconnf + \OpLacelamT\right)\right)^{-1}}(k) =  \Fourier{\OptlamT}(k).
\end{align}
That is, $\Fourier{\OptlamT}(k)$ is a bounded linear operator for $k\neq 0$.

The limit \eqref{eqn:pilimit} also allows us to extend our sub-critical bound on $\OpLacelam$ to apply at criticality. The limit \eqref{eqn:pilimit} may be phrased as saying that the function $\lambda\mapsto\norm*{\OpLacelam f}_2$ is continuous at $\lambda=\lambda_O$ (from the left) for all $f\in L^2\left(\X\right)$. Since $\lambda\mapsto \OpNorm{\OpLacelam}$ can be written as a supremum of such functions, it is a lower semi-continuous function. Our bound $\OpNorm{\OpLacelam} \leq c\beta$ for $\lambda < \lambda_O$ (from Proposition~\ref{thm:convergenceoflaceexpansion}) then implies
\begin{equation}
    \OpNorm{\OpLacelamT} \leq \liminf_{\lambda\uparrow\lambda_O}\OpNorm{\OpLacelam} \leq c\beta.
\end{equation}

The limit $\norm*{\left(\fOpLacelamT^{(n)}(0) - \fOpLacelam^{(n)}(0)\right)f}_2 \to 0$ for each $f\in L^2\left(\Ecal\right)$ can be derived in essentially the same way as \eqref{eqn:pilimit} above. The above lower semi-continuity argument can then be used to show that the bound from Proposition~\ref{thm:convergenceoflaceexpansion} extends to also hold at $\lambda=\lambda_O$. From Lemma~\ref{lem:NormBounds} and the positivity of $\pi^{(n)}_{\lambda_O}(x,y)$ this bound also holds for all $k\in\Rd$. By summing over $n$, we then get the desired bound on $\OpNorm{\fOpLacelamT(k)}$.
\end{proof}

\begin{lemma}
\label{lem:Bounded_Linear_Inverse}
For $\lambda = \lambda_O$ and $k\in\Rd\setminus\left\{0\right\}$, the bounded linear operator $\left(\Id - \lambda_O\left(\fOpconnf(k) + \fOpLacelamT(k)\right)\right)\colon L^2(\Ecal) \to L^2(\Ecal)$ has a bounded linear inverse.
\end{lemma}

\begin{proof}
We first introduce some notation. Let $A\colon L^2(\Ecal) \to L^2(\Ecal)$ be a self-adjoint linear operator. Then we define the \emph{minimum modulus} to be
\begin{equation}
    \MinModulus{A} := \inf_{0\ne f\in L^2(\Ecal)}\frac{\norm*{Af}_2}{\norm{f}_2}.
\end{equation}
Note that it follows from the Spectral Theorem (Theorem~\ref{thm:spectraltheorem}) that $A$ has a bounded linear inverse if and only if $\MinModulus{A}>0$. Now suppose that $A=A_\lambda$, and that for all $f\in L^2(\Ecal)$ the map $\lambda \mapsto \norm*{A_\lambda f}_2$ is continuous on $\left[0,\lambda_O\right]$. Then the map $\lambda\mapsto \MinModulus{A_\lambda}$ is upper semicontinuous on $\left[0,\lambda_O\right]$. That is, for any sequence $\lambda_n$ in $\left[0,\lambda_O\right]$ such that $\lambda_n\to \lambda$,
\begin{equation}
    \limsup_{n\to\infty}\MinModulus{A_{\lambda_n}} \leq \MinModulus{A_\lambda}.
\end{equation}

We now consider $A_\lambda = \Id - \lambda\left(\fOpconnf(k) + \fOpLacelam(k)\right)$. A similar argument to the way we proved \eqref{eqn:pilimit} shows that $\lambda \mapsto \norm*{A_\lambda f}$ is continuous on $\left[0,\lambda_O\right]$. Our result is then proven if we can show that for all $k\in\Rd\setminus\{0\}$ we have an $\varepsilon = \varepsilon(k)>0$ such that $\MinModulus{ \Id - \lambda\left(\fOpconnf(k) + \fOpLacelam(k)\right)} \geq \varepsilon$ uniformly in $\lambda\in\left(0,\lambda_O\right)$. We will in fact show the stronger result that $\lambda\SupSpec{\fOpconnf(k) + \fOpLacelam(k)} \leq 1 - \varepsilon$ uniformly in $\lambda\in\left(0,\lambda_O\right)$.

Recalling the definition of $\delta_\lambda(k)$ from \eqref{eqn:delta_def}, we have
\begin{align}
    1 - \lambda\SupSpec{\fOpconnf(k) + \fOpLacelam(k)} &= 1 - \lambda\SupSpec{\fOpconnf(k)} - \lambda\delta_{\lambda}(k) \nonumber\\
    &= \underbrace{1 - \lambda\SupSpec{\fOpconnf(0) + \fOpLacelam(0)}}_{>0 \text{ from Lemma~\ref{lem:Optlam-a}}} + \lambda\left(\SupSpec{\fOpconnf(0)} - \SupSpec{\fOpconnf(k)}\right) + \lambda\left(\delta_{\lambda}(0) - \delta_{\lambda}(k)\right) \nonumber\\
    & > \lambda\left(\SupSpec{\fOpconnf(0)} - \SupSpec{\fOpconnf(k)}\right)\left(1+\LandauBigO{\beta}\right).
\end{align}
Here we have used Lemma  \ref{thm:deltatoLacecoefficients} to bound the difference of $\delta$ with the difference of $\fOpLacelamT$. For sufficiently large $d$, \ref{Assump:Bound} then proves $\lambda\SupSpec{\fOpconnf(k) + \fOpLacelam(k)} \leq 1 - \varepsilon$ uniformly in $\lambda\in\left(0,\lambda_O\right)$.
\end{proof}

Note that in conjunction, Proposition~\ref{thm:convergenceoflaceexpansion}, Proposition~\ref{thm:bootstrapargument} and Proposition~\ref{lem:convergence_of_lace_expansion_corollary_lambda_T} prove the OZE equations of Theorem~\ref{thm:OZEtheorem}. In particular, the former two prove that the OZE holds in the subcritical regime, and the latter extends this to criticality where possible. Furthermore, Proposition~\ref{thm:convergenceoflaceexpansion} gives the bounds on $\OpLacelam$ and $\fOpLacelam(k)$ subcritically whilst Proposition~\ref{lem:convergence_of_lace_expansion_corollary_lambda_T} extends this bound to criticality.

Regarding Theorem~\ref{thm:InfraRed}, the bound on $\trilam$ and the lack of percolation at criticality were proven at the beginning of this section. The infrared bound component of the theorem is proven below.

\begin{proof}[Proof of Theorem~\ref{thm:InfraRed}]
First note that there is nothing to prove for $k=0$. From Proposition~\ref{lem:convergence_of_lace_expansion_corollary_lambda_T} and Lemma~\ref{lem:Bounded_Linear_Inverse}, we can write
\begin{equation}
    \fOptlam(k) = \left(\fOpconnf(k) + \fOpLacelam(k)\right)\left(\Id - \lambda\left(\fOpconnf(k) + \fOpLacelam(k)\right)\right)^{-1},
\end{equation}
for all $k\in\Rd$ if $\lambda<\lambda_O$, and for all $k\in\Rd\setminus\{0\}$ if $\lambda = \lambda_O$. 

If we use the Spectral Theorem (Theorem~\ref{thm:spectraltheorem}) to `diagonalize' the operator $\widehat{a}(k) =\fOpconnf(k) + \fOpLacelam(k)$, we get a multiplication operator on some Hilbert space $L^2(\mathfrak{E}_k,\mu_k)$ taking values $\widetilde{a}(k)(e)$ for $e\in\mathfrak{E}_k$. Since, $\fOptlam(k)$ is formed from $\widehat{a}(k) =\fOpconnf(k) + \fOpLacelam(k)$, it can also be diagonalized over $L^2(\mathfrak{E}_k,\mu_k)$, and takes values
\begin{equation}
\label{eqn:TauDiagonal}
    \widetilde{\tau}_\lambda(k)(e) = \frac{\widetilde{a}(k)(e)}{1 - \lambda\widetilde{a}(k)(e)}.
\end{equation}
Since $\lambda\SupSpec{\fOpconnf(k) + \fOpLacelam(k)} < 1$ (from the proof of Lemma~\ref{lem:Bounded_Linear_Inverse}), we know $\lambda\widetilde{a}(k)(e) >1$ and $\lambda\widetilde{\tau}_\lambda(k)(e) \geq -1$ for $\mu_k$-almost every $e\in\mathfrak{E}_k$. It therefore only remains to bound $\widetilde{\tau}_\lambda(k)(e)$ from above. Equivalently, we only need to bound $\SupSpec{\fOptlam(k)}$ from above. 

We also note that the proof of \eqref{eqn:pilimit} can be adapted to show that $\lambda\mapsto \left<f,\left(\fOpconnf(0) + \fOpLacelam(0)\right) f\right>$ is continuous on $\left[0,\lambda_O\right]$ for all $f\in L^2(\Ecal)$. Then because $\SupSpec{\fOpconnf(0) + \fOpLacelam(0)} = \sup_{f}\tfrac{\left<f,\left(\fOpconnf(0) + \fOpLacelam(0)\right) f\right>}{\left<f,f\right>}$, this proves that $\lambda\mapsto \SupSpec{\fOpconnf(0) + \fOpLacelam(0)}$ is lower-semicontinuous on $\left[0,\lambda_O\right]$. The bound $\lambda\SupSpec{\fOpconnf(0) + \fOpLacelam(0)} <1$ for $\lambda < \lambda_O$ (from Lemma~\ref{lem:Optlam-a}) then implies that $\lambda_O\SupSpec{\fOpconnf(0) + \fOpLacelamT(0)}\leq 1$.

From \eqref{eqn:TauDiagonal} the calculation proceeds:
\begin{align}
    \lambda \SupSpec{\fOptlam(k)} &=  \frac{\lambda\SupSpec{\fOpconnf(k) + \fOpLacelam(k)}}{1 - \lambda\SupSpec{\fOpconnf(k) + \fOpLacelam(k)}} \nonumber\\
    & = \frac{\lambda\SupSpec{\fOpconnf(k) + \fOpLacelam(k)}}{\underbrace{1 - \lambda\SupSpec{\fOpconnf(0) + \fOpLacelam(0)}}_{\geq 0} + \lambda\SupSpec{\fOpconnf(0) + \fOpLacelam(0)} - \lambda\SupSpec{\fOpconnf(k) + \fOpLacelam(k)}} \nonumber\\
    & \leq  \frac{\SupSpec{\fOpconnf(k) + \fOpLacelam(k)}}{\SupSpec{\fOpconnf(0)} + \delta_{\lambda}(0) - \SupSpec{\fOpconnf(k)} - \delta_{\lambda}(k)} \nonumber\\
    & \leq  \frac{\SupSpec{\fOpconnf(k)} + \LandauBigO{\beta}}{\left(\SupSpec{\fOpconnf(0)}  - \SupSpec{\fOpconnf(k)}\right)\left(1 + \LandauBigO{\beta}\right)} = \frac{\SupSpec{\fOpconnf(k)} + \LandauBigO{\beta}}{\SupSpec{\fOpconnf(0)}  - \SupSpec{\fOpconnf(k)}},
\end{align}
where we have bounded $\delta_{\lambda}(0)-\delta_{\lambda}(k)$ using Lemma~\ref{thm:deltatoLacecoefficients} and Observation~\ref{obs:convergenceoflaceexpansion_uniform}.
\end{proof}

\begin{appendix}

\section{Critical Intensities}
\label{sec:Critical_intensities}

\begin{proof}[Proof of Proposition~\ref{prop:lambda_T=lambda_0}]

By the Mecke equation \eqref{eq:prelim:mecke_n}, we find
\begin{align}
    \norm*{\chi_\lambda}_1 & = 1 + \lambda\int\tlam(\xbar;a,b)\dd \xbar\Pcal^{\otimes 2}\left(\dd a,\dd b\right) = 1 + \lambda\int\ftlam(0;a,b)\Pcal^{\otimes 2}\left(\dd a,\dd b\right),\\
    \norm*{\chi_\lambda}_\infty & = 1 + \lambda\esssup_{b}\int\tlam(\xbar;a,b)\dd \xbar \Pcal\left(\dd a\right) = 1 + \lambda\esssup_{b}\int\ftlam(0;a,b) \Pcal\left(\dd a\right) = 1 + \lambda\OneNorm{\fOptlam(0)}.
\end{align}
From Lemma~\ref{lem:BoundsonOperatorNorm} we have $\int\ftlam(0;a,b)\Pcal^{\otimes 2}\left(\dd a,\dd b\right) \leq \OpNorm{\fOptlam(0)} \leq \OneNorm{\fOptlam(0)}$. Therefore $\lambda^{(1)}_T \geq \lambda_O \geq \lambda_T^{(\infty)}$.

Also recall from the discussion in Section~\ref{sec:Results_criticalintensities} that $\lambda_c \geq \lambda^{(1)}_T$. Therefore we only now need to prove that $\lambda^{(1)}_T = \lambda^{(\infty)}_T$. We do this by proving that if $\norm*{\chi_\lambda}_1<\infty$ then $\norm*{\chi_\lambda}_\infty<\infty$.

Fix $b\in\Ecal$. Then by considering the vertices in the cluster $\C(\zerobar,b)$ adjacent to $\left(\zerobar,b\right)$, we find
\begin{equation}
    \E_\lambda\left[\abs*{\C(\zerobar,b)}\right] \leq 1 + \E_\lambda\left[\sum_{\left(\xbar,a\right)\in\eta:\left(\xbar,a\right)\sim\left(\zerobar,b\right)}\abs*{\C(\xbar,a)}\right]    = 1 + \lambda \int \E_\lambda\left[\abs*{\C(\xbar,a)}\right]\connf\left(\xbar;a,b\right)\dd \xbar \Pcal\left(\dd a\right).
\end{equation}
In this equality we have used Mecke's equation \eqref{eq:prelim:mecke_n}. From the spatial translation invariance of the model, $\E_\lambda\left[\abs*{\C(\xbar,a)}\right] = \E_\lambda\left[\abs*{\C(\zerobar,a)}\right]$ for all $\xbar\in\Rd$. Therefore an application of a supremum bound to the $a$-integral gives
\begin{multline}
    \norm*{\chi_\lambda}_\infty = \esssup_{b\in\Ecal}\E_\lambda\left[\abs*{\C(\zerobar,b)}\right] \leq 1 + \lambda \left(\esssup_{a,b\in\Ecal}\int \connf\left(\xbar;a,b\right)\dd \xbar\right)\int\E_\lambda\left[\abs*{\C(\zerobar,c)}\right]\Pcal\left(\dd c\right)\\
     = 1 + \lambda \left(\esssup_{a,b\in\Ecal}\int \connf\left(\xbar;a,b\right)\dd \xbar\right)\norm*{\chi_\lambda}_1.
\end{multline}
From the finiteness of the parenthesised factor, $\norm*{\chi_\lambda}_1<\infty$ implies $\norm*{\chi_\lambda}_\infty<\infty$ and the result is proven.
\end{proof}

\begin{proof}[Proof of Proposition~\ref{prop:Boolean_Critical_Intensities}]
The equality $\lambda_c\left(\dagger\right)= \lambda_T\left(\dagger\right)$ holds from \cite{men1988percolation}. The equality $\lambda_O = \lambda_T^{(p)}$ follows from Proposition~\ref{prop:lambda_T=lambda_0} because the condition \eqref{eqn:supDegree} clearly holds for the bounded radii model. Also recall from the discussion in Section~\ref{sec:Results_criticalintensities} that $\lambda_c \geq \lambda^{(p)}_T$ for all $p\in\left[1,\infty\right]$.

Now note that for all $a,b\in\Ecal$ we have that if $\left(\zerobar,\dagger\right)\sim \left(\xbar,b\right)$ then $\left(\zerobar,a\right)\sim \left(\xbar,b\right)$. Therefore $\lambda_c\left(\dagger\right) \geq \lambda_c$, and the equality $\lambda_c\left(\dagger\right)= \lambda_T\left(\dagger\right)$ then means we only need to prove $\lambda_T\left(\dagger\right) \leq \lambda_T^{(p)}$ for some $p\in\left[1,\infty\right]$. We will do so for $p=1$.

Suppose that $\lambda>\lambda_T^{(1)}$, so that $\int\E_\lambda\left[\C\left(\zerobar;a\right)\right]\Pcal\left(\dd a\right) = \infty$. That is, the expected cluster size of a vertex with a random ($\Pcal$-distributed) mark is infinite. Note that we have an ordering of the marks in that if two radii $a,b$ satisfy $a\geq b$, then $\connf\left(\xbar;a,c\right) \geq \connf\left(\xbar;b,c\right)$ for all $\xbar\in\Rd$ and $c\in\Ecal$. This also implies that $a\mapsto \E_\lambda\left[\abs*{\C\left(\zerobar,a\right)}\right]$ is a non-decreasing function. This in turn implies that if $\lambda>\lambda_T^{(1)}$ and $a$ satisfies $\int\Id\left\{b\geq a\right\}\Pcal\left(\dd b\right)>0$, then $\int \E_\lambda\left[\abs*{\C\left(\zerobar;b\right)}\right]\Id\left\{b\geq a\right\}\Pcal\left(\dd b\right) = \infty$.

Given some radius $a\in\Ecal$ such that $a>0$ and $\int\Id\left\{b\geq a\right\}\Pcal\left(\dd b\right)>0$, we define a new mark, $\star_a$, that forms connections according to
\begin{equation}
    \connf\left(\xbar;\star_a,b\right) = \Id\left\{\abs*{\xbar}\leq a\right\},
\end{equation}
for all $b\in\Ecal$. Note that if $b\geq a$ then $\connf\left(\xbar;\star_a,c\right) \leq \connf\left(\xbar;b,c\right)$ for all $c\in\Ecal$ and $\E_\lambda\left[\abs*{\C\left(\zerobar;\star_a\right)}\right] \leq \E_\lambda\left[\abs*{\C\left(\zerobar;b\right)}\right]$. We now consider the expected cluster size of vertices adjacent to $\left(\zerobar,\star_a\right)$ with radius $\geq a$. We also restrict to cases for which there is an unique neighbour, and that neighbour has radius $\geq a$. Since the probability of a connection is independent of the radius of the proposed neighbour, the distribution of the unique neighbour is equal to $\Pcal$ - conditioned upon having radius $\geq a$. Therefore we have the lower bound
\begin{equation}
    \E_\lambda\left[\abs*{\C\left(\zerobar;\star_a\right)}\right] \geq 1 + \pla\left(\exists ! \left(\xbar,b\right)\in\eta:\left(\xbar,b\right)\sim\left(\zerobar,\star_a\right),b\geq a\right)\int \E_\lambda\left[\abs*{\C\left(\zerobar;b\right)}\right]\Id\left\{b\geq a\right\}\Pcal\left(\dd b\right).
\end{equation}
As noted above, $\int \E_\lambda\left[\abs*{\C\left(\zerobar;b\right)}\right]\Id\left\{b\geq a\right\}\Pcal\left(\dd b\right)=\infty$. Furthermore, as $a>0$ and $\int\Id\left\{b\geq a\right\}\Pcal\left(\dd b\right)>0$, the random variable $\#\left\{\left(\xbar,b\right)\in\eta:\left(\xbar,b\right)\sim\left(\zerobar,\star_a\right),b\geq a\right\}$ obeys a Poisson distribution with strictly positive mean. Therefore $\pla\left(\exists ! \left(\xbar,b\right)\in\eta:\left(\xbar,b\right)\sim\left(\zerobar,\star_a\right),b\geq a\right)>0$ and $\E_\lambda\left[\abs*{\C\left(\zerobar;\star_a\right)}\right]=\infty$. This in turn implies that $\E_\lambda\left[\abs*{\C\left(\zerobar;a\right)}\right] =\infty$ for all $a>0$.

To relate this to $\E_\lambda\left[\abs*{\C\left(\zerobar,\dagger\right)}\right]$, we perform a similar lower bound. Given a radius $a>0$ and $\int\Id\left\{b\geq a\right\}\Pcal\left(\dd b\right)>0$, we restrict to cases where there is an unique neighbour of $\left(\zerobar,\dagger\right)$ with radius $\geq a$ and bound 
\begin{equation}
    \E_\lambda\left[\abs*{\C\left(\zerobar,\dagger\right)}\right] \geq 1+ \pla\left(\exists ! \left(\xbar,b\right)\in\eta:\left(\xbar,b\right)\sim\left(\zerobar,\dagger\right),b\geq a\right) \E_\lambda\left[\abs*{\C\left(\zerobar;a\right)}\right].
\end{equation}
As before, the probability factor is strictly positive and the expected cluster size term is infinite. Hence $\E_\lambda\left[\abs*{\C\left(\zerobar,\dagger\right)}\right]=\infty$ and $\lambda \geq \lambda\left(\dagger\right)$. Therefore $\lambda_T\left(\dagger\right)\leq \lambda^{(1)}_T$ as required.

\end{proof}

\section{Model Properties}
\label{sec:Model_properties}

Here we will prove that the models outlined in Section~\ref{subsection:Examples} do indeed satisfy Assumption~\ref{Model_Assumption}.

First we consider the single mark ``finite variance models'' considered by \cite{HeyHofLasMat19}. As proven in that reference, this includes the Poisson blob and Gaussian connection models.

\begin{lemma}
Single mark finite variance models satisfy the conditions of Assumption~\ref{Model_Assumption}.
\end{lemma}

\begin{proof}
First note that since $\Ecal$ is a singleton, the operators $\fOpconnf(k)$ are simply scalars. In particular, this means that the conditions \eqref{eqn:fconnfIsL2}, and \eqref{eqn:directional2ndMoment} of \ref{Assump:2ndMoment} hold trivially. The finiteness of $\SupSpec{\fOpconnf(0)}$ is also equivalent to the finiteness of $\int\connf(x)\dd x$ required by \cite{HeyHofLasMat19}.

Since $\fOpconnf(k)$ is simply a scalar, assumption \ref{Assump:Bound} follows directly from the third finite variance condition of \cite{HeyHofLasMat19}.

The condition \ref{Assump:BallDecay} is similarly a generalisation of the second finite variance condition, and therefore follows. Specifically, the sets $B(x)$ are the $\varepsilon$-balls around $x\in\Rd$ where $0<\varepsilon<r_d$ where $r_d$ is the radius of the ball of volume $1$.
\end{proof}

\begin{lemma}
Space-mark factorisation models satisfy the conditions of Assumption~\ref{Model_Assumption}.
\end{lemma}

\begin{proof}
Note that $\fOpconnf(k) = \widehat{\overline{\connf}}(k)\mathcal{K}$, where $\mathcal{K}\colon L^2(\Ecal) \to L^2(\Ecal)$ is the self-adjoint linear integral operator with kernel function $K$. Since $\OneNorm{\mathcal{K}} = \esssup_{b}\int K(a,b)\Pcal(\dd a)\leq 1$, $\mathcal{K}$ is a bounded operator and
\begin{equation}
    \SupSpec{\fOpconnf(k)} = \begin{cases}
        \widehat{\overline{\connf}}(k)\SupSpec{\mathcal{K}} &: \widehat{\overline{\connf}}(k) \geq 0\\
        -\widehat{\overline{\connf}}(k)\SupSpec{-\mathcal{K}} &: \widehat{\overline{\connf}}(k) < 0.
    \end{cases}
\end{equation}

Since $\OpNorm{\mathcal{K}}\leq 1$, the finiteness of $\widehat{\overline{\connf}}(0)$ implies that $\SupSpec{\fOpconnf(0)}$ is finite, and the $d$-independence of $\mathcal{K}$ implies that \ref{Assump:2ndMoment} holds.

Since $\overline{\connf}(\xbar)\geq 0$, we have $\abs*{\widehat{\overline{\connf}}(k)} < \widehat{\overline{\connf}}(0)$ for all $k\ne 0$. Furthermore, $\SupSpec{\mathcal{K}} = \OpNorm{\mathcal{K}}$ implies that $\SupSpec{\mathcal{K}}\geq \SupSpec{-\mathcal{K}}$. Then $\SupSpec{\fOpconnf(k)}$ inherits the properties required by \ref{Assump:Bound} from $\widehat{\overline{\connf}}(k)$.

For \ref{Assump:BallDecay}, we note that the $d$-independence of $\mathcal{K}$ means that required decay properties are again inherited from $\widehat{\overline{\connf}}(k)$.
\end{proof}

\begin{lemma}
The marked multivariate Gaussian model satisfies the conditions of Assumption~\ref{Model_Assumption}.
\end{lemma}

\begin{proof}
Firstly note that for each $\left(a,b\right)\in\Ecal^2$ the Gaussian structure of $\connf$ means that it factorises over the $d$ eigenvector directions of $\Sigma(a,b)$. This factorisation ensures that that the $d$-dimensional Fourier transform is the product of the $1$-dimensional Fourier transforms. We have
\begin{equation}
\label{eqn:GaussianFourierTransform}
    \fconnf(k;a,b) = \mathcal{A} \exp\left(-\frac{1}{2}k^{\intercal}\Sigma(a,b)k\right).
\end{equation}
In particular, when $k=0$ the Fourier transform is $a,b$-independent. This means $\fOpconnf(0) = \mathcal{A}\mathbf{1}$ where $\mathbf{1}$ is the integral operator with constant kernel function $1$. For $f\in L^2(\Ecal)$ we have
\begin{align}
    \inner{f}{\mathbf{1}f} &= \int f(a) \left(\int f(b) \Pcal\left(\dd b\right)\right) \Pcal\left(\dd a\right) = \E_\Pcal\left[f\right]^2,\\
    \inner{f}{f} &= \int f(a)^2 \Pcal\left(\dd a\right) = \E_\Pcal\left[f^2\right].
\end{align}
By Jensen's inequality we therefore have $\SupSpec{\mathbf{1}} = 1$ (equality follows by considering the test function $f(a)\equiv1$). It therefore follows that $\SupSpec{\fOpconnf(0)} = \mathcal{A}<\infty$. From \eqref{eqn:GaussianFourierTransform} we have
\begin{equation}
    \InfNorm{\fOpconnf(k)} = \mathcal{A}\esssup_{a,b\in\Ecal} \exp\left(-\frac{1}{2}k^{\intercal}\Sigma(a,b)k\right) \leq \mathcal{A}\exp\left(-\frac{1}{2}\Sigma_{\min}\abs*{k}^2\right).
\end{equation}
This bound then implies the condition \eqref{eqn:fconnfIsL2} of \ref{Assump:2ndMoment} holds with $C=1$.

Before we address \eqref{eqn:directional2ndMoment} of \ref{Assump:2ndMoment}, we turn to \ref{Assump:Bound}. From our above bounds we already have
\begin{equation}
    \SupSpec{\fOpconnf(k)} \leq \OneNorm{\fOpconnf(k)} \leq \InfNorm{\fOpconnf(k)} \leq \mathcal{A}\exp\left(-\frac{1}{2}\Sigma_{\min}\abs*{k}^2\right) = \SupSpec{\fOpconnf(0)} \exp\left(-\frac{1}{2}\Sigma_{\min}\abs*{k}^2\right).
\end{equation}
This exponential term ensures that \ref{Assump:Bound} is satisfied.

We now return to \eqref{eqn:directional2ndMoment} of \ref{Assump:2ndMoment}. For all $x\in\R$ we have $1-\cos(x) \leq \frac{1}{2}x^2$, and therefore
\begin{equation}
\label{eqn:OneInfNormQuadraticExpansion}
    \InfNorm{\fOpconnf(0) - \fOpconnf(k)} \leq \frac{1}{2}\abs*{k}^2 \esssup_{a,b\in\Ecal}\int \left(e\cdot \xbar\right)^2\connf\left(\xbar;a,b\right)\dd \xbar,
\end{equation}
for some unit vector $e\in\Rd$. Since the vector $e$ picks out the second moment of $\connf$ in only one direction, a standard calculation gives
\begin{equation}
    \esssup_{a,b\in\Ecal}\int \left(e\cdot \xbar\right)^2\connf\left(\xbar;a,b\right)\dd \xbar  \leq \Sigma_{\max}^2\InfNorm{\fOpconnf(0)}.
\end{equation}
Therefore $\InfNorm{\fOpconnf(0) - \fOpconnf(k)}$ is bounded by a $d$-independent quadratic function. In conjunction with the quadratic bound of \ref{Assump:Bound}, this then proves \eqref{eqn:directional2ndMoment} in the $k\to0$ regime. The remaining regime then holds because the triangle inequality implies $\InfNorm{\fOpconnf(0) - \fOpconnf(k)} \leq  2\InfNorm{\fOpconnf(0)} =  2\SupSpec{\fOpconnf(0)}$.

To address \ref{Assump:BallDecay}, recall that the spatial convolution of two multivariate Gaussian functions with means $\zerobar$ and covariance matrices $\Sigma_1$ and $\Sigma_2$ is a multivariate Gaussian function with mean $\zerobar$ and covariance matrix $\Sigma_1+\Sigma_2$. Therefore for all $\xbar\in\Rd$ we have
\begin{align}
    &\esssup_{a_1,a_2,a_3,a_4\in\Ecal}\int \connf(\xbar-\ybar;a_1,a_2)\connf(\ybar;a_3,a_4)\dd \ybar \nonumber \\ 
    &\hspace{0.5cm}= \mathcal{A}^2\left(2\pi\right)^{-d/2}\esssup_{a_1,a_2,a_3,a_4\in\Ecal}\left(\det \left(\Sigma(a_1,a_2) + \Sigma(a_3,a_4)\right)\right)^{-1/2}\exp\left(-\frac{1}{2}\xbar^{\intercal}\left(\Sigma(a_1,a_2) + \Sigma(a_3,a_4)\right)^{-1}\xbar\right) \nonumber\\
    & \hspace{0.5cm} \leq \mathcal{A}^2\left(4\pi \Sigma_{\min}\right)^{-d/2}\exp\left(-\frac{\abs*{\xbar}^2}{4\Sigma_{\min}}\right) \leq \mathcal{A}^2\left(4\pi \Sigma_{\min}\right)^{-d/2} \to 0.
\end{align}
In particular, this bounds the integral appearing in the definition of the sets $B(x)$ (see \eqref{eqn:BallDecay}) and therefore proves they are empty. For the bound of the convolution of three $\connf$ functions, we can get this immediately from the above calculation and $\InfNorm{\fOpconnf(0)} \leq C\SupSpec{\fOpconnf(0)}$. Specifically \ref{Assump:BallDecay} holds with $g(d) = \left(4\pi \Sigma_{\min}(d)\right)^{-d/2}$ and the sets $B(x) = \emptyset$.
\end{proof}

\begin{lemma}
The bounded-volume Boolean disc model satisfies the conditions of Assumption \ref{Model_Assumption}.
\end{lemma}

\begin{proof}
Crucial to the conditions of \ref{Assump:2ndMoment} holding is the upper bound on $R_d(a)$ in \eqref{eqn:RadiusBounds}. We use the shorthand
\begin{equation}
    R_{a,b} := R_d(a) + R_d(b),\qquad V_{a,b} := \frac{\pi^{\frac{d}{2}}}{\Gamma\left(\frac{d}{2}+1\right)}R^d_{a,b},
\end{equation}
so $V_{a,b}$ equals the Lebesgue volume of the $d$-dimensional Euclidean ball with radius $R_{a,b}$. Note that the upper bound $\Rmax$ on $R_d(a)$ in \eqref{eqn:RadiusBounds} ensures $V_{a,b} \leq \Vmax$ uniformly in $d$. We first find
\begin{equation}
    \SupSpec{\fOpconnf(0)} \leq \OneNorm{\fOpconnf(0)} = \esssup_{b\in\left[0,1\right]}\int\connf\left(\xbar;a,b\right)\dd \xbar\Pcal\left(\dd a\right) = \esssup_{b\in\left[0,1\right]}\int V_{a,b}\Pcal\left(\dd a\right) \leq \Vmax.
\end{equation}
This proves the statement that $\SupSpec{\fOpconnf(0)}$ is finite in \ref{Assump:2ndMoment}. From the condition \eqref{eqn:MeanRadiusBound} we also get
\begin{multline}
    \SupSpec{\fOpconnf(0)} \geq \int\connf\left(\xbar;a,b\right)\dd \xbar\Pcal^{\otimes2}\left(\dd a, \dd b\right) \geq \epsilon^2\frac{d\pi^\frac{d}{2}}{\Gamma\left(\frac{d}{2}+1\right)} \int^{2\left(1-\frac{c_2}{d}\right)\Rmax}_0 r^{d-1}\dd r \\= \epsilon^2\Vmax\left(1-\frac{c_2}{d}\right)^d \to \epsilon^2\e^{-c_2}\Vmax,
\end{multline}
where we have used the test function $f(a) \equiv 1$ to lower bound $\SupSpec{\fOpconnf(0)}$.  We then prove \eqref{eqn:fconnfIsL2} by bounding
\begin{equation}
    \InfNorm{\fOpconnf(0)} =\esssup_{a,b\in\left[0,1\right]}\int\connf\left(\xbar;a,b\right)\dd \xbar  = \esssup_{a,b\in\left[0,1\right]}V_{a,b} \leq \Vmax.
\end{equation}
We will return for \eqref{eqn:directional2ndMoment} after addressing \ref{Assump:Bound}.

For \ref{Assump:Bound}, it will be useful to have an expression for $\fconnf(k;a,b)$ for each $a,b\in\left[0,1\right]$. Suppose we are able to find $C_1\in\left(0,1\right)$ and $C_2 >0$ such that for all $a,b\in\left[0,1\right]$
\begin{equation}
\label{eqn:fconnf_a-b_bound}
    \abs*{\fconnf(k;a,b)} \leq \fconnf(0;a,b) \left(C_1 \vee \left(1 - C_2\abs*{k}^2\right)\right).
\end{equation}
Then since the $k$-dependent factor is $a,b$-independent we can use Lemma~\ref{lem:NormBounds} to show that
\begin{equation}
    \SupSpec{\fOpconnf(k)} \leq \SupSpec{\fOpconnf(0)} \left(C_1 \vee \left(1 - C_2\abs*{k}^2\right)\right),
\end{equation}
and therefore prove that \ref{Assump:Bound} is satisfied.

In finding an expression for $\fconnf(k;a,b)$ we are assisted by the spherical symmetry of the connection function and follow \cite[Appendix~B]{grafakos2008classical}. It transpires that
\begin{equation}
    \fconnf(k;a,b) = \left(\frac{2\pi R_{a,b}}{\abs*{k}}\right)^\frac{d}{2}J_{\frac{d}{2}}\left(R_{a,b}\abs*{k}\right),
\end{equation}
where $J_{\frac{d}{2}}$ is the Bessel function of the first kind of order $\frac{d}{2}$. This function has the expansion
\begin{equation}
\label{eqn:Bessel_Asymp}
    J_{\frac{d}{2}}\left(r\right) = \sum^\infty_{m=0} \frac{\left(-1\right)^{m}}{m!\Gamma\left(\frac{d}{2}+m+1\right)}\left(\frac{r}{2}\right)^{\frac{d}{2}+2m},
\end{equation}
converging for all $r\geq 0$. We will consider three different regimes for $\abs*{k}$. For $r \ll \sqrt{d}$ the expansion \eqref{eqn:Bessel_Asymp} is asymptotic, and since $\Rmax$ and $\Rmin$ are both asymptotically proportional to $\sqrt{d}$ we have
\begin{equation}
    \fconnf(k;a,b) \sim \frac{\pi^\frac{d}{2}}{\Gamma\left(\frac{d}{2}+1\right)}R_{a,b}^d - \frac{\pi^\frac{d}{2}}{4\Gamma\left(\frac{d}{2}+2\right)}R_{a,b}^{d+2}\abs*{k}^2 = V_{a,b}\left(1 - \frac{R^2_{a,b}}{2\left(d+2\right)}\abs*{k}^2\right)
\end{equation}
for $\abs*{k}\ll 1$. This proves the quadratic part of the desired bound.

The Bessel function $J_{\frac{d}{2}}$ is bounded and achieves its global maximum (in absolute value) at its first non-zero stationary point, denoted $j'_{\frac{d}{2},1}$. From \cite[p.371]{abramowitz1970handbook}, we have $j'_{\frac{d}{2},1} = \frac{d}{2} + \gamma_1\left(\frac{d}{2}\right)^\frac{1}{3} + \LandauBigO{d^{-\frac{1}{3}}}$ for a given $\gamma_1 \approx 0.81$, and $\abs*{J_{\frac{d}{2}}\left(j'_{\frac{d}{2},1}\right)} \leq M d^{-\frac{1}{3}}$ for a $d$-independent $M>0$. Since $R_{a,b} \in  \left[2c_1,1/\sqrt{2\pi\e}+o(1)\right]d^\frac{1}{2}$ (for $c_1\in\left(0,1/\sqrt{8\pi\e}\right]$ from \eqref{eqn:RadiusBounds}) and $j'_{\frac{d}{2},1} / R_{a,b} \in \left[\sqrt{\frac{\pi \e}{2}}+o(1),\frac{1}{4c_1}+o(1)\right]d^{\frac{1}{2}}$ we have
\begin{equation}
    \abs*{\fconnf(k;a,b)} \leq \left(\frac{2\pi R_{a,b}}{\abs*{k}}\right)^\frac{d}{2} M d^{-\frac{1}{3}} \leq \left(2e^{-1} + o(1)\right)^\frac{d}{2} M d^{-\frac{1}{3}} \to 0
\end{equation}
for $\abs*{k} \geq j'_{\frac{d}{2},1} / R_{a,b}$.

We are now left with the intermediate range for $\abs*{k}$. The first positive zero of $J_\frac{d}{2}$ occurs at $j_{\frac{d}{2},1}$, where $j_{\frac{d}{2},1} = \frac{d}{2} + \gamma_2\left(\frac{d}{2}\right)^\frac{1}{3} + 1+  \LandauBigO{d^{-\frac{1}{3}}}$ for a given $\gamma_2 \approx 1.86$ (see \cite{abramowitz1970handbook}). In particular, we will always have $j'_{\frac{d}{2},1} < j_{\frac{d}{2},1}$. From differential identities relating Bessel functions (see \cite{grafakos2008classical}), we have
\begin{equation}
    \frac{\dd }{\dd \abs*{k}} \fconnf(k;a,b) = -R_{a,b}\left(\frac{2\pi R_{a,b}}{\abs*{k}}\right)^\frac{d}{2}J_{\frac{d}{2}+1}\left(R_{a,b}\abs*{k}\right).
\end{equation}
The Bessel function $J_{\frac{d}{2}+1}(r)$ is positive for $r>0$ until its first positive zero at $j_{\frac{d}{2}+1,1}$, and $j_{\frac{d}{2}+1,1} > j_{\frac{d}{2},1}$. Therefore $\fconnf(k;a,b)$ is positive and decreasing on the whole region $\abs*{k}\in\left(0,j_{\frac{d}{2},1}/R_{a,b}\right]$. Since this overlaps with the high $\abs*{k}$ range, and we have the uniform quadratic behaviour near $k=0$, the function $k\mapsto \abs*{\fconnf(k;a,b)}$ can never increase and approach $\fconnf(0;a,b)$ again after leaving a neighbourhood of $k=0$. We therefore have the bound \eqref{eqn:fconnf_a-b_bound} and thus have proven that \ref{Assump:Bound} is satisfied.

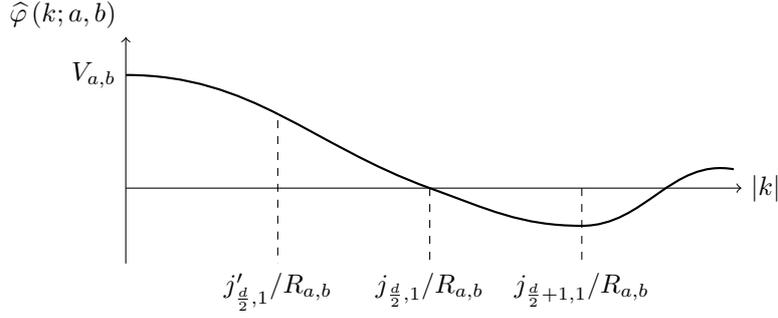
\begin{figure}
    \centering
    \begin{tikzpicture}
        \draw[->] (0,0) -- (8.1,0)node[right]{$\abs*{k}$};
        \draw[->] (0,-1) -- (0,2)node[above left]{$\fconnf\left(k;a,b\right)$};
        \draw[thick] (0,1.5)node[left]{$V_{a,b}$} to [out=0,in=160] (4,0) to [out=-20,in=180] (6,-0.5) to [out=0,in=170] (8,0.25);
        \draw[dashed] (4,0) -- (4,-1)node[below]{$j_{\frac{d}{2},1}/R_{a,b}$};
        \draw[dashed] (6,0) -- (6,-1)node[below]{$j_{\frac{d}{2}+1,1}/R_{a,b}$};
        \draw[dashed] (2,0.9) -- (2,-1)node[below]{$j'_{\frac{d}{2},1}/R_{a,b}$};
    \end{tikzpicture}
    \caption{Sketch of $\fconnf\left(k;a,b\right)$ against $\abs*{k}$. It approaches its maximum quadratically as $\abs*{k}\to0$. The first local maximum of $J_{\frac{d}{2}}$ occurs at $j'_{\frac{d}{2},1}\sim \frac{d}{2}+\gamma_1\left(\frac{d}{2}\right)^\frac{1}{3}$. The first zero of $\fconnf\left(k;a,b\right)$ occurs at $\abs*{k}R_{a,b} = j_{\frac{d}{2},1}\sim \frac{d}{2}+\gamma_2\left(\frac{d}{2}\right)^\frac{1}{3}$ where $\gamma_2>\gamma_1$. Furthermore, $\fconnf\left(k;a,b\right)$ is strictly decreasing until $\abs*{k}R_{a,b} = j_{\frac{d}{2}+1,1}\sim \frac{d}{2}+\gamma_2\left(\frac{d}{2}\right)^\frac{1}{3} + \frac{1}{2}$. }
    \label{fig:my_label}
\end{figure}

We return to \eqref{eqn:directional2ndMoment} in \ref{Assump:2ndMoment}. We first calculate the second moment:
\begin{equation}
    \int\abs*{\xbar}^2\connf\left(\xbar;a,b\right)\dd \xbar = \frac{d}{d+2}V_{a,b}R^2_{a,b}.
\end{equation}
From $\abs*{\xbar}^2 = \xbar_1^2 + \ldots + \xbar^2_d$ (for an orthonormal basis) and the rotational symmetry, we have
\begin{equation}
    \int\left(e\cdot\xbar\right)^2\connf\left(\xbar;a,b\right)\dd \xbar = \frac{1}{d+2}V_{a,b}R^2_{a,b}
\end{equation}
for all unit vectors $e\in\Rd$. Therefore
\begin{equation}
    \esssup_{a,b\in\Ecal}\int\left(e\cdot \xbar\right)^2\connf\left(\xbar;a,b\right)\dd \xbar \leq \frac{4}{d+2}\Vmax\left(\Rmax\right)^2.
\end{equation}
While $\Vmax$ is a constant, $\Rmax\sim \sqrt{d/8\pi\e}$ and therefore we have a bound uniform in $d$. The argument now proceeds similarly to the corresponding part of the proof for the multivariate Gaussian model. From \eqref{eqn:OneInfNormQuadraticExpansion} and \ref{Assump:Bound} we have both sides of \eqref{eqn:directional2ndMoment} satisfying quadratic bounds around $k=0$, and the triangle inequality with our other bounds prove that the bound also holds on the remainder.

For \ref{Assump:BallDecay} we first want to bound the convolution of three $\connf$ functions. Given marks $a_1,\ldots,a_6\in\left[0,1\right]$ and $\xbar\in\Rd$, let $\Bb_{a,b}(\xbar)$ denote the ball of radius $R_{a,b}$ centred on $\xbar$, and let $\Id_{a,b}\left(\xbar\right) = \Id\left\{\abs*{\xbar}\leq R_{a,b}\right\}$. Then
\begin{align}
    \left(\connf(\cdot;a_1,a_2)\star \connf(\cdot;a_3,a_4) \star \connf(\cdot;a_5,a_6)\right)\left(\xbar\right)
    &= \int \Id_{a_1,a_2}(\ybar)\Id_{a_3,a_4}(\zbar-\ybar)\Id_{a_5,a_6}(-\zbar)\dd \ybar \dd \zbar  \nonumber\\
    &= \int_{\Bb_{a_1,a_2}\left(\zerobar\right)}\left(\int_{\Bb_{a_5,a_6}\left(\zerobar\right)} \Id_{a_3,a_4}(\zbar-\ybar)\dd \zbar\right) \dd \ybar \nonumber \\
    &= \int_{\Bb_{a_1,a_2}\left(\zerobar\right)} \abs*{\Bb_{a_5,a_6}\left(\zerobar\right)\cap \Bb_{a_3,a_4}\left(\ybar\right)}\dd \ybar.
\end{align}
Then for $\delta\in\left(0,1\right)$ we have
\begin{align}
    &\int_{\Bb_{a_1,a_2}\left(\zerobar\right)} \abs*{\Bb_{a_5,a_6}\left(\zerobar\right)\cap \Bb_{a_3,a_4}\left(\ybar\right)}\dd \ybar \nonumber \\
    &\qquad\leq \int_{\delta\Bb_{a,b}\left(\zerobar\right)}\abs*{\Bb_{a_3,a_4}\left(\zerobar\right)}\dd\ybar + \abs*{\Bb_{a_1,a_2}\left(\zerobar\right)\setminus\delta\Bb_{a_1,a_2}\left(\zerobar\right)} \esssup_{\ybar\in\Bb_{a_1,a_2}\left(\zerobar\right)\setminus\delta\Bb_{a_1,a_2}\left(\zerobar\right)} \abs*{\Bb_{a_5,a_6}\left(\zerobar\right)\cap \Bb_{a_3,a_4}\left(\ybar\right)} \nonumber\\
    & \qquad\leq \left(\Vmax\right)^2\delta^d + \Vmax\esssup_{\ybar\in\Bb_{a_1,a_2}\left(\zerobar\right)\setminus\delta\Bb_{a_1,a_2}\left(\zerobar\right)} \abs*{\Bb_{a_5,a_6}\left(\zerobar\right)\cap \Bb_{a_3,a_4}\left(\ybar\right)}.
\end{align}
Let $\overline{e}_1\in\Rd$ be a unit vector. Then by spherical symmetry we have
\begin{equation}
\label{eqn:Ball-Ball_intersection}
    \esssup_{\ybar\in\Bb_{a_1,a_2}\left(\zerobar\right)\setminus\delta\Bb_{a_1,a_2}\left(\zerobar\right)} \abs*{\Bb_{a_5,a_6}\left(\zerobar\right)\cap \Bb_{a_3,a_4}\left(\ybar\right)} = \abs*{\Bb_{a_5,a_6}\left(\zerobar\right)\cap \Bb_{a_3,a_4}\left(\delta R_{a_1,a_2}\overline{e}_1\right)}.
\end{equation}
We bound the volume of this intersection with the volume of the $d$-ball with radius equal to the radius of the $d-1$-sphere formed by the intersection of their boundaries. This radius is equal to the length $h$ in Figure~\ref{fig:Triangle_of_Intersection_Spheres1}, which is maximised when $R_{a_5,a_6}$ and $R_{a_3,a_4}$ are maximised and $\delta R_{a_1,a_2}$ is minimised. Recall we defined the constant $c_1\in\left(0,1/\sqrt{8\pi\e}\right)$ when we defined the model in Section~\ref{subsection:Examples}. Since we have $R_{a_5,a_6},R_{a_3,a_4} \leq 2\Rmax$ and $\delta R_{a_1,a_2} \geq \delta c_1 d^\frac{1}{2}$, we have $h\leq \sqrt{\left(2 \Rmax\right)^2 - \delta^2 c^2_1 d}$ - we have used the fact that the triangle is isosceles if both the $R_{a_5,a_6}$ and $R_{a_3,a_4}$ edges attain their maximum. The ball with this radius then has the volume
\begin{equation}
    \frac{\pi^\frac{d}{2}}{\Gamma\left(\frac{d}{2}+1\right)}h^d \leq \Vmax\left(1 - \frac{\pi \delta^2 c^2_1 d}{\left(\Vmax\right)^\frac{2}{d}\Gamma\left(\frac{d}{2}+1\right)^\frac{2}{d}}\right)^\frac{d}{2}.
\end{equation}
Since $\frac{\pi \delta^2 c^2_1 d}{\left(\Vmax\right)^\frac{2}{d}\Gamma\left(\frac{d}{2}+1\right)^\frac{2}{d}} = 2\pi\e\delta^2c_1^2\left(1+o(1)\right)$ and $c_1>0$, this shows that 
\begin{equation}
    \abs*{\Bb_{a_5,a_6}\left(\zerobar\right)\cap \Bb_{a_3,a_4}\left(\delta R_{a_1,a_2}\overline{e}_1\right)} \leq \Vmax\left(1 - \pi\e\delta^2c_1^2\right)^\frac{d}{2}
\end{equation}
for $d$ sufficiently large. Sine we now have a mark independent bound with the required decay, we have proved the bound for the convolution of three $\connf$ functions.

\begin{figure}
    \centering
    \begin{subfigure}[b]{0.45\textwidth}
    \centering
    \begin{tikzpicture}
        \draw[thick] (0,0) -- (2,2) -- (5,0) -- (0,0);
        \draw[dashed] (2,2) -- (2,0);
        \draw (2.2,0) -- (2.2,0.2) -- (2,0.2);
        \draw (1,1) circle (0pt) node[above left]{$R_{a_5,a_6}$};
        \draw (3.5,1) circle (0pt) node[above right]{$R_{a_3,a_4}$};
        \draw (2.5,0) circle (0pt) node[below]{$\delta R_{a_1,a_2}$};
        \draw (2,1) circle (0pt) node[right]{$h$};
    \end{tikzpicture}
    \subcaption{Diagram demonstrating the radius $h$ in the three $\connf$ function argument.}
    \label{fig:Triangle_of_Intersection_Spheres1}
    \end{subfigure}
    \hfill
    \begin{subfigure}[b]{0.45\textwidth}
    \centering
    \begin{tikzpicture}
        \draw[thick] (0,0) -- (2,2) -- (5,0) -- (0,0);
        \draw[dashed] (2,2) -- (2,0);
        \draw (2.2,0) -- (2.2,0.2) -- (2,0.2);
        \draw (1,1) circle (0pt) node[above left]{$R_{a,c}$};
        \draw (3.5,1) circle (0pt) node[above right]{$R_{c,b}$};
        \draw (2.5,0) circle (0pt) node[below]{$\abs*{\xbar}$};
        \draw (2,1) circle (0pt) node[right]{$h$};
    \end{tikzpicture}
    \subcaption{Diagram demonstrating the radius $h$ in the two $\connf$ function argument.}
    \label{fig:Triangle_of_Intersection_Spheres2}
    \end{subfigure}
    \caption{For both the arguments, we bound the intersection of two $d$-balls with a new $d$-ball that has the radius $h$ in these diagrams.}
\end{figure}
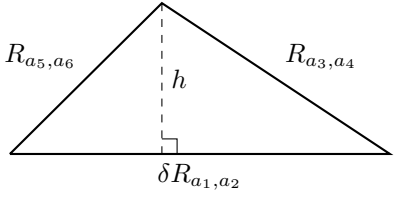
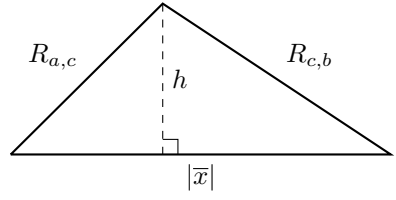

For finding the sets $B(x)$ in \eqref{eqn:BallDecay}, we write
\begin{equation}
    \int\connf\left(\xbar- \ybar;a,c\right)\connf\left(\ybar;c,b\right)\dd \ybar\Pcal\left(\dd c\right) = \int \abs*{\Bb_{a,c}\left(\xbar\right)\cap\Bb_{c,b}\left(\zerobar\right)}\Pcal\left(\dd c \right).
\end{equation}
As before, we can bound the volume of this intersection independently of the mark $c$. We can bound it with the volume of the ball with radius $h$ where $h\leq \sqrt{\left(2\Rmax\right)^2 - \frac{1}{4}\abs*{\xbar}^2}$ (as can be seen in Figure~\ref{fig:Triangle_of_Intersection_Spheres2}). As before, if $\abs*{\xbar} \geq \kappa d^\frac{1}{2}$ for some $\kappa>0$, then the volume of the ball with radius $h$ vanishes in the $d\to\infty$ limit. Furthermore, if $\kappa\leq \frac{1}{\sqrt{2\pi\e}}$ then the volume of the ball with radius $\kappa d^\frac{1}{2}$ vanishes in the $d\to\infty$ limit. Therefore assumption \ref{Assump:BallDecay} is satisfied with the sets $B(x)\equiv B\left(\xbar,a\right)=\left\{\left(\ybar,b\right)\in\X:\abs*{\xbar-\ybar} \leq \kappa d^\frac{1}{2}\right\}$ for any $\kappa\in\left(0,\frac{1}{\sqrt{2\pi\e}}\right)$.

\end{proof}

\section{Proofs for Linear Operator Lemmas}
\label{appendix:LinearOperatorProofs}

\begin{proof}[Proof of Lemma~\ref{lem:BoundsonOperatorNorm}]
The first inequality in \eqref{eqn:OpnormBounds} holds from the expressions of $\SupSpec{\cdot}$ and $\OpNorm{\cdot}$ in terms of the spectrum (see \eqref{eqn:Opnorm_Supspec_Spectrum}). The boundedness and self-adjointness of $H$ implies that the second inequality follows from an application of the Schur Test and the conjugate symmetry of its kernel function (see, for example \cite[Theorem~5.2]{halmos1978bounded}). For the lower bound on $\SupSpec{\cdot}$, we use the definition \eqref{eqn:SupSpecDef} and consider the test function $f(a) = 1$.
\end{proof}

\begin{proof}[Proof of Lemma~\ref{lem:NormBounds}]
For the $\OneNorm{\cdot}$, $\TwoNorm{\cdot}$, and $\InfNorm{\cdot}$ norms the inequality is clear from the definitions. For the $\OpNorm{\cdot}$ version, let $f\in L^2\left(\Xcal\right)$, and use $\rm Abs$ to denote the operator $f\mapsto \abs*{f}$. While this operator is not linear, it is clear that $\norm*{{\rm Abs}f}_2 = \norm*{f}_2$. With this setup,
\begin{multline}
    \norm*{Gf}^2_2 = \int \abs*{\int g\left(x,y\right) f\left(y\right) \nu\left(\dd y\right)}^2\nu\left(\dd x\right) \leq \int \left(\int \abs*{g\left(x,y\right)} \abs*{f\left(y\right)} \nu\left(\dd y\right)\right)^2\nu\left(\dd x\right)\\
    \leq \int \left(\int h\left(x,y\right) \abs*{f\left(y\right)} \nu\left(\dd y\right)\right)^2\nu\left(\dd x\right) = \norm*{H{\rm Abs}f}^2_2.
\end{multline}
Therefore 
\begin{equation}
    \OpNorm{G} = \sup_f\frac{\norm*{Gf}_2}{\norm*{f}_2} \leq \sup_f\frac{\norm*{H{\rm Abs}f}_2}{\norm*{f}_2} = \sup_f\frac{\norm*{H{\rm Abs}f}_2}{\norm*{{\rm Abs}f}_2}\leq \sup_f\frac{\norm*{Hf}_2}{\norm*{f}_2} =\OpNorm{H}.
\end{equation}
Here we used that the image of $\rm Abs$ is a subset of $L^2\left(\Xcal\right)$.

A similar argument works for the $\SupSpec{\cdot}$ inequality. Given $f\in L^2(\Xcal)$,
\begin{equation}
    \abs*{\left<f,Gf\right>}  \leq \int\abs*{f(x)}\abs*{g(x,y)}\abs*{f(y)}\nu^{\otimes2}\left(\dd x,\dd y\right) \leq \left<{\rm Abs}f,H{\rm Abs}f\right>.
\end{equation}
Therefore
\begin{equation}
    \SupSpec{G} = \sup_f\frac{\left<f,Gf\right>}{\norm*{f}^2_2} \leq \sup_f\frac{\left<{\rm Abs}f,H{\rm Abs}f\right>}{\norm*{f}^2_2} = \sup_f\frac{\left<{\rm Abs}f,H{\rm Abs}f\right>}{\norm*{{\rm Abs}f}^2_2}\leq \sup_f\frac{\left<f,Hf\right>}{\norm*{f}^2_2} =\SupSpec{H}.
\end{equation}
\end{proof}

\begin{proof}[Proof of Lemma~\ref{lem:SupSpecTriangle}]
It is a well-known result (for example see \cite{kato1995perturbation}) that $\sigma(G+H)$ is contained in the closed $\OpNorm{H}$-neighbourhood of $\sigma(G)$. Perturbing $G$ by $+H$, and perturbing $G+H$ by $-H$ then implies the result.
\end{proof}

\begin{proof}[Proof of Lemma~\ref{thm:Spectrum_Subset}]
We present the proof of \eqref{eqn:spectrum_equality} for the connection operator case. The two-point operator case follows similarly.

First suppose $\zeta\in \Ess \bigcup_{k\in\Rd}\sigma\left(\fOpconnf(k)\right)$. Therefore for all $\epsilon>0$ there exists $P_\epsilon\subset\Rd$ such that $\Leb\left(P_\epsilon\right)>0$ and for all $k\in P_\epsilon$ there exists $z\in\sigma\left(\fOpconnf\left(k\right)\right)$ such that $\abs*{\zeta-z}<\epsilon$. Furthermore, given $z\in\sigma\left(\fOpconnf\left(k\right)\right)$ there exists $g\in L^2\left(\Ecal\right)$ with $\norm*{g}_2=1$ such that $\norm*{\left(z\Id - \fOpconnf\left(k\right)\right)g}_2<\epsilon$. For all $k\in P_\epsilon$ we let $z^{(k)}$ and $g^{(k)}$ denote such choices. Now let $\widehat{f}$ be any element of $L^2\left(\Rd\right)$ supported on $P_\epsilon$ with $\norm*{\widehat{f}}_2=1$. Then define
\begin{equation}
    h\left(\xbar,a\right) := \int_{P_\epsilon} \e^{-i k\cdot \xbar}\widehat{f}(k)g^{(k)}(a)\frac{\dd k}{\left(2\pi\right)^{d}}.
\end{equation}
By the Fourier Inversion Theorem, the Fourier transform of $h$ is given by $\widehat{h}\left(k,a\right) = \widehat{f}(k)g^{(k)}(a)$ for $k\in P_\epsilon$ and vanishes elsewhere. By Plancherel's Theorem (that is, the Fourier transform is unitary), $h$ and $\widehat{h}$ have the same $L^2$-norm:
\begin{multline}
    \norm*{h}^2_2= \norm*{\widehat{h}}^2_2 = \int_{P_\epsilon}\abs*{\widehat{f}(k)}^2\left(\int \abs*{g^{(k)}(a)}^2\Pcal\left(\dd a\right) \right)\frac{\dd k}{\left(2\pi\right)^{d}} \\= \int_{P_\epsilon}\abs*{\widehat{f}(k)}^2\norm*{g^{(k)}}^2_2\frac{\dd k}{\left(2\pi\right)^{d}} = \int_{P_\epsilon}\abs*{\widehat{f}(k)}^2\frac{\dd k}{\left(2\pi\right)^{d}} = \norm*{\widehat{f}}^2_2 = 1.
\end{multline}
By applying Plancherel's Theorem once again, we can write 
\begin{multline}
    \norm*{\left(\zeta\Id - \Opconnf\right)h}^2_2  = \int_{P_\epsilon}\abs*{\widehat{f}(k)}^2\int\abs*{\zeta g^{(k)}(a) - \int\fconnf(k;a,b)g^{(k)}(b)\Pcal\left(\dd b\right)}^2\Pcal\left(\dd a\right) \frac{\dd k}{\left(2\pi\right)^{d}} \\
    = \int_{P_\epsilon}\abs*{\widehat{f}(k)}^2\norm*{\left(\zeta\Id - \fOpconnf(k)\right)g^{(k)}}^2_2 \frac{\dd k}{\left(2\pi\right)^{d}} < \int_{P_\epsilon}\abs*{\widehat{f}(k)}^2\left(\epsilon + \underbrace{\norm*{\left(z(k)\Id - \fOpconnf(k)\right)g^{(k)}}_2}_{<\epsilon}\right)^2 \frac{\dd k}{\left(2\pi\right)^{d}} < 4\epsilon^2.
\end{multline}
Here we use the triangle inequality and $\norm*{g^{(k)}}_2=1$ to replace $\zeta$ with $z(k)$ at the cost of an $\epsilon$. Since for each $\epsilon>0$ we have constructed a unit $h\in L^2\left(\X\right)$ such that $\norm*{\left(\zeta\Id - \Opconnf\right)h}_2<2\epsilon$, we have proven that $\left(\zeta\Id - \Opconnf\right)$ has no bounded linear inverse and therefore $\zeta\in\sigma\left(\Opconnf\right)$.

We now show the converse. Suppose $\zeta\in\sigma\left(\Opconnf\right)$. Then for all $\epsilon>0$ there exists $h \in L^2\left(\X\right)$ such that $\norm*{h}_2=1$ and $\norm*{\left(\zeta\Id - \Opconnf\right)h}_2 < \epsilon$. By using Plancherel's Theorem to replace the original functions with their Fourier transforms, we get
\begin{multline}
    \norm*{\left(\zeta\Id - \Opconnf\right)h}^2_2 = \int\abs*{\zeta h\left(\xbar,a\right) - \int\connf\left(\xbar-\ybar;a,b\right)h\left(\ybar,b\right)\dd\ybar\Pcal\left(\dd b\right)}^2\dd\xbar\Pcal\left(\dd a\right) \nonumber\\
    = \int\abs*{\zeta \widehat{h}\left(k,a\right) - \int\fconnf\left(k;a,b\right)\widehat{h}\left(k,b\right)\Pcal\left(\dd b\right)}^2\frac{\dd k}{\left(2\pi\right)^d}\Pcal\left(\dd a\right).
\end{multline}
Now let us define the family of functions $g^{(k)}\colon \Ecal \to \Complex$ by $g^{(k)}\left(a\right) := \widehat{h}\left(k,a\right)$ for all $k\in\Rd$. Then we can again re-write the equality above as
\begin{equation}
\label{eqn:kintegralofnorm}
    \norm*{\left(\zeta\Id - \Opconnf\right)h}^2_2 = \int\norm*{\left(\zeta\Id - \fOpconnf\left(k\right)\right)g^{(k)}}^2_2\frac{\dd k}{\left(2\pi\right)^d}.
\end{equation}
For each $k\in\Rd$, $\fOpconnf(k)$ is self-adjoint and therefore by the Spectral Theorem (Theorem~\ref{thm:spectraltheorem}) there exists a Hilbert space $L^2\left(\Efrak_k,\mu_k\right)$, a unitary operator $U_k\colon L^2\left(\Ecal\right)\to L^2\left(\Efrak_k\right)$, and a measurable function $e\mapsto \widetilde{\connf}\left(k\right)\left(e\right)$ such that $\fOpconnf(k)$ is unitarily equivalent to pointwise multiplication by $\widetilde{\connf}\left(k\right)\left(e\right)$ on $L^2\left(\Efrak_k\right)$. If we define $f^{\left(k\right)\left(e\right)} := \left(U_kg^{\left(k\right)}\right)\left(e\right)\in \Complex$, then \eqref{eqn:kintegralofnorm} becomes
\begin{equation}
    \norm*{\left(\zeta\Id - \Opconnf\right)h}^2_2 = \int\left(\int\abs*{\zeta - \widetilde{\connf}\left(k\right)\left(e\right)}^2\abs*{f^{(k)(e)}}^2\mu_k\left(\dd e\right)\right)\frac{\dd k}{\left(2\pi\right)^d}.
\end{equation}
Note that since the Fourier transform and the maps $\left\{U_k\right\}_k$ are all unitary, the normalisation of $h$ passes on to $g^{(k)}$ and on to $f^{(k)(e)}$ so that
\begin{equation}
    \int\left(\int\abs*{f^{(k)(e)}}^2\mu_k\left(\dd e\right)\right)\frac{\dd k}{\left(2\pi\right)^d}=1.
\end{equation}

Now suppose for contradiction that $\abs*{\zeta - \widetilde{\connf}\left(k\right)\left(e\right)}\geq \epsilon$ for $\mu_k$-a.e. $e\in\Efrak_k$, for $\Leb$-a.e. $k\in\Rd$. Then the normalisation of $f^{(k)(e)}$ implies that
\begin{equation}
    \int\left(\int\abs*{\zeta - \widetilde{\connf}\left(k\right)\left(e\right)}^2\abs*{f^{(k)(e)}}^2\mu_k\left(\dd e\right)\right)\frac{\dd k}{\left(2\pi\right)^d} \geq \epsilon^2.
\end{equation}
However this contradicts the condition that $\norm*{\left(\zeta\Id - \Opconnf\right)h}_2 < \epsilon$. Instead, for each $\epsilon>0$ there exists $P_\epsilon\subset\Rd$ with $\Leb\left(P_\epsilon\right)>0$ such that for all $k\in P_\epsilon$ there exists $Q^{(k)}_\epsilon\subset \Efrak_k$ with $\mu_k\left(Q^{(k)}_\epsilon\right)>0$ such that for all $e\in Q^{(k)}_\epsilon$ we have $\abs*{\zeta - \widetilde{\connf}\left(k\right)\left(e\right)}< \epsilon$.

For each $k\in\Rd$, the operator $U_k\fOpconnf(k)U_k^{-1}$ is a multiplication operator on $L^2\left(\Efrak_k\right)$ with function $e\mapsto\widetilde{\connf}(k)(e)$, and therefore it is easy to calculate its spectrum. Since the operator is unchanged by changing $\widetilde{\connf}(k)(e)$ on a $\mu_k$-null set, the spectrum is given by the essential image of $\widetilde{\connf}(k)$:
\begin{equation}
    \EssIm{\widetilde{\connf}\left(k\right)} = \left\{z\in\Complex:\forall\epsilon>0,\mu_k\left(e\in\Efrak_k:\abs*{\widetilde{\connf}\left(k\right)\left(e\right)-z}>\epsilon\right)>0\right\}.
\end{equation}
Since the unitary equivalence preserves the spectrum, this gives us an expression for $\sigma\left(\fOpconnf\left(k\right)\right)$. In particular, we can write our previous conclusion as the following. For all $\epsilon>0$ there exists $P_\epsilon\subset\Rd$ with $\Leb\left(P_\epsilon\right)>0$ such that for all $k\in P_\epsilon$ there exists $z\in \sigma\left(\fOpconnf\left(k\right)\right)$ such that $\abs*{\zeta-z}<\epsilon$. Equivalently, this means that $\zeta\in \Ess \bigcup_{k\in\Rd}\sigma\left(\fOpconnf(k)\right)$.

Since $\OpNorm{A} = \sup\left\{\abs*{z}:z\in\sigma\left(A\right)\right\}$ and $\SupSpec{A} = \sup\left\{z:z\in\sigma\left(A\right)\subset \R\right\}$ for self-adjoint operator $A$, the remaining equalities follow from \eqref{eqn:spectrum_equality}.
\end{proof}

\section{Differentiating the Two-Point Operator}

\label{sec:Prelim:differentiating}

The ideas of this section can be found in \cite{HeyHofLasMat19}, but we will need extra care to account for the fact that we are dealing with operators and marks and not just functions on $\R^d$.

Recall $\tlam\left(x,y\right) := \pla(\conn{x}{y}{\xi^{x,y}})$. Using the notation $x = \left(\overline{x},a\right)\in\R^d\times \Ecal$, define $\Lambda_n\left(x\right) = \left(\overline{x} + \left[-n,n\right)^d\right)\times\Ecal$. Then we can define the truncated two-point function
\begin{equation}
    \tlam^n\left(x,y\right) := \pla\left(\conn{x}{y}{\xi^{x,y}_{\Lambda_n\left(y\right)}}\right).
\end{equation}

We will want to give meaning to the event $\left\{\conn{x}{\Lambda_n^c\left(y\right)}{\xi^x_{\Lambda_n\left(y\right)}}\right\}$. To do so we add a ``ghost vertex'' $g$ in the same way we added deterministic vertices, and add an edge between $v \in \xi_{\Lambda_n\left(y\right)}$ and $g$ with probability $1-\exp\left(-\int_{\Lambda_n^c\left(y\right)} \connf(u,v) \nu\left(\dd u\right)\right)$. We now identify $\Lambda_n^c\left(y\right)$ with $g$.

\begin{lemma}[Differentiability of $\tlam$]\label{lem:tlam_differentiability}
Let $x,y\in\X$ and $\varepsilon>0$ be arbitrary. The function $\lambda\mapsto\tau_\lambda^n(x,y)$ is differentiable on $[0, \lambda_O-\varepsilon]$ for any $n\in\N$. Furthermore, $ \tau_\lambda^n(x,y)$ converges to $\tlam(x,y)$ uniformly in $\lambda$ and $\tfrac{\dd}{\dd\lambda} \tau_\lambda^n(x,y)$ converges to a limit uniformly in $\lambda$. Consequently, $\tlam(x,y)$ is differentiable w.r.t.~$\lambda$ on $[0,\lambda_O)$ and
\begin{equation}\label{eq:prelim:tlam_derivative_limit_exchange}
    \lim_{n\to\infty} \frac{\dd}{\dd\lambda} \tau_\lambda^n(x,y) = \frac{\dd}{\dd\lambda}\tlam(x,y) =
							\int \pla(\conn{y}{x}{\xi^{x,y,u}}, \nconn{y}{x}{\xi^{x,y}}) \nu\left(\dd u\right).
\end{equation}
\end{lemma}

\begin{proof}
    The proof of this lemma follows as in \cite[Lemma~2.2]{HeyHofLasMat19}. The key idea is to uniformly bound the probabilities that a path that leaves $\Lambda_n(y)$ is used by the various events. The importance of $\lambda_O$ arises from the observation that for $\lambda<\lambda_O$
    \begin{equation}
    \label{eqn:OneOneNormFinite}
    \int\tlam(\xbar;a,b)\dd \xbar \Pcal^{\otimes 2}\left(\dd a,\dd b\right) \leq \SupSpec{\fOptlam(0)}\leq \OpNorm{\fOptlam(0)} <\infty,
\end{equation}
which is how we know that for $\nu$-almost every $y\in\X$ we have $\int_{\Lambda_n^c\left(y\right)} \tau_{\lambda_O-\varepsilon}(u,y) \nu\left(\dd u\right)\to0$ as $n\to\infty$.
\end{proof}

A family of operators $\left\{H_x\right\}_{x\in\R}$ is differentiable at $x$ if there exists an operator $F_x$ such that
\begin{equation}
    \lim_{\varepsilon\to0}\OpNorm{\frac{1}{\varepsilon}\left(H_{x+\varepsilon}-H_x\right) - F_x} = 0.
\end{equation}
Using the notation $x = \left(\xbar,a\right)\in\R^d\times \Ecal$, recall that we require that the model has $\R^d$-translation invariance and so we can write $\tlam\left(x,y\right) = \tlam\left(\xbar-\ybar;a,b\right)$. Also recall that we defined the Fourier transform $\fOptlam\left(k\right)\colon L^2\left(\Ecal\right) \to L^2\left(\Ecal\right)$ as the integral operator with kernel $\ftlam\left(k;a,b\right) = \int\cos\left(k\cdot\xbar\right)\tlam\left(\xbar;a,b\right)\dd \xbar$. We now define the \emph{displaced} Fourier transform $\fOptklam\left(l\right)\colon L^2\left(\Ecal\right)\to L^2\left(\Ecal\right)$ as the integral operator with kernel $\ftklam\left(l;a,b\right) = \int\left(1-\cos\left(k\cdot \xbar\right)\right)\cos\left(l\cdot\xbar\right)\tlam\left(\xbar;a,b\right)\dd \xbar$.

\begin{corollary}[Differentiability of $\Optlam$]
\label{thm:differentiabilityofOpT}
The operators $\Optlam$, $\fOptlam\left(k\right)$, and $\fOptklam\left(l\right)$ are differentiable w.r.t.~$\lambda$ on $\left(0,\lambda_O\right)$ and their derivatives are the bounded linear operators $\frac{\dd}{\dd\lambda}\Optlam\colon L^2\left(\X\right) \to L^2\left(\X\right)$, $\frac{\dd}{\dd\lambda}\fOptlam\left(k\right)\colon L^2\left(\Ecal\right) \to L^2\left(\Ecal\right)$, and $\frac{\dd}{\dd\lambda}\fOptklam\left(l\right)\colon L^2\left(\Ecal\right) \to L^2\left(\Ecal\right)$ with kernel functions $\frac{\dd}{\dd \lambda}\tlam\left(x,y\right)$, $\frac{\dd}{\dd \lambda}\ftlam\left(k;a,b\right)$, and $\frac{\dd}{\dd \lambda}\ftklam\left(l;a,b\right)$ respectively. Specifically, we have the following bounds on the derivatives:
\begin{align}
    \OpNorm{\frac{\dd}{\dd \lambda}\Optlam} &\leq \OpNorm{\Optlam}^2,\label{eqn:differentiabilityofOpT -1}\\
    \OpNorm{\frac{\dd}{\dd \lambda}\fOptlam\left(k\right)} &\leq \OpNorm{\fOptlam\left(0\right)}^2,\label{eqn:differentiabilityofOpT -2}\\
    \OpNorm{\frac{\dd}{\dd \lambda}\fOptklam\left(l\right)} &\leq 4\OpNorm{\fOptlam\left(0\right) - \fOptlam\left(k\right)}\OpNorm{\fOptlam\left(0\right)}. \label{eqn:differentiabilityofOpT -3}
\end{align}
\end{corollary}

\begin{proof}

We use Lemma~\ref{lem:tlam_differentiability} in conjunction with the Leibniz integral rule to differentiate the two-point operator.

Fix $f\in L^2\left(\X\right)$ and let $\varepsilon\in\left(0,\lambda_O\right)$. From the reasoning of \eqref{eqn:OneOneNormFinite} we know that the subcritical integral $\int \tau_{\lambda_O-\varepsilon}(\xbar;a,b)\dd \xbar \Pcal^{(\otimes 2)}\left(\dd a,\dd b\right)<\infty$, and therefore 
\begin{equation}
    \int\tau_{\lambda_O-\varepsilon}\left(x,u\right)^2\nu\left(\dd u\right) \leq \int\tau_{\lambda_O-\varepsilon}\left(x,u\right)\nu\left(\dd u\right) < \infty,
\end{equation}
for $\nu$-almost every $x\in\X$. Here we have also used that $\tlam(x,u)$ is $\left[0,1\right]$-valued. This implies that $\tau_{\lambda_O-\varepsilon}\left(x,\cdot\right)\in L^2\left(\X\right)$ for $\nu$-almost every $x\in\X$. In conjunction with $f\in L^2\left(\X\right)$, H{\"o}lder's inequality implies that the function $\Theta_{\varepsilon,x}\colon y\mapsto \tau_{\lambda_O-\varepsilon}\left(x,y\right)\abs*{f\left(y\right)}$ is in $L^1\left(\X\right)$ for $\nu$-almost every $x\in\X$. Furthermore, $\tlam\left(x,y\right)$ is non-decreasing in $\lambda$ and so for all $\lambda<\lambda_O$ there exists $\varepsilon>0$ such that $\Theta_{\varepsilon,x}$ dominates $y\mapsto \tlam\left(x,y\right)f\left(y\right)$. This domination by an $L^1\left(\X\right)$ function allows us to use the Leibniz integral rule to exchange the integral and derivative to get (for $\nu$-almost every $x\in\X$)
\begin{equation}
    \frac{\dd}{\dd \lambda}\left(\left[\Optlam f\right]\left(x\right)\right) = \frac{\dd}{\dd \lambda}\int \tlam\left(x,u\right)f\left(u\right)\nu\left(\dd u\right) = \int \left(\frac{\dd}{\dd \lambda} \tlam\left(x,u\right)\right)f\left(u\right)\nu\left(\dd u\right) = \left[\left(\frac{\dd}{\dd \lambda}\Optlam\right)f\right]\left(x\right),
\end{equation}
where Lemma~\ref{lem:tlam_differentiability} gives the differentiability of the function $\tlam\left(x,y\right)$. 

We now prove that the operator $\frac{\dd}{\dd \lambda}\Optlam$ is bounded. From Lemma~\ref{lem:tlam_differentiability} and the BK inequality, we get the bound
\begin{multline}
    \label{eq:prelim:tlam_finite_derivative_bound}
    0\leq \frac{\dd}{\dd\lambda} \tlam^n(x,y) \leq \int \pla \left( \{\conn{y}{u}{\xi^{y,u}_{\Lambda_n\left(y\right)}} \}\circ\{ \conn{u}{x}{\xi^{u,x}_{\Lambda_n\left(y\right)}} \} \right) \nu\left(\dd u\right) \\\leq \int \tlam^n(u,y)\pla\left(\conn{u}{x}{\xi^{u,x}_{\Lambda_n\left(y\right)}}\right) \nu\left(\dd u\right) \leq \int \tlam\left(u,y\right)\tlam\left(x,u\right)\nu\left(\dd u\right),
\end{multline}
where the last inequality holds because we are removing the $\Lambda_n\left(y\right)$ restriction. Note that the right hand side of this equation is the kernel function of the operator $\Optlam^2$. Choose $f\in L^2\left(\X\right)$, and use $\rm Abs$ to denote the map $f\mapsto \abs*{f}$. Now the above inequality implies that
\begin{align}
    \norm*{\left[\frac{\dd}{\dd \lambda}\Optlam\right] f}^2_2 & \leq \int \left(\int\abs*{\frac{\dd}{\dd \lambda}\tlam\left(x,y\right)}\abs*{f\left(y\right)}\nu\left(\dd y\right)\right)^2 \nu\left(\dd x\right) \nonumber\\
    & \leq \int \left(\int\int \tlam\left(x,u\right)\tlam\left(u,y\right)\abs*{f\left(y\right)}\nu\left(\dd u\right)\nu\left(\dd y\right)\right)^2 \nu\left(\dd x\right) \nonumber\\
    &= \int \left(\left[\Optlam^2 {\rm Abs}{f}\right]\left(x\right)\right)^2 \nu\left(\dd x\right) = \norm*{\Optlam^2 {\rm Abs}f}^2_2.
\end{align}
Since $\norm*{{\rm Abs}f}_2 = \norm*{f}_2$, this then gives us
\begin{equation}
\label{eqn:Remove_Abs}
    \OpNorm{\frac{\dd}{\dd \lambda}\Optlam} \leq \sup_f\frac{\norm*{\Optlam^2 {\rm Abs}f}_2}{\norm*{f}_2} = \sup_f\frac{\norm*{\Optlam^2 {\rm Abs}f}_2}{\norm*{{\rm Abs}f}_2}\leq \OpNorm{\Optlam^2} = \OpNorm{\Optlam}^2 < \infty.
\end{equation}

Now we address the Fourier transform. Pick $\varepsilon\in(0,\lambda_O-\lambda)$. Recall $\abs*{\e^{ik\cdot\xbar}\tlam\left(\xbar;a,b\right)}$ is bounded by $\tau_{\lambda_O-\varepsilon}\left(\xbar;a,b\right)$ for all $\xbar\in\R^d$ and $a,b\in\Ecal$. From the reasoning of \eqref{eqn:OneOneNormFinite}, we know that $\tau_{\lambda_O-\varepsilon}\left(\cdot;a,b\right)$ is Lebesgue integrable for $\Pcal$-a.e. $a,b\in\Ecal$. We are then justified in using the Leibniz integral rule to say that for every $b\in\Ecal$ and $\Pcal$-a.e. $a\in\Ecal$, $\tfrac{\dd}{\dd\lambda} \int\e^{ik\cdot\xbar} \tlam(\xbar;a,b) \dd \xbar  = \int \e^{ik\cdot\xbar}\tfrac{\dd}{\dd\lambda} \tlam(\xbar;a,b) \dd \xbar$. Applying Lemma~\ref{lem:tlam_differentiability} as well as~\eqref{eq:prelim:tlam_finite_derivative_bound}, we derive that for such $a,b$ we have
\begin{equation}
    \abs*{\frac{\dd}{\dd \lambda} \ftlam(k;a,b)} \leq \int \lim_{n\to\infty} \frac{\dd}{\dd\lambda} \tau_\lambda^n(\xbar;a,b) \dd \xbar 
				 \leq \int \tlam(\xbar-\ybar;a,c) \tlam(\ybar;c,b) \,\dd \xbar  \,\dd \ybar \,\Pcal\left(\dd c \right).
\end{equation}
Note that the right hand side of this equation is the kernel function of the operator $\fOptlam\left(0\right)^2$. By repeating the argument for $\frac{\dd}{\dd \lambda}\Optlam$ - with $\X$ replaced by $\Ecal$ - we similarly arrive at
\begin{equation}
    \OpNorm{\frac{\dd}{\dd \lambda}\fOptlam\left(k\right)}\leq \OpNorm{\fOptlam\left(0\right)}^2.
\end{equation}

For the displaced Fourier transform, note that $0\leq 1-\cos\left(k\cdot\xbar\right)\leq 2$, and so the same argument as for the Fourier transform allows us to use the Leibniz integral rule to exchange derivative and integral. We can also improve on the immediate bound of $2\OpNorm{\fOptlam\left(0\right)}^2$ for the operator norm of the derivative. Let $f\in L^2\left(\Ecal\right)$,
\begin{align}
    \norm*{\left[\frac{\dd}{\dd \lambda}\fOptklam\left(l\right)\right]f}^2_2 &\leq \int\left(\int\abs*{\e^{il\cdot \xbar}}\left(1 - \cos\left(k\cdot \xbar\right)\right)\tlam(\xbar-\ybar;a,c) \tlam(\ybar;c,b) \abs*{f\left(b\right)}\dd \xbar  \,\dd \ybar \,\Pcal\left(\dd c \right)\Pcal\left(\dd b\right)\right)^2\Pcal\left(\dd a\right) \nonumber\\
    & \leq \int\left(\int2\Big[\left(1 - \cos\left(k\cdot \left(\xbar-\ybar\right)\right)\right)\tlam(\xbar-\ybar;a,c) \tlam(\ybar;c,b) \right. \nonumber\\ 
    &\qquad + \left(1 - \cos\left(k\cdot\ybar\right)\right)\tlam(\xbar-\ybar;a,c) \tlam(\ybar;c,b)  \Big]\abs*{f\left(b\right)}\,\dd \xbar  \,\dd \ybar \,\Pcal\left(\dd c \right)\Pcal\left(\dd b\right)\Big)^2\Pcal\left(\dd a\right) \nonumber\\
    & = 4 \norm*{\left[\left(\fOptlam\left(0\right) - \fOptlam\left(k\right)\right)\fOptlam\left(0\right) + \fOptlam\left(0\right)\left(\fOptlam\left(0\right) - \fOptlam\left(k\right)\right)\right]{\rm Abs} f}^2_2,
\end{align}
where we have used the cosine-splitting lemma (Lemma~\ref{lem:cosinesplitlemma}) to divide the displacement factor $\left(1-\cos\left(k\cdot \xbar\right)\right)$ into one spanning $\xbar-\ybar$ and one spanning $\ybar$. Using the above approach of \eqref{eqn:Remove_Abs} and the triangle inequality then produces the required bound.
\end{proof}

We use the following lemma in the proof of Proposition~\ref{thm:bootstrapargument} to bound the displacement on the two-point operator using the displacement of the connection operator when working in the sub-critical regime away from the critical threshold.

\begin{lemma}
\label{lem:tau-connf subcritical}
For $\lambda<\lambda_O$,
\begin{equation}
    \OpNorm{\fOptlam\left(l\right) - \fOptlam\left(l+k\right)} \leq \OpNorm{\fOptlam\left(0\right) - \fOptlam\left(k\right)} \leq \e^{4\lambda \OpNorm{\fOptlam\left(0\right)}}\OpNorm{\fOpconnf\left(0\right) - \fOpconnf\left(k\right)}.
\end{equation}
\end{lemma}

\begin{proof}
For the inequality removing $l$, let $f\in L^2\left(\Ecal\right)$. Then
\begin{align}
    \norm*{\left(\fOptlam\left(l\right) - \fOptlam\left(l+k\right)\right)f}^2_2  &\leq \int\left(\int\abs*{\e^{il\cdot \xbar}}\left(1 - \cos\left(k\cdot \xbar\right)\right)\tlam(\xbar;a,b) \abs*{f\left(b\right)}\,\dd \xbar\, \Pcal\left(\dd b\right)\right)^2\Pcal\left(\dd a\right) \nonumber \\
    & = \norm*{\left(\fOptlam\left(0\right) - \fOptlam\left(k\right)\right){\rm Abs} f}^2_2.
\end{align}
Since $\norm*{{\rm Abs}f}_2=\norm*{f}_2$, we can use the approach of \eqref{eqn:Remove_Abs} to get the first inequality.

For the second inequality, we first note that $\widehat{\mathcal T}_0\left(k\right) = \fOpconnf\left(k\right)$. We then use the differential inequality \eqref{eqn:differentiabilityofOpT -3} to extend to $\lambda<\lambda_O$. Using the reverse triangle inequality, we have
\begin{align}
    \limsup_{\varepsilon\downarrow 0}\frac{1}{\varepsilon}\abs*{\OpNorm{\widehat{\mathcal T}_{s+\varepsilon}\left(0\right) - \widehat{\mathcal T}_{s+\varepsilon}\left(k\right)} - \OpNorm{\widehat{\mathcal T}_{s}\left(0\right) - \widehat{\mathcal T}_{s}\left(k\right)}} &\leq \OpNorm{\left.\frac{\dd}{\dd \lambda}\left(\widehat{\mathcal T}_{\lambda}\left(0\right) - \widehat{\mathcal T}_{\lambda}\left(k\right)\right)\right|_{\lambda=s}} \nonumber \\
    &\leq 4 \OpNorm{\widehat{\mathcal T}_{s}\left(0\right)}\OpNorm{\widehat{\mathcal T}_{s}\left(0\right) - \widehat{\mathcal T}_{s}\left(k\right)}.
\end{align}
Since $\OpNorm{\widehat{\mathcal T}_{s}\left(0\right)}$ is monotone increasing for $s \in\left[0,\lambda\right]$, the right hand side of the inequality is bounded by $4 \OpNorm{\widehat{\mathcal T}_{\lambda}\left(0\right)}\OpNorm{\widehat{\mathcal T}_{s}\left(0\right) - \widehat{\mathcal T}_{s}\left(k\right)}$, giving the exponential factor in the result.
\end{proof}

\section{Proofs of Diagrammatic Bounds}
\label{appendix:DiagrammaticBoundsProofs}

\remainderBound*

\begin{proof}
As $\pla(\xconn{u_n}{x}{\xi^{u_n,x}_{n+1}}{A} ) \leq \tlam(x,u_n)$ for an arbitrary locally finite set $A$, the definition \eqref{eq:LE:Rn_def} gives
\begin{equation}
\label{eqn:remainderfunctionbound}
    \abs*{r_{\lambda,n}(x)} \leq \int \tlam(x,u)\pi_\lambda^{(n)}(u,y)  \nu\left(\dd u\right).
\end{equation}
Let $f\in L^2\left(\X\right)$. Then this gives
\begin{align}
    \norm*{R_{\lambda,n}f}^2_2 &\leq   \int\left(\int\abs*{r_{\lambda,n}(x,y)}\abs*{f\left(y\right)}\nu\left(\dd y\right)\right)^2\nu\left(\dd x\right) \nonumber \\
    &\leq   \lambda^2\int\left(\int\left(\int \tlam(x,u)\pi_\lambda^{(n)}(u,y)  \nu\left(\dd u\right)\right)\abs*{f\left(y\right)}\nu\left(\dd y\right)\right)^2\nu\left(\dd x\right) \nonumber\\
    &=   \lambda^2\int\left(\int \tlam(x,u)\left(\int\pi_\lambda^{(n)}(u,y)\abs*{f\left(y\right)}\nu\left(\dd y\right)\right)  \nu\left(\dd u\right)\right)^2\nu\left(\dd x\right) \nonumber\\
    & = \lambda^2 \norm*{\Optlam\OpLacelam^{(n)}{\rm Abs}f}^2_2.
			\label{eq:LE:Rn_Pin_bound-Operator} 
\end{align}
Note that we were able to use Tonelli's theorem to swap the $u$ and $y$ integrals because $\tlam$ and $\pi^{(n)}_\lambda$ are both non-negative and measurable. Finally, the approach of \eqref{eqn:Remove_Abs} implies the result.

For each $n\geq 0$, either $r_{\lambda,n}(x)\geq 0$ for all $x\in\X$ or $r_{\lambda,n}(x)\leq 0$ for all $x\in\X$. Therefore $\abs*{\widehat{r}_{\lambda,n}(k;a,b)}\leq \abs*{\widehat{r}_{\lambda,n}(0;a,b)}$ for all $a,b\in\Ecal$ and $\OpNorm{\widehat{R}_{\lambda,n}(k)}\leq \OpNorm{\widehat{R}_{\lambda,n}(0)}$ (by Lemma~\ref{lem:NormBounds}) for all $k\in\Rd$. From \eqref{eqn:remainderfunctionbound} we have
\begin{equation}
    \abs*{\widehat{r}_{\lambda,n}(0;a,b)} \leq \int \ftlam(0;a,c)\widehat{\pi}_\lambda^{(n)}(0;c,b)  \Pcal\left(\dd c\right).
\end{equation}
Therefore the same argument as above gives the second bound in the result.
\end{proof}

\DBPiObounds*

\begin{proof}
Our first comment relates to these bounds and every bound we perform hereafter in Section~\ref{sec:diagrammaticbounds}. From the positivity of $\pi^{(0)}_{\lambda}(x,y)$ and Lemma~\ref{lem:NormBounds} we have $\OpNorm{\fOpLacelam^{(0)}(k)}\leq \OpNorm{\fOpLacelam^{(0)}(0)}$, and by Lemma~\ref{lem:BoundsonOperatorNorm} we have $\OpNorm{\OpLacelam^{(0)}} \leq \OneNorm{\OpLacelam^{(0)}}$ and $\OpNorm{\fOpLacelam^{(0)}(0)} \leq \OneNorm{\fOpLacelam^{(0)}(0)}$. However from the definitions of the norm we also have $\OneNorm{\fOpLacelam^{(0)}(0)} = \OneNorm{\OpLacelam^{(0)}}$, and so we will be interested in bounding this last norm.

Now note that for the event $\{ \dconn{y}{x}{\xi^{y,x}} \}$ to hold, either there is a direct edge between $y$ and $x$, or there exist vertices $w,z$ in $\eta$ that are direct neighbours of $y$ and have respective disjoint paths to $x$ that both do not contain $y$. Hence, by the Mecke equation~\eqref{eq:prelim:mecke_n},
\begin{align}
    \pla( \dconn{y}{x}{\xi^{\orig,x}}) &\leq \connf(x,y) + \tfrac {1}{2} \E_\lambda\Big[ \sum_{(w,z) \in \eta^{(2)}} \mathds 1_{(\{y\sim w \text{ in } \xi^{y}\}
			\cap \{\conn{w}{x}{\xi^x}\}) \circ (\{y \sim z \text{ in } \xi^{y}\} \cap \{ \conn{z}{x}{\xi^x} \})} \Big] \nonumber\\
		& = \connf(x,y) + \tfrac {1}{2} \lambda^2 \int \pla\big( (\{y\sim y \text{ in } \xi^{y, w}\} \cap \{\conn{w}{x}{\xi^{x,w}}\})\nonumber \\
		& \hspace{3.5cm} \circ (\{y \sim z \text{ in } \xi^{y, z}\} \cap \{ \conn{z}{x}{\xi^{x,z}} \}) \big) \nu^{\otimes2}\left(\dd w,\dd z\right). 
\end{align}
After applying the BK inequality to the above probability,
\begin{multline}
    \pla\left( \dconn{y}{x}{\xi^{y,x}}\right) \leq \connf(x,y) + \tfrac {1}{2} \lambda^2 \left( \int \pla\left( \left\{y\sim w \text{ in } \xi^{y,w}\right\} \cap \left\{\conn{w}{x}{\xi^{w,x}}\right\}\right) \nu\left(\dd w\right) \right)^2 \notag \\
	= \connf(x,y) + \tfrac {1}{2} \lambda^2 \left(\int\tlam\left(x,w\right)\connf\left(w,y\right)\nu\left(\dd w\right)\right)^2. \label{eq:DB:Pi0_dconn_bound}
\end{multline}
Thus, recalling that $\pi_\lambda^{(0)}(x,y) = \pla(\dconn{y}{x}{\xi^{y,x}})-\connf(x,y) \geq 0$, and using the symmetry of $\connf$ and $\tlam$,
\begin{align}
    \OneNorm{\OpLacelam^{(0)}} &= \esssup_{y\in\X}\int\pi^{(0)}_\lambda\left(x,y\right)\nu\left(\dd x\right)\nonumber\\
    & \leq \tfrac{1}{2}\lambda^2\esssup_{y\in\X}\int \left(\int\tlam\left(x,w\right)\connf\left(w,y\right)\nu\left(\dd w\right)\right)^2\nu\left(\dd x\right)\nonumber\\
    & =  \tfrac{1}{2}\lambda^2\esssup_{y\in\X}\int \connf\left(y,u\right)\tlam\left(u,x\right)\tlam\left(x,w\right)\connf\left(w,y\right)\nu^{\otimes3}\left(\dd x,\dd w,\dd u\right)\nonumber\\
    &\leq \tfrac{1}{2}\lambda^2\InfNorm{\Opconnf\Optlam^2\Opconnf}.
\end{align}
We then use a supremum bound on this integral to split one of the $\Opconnf$ off from the others, and bound the other from above with $\Optlam$ to get the first two bounds of the result.

For the last bound of Proposition~\ref{thm:DB:Pi0_bounds}, we apply~\eqref{eq:DB:Pi0_dconn_bound} and obtain
\begin{align}
    &\OpNorm{\fOpLacelam^{(0)}\left(0\right) - \fOpLacelam^{(0)}\left(k\right)}  \leq \OneNorm{\fOpLacelam^{(0)}\left(0\right) - \fOpLacelam^{(0)}\left(k\right)}\nonumber \\
    &\qquad= \esssup_{b\in\Ecal}\int \left(1 - \cos\left(k\cdot \xbar\right)\right)\pi^{(0)}_\lambda\left(\xbar;a,b\right) \dd \xbar \Pcal\left(\dd a\right)\nonumber\\
    & \qquad\leq \tfrac{1}{2}\lambda^2\esssup_{b\in\Ecal}\int \left(1 - \cos\left(k\cdot \xbar\right)\right)\left(\int\tlam\left(\xbar-\ubar;a,c\right)\connf\left(\ubar;c,b\right)\dd \ubar \Pcal\left(\dd c\right)\right)^2 \dd \xbar \Pcal\left(\dd a\right)\nonumber\\
    & \qquad= \tfrac{1}{2}\lambda^2\esssup_{b\in\Ecal}\int \left(1 - \cos\left(k\cdot \xbar\right)\right)\connf\left(-\ubar^\prime;b,c^\prime\right)\tlam\left(\ubar^\prime-\xbar;c^\prime,a\right)\tlam\left(\xbar-\ubar;a,c\right)\connf\left(\ubar;c,b\right)\nonumber \\
    &\qquad\hspace{10cm}\times\dd \ubar \dd \ubar^\prime \dd \xbar \Pcal^{\otimes3}\left(\dd c,\dd c^\prime,\dd a\right)\nonumber\\
    & \qquad\leq \lambda^2\esssup_{b,b^\prime\in\Ecal}\int \Big[\left(1 - \cos\left(k\cdot \left(\xbar-\ubar\right)\right)\right)\connf\left(-\ubar^\prime;b^\prime,c^\prime\right)\tlam\left(\ubar^\prime-\xbar;c^\prime,a\right)\tlam\left(\xbar-\ubar;a,c\right)\connf\left(\ubar;c,b\right)\nonumber \\
    &\qquad\hspace{3cm}+ \left(1 - \cos\left(k\cdot \ubar\right)\right)\connf\left(-\ubar^\prime;b,c^\prime\right)\tlam\left(\ubar^\prime-\xbar;c^\prime,a\right)\tlam\left(\xbar-\ubar;a,c\right)\connf\left(\ubar;c,b\right)\Big] \nonumber \\
    &\qquad\hspace{10cm}\times\dd \ubar \dd \ubar^\prime \dd \xbar \Pcal^{\otimes3}\left(\dd c,\dd c^\prime,\dd a\right)\nonumber\\
    & \qquad= \lambda^2\InfNorm{\Opconnf\Optlam\Optklam\Opconnf + \Opconnf\Optlam\Optlam\Opconnf_k}.
\end{align}
Here we have again used the symmetry of $\connf$ and $\tlam$, and also used  Lemma~\ref{lem:cosinesplitlemma} to split the cosine factor over $\xbar-\ubar$ and $\ubar$.

We can consider each term individually by using the triangle inequality. By using a supremum bound, the symmetry of $\connf$, and by bounding $\connf$ with $\tlam$, we get the bound
\begin{equation}
    \InfNorm{\Opconnf\Optlam\Optlam\Opconnf_k} \leq \InfNorm{\Opconnf\Optlam\Optlam}\OneNorm{\Opconnf_k} \leq \InfNorm{\Optlam^3}\OneNorm{\Opconnf_k}.
\end{equation}
\end{proof}

\singleandpairbounds*

\begin{proof}
We begin with the single operators. For $\Psi^{(1)}$ we find
\begin{align}
    \OneNorm{\Psi^{(1)}} &= \lambda^4\esssup_{r,s\in\X}\int\tlam(w,u)\tlamo(t,s) \tlam(t,w)\tlam(u,z)\tlam(z,t)\tlam(z,r)\nu^{\otimes4}\left(\dd z, \dd t, \dd w, \dd u\right), \nonumber\\
    & \leq \lambda^4\left(\esssup_{r,s\in\X}\int\tlam(r,z) \tlam(z,t)\tlamo(t,s)\nu^{\otimes2}\left(\dd z, \dd t \right)\right) \nonumber \\
    &\hspace{5cm} \times\left(\esssup_{t',z'\in\X}\int\tlam(t',w)\tlam(w,u)\tlam(u,z')\nu^{\otimes2}\left(\dd w, \dd u\right)\right),\nonumber\\
    & = \trilamo\trilam,
\end{align}
where we have used a supremum bound on the $t$ and $z$ integrals to split them into $L^1$ and $L^\infty$ bounds. It is also clear that $\OneNorm{\Psi^{(3)}}= \trilam$ and $\OneNorm{\Psi^{(4)}} \leq \trilamo$. The $\OneNorm{\Psi^{(2)}}$ bound is more involved and we return to it in a moment.

We are able to represent these $\OneNorm{\cdot}$ norms and their bounds pictorially. The vertices of the diagrams represent the variables $w,u,z,t,a,b$, etc.\ appearing in the $\OneNorm{\cdot}$ expression. The variables over which a $\nu$-integral is taken are represented by the \IntegralDot~vertices, and the variables over which a supremum is taken are represented by the \SupremumDot~vertices. The presence of a $\tlam$ connecting two variables is then represented by a standard edge \tlamline~between their vertices, and the presence of a $\tlamo$ connecting two variables is correspondingly represented by an edge \tlamoline~between their vertices. When using a supremum bound to split the diagrams, an integrated vertex gets split into two new vertices, precisely one integrated vertex and one supremum vertex. Each edge connected to such a split vertex can independently choose which of the new vertices to associate to. The calculations producing the bounds on $\OneNorm{\Psi^{(j)}}$ for $j=1,3,4$ can then be represented by:
\begin{align}
    \OneNorm{\Psi^{(1)}} \quad&=\quad \psioneintegral \quad\leq\quad \DiagramTriangleO \quad\times\quad \DiagramTriangle \\
    \OneNorm{\Psi^{(3)}} + \OneNorm{\Psi^{(4)}} \quad &= \quad \DiagramTriangleO \quad.
\end{align}

Our strategy for $\OneNorm{\Psi^{(2)}}$ is to use the spatial translation invariance to shift the origin to the bottom right vertex, and then use a supremum bound to split the diagram in two. Unfortunately, this spatial shift decouples the spatial and mark components of the vertices so - for example - at some vertices we end up taking suprema over marks whilst integrating over space. Let us introduce some more notation. The vertex \MarkSupremumDot~indicates that the spatial component is integrated over whilst the mark component is fixed and its supremum is taken after all the integrals. We also introduce the edge-vertex combination \supline~which indicates that there is no term connecting the associated vertices, but there is a supremum taken over the spatial displacement. Note that in principle we could use a vertex that had a spatial supremum and a mark integral, but we will always bound that probability integral by the supremum. The bound for $\OneNorm{\Psi^{(2)}}$ can therefore be expressed in the following:
\begin{equation}
\label{eqn:translate}
    \OneNorm{\Psi^{(2)}} = \quad\psitwointegralPtOne \quad \leq \quad\psitwointegralPtOneTranslate \quad \leq \quad \psitwointegralPtOneTranslatePtOne \quad \times \quad \DiagramTriangleClosed.
\end{equation}
To clarify the calculation for $\OneNorm{\Psi^{(2)}}$, we write out the integral here:
\begin{align}
    &\OneNorm{\Psi^{(2)}} = \esssup_{\sbar\in\Rd,a_1,a_2\in\Ecal}\int \tlamo(\wbar-\sbar;a_3,a_1)\tlamo(\tbar - \wbar; a_4,a_3)\tlam(\ubar-\tbar;a_6,a_4)\tlam(\zbar-\ubar;a_5,a_6) \nonumber\\
    &\hspace{5cm}\times\tlam(\tbar - \zbar; a_4,a_5)\tlam(\zbar; a_5,a_2) \dd \tbar \dd \ubar \dd \wbar \dd \zbar \Pcal^{\otimes 4}\left(\dd a_3, \dd a_4, \dd a_5, \dd a_6\right)\nonumber\\
    &\hspace{0.5cm}\leq \esssup_{\sbar\in\Rd,a_1,a_2,a_6\in\Ecal}\int \tlamo(\wbar'-\sbar+\ubar;a_3,a_1)\tlamo(\tbar' - \wbar'; a_4,a_3)\tlam(-\tbar';a_6,a_4)\tlam(\zbar';a_5,a_6) \nonumber\\
    &\hspace{2cm}\times\tlam(\tbar' - \zbar'; a_4,a_5)\tlam(\zbar' + \ubar; a_5,a_2) \dd \tbar' \dd \ubar \dd \wbar' \dd \zbar' \Pcal^{\otimes 3}\left(\dd a_3, \dd a_4, \dd a_5\right)\nonumber\\
    &\hspace{0.5cm}\leq \left(\esssup_{\sbar,\tbar',\zbar'\in\Rd,a_1,a_2,a_4,a_5\in\Ecal}\int \tlamo(\wbar'-\sbar+\ubar;a_3,a_1)\tlamo(\tbar' - \wbar'; a_4,a_3)\tlam(\zbar' + \ubar; a_5,a_2) \dd \ubar \dd \wbar' \Pcal\left(\dd a_3\right)\right)\nonumber\\
    &\hspace{2cm}\times\left(\esssup_{a_6\in\Ecal}\int \tlam(-\tbar';a_6,a_4)\tlam(\zbar';a_5,a_6)\tlam(\tbar' - \zbar'; a_4,a_5) \dd \tbar' \dd \zbar' \Pcal^{\otimes 2}\left(\dd a_4, \dd a_5\right)\right).
\end{align}
Then the two components of the first factor that are joined by the $\supline$~edge can be spatially translated together to form three consecutive edges. Note however, that the marks on either side of this join are not necessarily equal. We can therefore bound this first term by $\trilamooBar$. In summary, these manipulations then imply that $\OneNorm{\Psi^{(2)}}\leq \trilamooBar\trilam$. 

For $\Psi^{(1)}_0$ and $\Psi^{(2)}_0$ the calculations are identical to each other, and we can once again represent the calculations pictorially:
\begin{equation}
    \OneNorm{\Psi^{(1)}_0} = \OneNorm{\Psi^{(2)}_0} =\quad \psionenoughtintegral \quad\leq\quad \DiagramTriangle \quad = \trilam,
\end{equation}
where splitting the single supremum into two produces an upper bound. For the third case we have
\begin{equation}
    \OneNorm{\Psi^{(3)}} = \quad \psiZerothreeintegral \quad = \lambda \OneNorm{\Opconnf}.
\end{equation}

The diagrams representing the calculations for $\Psi_n$ are as follows:
\begin{align}
    \OneNorm{\Psi_n^{(1)}} & = \quad\psioneNintegral \quad = \quad \psioneNintegralFull \quad + \quad \psitwointegralPtThree \nonumber\\ &\hspace{4cm}\leq \quad \DiagramTriangle \quad\times\quad \psitwoNintegral \qquad + \quad \DiagramTriangle \quad \leq \trilam\trilamoo, \nonumber\\
    \OneNorm{\Psi^{(2)}_n} & = \quad\psitwoNintegral \quad \leq \trilamo.
\end{align}
In the first case we expanded a $\tlamo$ edge into a $\tlam$ edge and a contraction. Then in the contracted case we bound $\tlam\leq 1$ on the diagonal edge.

We have now proven the bounds for the single operator norms. For most of the operator pairs, we can then use the sub-multiplicity of $\OneNorm{\cdot}$ to immediately get sufficient bounds. The only pairs for which this strategy fails are: $\OneNorm{\Psi^{(4)}\Psi^{(4)}_0}$, $\OneNorm{\Psi^{(2)}_n\Psi^{(3)}_0}$, $\OneNorm{\Psi^{(4)}\Psi^{(4)}}$, and $\OneNorm{\Psi^{(2)}_n\Psi^{(4)}}$. We now deal with these cases by hand:
\begin{align}
    \OneNorm{\Psi^{(4)}\Psi^{(3)}_0} = \OneNorm{\Psi^{(2)}_n\Psi^{(3)}_0} &= \quad \PairThreeFive \quad \leq \trilam, \\
    \OneNorm{\Psi^{(4)}\Psi^{(4)}} = \OneNorm{\Psi^{(2)}_n\Psi^{(4)}} &= \quad\PairFiveFive\quad \leq \trilam.
\end{align}
For the first bound we used $\connf \leq \tlam$, and then split the single supremum into two. For the second bound we used $\tlam\leq 1$ on the diagonal edge.
\end{proof}

\LaceCoefficientBound*

\begin{proof}
Recall from Lemma~\ref{thm:DB:Pi_bound_Psi} that $\lambda\OneNorm{\OpLacelam^{(n)}} \leq \OneNorm{\Psi_n\Psi^{n-1}\Psi_0}$. We therefore aim to bound $\OneNorm{\Psi^{(j_n)}_n\Psi^{(j_{n-1})}\ldots \Psi^{(j_1)}\Psi^{(j_0)}_0}$ for $j_0\in\left\{1,2,3\right\}$, $j_1,\ldots,j_{n-1}\in\left\{1,2,3,4\right\}$, and $j_n\in\left\{1,2\right\}$. We have two cases depending upon the parity of $n$.

If $n$ is odd then we have an even number of operators. We pair them off, use the sub-multiplicity of $\OneNorm{\cdot}$, and Lemma~\ref{lem:singleandpairbounds} to get
\begin{equation}
    \OneNorm{\Psi^{(j_n)}_n\Psi^{(j_{n-1})}\ldots \Psi^{(j_1)}\Psi^{(j_0)}_0} \leq \left(\Vlam^2\right)^{\frac{n+1}{2}} = \Vlam^{n+1}.
\end{equation}
If $n$ is even, then we pair off the first $n$ operators and treat the last $\Psi^{(j_n)}_n$ operator by itself. We then get
\begin{equation}
    \OneNorm{\Psi^{(j_n)}_n\Psi^{(j_{n-1})}\ldots \Psi^{(j_1)}\Psi^{(j_0)}_0} \leq \Ulam \left(\Vlam^2\right)^{\frac{n}{2}} = \Ulam\Vlam^{n}.
\end{equation}

Since $\Ulam\geq \Vlam$, we have the bound $\Ulam\Vlam^{n}$ for all $n\geq 1$. Since there are $3$ choices for $j_0$, $4$ choices for each of $j_1,\ldots,j_{n-1}$, and $2$ choices for $j_n$, we get the pre-factor of $6\times4^{n-1}$ in our result by applying the triangle inequality.
\end{proof}

\StartEndBounds*

\begin{proof}
The argument for $\OneNorm{\Psi^{(j_m)}\ldots\Psi^{(j_1)}\Psi^{(j_0)}_0}$ and $\OneNorm{\Psi^{(j_n)}_n\Psi^{(j_{n-1})}\ldots\Psi^{(j_{m})}}$ is the same as that in Lemma~\ref{lem:singleandpairbounds} and Proposition~\ref{prop:LaceCoefficientBound}. We consider bounds for single operators and pairs of operators, and then use sub-multiplicity to get a bound for each diagram.

For $\OneNorm{\OpEndBlock^{(j_m)}\ldots\OpEndBlock^{(j_1)}\OpEndBlock^{(j_0)}_0}$ the argument is a little more complicated. For $m=0,1$ we can easily check each diagram satisfies the bound. For $m\geq 2$ we have the issue that there exists one pair of operators which we can't bound by $\Vlam^2$:
\begin{equation}
    \OneNorm{\OpEndBlock^{(4)}\OpEndBlock^{(2)}_0} = \quad \ProblemFiveTwo \quad \leq \quad \ProblemFiveTwoPtOne \quad \leq\trilamo \leq \Ulam^2.
\end{equation}
If $m\geq 2$ is even (an odd number of operators) then we use sub-multiplicity to pull off the $\OpEndBlock_0$ term which we bound with $\Ulam$. We then pull off pairs which we can bound with $\Vlam^2$ to get the bound $\Ulam\Vlam^m$. If $m\geq 2$ is odd (an even number of operators) then we use sub-multiplicity to pull off the $\OpEndBlock\OpEndBlock_0$ pair which we bound with $\Ulam^2$. We then pull off pairs which we can bound with $\Vlam^2$ to get the bound $\Ulam^2\Vlam^{m-1}$.
\end{proof}

\DisplacementNgeqTwo*

\begin{proof}[Proof of Proposition~\ref{thm:DisplacementNgeq2}]
In bounding the diagrams with one displacement segment, we have three broad cases. Let $i\in\left\{0,1,\ldots,n\right\}$ denote which segment the displacement lies upon. Our three cases are then:
\begin{enumerate}[label=\textbf{\arabic*)}]
    \item \label{Case_i=n} the displacement lies on the $\psi_n$ segment (i.e. $i=n$),
    \item \label{Case_i=0} the displacement lies on the $\psi_0$ segment (i.e. $i=0$),
    \item \label{Case_i=i} the displacement lies on a $\psi$ segment (i.e. $i\in\left\{1,\ldots,n-1\right\}$).
\end{enumerate}
We will aim to get a single bound for each diagram having a displacement crossing a single segment. This will be uniform in the sequence $\left(j_0,j_1,\ldots,j_n\right)$. We then account for the number of possible choices of $\left(j_0,j_1,\ldots,j_n\right)$, and finally include the factor of $(n+1)^2$ arising from using cosine-splitting (the explicit factor and the number of diagrams).

Case~\ref{Case_i=n} is the simplest. We use a supremum bound to pull off the displaced $\Psi_n$ term. We then bound this by hand and use Lemma~\ref{lem:StartEndBounds} to bound the remainder. In representing this scheme we use a shorthand $\Psi^{n-1}\Psi_0$ to denote a particular sequence $\Psi^{(j_{n-1})}\ldots\Psi^{(j_1)}\Psi^{(j_0)}_0$ rather than actually the operators $\Psi$ and $\Psi_0$ defined above. We also use labelled grey shapes to represent these groupings of diagrams, highlighting only end vertices or vertices that connect to neighbouring segments. This scheme look like:
\begin{align}
    &\CaseTwoFirst \quad +\quad \CaseTwoFirstPtTwo \nonumber\\ &\hspace{1cm}\leq \Ulam\Vlam^{n-1} \left(2\left( \quad \CaseTwoPtOne \quad + \quad \CaseTwoPtTwo\quad+ \quad \CaseTwoPtThree \quad\right) + \quad \DiagramWk\right)\nonumber\\
    &\hspace{1cm}\leq 2\Ulam\Vlam^{n-1}\left(\WkBar\trilamoo + \Wk \trilamo\right)\nonumber\\
    &\hspace{1cm}\leq 4\WkBar\Ulam^2\Vlam^{n-1}.
\end{align}
Note that this bound also holds if the displacement runs along the bottom of the $\Psi_n$ segment. Since there are $3\times 4^{n-1}$ possible sequences of indices for the $\Psi^{n-1}\Psi_0$ diagram, this case contributes that as a pre-factor.

For Case~\ref{Case_i=0}, we need a few more tricks. Since we are taking the displacement across the top of the diagram, the only term that will make a contribution is $\Psi^{(1)}_0$. To tackle this case we first use the observation of \cite{HeyHofLasMat19} that the diagram we want to bound can in turn be bounded by a diagram using $\OpEndBlock^{n-1}\OpEndBlock_0$. As above, we use $\OpEndBlock^{n-1}\OpEndBlock_0$ as a shorthand for a particular sequence $\OpEndBlock^{(j_{n-1})}\ldots\OpEndBlock^{(j_1)}\OpEndBlock^{(j_0)}_0$. We are also imprecise in this notation over whether it is this operator or its adjoint. Nevertheless the diagrams we use are clear on which vertices are integrated over and which have suprema, and this will avoid ambiguity. We then spatially translate the supremum to the far end and pull off the $\OpEndBlock^{n-2}\OpEndBlock_0$ diagram using a supremum bound. This looks like
\begin{multline}
\label{eqn:CaseOneFirst}
    \CaseOneFirst \quad\leq\quad \CaseOneFirstShifted \\\leq\quad \CaseOneFirstShiftedPtOne \quad\times\quad \CaseOneFirstShiftedPtTwo.
\end{multline}
For $\OpEndBlock^{(j)}$ with $j=1,2,3$, a simple supremum bound shows that this starting diagram can be bounded by $\WkBar\trilamo\OneNorm{\OpEndBlock^{(j)}}\leq  \WkBar\trilamo\trilam\trilamooBar \leq \WkBar\Vlam^2$. For $j=4$ a little more care is needed:
\begin{equation}
    \CaseOneFive \quad \leq \quad \CaseOneFivePtOne \quad + \quad \CaseOneFivePtTwo \quad \leq \WkBar\trilam\trilamoo \leq \WkBar\Vlam^2.
\end{equation}
We therefore find that 
\begin{equation}
    \CaseOneFirst\quad \leq \begin{cases}
    \WkBar\Vlam^{n-1}\Ulam^2 &: n\geq 3\\
    \WkBar\Vlam^2\Ulam &: n=2.
    \end{cases}
\end{equation}
Since there are $4$ choices for the appended $\OpEndBlock^{(j)}$ term, and $4^{n-2}\times2$ for the remaining $\OpEndBlock^{n-2}\OpEndBlock_0$ diagram, this case contributes $2\times 4^{n-1}$ as a pre-factor.

For Case~\ref{Case_i=i}, we shall find we have $3$ sub-cases, depending not only on the $j$-index of the displaced segment, but on which term we are looking at once it has been expanded. Sub-cases $\SubCaseOne$ and $\SubCaseTwo$ deal with $j=1,2,3$. In Appendix~\ref{appendix:DisplacedBounds} we expand out the various displaced segments using cosine-splitting and by expanding $\tlamo$ edges into $\tlam$ edges and points, and label each of the resulting diagrams as sub-case $\SubCaseOne$ and $\SubCaseTwo$. These sub-cases are distinguished by the strategy we use to bound them. Sub-case $\SubCaseThree$ then addresses the case with $j=4$.

For sub-case $\SubCaseOne$, we just use the $\Psi$ structure of the diagrams. We first split off segments from the left of the displaced term like we did in Case~\ref{Case_i=n}. If $i=n-1$ we can then bound the remaining two pair of segments with $\WkBar\Vlam^2$ (Appendix~\ref{appendix:DisplacedBounds} explains how this bound is found). This can be described diagrammatically as
\begin{multline}
\label{eqn:explanationcaseOnePtTwo}
    \ExplanationCaseOnePtTwo \quad \leq \Ulam\Vlam^{n-2} \quad \ExplanationCaseOnePtTwoPtOne \\ \leq \WkBar\Ulam\Vlam^{n}.
\end{multline}
If $i\leq n-2$, then we also split off the segments to the right of this pair, and then bound the pair by $\WkBar\Vlam^2$ again. Diagrammatically this looks like
\begin{multline}
\label{eqn:explanationcaseOnePtOne}
    \ExplanationCaseOnePtOne \quad \leq \Ulam^2\Vlam^{n-3} \quad \ExplanationCaseOnePtOnePtOne \\ \leq \WkBar\Ulam^2\Vlam^{n-1}.
\end{multline}
Note that the since $\Ulam\geq \Vlam$, the bound in \eqref{eqn:explanationcaseOnePtOne} is greater than or equal to the bound in \eqref{eqn:explanationcaseOnePtTwo}. It will be convenient to find an $i$-independent bound. To this end, since the $i\leq n-2$ case only exists for $n\geq 3$, we can bound an instance of a displaced diagram in Case~\ref{Case_i=i}, sub-case $\SubCaseOne$ with 
\begin{equation}
    \begin{cases}
    \WkBar\Ulam^2\Vlam^{n-1}&: n\geq 3\\
    \WkBar\Ulam\Vlam^{2}&: n=2,
    \end{cases}
\end{equation}
uniformly in $i\in\left\{1,\ldots,n-1\right\}$.

For sub-case $\SubCaseTwo$, we use a different strategy. First we split off the earlier (left) segments as we did for the first sub-case and in Case~\ref{Case_i=n}. However we arrange the later terms differently - more like we did in Case~\ref{Case_i=0}. We associate our displaced $\Psi$ term with $\tlam$ and $\tlamo$ edges from the subsequent term. If the subsequent term would be $\Psi^{(2)}_n$, these are actually two $\tlam$ edges, but in our schematic diagrams we will draw a $\tlamo$ edge and bear this special case in mind when it is relevant. While the displacement portrayed on the `top' of the segment could equally well be on the `bottom', the $\tlamo$ edge will always be attached to the bottom in the usual orientation. Everything to the right of these $\tlam$ and $\tlamo$ edges can then be described as a sequence of $\OpEndBlock$ and $\OpEndBlock_0$ terms. If $i=n-1$, then we pair our augmented displaced segment with the subsequent $\OpEndBlock_0$ segment. We then perform a spatial translation to make the right-most vertex a supremum vertex like we did in \eqref{eqn:CaseOneFirst}. The remaining diagram of a $\Psi^{(j)}$ segment and $\OpEndBlock_0$ segment connected by a $\tlamo$ and $\tlam$ edge can then be bounded by $\left(\WkBar\Vlam^2 + \HkBar\Ulam\right)$ (see Appendix~\ref{appendix:DisplacedBounds}). Recall that if $\OpEndBlock_0$ is in fact $\OpEndBlock^{(2)}_0$ then we have two $\tlam$ edges connecting the terms, not a $\tlam$ edge and a $\tlamo$ edge. We need to note this to get the $\left(\WkBar\Vlam^2 + \HkBar\Ulam\right)$ bound. The scheme is then represented diagrammatically as
\begin{multline}
    \ExplanationCaseTwoPtTwo \quad \leq \Ulam\Vlam^{n-2} \quad \ExplanationCaseTwoPtTwoPtOne \\ \leq \left(\WkBar\Vlam^2 + \HkBar\Ulam\right)\Ulam\Vlam^{n-2}.
\end{multline}
If $i\leq n-2$, we also perform a spatial translation to move a supremum to the right-most vertex, but then also split off all $\OpEndBlock$ and $\OpEndBlock_0$ segments after the $\OpEndBlock$ segment immediately following the displaced term. Once again the remaining pair can be bounded by $\left(\WkBar\Vlam^2 + \HkBar\Ulam\right)$ (see Appendix~\ref{appendix:DisplacedBounds}) and the split-off diagrams bounded using Lemma~\ref{lem:StartEndBounds}. The scheme can be represented diagrammatically as
\begin{align}
    &\ExplanationCaseTwoPtOne \nonumber\\
    &\hspace{6cm}\leq \quad \ExplanationCaseTwoPtOnePtOne \quad\times\begin{cases}
    \Ulam^2 \Vlam^{n-4}&: i\leq n-3 \\
    \Ulam \Vlam^{n-3}&: i=n-2
    \end{cases} \nonumber\\
    &\hspace{6cm}\leq \left(\WkBar\Vlam^2 + \HkBar\Ulam\right)\begin{cases}
    \Ulam^2 \Vlam^{n-4}&: i\leq n-3 \\
    \Ulam \Vlam^{n-3}&: i=n-2
    \end{cases}\label{eqn:SubCaseTwoExplanation}
\end{align}
Again for simplicity, we would like to get and $i$-independent bound for such diagrams. Since the $i = n-2$ case only exists for $n\geq 3$, and the $i\leq n-3$ case only exists for $n\geq 4$, we can bound an instance of a displaced diagram in Case~\ref{Case_i=i}, sub-case $\SubCaseTwo$ with 
\begin{equation}
    \begin{cases}
    \left(\WkBar\Vlam^2 + \HkBar\Ulam\right)\Ulam^3\Vlam^{n-4}&: n\geq 4\\
    \left(\WkBar\Vlam^2 + \HkBar\Ulam\right)\Ulam^2&: n=3\\
    \left(\WkBar\Vlam^2 + \HkBar\Ulam\right)\Ulam&: n=2,
    \end{cases}
\end{equation}
uniformly in $i\in\left\{1,\ldots,n-1\right\}$.

We now address sub-case $\SubCaseThree$, that is when the displacement crosses a $\Psi^{(4)}$ segment. This is the simplest sub-case because when we take a displacement across $\Psi^{(4)}$ in isolation, we get exactly
\begin{equation}
    \DisplacedFive \quad =\Wk.
\end{equation}
We will be able to use the $\Psi$-arrangement of segments here. For $n\geq 4$ we split off the earlier and later segments without pairing the displaced term with any of them. With Lemma~\ref{lem:StartEndBounds} this produces the bound
\begin{equation}
    \left(\Ulam\Vlam^{i-1}\right)\Wk\left(\Ulam\Vlam^{n-i-1}\right) = \WkBar\Ulam^2\Vlam^{n-2} \qquad:n\geq 4.
\end{equation}
For $n=3$ we split off two segments from one side and one from the other. This gives the finer bound
\begin{equation}
    \Wk\Ulam\Vlam^2 \quad:n=3.
\end{equation}
For $n=2$, we use $j_0$ to indicate which $\Psi_0$ term we have and $j_2$ to denote which $\Psi_n$ term we have. We note that if $j_0=1,2$ then $\OneNorm{\Psi^{(j_0)}_0}\leq \trilam\trilamoo \leq \Vlam$, and if $j_2 = 1$ then $\OneNorm{\Psi^{(j_2)}_n}\leq \trilam\trilamoo \leq \Vlam$. Therefore if $\left(j_0,j_2\right) \ne \left(3,2\right)$ we have the bound $\Wk\Ulam\Vlam$. If $\left(j_0,j_2\right) = \left(3,2\right)$ we proceed by hand. We bound $\connf \leq \tlam$ and spatially shift one vertex to show
\begin{equation}
    \ExtraThreeFiveTwo \quad \leq \quad \ExtraThreeFiveTwoPtTwo \quad\leq \WkBar\trilam \leq \WkBar \Vlam.
\end{equation}

Carefully comparing the above bounds gives the required bounds for $n\leq 2$. Including the multiplicity of each diagram (including that arising from cosine-splitting), we find that we can bound the displaced and subsequent non-displaced pair by $\left(24\WkBar\Vlam^2 + 2\HkBar\right)$. Note that this includes summing over the $j$-index of the displaced term but not the index of the subsequent term. Since there are $3\times4^{n-2}\times2$ options for the $j$-indices of the other segments, this case contributes this to the counting pre-factors in our result.
\end{proof}

\DisplacementNeqOne*

\begin{proof}[Proof of Proposition~\ref{thm:DisplacementNeq1}]
For $n=1$, we can consider each of the six cases by hand. Most of the diagrams can be dealt with routinely using cosine-splitting, supremum bounds, and spatial shifts. We enumerate the diagrams by their $\left(j_0,j_1\right)$ values corresponding to the constituent segments.
\begin{eqnarray*}
    (1,1):& \quad \ExtraOneOne \quad &\leq 3\left(\WkBar\trilamo\trilam + \trilam\WkBar\trilam + \trilam\trilamo\Wk\right)\\
    (1,2):& \quad \ExtraOneTwo \quad &\leq 2\left(\WkBar\trilam + \trilam\Wk\right)\\
    (2,1):& \quad \ExtraTwoOne \quad &\leq 2\left(\trilam\WkBar\trilamo + \trilam\trilamo\Wk\right)\\
    (2,2):& \quad \ExtraTwoTwo \quad &\leq \trilam\Wk\\
    (3,1):& \quad \ExtraThreeOne \quad &\leq 2\left(\lambda\OneNorm{\Opconnf}\WkBar\trilam + \WkBar\trilam + \lambda\OneNorm{\Opconnf}\trilam\Wk + \trilam\Wk\right)
\end{eqnarray*}
These terms sum to no more than
\begin{equation}
    13\WkBar\Vlam\Ulam.
\end{equation}

For the $\left(3,2\right)$ diagram, we use a different `trick.' In Lemma~\ref{lem:tau-phi-Expansion} we will prove the bound $\tlam\left(x,y\right) \leq \connf(x,y) + \lambda\int \tlam(x,u)\connf(u,y)\nu\left(\dd u\right)$. When applied to this diagram, this looks like
\begin{equation}
    \CaseThreeTwo \quad \leq \quad \ThreeTwoDiagramPtOne \quad + \quad 2\left(\quad \ThreeTwoDiagramPtTwo \quad + \quad \ThreeTwoDiagramPtThree \quad\right),
\end{equation}
Here we have also used the cosine-splitting result on the second part. For the first of these $4$-vertex diagrams, we split the supremum in two and get the bound $\lambda^3\InfNorm{\Opconnf\Optlam\Optlam\Opconnf_k} \leq \lambda\trilam \OneNorm{\Opconnf_k}$. For the second of the $4$-vertex diagrams, we also split the supremum in two and get the bound $\lambda^3\InfNorm{\Opconnf\Optlam\Optklam\Opconnf}$. For the $3$-vertex diagram we condition on whether two of the vertices are in each other's $B$-set (defined in Assumption~\ref{Assump:BallDecay}). In the following diagrams \tlamBline~indicates that an indicator of the form $\mathds{1}_{\left\{w \in B\left(u\right)\right\}}$ relates two variables, and \NOTtlamBline~indicates that an indicator of the form $\mathds{1}_{\left\{w \not\in B\left(u\right)\right\}}$ relates them:
\begin{align}
    \ThreeTwoDiagramPtOne \quad &= \quad \ThreeTwoDiagramPtFour \quad + \quad \ThreeTwoDiagramPtFive  \nonumber\\ &\leq \quad \ThreeTwoDiagramPtSix \quad + 2\left(\quad \ThreeTwoDiagramPtSeven \quad + \quad\ThreeTwoDiagramPtEight\quad\right) \nonumber\\
    &\leq \left(\quad\ThreeTwoDiagramPtNine \quad \times \quad \psiNrighttwointegral\quad\right) +4\left(\quad\psiZeroleftthreeintegralPtOne\quad\times \quad \DiagramWk\quad\right)\nonumber\\
    & \leq \lambda\OneNorm{\Opconnf_k}\trilamB + 4\lambda\mathbb{B}W_k.
\end{align}

\end{proof}

\subsection{Displacement Diagram Case Distinctions}
\label{appendix:DisplacedBounds}

This section provides the details omitted from the proof of Proposition~\ref{thm:DisplacementNgeq2}. Here we show which diagrams arise from displacements crossing a $\Psi$ segment, and whether they can be dealt with by the first sub-case, $\SubCaseOne$, or by the second sub-case, $\SubCaseTwo$. Recall that $\SubCaseOne$ are those that use a $\Psi$ decomposition of the larger diagram only (and the displacement is not crossing a $\Psi^{(4)}$ segment), while $\SubCaseTwo$ are those that will require a combination of a $\Psi$ decomposition and a $\OpEndBlock$ decomposition.

We first expand out the displaced $\Psi$ terms and label which are in sub-case $\SubCaseOne$ and which are in sub-case $\SubCaseTwo$. Note that the displacement may be across the `top' or the `bottom' of the segments in the usual orientation. If we consider a displacement across the `top' or the `bottom' of a $\Psi^{(1)}$ term, we use cosine-splitting and sometimes expand $\tlamo$ edges into a $\tlam$ edge and a contraction to get
\begin{align}
    \DisplacementPsiOneTop \quad &\leq 2\left(\quad \underbrace{\DisplacementPsiOneTopPtOne}_{\SubCaseTwo} \quad + \quad \underbrace{\DisplacementPsiOneTopPtTwo}_{\SubCaseTwo} \quad\right)\\
    \DisplacementPsiOneBottom \quad &\leq 2\left(\quad \underbrace{\DisplacementPsiOneBottomPtOne}_{\SubCaseOne} \quad + \quad \underbrace{\DisplacementPsiOneBottomPtTwo}_{\SubCaseTwo} \quad + \quad \underbrace{\DisplacementPsiOneBottomPtThree}_{\SubCaseTwo}\quad\right).
\end{align}
Using the same ideas, if we consider a displacement across the `top' or the `bottom' of a $\Psi^{(2)}$ term we get
\begin{align}
    \DisplacementPsiTwoTop \quad &\leq \quad \underbrace{\DisplacementPsiTwoTopPtOne}_{\SubCaseTwo} \quad + \quad \underbrace{\DisplacementPsiTwoTopPtTwo}_{\SubCaseOne}\\
    \DisplacementPsiTwoBottom \quad &\leq 2\left(\quad \underbrace{\DisplacementPsiTwoBottomPtOne}_{\SubCaseTwo} \quad+\quad 2\times\underbrace{\DisplacementPsiTwoBottomPtTwo}_{\SubCaseOne} \quad+\quad \underbrace{\DisplacementPsiTwoBottomPtThree}_{\SubCaseOne}\right.\nonumber\\
    &\hspace{1cm}\left.\underbrace{\DisplacementPsiTwoBottomPtFour}_{\SubCaseTwo} \quad+\quad 2\times\underbrace{\DisplacementPsiTwoBottomPtFive}_{\SubCaseOne} \quad+\quad \underbrace{\DisplacementPsiTwoBottomPtSix}_{\SubCaseOne}\quad\right). 
\end{align}
If we consider a displacement across the `top' or the `bottom' of a $\Psi^{(3)}$ term, then by symmetry in both cases we get
\begin{equation}
    \underbrace{\DisplacementPsiFourTop}_{\SubCaseTwo}.
\end{equation}

We first show the bounds for the single segments in sub-case $\SubCaseOne$. Recall from the proof of Proposition~\ref{thm:DisplacementNgeq2} that in this sub-case the aim is prove that the schematic diagrams can be bounded:
\begin{equation}
    \ExplanationCaseOnePtTwoPtOne \quad\vee\quad \ExplanationCaseOnePtOnePtOne \quad\leq \WkBar\Vlam^2.
\end{equation}
We first consider the displaced $\Psi^{(j)}$ term by itself. By using supremum bounds and spatial translations we arrive at
\begin{eqnarray*}
    \DisplacementPsiOneBottomPtOne \quad\leq \Wk\trilam, & \DisplacementPsiTwoTopPtTwo \quad \leq \WkBar\trilam, & \DisplacementPsiTwoBottomPtTwo \quad \leq \WkBar\trilam, \\
    \DisplacementPsiTwoBottomPtThree \quad \leq \Wk\trilamo, & \DisplacementPsiTwoBottomPtFive \quad \leq \Wk\trilam, & \DisplacementPsiTwoBottomPtSix \quad \leq \Wk\trilamo.
\end{eqnarray*}
For most of these, it is sufficient to separate the pair using a supremum bound and use the individual bound above for the displaced $\Psi^{(j)}$ term and the bound on the $\Psi$ or $\Psi_n$ term from Lemma~\ref{lem:singleandpairbounds}. The only pairs that need a more careful approach are those that have the bound $\Wk\trilamo$ above for the displaced $\Psi^{(j)}$, followed by a $\Psi^{(4)}$ or $\Psi^{(2)}_n$ term. In these cases we consider the pair together and get the diagram bounds
\begin{equation*}
    \DisplacedTwoPtTwoPlusFive \quad \leq \Wk\trilam \leq \Wk\Vlam^2,\qquad
    \DisplacedTwoPtFourPlusFive \quad \leq \Wk\trilam \leq \Wk\Vlam^2.
\end{equation*}
In both of these cases we bounded a $\tlam$ edge by $1$ (essentially omitting the edge) before applying a supremum bound. In the first we omit the diagonal edge and split the upper supremum vertex in two. In the second we omit the left-most vertical edge and apply a supremum bound at the bottom right vertex.

We now turn our attention to sub-case $\SubCaseTwo$. Recall from the proof of Proposition~\ref{thm:DisplacementNgeq2} that in this sub-case the aim is prove that the schematic diagrams can be bounded:
\begin{equation}
    \ExplanationCaseTwoPtTwoPtOne \quad\vee\quad \ExplanationCaseTwoPtOnePtOne \quad\leq \WkBar\Vlam^2 + \HkBar\Ulam.
\end{equation}
Also recall that if the $\OpEndBlock_0$ term is actually a $\OpEndBlock^{(2)}_0$ term, then the $\tlamo$ edge connecting them is actually just a $\tlam$ edge. We first investigate bounds for the displaced $\Psi^{(j)}$ term with the connecting $\tlam$ and $\tlamo$ edges. By using supremum bounds, spatial translations, splitting supremum vertices in two, and sometimes expanding $\tlamo$ edges into a $\tlam$ edge and a contraction, we get
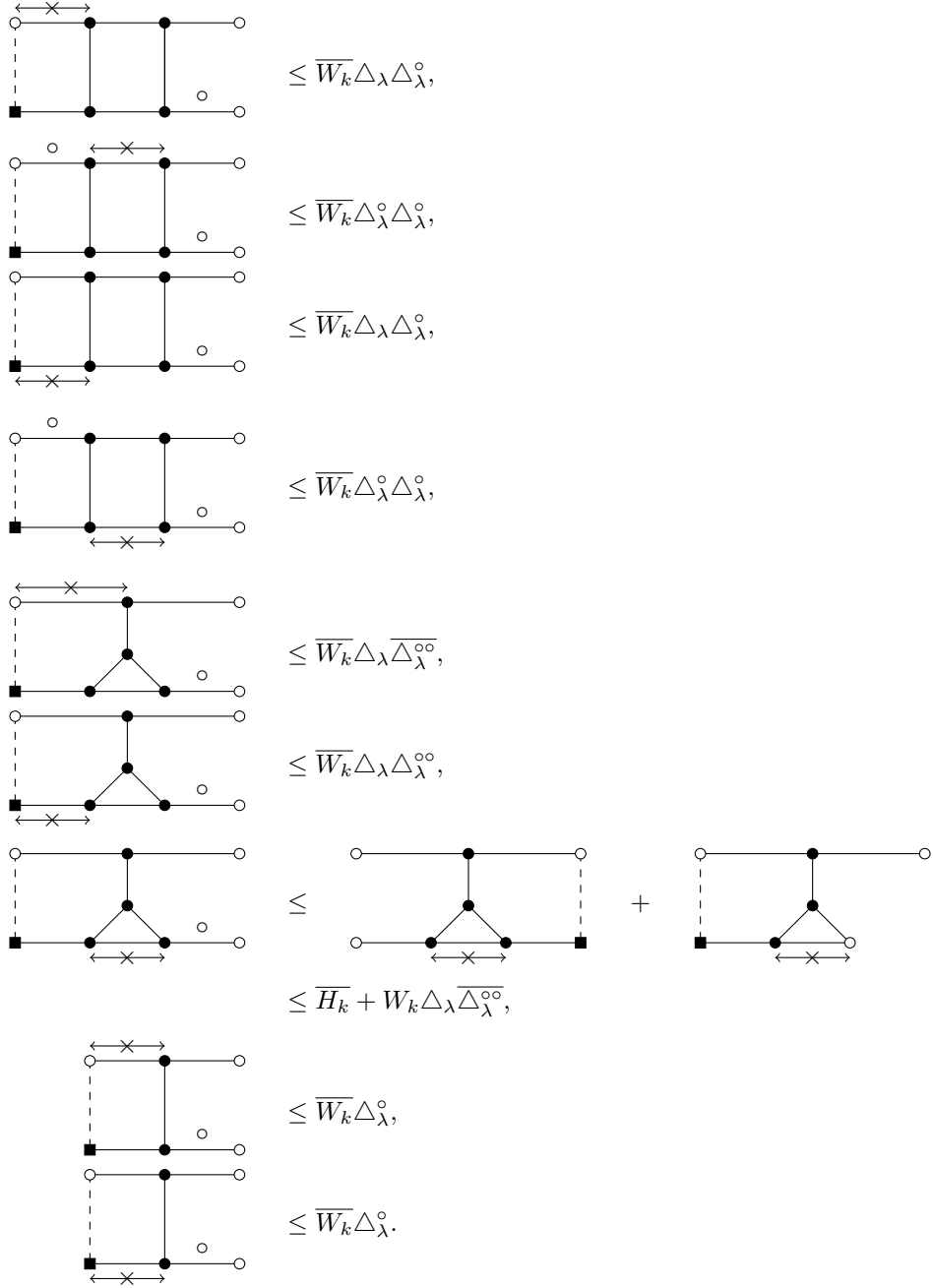
\begin{align*}
    \DisplacementPsiOneTopPtOneAug \quad &\leq \WkBar\trilam\trilamo,\\
    \DisplacementPsiOneTopPtTwoAug \quad &\leq \WkBar\trilamo\trilamo,\\
    \DisplacementPsiOneBottomPtTwoAug \quad &\leq \WkBar\trilam\trilamo,\\
    \DisplacementPsiOneBottomPtThreeAug \quad &\leq \WkBar\trilamo\trilamo,\\
    \DisplacementPsiTwoTopPtOneAug \quad &\leq \WkBar\trilam\trilamooBar,\\
    \DisplacementPsiTwoBottomPtOneAug \quad &\leq \WkBar\trilam\trilamoo,\\
    \DisplacementPsiTwoBottomPtFourAug \quad &\leq \quad \DisplacedTwoPtFivePtOne \quad+\quad \DisplacedTwoPtFivePtTwo \nonumber\\
    &\leq \HkBar + \Wk\trilam\trilamooBar,\\
    \DisplacementPsiFourTopAugVOne \quad &\leq \WkBar\trilamo,\\
    \DisplacementPsiFourTopAugVTwo \quad &\leq \WkBar\trilamo.
\end{align*}
For most of these, it is clear from a crude supremum bound that we have the bound $\WkBar\Vlam^2 + \HkBar\Ulam$ for the pair of these with the subsequent $\OpEndBlock$ or $\OpEndBlock_0$ term (like we did for sub-case $\SubCaseOne$). The only pairs that need a more careful approach are those that have the bound $\WkBar\trilamo$ or $\WkBar\trilamo\trilamo$ above, followed by a $\OpEndBlock^{(4)}$ or $\OpEndBlock^{(2)}_0$ term. We first consider the cases where they are followed by a $\OpEndBlock^{(4)}$ term. Here we use supremum bounds, spatial translations, and expanding $\tlamo$ edges into a $\tlam$ edge and a contraction to get
\begin{align*}
    \DisplacementPsiOneTopPtTwoAugPlusFive \quad &\leq \trilamo\WkBar\trilam\trilamo + \trilamo\WkBar\trilam = \WkBar\trilam\trilamo\trilamoo \leq \WkBar\Vlam^2,\\
    \DisplacementPsiOneBottomPtThreeAugPlusFive \quad &\leq \WkBar\trilam\trilamo\trilamoo\leq \WkBar\Vlam^2,\\
    \DisplacementPsiFourTopAugPlusFiveVOne \quad &\leq \WkBar\trilam\trilamo + \WkBar\trilam = \WkBar\trilam\trilamooBar\leq \WkBar\Vlam^2,\\
    \DisplacementPsiFourTopAugPlusFiveVTwo \quad &\leq \WkBar\trilam\trilamooBar\leq \WkBar\Vlam^2.
\end{align*}
We now consider the cases where they are followed by a $\OpEndBlock^{(2)}_0$ term. Recall that for this case, the connecting $\tlamo$ edge is in fact only a $\tlam$ edge. We then get
\begin{align*}
    \DisplacementPsiOneTopPtTwoAugPlusTwo \quad &\leq \WkBar\trilam\trilamo\leq \WkBar\Vlam^2,\\
    \DisplacementPsiOneBottomPtThreeAugPlusTwo \quad &\leq \WkBar\trilam\trilamo\leq \WkBar\Vlam^2,\\
    \DisplacementPsiFourTopAugPlusTwoVOne \quad &\leq \WkBar\trilam\leq \WkBar\Vlam^2,\\
    \DisplacementPsiFourTopAugPlusTwoVTwo \quad &\leq \WkBar\trilam\leq \WkBar\Vlam^2.
\end{align*}

The remaining sub-case $\SubCaseThree$ was fully dealt with in the main text.

\section{Bounding diagrams with the Bootstrap Function}
\label{appendix:BoundingwithBootstrap}

\begin{lemma}
    \label{lem:tau-phi-Expansion}
Let $x,y \in \X$. Then
\begin{align}
\label{eqn:tau-phi-Expansionfunction}
    \tlam(x,y) &\leq \connf(x,y) + \lambda \int \connf(x,u)\tlam(u,y)\nu\left(\dd u\right),\\
\label{eqn:tau-phi-Expansionfunction_2}
    \tlam(x,y) &\leq \connf(x,y) + \lambda \int \tlam(x,u)\connf(u,y)\nu\left(\dd u\right).
\end{align}
Then for $n\geq 2$,
\begin{equation}
\label{eqn:tau-phi-Expansion}
    \InfNorm{\Optlam^n} \leq \sum_{m=0}^n\genfrac(){0pt}{0}{n}{m}\lambda^{m} \InfNorm{\Opconnf^{n-1}\Optlam^m\Opconnf}.
\end{equation}
\end{lemma}

\begin{proof}
By combining Mecke's formula and the BK inequality, we obtain
\begin{align}
    \tlam(x,y) & \leq  \connf(x,y) + \E_\lambda\left[\sum_{u \in \eta} \mathds 1_{\{x \sim u \text{ in } \xi^{x} \} \cap \{\conn{u}{y}{\xi^{y}}\}} \right] \nonumber\\
		& = \connf(x,y) + \lambda \int \pla\left(\{x \sim u \text{ in } \xi^{x,u} \} \cap \{\conn{u}{y}{\xi^{y,u}}\}\right) \nu\left(\dd u\right) \nonumber\\
		& = \connf(x,y) +\lambda \int \connf(x,u) \tlam(u,y) \nu\left(\dd u\right).
\end{align}
In the last line, we have used that the two intersected events are independent. This is due to the fact that $\nu$ is non-atomic and $x\notin \eta^{y,u}$ a.s. 

The second inequality is then proven nearly identically, but with the adjacency to the intermediate point $u$ holding for $y$ instead of $x$.

To get \eqref{eqn:tau-phi-Expansion}, we first use \eqref{eqn:tau-phi-Expansionfunction_2} to expand out the right-most $\Optlam$. This produces one term with $\Opconnf$ replacing the $\Optlam$, and one term with $\lambda\Optlam\Opconnf$ replacing it. We then repeatedly use \eqref{eqn:tau-phi-Expansionfunction} to expand the left-most $\Optlam$ to get a $\Opconnf$ factor and a $\lambda\Opconnf\Optlam$ factor until we have a total of $n$ factors of $\Opconnf$ in the expression. Counting the number of ways to get to each possible end term then gives the binomial $\genfrac(){0pt}{1}{n}{m}$ factor as required.
\end{proof}

\begin{lemma}\label{lem:ApproxPointWithL2}
Let $H\colon L^2\left(\Ecal\right) \to L^2\left(\Ecal\right)$ be an integral linear operator with kernel function $h\colon\Ecal^2\to\R$. Then for all $h\left(a,b\right)\in \EssIm{h}$  and $\varepsilon>0$, there exist $\Pcal$-positive sets $E_1,E_2\subset\Ecal$ such that $g_i(c) := \frac{1}{\Pcal\left(E_i\right)}\Id_{E_i}\left(c\right)$ for $i=1,2$ satisfy
\begin{equation}
    \abs*{h\left(a,b\right) - \inner{g_1}{Hg_2}} \leq \varepsilon.
\end{equation}
\end{lemma}

\begin{proof}
Since $h$ is measurable, the set $A=\left\{c,c'\in\Ecal:\abs*{h\left(c,c'\right)-h\left(a,b\right)}\leq \varepsilon\right\}$ is measurable. Since $h\left(a,b\right)\in \EssIm{h}$, it has positive measure: $\Pcal\left(A\right)>0$. Furthermore, by considering the $\pi$-system of product sets $\left\{E_1\times E_2: E_1,E_2\subset \Ecal \text{ measurable}\right\}$ that generates the $\sigma$-algebra on $\Ecal^2$, we can find $B = E_1\times E_2$ such that $E_1,E_2\subset \Ecal$ are both measurable, $B\subset A$, and $\Pcal^{\otimes 2}\left(B\right) = \Pcal\left(E_1\right)\Pcal\left(E_2\right)>0$. Then use $E_1,E_2$ to define $g_1,g_2$ as in the statement on the lemma. Since $E_1,E_2$ are both measurable and have positive measure, both $g_1,g_2\in L^2\left(\Ecal\right)$. We also have
\begin{equation}
    \inner{g_1}{Hg_2} = \frac{1}{\Pcal^{\otimes 2}\left(B\right)}\int_B h\left(c,c'\right)\Pcal^{\otimes 2}\left(\dd c,\dd c'\right) \in \left[h(a,b) - \varepsilon,h(a,b) +\varepsilon\right],
\end{equation}
proving the result.
\end{proof}

\begin{lemma}
\label{lem:phi-tau-phi_bounds}
Let $\lambda<\lambda_O$, $n\geq 2$, and $d>2m$. Then there exists finite $\const= \const(m,n,f(\lambda))$ (increasing in $f(\lambda)$ and independent of $d$) such that
\begin{equation}
    \InfNorm{\Opconnf^{n-1}\Optlam^m\Opconnf} \leq \begin{cases}
    \const &: n=2\\
    \const g(d)^{\frac{1}{2}-m\overline{\epsilon}(d)} &: n\geq 3.
    \end{cases}
\end{equation}
\end{lemma}

The condition $\lambda < \lambda_O$ here is required to allow us to swap integrals over the mark space with integrals over the Fourier space. We are able to use Fubini's Theorem because sub-criticality implies that the relevant integrals are finite.

\begin{proof}
For clarity, let us introduce the notation
\begin{equation}
\label{eqn:phi^n_tau^n_notation}
    \connf^{(n)}(x,y) = \int\prod^{n}_{j=1}\connf(u_{j-1},u_{j})\nu^{\otimes (n-1)}\left(\dd u_{[1,n-1]}\right),\quad \tlam^{(n)}(x,y) = \int\prod^{n}_{j=1}\tlam(u_{j-1},u_{j})\nu^{\otimes (n-1)}\left(\dd u_{[1,n-1]}\right),
\end{equation}
where $u_0=x$ and $u_n=y$. Furthermore, for each $\xbar\in\Rd$ let us define the integral linear operator $\left(\Opconnf^{n-1}\Optlam^m\Opconnf\right)\left(\xbar\right) \colon L^2\left(\Ecal\right) \to L^2\left(\Ecal\right)$ as the integral operator with kernel function 
\begin{equation}
(a,b)\mapsto\left(\Opconnf^{n-1}\Optlam^m\Opconnf\right)\left(\xbar;a,b\right) := \int\connf^{(n-1)}\left(\xbar-\ubar;a,c\right)\tlam^{(m)}\left(\ubar-\ubar';c,c'\right)\connf\left(\ubar';c',b\right)\dd \ubar\dd\ubar'\Pcal^{\otimes 2}\left(\dd c,\dd c'\right).
\end{equation}
Given $\delta>0$, $\xbar\in\Rd$, and $a,b\in\Ecal$ such that $\left(\Opconnf^{n-1}\Optlam^m\Opconnf\right)\left(\xbar;a,b\right)\in \EssIm{\left(\Opconnf^{n-1}\Optlam^m\Opconnf\right)\left(\xbar;\cdot,\cdot\right)}$, we can use Lemma~\ref{lem:ApproxPointWithL2} to get measurable and $\Pcal$-positive sets $E^{(\xbar,a,b,\delta)}_1,E^{(\xbar,a,b,\delta)}_2\subset \Ecal$ such that
\begin{equation}
    \left(\Opconnf^{n-1}\Optlam^m\Opconnf\right)\left(\xbar;a,b\right) \leq \inner*{g^{(\xbar,a,b,\delta)}_1}{\left(\Opconnf^{n-1}\Optlam^m\Opconnf\right)\left(\xbar\right) g^{(\xbar,a,b,\delta)}_2} + \delta,
\end{equation}
where $g^{(\xbar,a,b,\delta)}_i(c) = \frac{1}{\Pcal\left(E^{(\xbar,a,b,\delta)}_i\right)}\Id_{E^{(\xbar,a,b,\delta)}_i}\left(c\right)$ for $i=1,2$.

Now let $E_1,E_2\subset \Ecal$ be arbitrary measurable $\Pcal$-positive sets and $g_1,g_2$ be their associated functions. We aim to produce bounds that are independent of the choice of $E_1,E_2$. We first use the Fourier inversion theorem to reformulate the position behaviour in terms of a $k$-integral of the Fourier transforms, and use \ref{Assump:2ndMoment} and $\lambda<\lambda_O$ (so the relevant integrals are finite) to allow us to swap the $k$-integral and the mark integrals. This produces
\begin{multline}
    \abs*{\inner*{g_1}{\left(\Opconnf^{n-1}\Optlam^m\Opconnf\right)\left(\xbar\right) g_2}} = \abs*{\int\e^{-i\xbar\cdot k}\inner*{g_1}{\fOpconnf(k)^{n-1}\fOptlam(k)^m\fOpconnf(k)g_2}\frac{\dd k}{\left(2\pi\right)^d}}\\ \leq \int\abs*{\inner*{g_1}{\fOpconnf(k)^{n-1}\fOptlam(k)^m\fOpconnf(k)g_2}}\frac{\dd k}{\left(2\pi\right)^d}.
\end{multline}
Recall that $\fOpconnf(k)$ is self-adjoint, so we can move $\fOpconnf(k)^{n-1}$ to the other side of the inner product, and then we can use Cauchy-Schwarz and the definition of the operator norm to get
\begin{multline}
\label{eqn:ExtractNormT}
    \inner*{g_1}{\fOpconnf(k)^{n-1}\fOptlam(k)^m\fOpconnf(k)g_2} = \inner*{\fOpconnf(k)^{n-1}g_1}{\fOptlam(k)^m\fOpconnf(k)g_2} \\\leq \norm*{\fOpconnf(k)^{n-1}g_1}_2\norm*{\fOptlam(k)^m\fOpconnf(k)g_2}_2 \leq \OpNorm{\fOptlam(k)}^m\norm*{\fOpconnf(k)^{n-1}g_1}_2\norm*{\fOpconnf(k)g_2}_2.
\end{multline}
We can then use the definition of the bootstrap function $f(\lambda)$ to replace the $\OpNorm{\fOptlam(k)}$ factors with factors of $f(\lambda)$ and $\fgmu(k)$. Also writing the norm $\norm*{\cdot}_2$ as the square root of an inner product and using the self-adjoint property of $\fOpconnf(k)$ gives
\begin{equation}
    \inner*{g_1}{\fOpconnf(k)^{n-1}\fOptlam(k)^m\fOpconnf(k)g_2} \leq f(\lambda)^m\fgmu(k)^m\inner*{g_1}{\fOpconnf(k)^{2n-2}g_1}^\frac{1}{2}\inner*{g_2}{\fOpconnf(k)^2g_2}^\frac{1}{2}.
\end{equation}
We therefore want to derive bounds for the integral
\begin{equation}\label{eqn:fullkintegral}
    \int\fgmu(k)^m\inner*{g_1}{\fOpconnf(k)^{2n-2}g_1}^\frac{1}{2}\inner*{g_2}{\fOpconnf(k)^2g_2}^\frac{1}{2}\frac{\dd k}{\left(2\pi\right)^d}.
\end{equation}

We first derive pointwise and integral bounds for the $\fOpconnf$ factors. Given $N\geq 1$, and $g(c)=\frac{1}{\Pcal\left(E\right)}\Id_E(c)$ for some measurable and $\Pcal$-positive set $E\subset \Ecal$,
\begin{equation}
    \inner*{g}{\fOpconnf(k)^{2N}g}^\frac{1}{2} = \left(\frac{1}{\Pcal(E)^2}\int_E\int_E\fconnf^{(2N)}(k;c,c')\Pcal(\dd c)\Pcal(\dd c')\right)^\frac{1}{2} \leq \InfNorm{\fOpconnf(k)^{2N}}^\frac{1}{2}.
\end{equation}
By using Cauchy-Schwarz and supremum bounds we can split this into norms of $\fOpconnf(k)$ only, then use Lemma~\ref{lem:NormBounds} to replace $k$ with $0$, and finally use \ref{Assump:2ndMoment} to bound these terms:
\begin{equation}
\label{eqn:phi^k_pointwise}
    \inner*{g}{\fOpconnf(k)^{2N}g}^\frac{1}{2}\leq \InfNorm{\fOpconnf(k)^{2N}}^\frac{1}{2} \leq \TwoNorm{\fOpconnf(k)}\OneNorm{\fOpconnf(k)}^{N-1} \leq \TwoNorm{\fOpconnf(0)}\OneNorm{\fOpconnf(0)}^{N-1} \leq C^N.
\end{equation}
For the integral bound, we apply Cauchy-Schwarz to the $k$-integral and swap the $k$-integrals with the mark integrals to un-do the Fourier transforms. We get
\begin{multline}
\label{eqn:IntboundN=1Begin}
    \int\inner*{g_1}{\fOpconnf(k)^{2n-2}g_1}^\frac{1}{2}\inner*{g_2}{\fOpconnf(k)^2g_2}^\frac{1}{2}\frac{\dd k}{\left(2\pi\right)^d} \leq \left(\int\inner*{g_1}{\fOpconnf(k)^{2n-2}g_1}\frac{\dd k}{\left(2\pi\right)^d}\right)^\frac{1}{2}\left(\int\inner*{g_2}{\fOpconnf(k)^2g_2}\frac{\dd k}{\left(2\pi\right)^d}\right)^\frac{1}{2} \\= \inner*{g_1}{\Opconnf^{(2n-2)}\left(\zerobar\right)g_1}^\frac{1}{2}\inner*{g_2}{\Opconnf^{(2)}\left(\zerobar\right)g_2}^\frac{1}{2},
\end{multline}
where $\Opconnf^{(N)}\left(\zerobar\right)\colon L^2\left(\Ecal\right) \to L^2\left(\Ecal\right)$ is the integral operator with kernel function $\left(a,b\right)\mapsto \connf^{(N)}\left(\zerobar;a,b\right)$. Given $N\geq 1$, and $g(c)=\frac{1}{\Pcal\left(E\right)}\Id_E(c)$ for some measurable and $\Pcal$-positive set $E\subset \Ecal$,
\begin{equation}
    \inner*{g}{\Opconnf^{(2N)}\left(\zerobar\right)g} = \left(\frac{1}{\Pcal(E)^2}\int_E\int_E\connf^{(2N)}\left(\zerobar;c,c'\right)\Pcal(\dd c)\Pcal(\dd c')\right) \leq \InfNorm{\Opconnf^{(2N)}\left(\zerobar\right)}.
\end{equation}
Here for $N\geq 2$ we bound with $\InfNorm{\Opconnf^{(2N)}\left(\zerobar\right)}^\frac{1}{2} \leq \InfNorm{\Opconnf^{2N}}^\frac{1}{2} \leq g(d)^\frac{1}{2}$ from \ref{Assump:BallDecay}. For $N=1$ we bound using $\connf\left(\xbar;a,b\right)\in\left[0,1\right]$ and \ref{Assump:2ndMoment}:
\begin{multline}
\label{eqn:IntboundN=1End}
    \InfNorm{\Opconnf^{(2)}\left(\zerobar\right)} = \esssup_{a,b\in\Ecal}\int \connf\left(\xbar;a,c\right)\connf\left(-\xbar;c,b\right)\dd \xbar \Pcal\left(\dd c\right) \leq \esssup_{b\in\Ecal}\int \connf\left(\xbar;a,b\right)\dd \xbar \Pcal\left(\dd a\right) \\= \OneNorm{\fOpconnf(0)}\leq C.
\end{multline}
Therefore 
\begin{equation}
\label{eqn:phi^k_integral}
    \int\inner*{g_1}{\fOpconnf(k)^{2n-2}g_1}^\frac{1}{2}\inner*{g_2}{\fOpconnf(k)^2g_2}^\frac{1}{2}\frac{\dd k}{\left(2\pi\right)^d} \leq \begin{cases}
    C &: n=2,\\
    C^\frac{1}{2}g(d)^\frac{1}{2} &: n\geq 3.
    \end{cases}
\end{equation}

We now return to bounding \eqref{eqn:fullkintegral}. To perform the $k$-integral, we split the domain. We first consider $B_\varepsilon(0) = \{k\in\Rd: \abs*{k} < \varepsilon\}$, where we choose $\varepsilon>0$ such that $\varepsilon^2 \leq \tfrac{1-C_1}{C_2}$ - the constants $C_1,C_2$ coming from the assumption \ref{Assump:Bound}. From \ref{Assump:Bound}, we have $\fgmu(k) = \left(1 - \mulam\SupSpec{\fOpconnf(k)}\right)^{-1}\leq 1/C_2\abs*{k}^2$ on $B_\varepsilon(0)$ and $\fgmu(k)\leq 1/\left(C_2\varepsilon^2\right)$ on $B_\varepsilon(0)^c$. Our pointwise bound on the $\fOpconnf$ factors then gives 
\begin{multline}
    \int_{B_\varepsilon(0)}\fgmu(k)^m\inner*{g_1}{\fOpconnf(k)^{2n-2}g_1}^\frac{1}{2}\inner*{g_2}{\fOpconnf(k)^2g_2}^\frac{1}{2}\frac{\dd k}{\left(2\pi\right)^d} \\\leq \frac{C^n}{C^m_2}\frac{\mathfrak{S}_{d-1}}{\left(2\pi\right)^d}\int^\varepsilon_0\frac{1}{r^{2m}}r^{d-1}\dd r = \frac{C^n}{C^m_2}\frac{\mathfrak{S}_{d-1}}{d-2m}\frac{\varepsilon^{d-2m}}{\left(2\pi\right)^d},
\end{multline}
where $\mathfrak{S}_{d-1}= d\pi^\frac{d}{2}/\Gamma\left(\tfrac{d}{2}+1\right)$ is the surface area of an unit $d$-sphere. Note that for all fixed $\varepsilon>0$ this term vanishes in the $d\to\infty$ limit.

To deal with the integral over $B_\varepsilon(0)^c$, we find the upper bound
\begin{multline}
    \int_{B_\varepsilon(0)^c}\fgmu(k)^m\inner*{g_1}{\fOpconnf(k)^{2n-2}g_1}^\frac{1}{2}\inner*{g_2}{\fOpconnf(k)^2g_2}^\frac{1}{2}\frac{\dd k}{\left(2\pi\right)^d} \\\leq \frac{1}{C^m_2\varepsilon^{2m}}\int\inner*{g_1}{\fOpconnf(k)^{2n-2}g_1}^\frac{1}{2}\inner*{g_2}{\fOpconnf(k)^2g_2}^\frac{1}{2}\frac{\dd k}{\left(2\pi\right)^d} \leq \frac{C^\frac{1}{2}}{C^m_2\varepsilon^{2m}}\begin{cases}
    C^\frac{1}{2} &: n=2\\
    g(d)^\frac{1}{2} &: n\geq 3.
    \end{cases}
\end{multline}

Note that the bounds we found were independent of the sets $E_1,E_2$, and therefore these bounds also apply to $\inner*{g^{(\xbar,a,b,\delta)}_1}{\left(\Opconnf^{n-1}\Optlam^m\Opconnf\right)\left(\xbar\right) g^{(\xbar,a,b,\delta)}_2}$ uniformly over every $\delta>0$, $\xbar\in\Rd$, and $\Pcal$-almost every $a,b\in\Ecal$. Since we can take $\delta\to 0$, our bound also applies to $\InfNorm{\Opconnf^{n-1}\Optlam^m\Opconnf}$.

For $n=2$, the $B_\varepsilon(0)^c$ integral dominates the $B_\varepsilon(0)$ integral, and so the result is proven.

For $n\geq 3$, both parts of the integral approach $0$ as $d\to\infty$. If both are of the same order or if the $B_\varepsilon(0)^c$ integral dominates, then we can fix $\varepsilon>0$ and get the result. On the other hand, if $g(d)^\frac{1}{2} \ll \frac{1}{d}\mathfrak{S}_{d-1}\varepsilon^{d}$ for all fixed $\varepsilon>0$ and the $B_\varepsilon(0)$ integral always dominates, we can improve the overall bound by letting $\varepsilon=\varepsilon(d)$ and having $\varepsilon(d)\to0$ as $d\to\infty$. Having a smaller value of $\varepsilon$ produces a smaller bound for the $B_\varepsilon(0)$ integral, but a larger bound for the $B_\varepsilon(0)^c$ integral. To get an optimal $\varepsilon$ we can set both terms to be of the same order - that is by having $\varepsilon(d) = 2\pi^{\frac{1}{2}}A^\frac{1}{d}g(d)^\frac{1}{2d}\left(1-\frac{2m}{d}\right)^\frac{1}{d}\Gamma\left(\frac{d}{2}+1\right)^\frac{1}{d}$ for some fixed $A>0$. For this choice, we have both integrals of the order
\begin{multline}
    \varepsilon^{-2m}g(d)^\frac{1}{2} = 2^{-2m}\pi^{-m}A^{-\frac{2m}{d}} g(d)^{\frac{1}{2}-\frac{m}{d}}\left(1-\frac{2m}{d}\right)^{-\frac{2m}{d}}\Gamma\left(\frac{d}{2}+1\right)^{-\frac{2m}{d}} \\= \left(\frac{\e}{2\pi}\right)^m g(d)^{\frac{1}{2}-\frac{m}{d}}\frac{1}{d^m}\left(1+o\left(1\right)\right),
\end{multline}
as $d\to\infty$.

\end{proof}

\TauNbound*

\begin{proof}
The $n=1$ case is trivial because $\tlam\left(\xbar,a,b\right)\in\left[0,1\right]$.

For $n=2$, we first bound $\InfNorm{\Opconnf^2}$. We get
\begin{multline}
    \InfNorm{\Opconnf^2} = \esssup_{x,y\in\X}\int \abs*{\connf(x,u)\connf(u,y)}\nu(\dd u) \leq \esssup_{x\in\X}\int \abs*{\connf(x,u)}^2\nu(\dd u) \\\leq \esssup_{x\in\X}\int \abs*{\connf(x,u)}\nu(\dd u) = \OneNorm{\Opconnf},
\end{multline}
by using Cauchy-Schwarz and symmetry of $\connf$ in the first inequality and $\connf(x,u)\in\left[0,1\right]$ in the second. We then use \ref{Assump:2ndMoment} to bound this by the constant $C$. We use Lemma~\ref{lem:phi-tau-phi_bounds} to get the required bound for the remaining terms arising from \eqref{eqn:tau-phi-Expansion} in Lemma~\ref{lem:tau-phi-Expansion}.

For $n\geq 3$, we once again bound the terms in the relevant expansion in \eqref{eqn:tau-phi-Expansion}. For $\InfNorm{\Opconnf^n}$ we use \ref{Assump:BallDecay} to get $\InfNorm{\Opconnf^n}= O\left(g(d)^\frac{1}{2}\right)$. We then use Lemma~\ref{lem:phi-tau-phi_bounds} to bound the remaining terms.
\end{proof}

For our proofs of Lemma~\ref{lem:BoundTrilamBar} and Lemma~\ref{lem:Martini_Bound}, we adopt the following notation. Given $g(\xbar;a,b)$ and $h(\xbar;a,b)$, we define
\begin{equation}
    (gh)(\xbar;a,b) := \int g(\xbar-\ybar;a,c)h(\ybar;c,b)\dd \ybar \Pcal(\dd c).
\end{equation}
This notational convention is associative and therefore can generalise to three or more terms unambiguously.

\BoundTrilamBar*

\begin{proof}
By applying \eqref{eqn:tau-phi-Expansionfunction} and \eqref{eqn:tau-phi-Expansionfunction_2}, we can get
\begin{equation}
    \tlam(\xbar;a,b) \leq \connf(\xbar;a,b) + \lambda (\connf\connf)(\xbar;a,b)+ \lambda^2(\connf\tlam\connf)(\xbar;a,b).
\end{equation}

We will in fact prove the more general result for the convolution of $m\geq 2$ $\tlam$-functions in dimensions $d>2m$. By bounding $\lambda\leq f(\lambda)$, and using Lemma~\ref{lem:ApproxPointWithL2} and the Fourier inversion theorem we can realise that we only need to get further bounds for the following objects. For $\vec{j}\in\left\{1,2,3\right\}^m$, we want to bound
\begin{equation}
\label{eqn:PQexpression}
    \int \prod^m_{i=1} P^{j_i}_{2i-1,2i}(k)\frac{\dd k}{\left(2\pi\right)^{d}},
\end{equation}
where
\begin{equation}
    P^{j}_{p,q}(k) := \begin{cases}
    \abs*{\inner*{g_p}{\fOpconnf(k) g_q}} &:j=1\\
    \abs*{\inner*{g_p}{\fOpconnf(k)^2g_q}} &:j=2\\
    \abs*{\inner*{g_p}{\fOpconnf(k)\fOptlam(k)\fOpconnf(k)g_q}} &: j=3,
    \end{cases}
\end{equation}
where $\left\{g_i\right\}_{i=1}^{2m}$ are the functions of the form $g_i(a) = \frac{1}{\Pcal(E_i)}\Id_{E_i}(a)$ arising from the application of Lemma~\ref{lem:ApproxPointWithL2}.

We now bound the following terms in the same way as \eqref{eqn:ExtractNormT}:
\begin{align}
    \abs*{\inner*{g_1}{\fOpconnf(k)^2g_2}} &\leq \inner*{g_1}{\fOpconnf(k)^2g_1}^{\frac{1}{2}}\inner*{g_2}{\fOpconnf(k)^2g_2}^{\frac{1}{2}}\\
    \abs*{\inner*{g_1}{\fOpconnf(k)\fOptlam(k)\fOpconnf(k)g_2}} &\leq f\left(\lambda\right)\fgmu(k)\inner*{g_1}{\fOpconnf(k)^2g_1}^{\frac{1}{2}}\inner*{g_2}{\fOpconnf(k)^2g_2}^{\frac{1}{2}}.
\end{align}
In terms of pointwise uniform bounds for the $\fOpconnf$ terms, we have $\abs*{\inner*{g_1}{\fOpconnf(k)g_2}}\leq \InfNorm{\fOpconnf(0)} \leq C$ and $\inner*{g}{\fOpconnf(k)^2g}\leq C^2$ (the latter also having been used in the proof of Lemma~\ref{lem:phi-tau-phi_bounds}).

As in Lemma~\ref{lem:phi-tau-phi_bounds}, we will require various integral bounds for these $\fOpconnf$ terms. Suppose we have $r \in\left\{1,\ldots,m-1\right\}$ instances of $P^{2}_{p,q}(k)$ and $P^{3}_{p,q}(k)$, and $m-r$ instances of $P^{1}_{p,q}(k)$. Then by applying Cauchy-Schwarz and the definitions of the $g_i$ we get
\begin{align}
    &\int \prod^{r}_{i=1} \inner*{g_{2i-1}}{\fOpconnf(k)^2g_{2i-1}}^{\frac{1}{2}}\inner*{g_{2i}}{\fOpconnf(k)^2g_{2i}}^{\frac{1}{2}} \prod^{m}_{i'=r+1}\abs*{\inner*{g_{2i'-1}}{\fOpconnf(k)g_{2i'}}}\frac{\dd k}{\left(2\pi\right)^{d}} \nonumber\\
    &\hspace{1cm}\leq \left(\int \inner*{g_1}{\fOpconnf(k)^2g_1}\inner*{g_2}{\fOpconnf(k)^2g_2}\frac{\dd k}{\left(2\pi\right)^{d}}\right)^{\frac{1}{2}}\nonumber\\
    &\hspace{2cm}\times
    \left(\int \prod^{r}_{i=2} \inner*{g_{2i-1}}{\fOpconnf(k)^2g_{2i-1}}\inner*{g_{2i}}{\fOpconnf(k)^2g_{2i}}\prod^{m}_{i'=r+1}\inner*{g_{2i'-1}}{\fOpconnf(k)g_{2i'}}^2\frac{\dd k}{\left(2\pi\right)^{d}}\right)^{\frac{1}{2}} \nonumber\\
    &\hspace{1cm} \leq \left(\esssup_{a_1,\ldots,a_4\in\Ecal}\left(\connf^{(2)}(\cdot;a_1,a_2)\star\connf^{(2)}(\cdot;a_3,a_4)\right)\left(\zerobar\right)\right)^{\frac{1}{2}}\nonumber\\
    &\hspace{2cm}\times\left(\esssup_{a_5,\ldots,a_{2m}\in\Ecal}\left(\connf^{(2)}(\cdot;a_5,a_6)\star\ldots\star\connf^{(2)}(\cdot;a_{2r-1},a_{2r})\right.\right.\nonumber\\
    &\hspace{7cm}\left.\left.\star\connf(\cdot;a_{2r+1},a_{2r+2})\star\ldots\star\connf(\cdot;a_{2m-1},a_{2m})\right)\left(\zerobar\right)\right)^{\frac{1}{2}} \nonumber\\
    &\hspace{1cm} \leq C^{m+r}.
\end{align}
In this last inequality we have used $\InfNorm{\fOpconnf(0)}\leq C$ and $\InfNorm{\fOpconnf(0)^2} \leq \TwoNorm{\fOpconnf(0)}^2 \leq C^2$ to extract off $\connf$ and $\connf^{(2)}$ via supremum bounds on the spatial position. We will also require integral bounds where there are $m$ factors of $P^{2}_{p,q}(k)$ and $P^{3}_{p,q}(k)$, and where there are $m$ factors of $P^{1}_{p,q}(k)$:
\begin{align}
    &\int \prod^{m}_{i=1} \inner*{g_{2i-1}}{\fOpconnf(k)^2g_{2i-1}}^{\frac{1}{2}}\inner*{g_{2i}}{\fOpconnf(k)^2g_{2i}}^{\frac{1}{2}} \frac{\dd k}{\left(2\pi\right)^{d}} \nonumber\\
    &\hspace{1cm}\leq \left(\int \inner*{g_1}{\fOpconnf(k)^2g_1}\inner*{g_2}{\fOpconnf(k)^2g_2}\frac{\dd k}{\left(2\pi\right)^{d}}\right)^{\frac{1}{2}}
    \left(\int \prod^{m}_{i=2} \inner*{g_{2i-1}}{\fOpconnf(k)^2g_{2i-1}}\inner*{g_{2i}}{\fOpconnf(k)^2g_{2i}}\frac{\dd k}{\left(2\pi\right)^{d}}\right)^{\frac{1}{2}} \nonumber\\
    &\hspace{1cm} \leq \left(\esssup_{a_1,\ldots,a_4\in\Ecal}\left(\connf^{(2)}(\cdot;a_1,a_2)\star\connf^{(2)}(\cdot;a_3,a_4)\right)\left(\zerobar\right)\right)^{\frac{1}{2}}\nonumber\\
    &\hspace{6cm}\times\left(\esssup_{a_5,\ldots,a_{2m}\in\Ecal}\left(\connf^{(2)}(\cdot;a_5,a_6)\star\ldots\star\connf^{(2)}(\cdot;a_{2m-1},a_{2m})\right)\left(\zerobar\right)\right)^{\frac{1}{2}} \nonumber\\
    &\hspace{1cm} \leq C^{2m}.\\
    &\int \prod^{m}_{i=1}\abs*{\inner*{g_{2i-1}}{\fOpconnf(k)g_{2i}}}\frac{\dd k}{\left(2\pi\right)^{d}} \nonumber\\
    &\hspace{1cm}\leq \left(\int \inner*{g_1}{\fOpconnf(k)g_2}^2\frac{\dd k}{\left(2\pi\right)^{d}}\right)^{\frac{1}{2}} \left(\prod^{m}_{i=2}\inner*{g_{2i-1}}{\fOpconnf(k)g_{2i}}^2\frac{\dd k}{\left(2\pi\right)^{d}}\right)^{\frac{1}{2}} \nonumber\\
    &\hspace{1cm} \leq \left(\esssup_{a_1,\ldots,a_4\in\Ecal}\left(\connf(\cdot;a_1,a_2)\star\connf(\cdot;a_3,a_4)\right)\left(\zerobar\right)\right)^{\frac{1}{2}}\nonumber\\
    &\hspace{6cm}\times\left(\esssup_{a_5,\ldots,a_{2m}\in\Ecal}\left(\connf(\cdot;a_5,a_6)\star\ldots\star\connf(\cdot;a_{2m-1},a_{2m})\right)\left(\zerobar\right)\right)^{\frac{1}{2}} \nonumber\\
    &\hspace{1cm} \leq C^{m}.
\end{align}
Note that for $m\geq 3$, better bounds are available via \ref{Assump:BallDecay}, but will not be required here.

We can now proceed to bound each of the expressions of the form \eqref{eqn:PQexpression} in much the same way as we did in the proof of Lemma~\ref{lem:phi-tau-phi_bounds}. We set $\varepsilon>0$ and partition the integral over $\Rd$ into one over $B_\varepsilon(0)$ and one over $B_\varepsilon(0)^c$. For the $B_\varepsilon(0)$ integral we use the uniform bounds on the $\fOpconnf$ terms and perform the integral of the $\fgmu$ terms (if there are any) as before. Note that we require $d>6$ for all of these to be finite. For the $B_\varepsilon(0)^c$ integrals we uniformly bound $\fgmu(k) \leq 1/\left(C_2\varepsilon^2\right)$ and use the integral bounds we calculated above. Since we are only asking for a constant bound, we don't need to worry about having $\varepsilon \to 0$ as $d\to\infty$.

\end{proof}

\connfComparison*

\begin{proof}
The first two inequalities in \eqref{eqn:connfComparison} holds from the general inequalities that holds for these norms (on a probability space for the second inequality). Then we can remove the $l$-dependence by performing the following calculation:
\begin{multline}
    \label{eqn:l-independence}
    \InfNorm{\fOpconnf\left(l\right) - \fOpconnf\left(l+ k\right)} \leq \esssup_{a,b\in\Ecal}\int \abs*{\e^{il\cdot\xbar}}\abs*{\left(1-\cos\left( k\cdot \xbar\right)\right)\connf\left(\xbar;a,b\right)}\dd \xbar \\= \esssup_{a,b\in\Ecal}\int \abs*{\left(1-\cos\left(k\cdot \xbar\right)\right)\connf\left(\xbar;a,b\right)}\dd \xbar = \InfNorm{\fOpconnf\left(0\right) - \fOpconnf\left( k\right)}.
\end{multline}
The final step relating the $\InfNorm{\cdot}$-norm to the difference of the spectral suprema is precisely that given by \eqref{eqn:directional2ndMoment} in \ref{Assump:2ndMoment}.
\end{proof}

We now address $\InfNorm{\Opconnf\Optklam\Optlam\Opconnf}$.

\PTkTPBound*

\begin{proof}
First note that for $k=0$ the kernel function $\tklam\left(x,y\right)=0$ for all $x,y\in\X$, and so the required inequality holds trivially. In this proof we hereafter may assume $k\ne 0$.

We use \eqref{eqn:tau-phi-Expansionfunction} to write
\begin{equation}
    \InfNorm{\Opconnf\Optlam\Optklam\Opconnf} \leq \InfNorm{\Opconnf^2\Optklam\Opconnf} + \lambda\InfNorm{\Opconnf^2\Optlam\Optklam\Opconnf}.
\end{equation}

We first get the bound on $\InfNorm{\Opconnf^2\Optlam\Optklam\Opconnf}$, using a similar approach to Lemma~\ref{lem:phi-tau-phi_bounds}. For each $\xbar\in\Rd$ we define the integral linear operator $\left(\Opconnf^2\Optlam\Optklam\Opconnf\right)\left(\xbar\right)\colon L^2\left(\Ecal\right)\to L^2\left(\Ecal\right)$ as the integral operator with kernel function
\begin{multline}
    \left(a,b\right)\mapsto \left(\Opconnf^2\Optlam\Optklam\Opconnf\right)\left(\xbar;a,b\right) \\:= \int \connf^{(2)}\left(\xbar-\ubar;a,c\right)\tlam\left(\ubar-\ubar';c,c'\right)\tklam\left(\ubar'-\ubar'';c',c''\right)\connf\left(\ubar'';c'',b\right)\dd \ubar \dd \ubar' \dd \ubar''\Pcal^{\otimes 3}\left(\dd c,\dd c',\dd c''\right).
\end{multline}
Recall the notation $\connf^{(2)}$ used in \eqref{eqn:phi^n_tau^n_notation}. As in the proof of Lemma~\ref{lem:phi-tau-phi_bounds}, we let $E_1,E_2\subset \Ecal$ be measurable and $\Pcal$-positive sets, and aim to bound $\inner{g_1}{\left(\Opconnf\Optlam\Optklam\Opconnf\right)\left(\xbar\right)g_2}$ (where $g_i(c) := \frac{1}{\Pcal\left(E_i\right)}\Id_{E_i}\left(c\right)$ for $i=1,2$) independently of the choice of $E_1,E_2$. If we do this we will have proven the result.

As before, we use the Fourier inversion theorem to write $\inner{g_1}{\left(\Opconnf\Optlam\Optklam\Opconnf\right)\left(\xbar\right)g_2}$ as the integral of an inner product of Fourier transformed operators over the Fourier argument. Having $\lambda<\lambda_O$ allows us to swap the integrals in this step. Then we can use Cauchy-Schwarz and the definitions of the operator norm and the bootstrap functions to extract factors of $\fgmu(l)$ and $\widehat{J}_{\mulam}\left(k,l\right)$ from the inner product. The net result is
\begin{align}
    \inner*{g_1}{\left(\Opconnf^2\Optlam\Optklam\Opconnf\right)\left(\xbar\right)g_2} &\leq \int\abs*{\inner*{g_1}{\fOpconnf(l)^{2}\fOptlam(l)\fOptklam(l)\fOpconnf(l)g_2}}\frac{\dd l}{\left(2\pi\right)^d} \nonumber\\
    &\leq f(\lambda)^2\int \fgmu(l) \widehat{J}_{\mulam}\left(k,l\right) \inner*{g_1}{\fOpconnf(l)^4g_1}^\frac{1}{2}\inner*{g_2}{\fOpconnf(l)^2g_2}^\frac{1}{2}\frac{\dd l}{\left(2\pi\right)^d}.
\end{align}
The $\widehat{J}_{\mulam}\left(k,l\right)$ term produces the required factor of $\left(1-\SupSpec{\fOpconnf(k)}\right)$, so (after using symmetry in $k$) we only need to bound the expression
\begin{equation}
    \int \left(2\fgmu(l)^2\fgmu(l-k) + \fgmu(l+k)\fgmu(l)\fgmu(l-k)\right)\inner*{g_1}{\fOpconnf(l)^4g_1}^\frac{1}{2}\inner*{g_2}{\fOpconnf(l)^2g_2}^\frac{1}{2}\frac{\dd l}{\left(2\pi\right)^d}.
\end{equation}

We partition our integral using open $\varepsilon$-balls around the poles where $\varepsilon^2 \leq \tfrac{1-C_1}{C_2}$ - the constants $C_1$ and $C_2$ coming from \ref{Assump:Bound}. Let $B_{\varepsilon}(p)$ denote the open $\varepsilon$-ball around $p\in\Rd$. For $n\in\left\{0,1,2,3\right\}$, let 
\begin{equation}
\label{eqn:DisplacedEpsilonBalls}
    A_n := \left\{l\in\Rd: l\in B_\varepsilon(p) \text{ for precisely $n$ elements }p\in\left\{k,0,-k\right\}\right\}.
\end{equation}
Firstly it is possible for $A_3,A_2 = \emptyset$ if $\abs*{k}$ is sufficiently large compared to $\varepsilon$. It is also easy to see that $A_3\subset B_{\varepsilon}(0)$ and $A_2\subset B_{2\varepsilon}(0)$. We will partition $\Rd$ into $A_3$, $A_2$, $A_1$, and $A_0$ and bound the $l$-integral on each part. Recall that \ref{Assump:Bound} implies that $\fgmu(l) \leq C^{-1}_2\abs*{l}^{-2}$ for $\abs*{l}\leq \varepsilon$ and $\fgmu(l) \leq C^{-1}_2\varepsilon^2$ for $\abs*{l}\geq \varepsilon$. Also recall the pointwise \eqref{eqn:phi^k_pointwise} and integral \eqref{eqn:phi^k_integral} bounds we derived for the $\fOpconnf$ factors. We treat $A_3$ first:
\begin{align}
    &\int_{A_3}\fgmu(l)^2\fgmu(l-k)\inner*{g_1}{\fOpconnf(l)^4g_1}^\frac{1}{2}\inner*{g_2}{\fOpconnf(l)^2g_2}^\frac{1}{2}\frac{\dd l}{\left(2\pi\right)^d}\leq \frac{C^3}{C^3_2}\int_{B_\varepsilon(0)\cap B_\varepsilon(k)}\frac{1}{\abs*{l}^4\abs*{l-k}^2} \frac{\dd l}{\left(2\pi\right)^d} \nonumber \\
    &\qquad\leq \frac{C^3}{C^3_2}\left(\int_{B_\varepsilon(0)}\frac{1}{\abs*{l}^6} \frac{\dd l}{\left(2\pi\right)^d}\right)^{\frac{2}{3}}\left(\int_{B_\varepsilon(k)}\frac{1}{\abs*{l-k}^6} \frac{\dd l}{\left(2\pi\right)^d}\right)^{\frac{1}{3}} \leq \frac{C^3}{C^3_2}\int_{B_\varepsilon(0)}\frac{1}{\abs*{l}^6} \frac{\dd l}{\left(2\pi\right)^d} = \frac{C^3}{C^3_2}\frac{\mathfrak{S}_{d-1}}{d-6}\frac{\varepsilon^{d-6}}{\left(2\pi\right)^d}\\
    &\int_{A_3}\fgmu(l+k)\fgmu(l)\fgmu(l-k)\inner*{g_1}{\fOpconnf(l)^4g_1}^\frac{1}{2}\inner*{g_2}{\fOpconnf(l)^2g_2}^\frac{1}{2}\frac{\dd l}{\left(2\pi\right)^d} \nonumber\\
    &\hspace{2cm}\leq \frac{C^3}{C^3_2}\int_{A_3}\frac{1}{\abs*{l+k}^2\abs*{l}^2\abs*{l-k}^2} \frac{\dd l}{\left(2\pi\right)^d}\nonumber \\
    &\hspace{2cm}\leq \frac{C^3}{C^3_2}\left(\int_{B_\varepsilon(-k)}\frac{1}{\abs*{l+k}^6} \frac{\dd l}{\left(2\pi\right)^d}\right)^{\frac{1}{3}}\left(\int_{B_\varepsilon(0)}\frac{1}{\abs*{l}^6} \frac{\dd l}{\left(2\pi\right)^d}\right)^{\frac{1}{3}}\left(\int_{B_\varepsilon(k)}\frac{1}{\abs*{l-k}^6} \frac{\dd l}{\left(2\pi\right)^d}\right)^{\frac{1}{3}} \nonumber \\
    &\hspace{2cm}\leq \frac{C^3}{C^3_2}\int_{B_\varepsilon(0)}\frac{1}{\abs*{l}^6} \frac{\dd l}{\left(2\pi\right)^d} = \frac{C^3}{C^3_2}\frac{\mathfrak{S}_{d-1}}{d-6}\frac{\varepsilon^{d-6}}{\left(2\pi\right)^d}\label{eqn:Holder_trick}
\end{align}
Recall $\mathfrak{S}_{d-1}= d\pi^\frac{d}{2}/\Gamma\left(\tfrac{d}{2}+1\right)$ is the surface area of an unit $d$-sphere. In these calculations we used H\"older's inequality and increased the domain of integration to get upper bounds. These integrals are finite for $d>6$, and approach $0$ if we take bounded $\varepsilon$.

For $A_2$ the counting of cases is the only extra complication. A precise counting is entirely possible, but unnecessary for our purposes. It is simple to see that the arguments outlined for $A_3$ can be applied to get the bound as an integer multiple of 
\begin{equation}
    \frac{C^3}{C^3_2}\left(\int_{B_\varepsilon(0)}\left(\frac{1}{\abs*{l}^6}  + \frac{1}{\varepsilon^2\abs*{l}^4}  + \frac{1}{\varepsilon^4\abs*{l}^2} \right)\frac{\dd l}{\left(2\pi\right)^d}\right) = \frac{C^3}{C^3_2}\frac{\mathfrak{S}_{d-1}}{\left(2\pi\right)^d}\left(\frac{1}{d-6} + \frac{1}{d-4} + \frac{1}{d-2}\right)\varepsilon^{d-6}.
\end{equation}
Essentially, these terms arise because given a pair of two distinct points from $\left\{k,0,-k\right\}$, there are terms with three, two, and one factor(s) of $\fgmu$ centred on these two points. This bound is finite for $d>6$, and approaches $0$ if we take bounded $\varepsilon$.

Repeating this for $A_1$, we get the bound as an integer multiple of
\begin{equation}
    \frac{C^3}{C^3_2}\left(\int_{B_\varepsilon(0)}\left(\frac{1}{\varepsilon^2\abs*{l}^4}  + \frac{1}{\varepsilon^4\abs*{l}^2} + \frac{1}{\varepsilon^6} \right)\frac{\dd l}{\left(2\pi\right)^d}\right) = \frac{C^3}{C^3_2}\frac{\mathfrak{S}_{d-1}}{\left(2\pi\right)^d}\left(\frac{1}{d-4} + \frac{1}{d-2} + \frac{1}{d}\right)\varepsilon^{d-6}.
\end{equation}
Essentially, these terms arise because given a single point from $\left\{k,0,-k\right\}$, there are terms with two, one, and zero factor(s) of $\fgmu$ centred on this single point. This bound is finite for $d>4$, and approaches $0$ if we take bounded $\varepsilon$.

The $A_0$ integral is qualitatively different. By taking the integral bound for the $\fOpconnf$ factors and a uniform (on $A_0$) bound for the factors of $\fgmu$, we get the bound
\begin{equation}
    \frac{C^3}{C^3_2}\varepsilon^{-6}g(d)^\frac{1}{2}.
\end{equation}

As in Lemma~\ref{lem:phi-tau-phi_bounds}, if the $g(d)^\frac{1}{2}$ term dominates, then we have a sufficient bound for our result. The largest of the other terms (for $d$ sufficiently large and $\varepsilon\leq 1$) is the $\frac{\mathfrak{S}_{d-1}}{d-6}\frac{\varepsilon^{d-6}}{\left(2\pi\right)^d}$ term. Like we did in Lemma~\ref{lem:phi-tau-phi_bounds}, we can make these other terms smaller by decreasing $\varepsilon$ at the cost of making the $g(d)^\frac{1}{2}$ bound greater. The largest term can be minimised (up to a constant) by taking $\varepsilon(d) = 2\pi^{\frac{1}{2}}A^\frac{1}{d}g(d)^\frac{1}{2d}\left(1-\frac{6}{d}\right)^\frac{1}{d}\Gamma\left(\frac{d}{2}+1\right)^\frac{1}{d}$ for some fixed $A>0$. For this choice, we have both integrals of the order
\begin{equation}
    \varepsilon^{-6}g(d)^\frac{1}{2} = 2^{-6}\pi^{-3}A^{-\frac{6}{d}}g(d)^{\frac{1}{2}-\frac{3}{d}}\left(1-\frac{6}{d}\right)^{-\frac{6}{d}}\Gamma\left(\frac{d}{2}+1\right)^{-\frac{6}{d}} = \left(\frac{\e}{2\pi}\right)^3 g(d)^{\frac{1}{2}-\frac{3}{d}}\frac{1}{d^3}\left(1+o\left(1\right)\right),
\end{equation}
as $d\to\infty$.

Bounding the $\InfNorm{\Opconnf^2\Optklam\Opconnf}$ term is similar. For each $\xbar\in\Rd$ we define the integral linear operator $\left(\Opconnf^2\Optklam\Opconnf\right)\left(\xbar\right)\colon L^2\left(\Ecal\right)\to L^2\left(\Ecal\right)$ as the integral operator with kernel function
\begin{multline}
    \left(a,b\right)\mapsto \left(\Opconnf^2\Optklam\Opconnf\right)\left(\xbar;a,b\right) \\:= \int \connf^{(2)}\left(\xbar-\ubar;a,c\right)\tklam\left(\ubar-\ubar';c,c''\right)\connf\left(\ubar';c',b\right)\dd \ubar \dd \ubar' \Pcal^{\otimes 2}\left(\dd c,\dd c'\right).
\end{multline}
We once again take $\Pcal$-positive sets $E_1,E_2$ and corresponding $g_1,g_2$. Then we use the Fourier inversion theorem and Cauchy-Schwarz to get
\begin{equation}
    \inner*{g_1}{\left(\Opconnf^2\Optklam\Opconnf\right)\left(\xbar\right)g_2} \leq f(\lambda)\int\widehat{J}_{\mulam}\left(k,l\right) \inner*{g_1}{\fOpconnf(l)^4g_1}^\frac{1}{2}\inner*{g_2}{\fOpconnf(l)^2g_2}^\frac{1}{2}\frac{\dd l}{\left(2\pi\right)^d}.
\end{equation}
Then once again the $\widehat{J}_{\mulam}\left(k,l\right)$ term produces the required factor of $\left(1-\SupSpec{\fOpconnf(k)}\right)$, so (after using symmetry in $k$) we only need to bound the expression
\begin{equation}
\label{eqn:phiphitaukphiboundintegral}
    \int \left(2\fgmu(l)\fgmu(l-k) + \fgmu(l+k)\fgmu(l-k)\right)\inner*{g_1}{\fOpconnf(l)^4g_1}^\frac{1}{2}\inner*{g_2}{\fOpconnf(l)^2g_2}^\frac{1}{2}\frac{\dd l}{\left(2\pi\right)^d}.
\end{equation}

Counting the number of factors of $\fgmu$ centred on each of the poles $\left\{k,0,-k\right\}$ then tells us that there is a uniform bound that is some integer multiple of
\begin{multline}
    \frac{C^2}{C^2_2}\left(\int_{B_\varepsilon(0)}\left(\frac{1}{\abs*{l}^{4}}  + \frac{1}{\varepsilon^2\abs*{l}^{2}}  + \frac{1}{\varepsilon^4}\right)\frac{\dd l}{\left(2\pi\right)^d} + \varepsilon^{-4}g(d)^\frac{1}{2}\right) \\= \frac{C^2}{C^2_2}\frac{\mathfrak{S}_{d-1}}{\left(2\pi\right)^d}\left(\frac{1}{d-4} + \frac{1}{d-2} + \frac{1}{d}\right)\varepsilon^{d-4} + \frac{C^2}{C^2_2}\varepsilon^{-4}g(d)^\frac{1}{2}.
\end{multline}
Once again, if the $g(d)^\frac{1}{2}$ term dominates the result is proven. Otherwise, we let $\varepsilon$ vary with $d$. The largest term can be minimised (up to a constant) by taking $\varepsilon(d) = 2\pi^{\frac{1}{2}}A^\frac{1}{d}g(d)^\frac{1}{2d}\left(1-\frac{4}{d}\right)^\frac{1}{d}\Gamma\left(\frac{d}{2}+1\right)^\frac{1}{d}$ for some fixed $A>0$. For this choice, we have both integrals of the order
\begin{equation}
    \varepsilon^{-4}\beta^2 = 2^{-4}\pi^{-2}A^{-\frac{4}{d}} g(d)^{\frac{1}{2}-\frac{2}{d}}\left(1-\frac{4}{d}\right)^{-\frac{4}{d}}\Gamma\left(\frac{d}{2}+1\right)^{-\frac{4}{d}} = \left(\frac{\e}{2\pi}\right)^2 g(d)^{\frac{1}{2}-\frac{2}{d}}\frac{1}{d^2}\left(1+o\left(1\right)\right),
\end{equation}
as $d\to\infty$. For large $d$, this is dominated by $g(d)^{\frac{1}{2}-\frac{3}{d}}d^{-3}$ and therefore $\InfNorm{\Opconnf^2\Optklam\Opconnf}$ is dominated by $\InfNorm{\Opconnf^2\Optlam\Optklam\Opconnf}$ and we prove the result.

\end{proof}

The following lemma allows us to deal with occurrences of $\WkBar$.

\WkBound*

\begin{proof}
The first inequality follows by applying the cosine-splitting result (Lemma~\ref{lem:cosinesplitlemma}) to \eqref{eqn:tau-phi-Expansionfunction_2}.

For the $\tlam$ term in $\WkBar$ we apply \eqref{eqn:tau-phi-Expansionfunction} and \eqref{eqn:tau-phi-Expansionfunction_2} to get
\begin{equation}
    \tlam(\xbar;a,b) \leq \connf(\xbar;a,b) + \lambda (\connf\connf)(\xbar;a,b)+ \lambda^2(\connf\tlam\connf)(\xbar;a,b).
\end{equation}
Similarly, with the additional use of \eqref{eqn:tau-phiExpansionDisplacement}, we get
\begin{multline}
    \tklam(\xbar;a,b) \leq \connf_k(\xbar;a,b) + 2\lambda\left( (\connf_k\connf)(\xbar;a,b) +(\connf\connf_k)(\xbar;a,b)\right)+ 4\lambda^2(\connf\tlam\connf_k)(\xbar;a,b) \\ + 4\lambda^2(\connf\tklam\connf)(\xbar;a,b).
\end{multline}
By using Lemma~\ref{lem:ApproxPointWithL2}, the Fourier inversion theorem, and $\lambda\leq f(\lambda)$, we can realise that we can bound $\WkBar$ by bounding the following objects. For $\vec{j}\in\left\{1,2,3\right\}\times\left\{1,2,3,4,5\right\}$, define
\begin{equation}
\label{eqn:Wk_target_integral}
    \mathcal{W}(k;\vec{j}) := \int P^{j_1}_{1,2}(l)Q^{j_2}_{3,4}(l;k)\frac{\dd l}{\left(2\pi\right)^{d}},
\end{equation}
where
\begin{align}
    P^{j}_{1,2}(l) &:= \begin{cases}
    \abs*{\inner*{g_1}{\fOpconnf(l) g_2}} &:j=1\\
    \abs*{\inner*{g_1}{\fOpconnf(l)^2g_2}} &:j=2\\
    \abs*{\inner*{g_1}{\fOpconnf(l)\fOptlam(l)\fOpconnf(l)g_2}} &: j=3,
    \end{cases}\\
    Q^{j}_{3,4}(l;k) &:= \begin{cases}
    \abs*{\inner*{g_{3}}{\fOpconnf_k(l) g_{4}}} &:j=1\\
    \abs*{\inner*{g_{3}}{\fOpconnf_k(l)\fOpconnf(l) g_{4}}} &:j=2\\
    \abs*{\inner*{g_{3}}{\fOpconnf(l)\fOpconnf_k(l) g_{4}}} &:j=3\\
    \abs*{\inner*{g_{3}}{\fOpconnf(l)\fOptlam(l)\fOpconnf_k(l) g_{4}}} &:j=4\\
    \abs*{\inner*{g_{3}}{\fOpconnf(l)\fOptklam(l)\fOpconnf(l) g_{4}}} &:j=5,
    \end{cases}
\end{align}
where $\left\{g_m\right\}_{m=1}^{4}$ are the functions of the form $g_m(a) = \frac{1}{\Pcal(E_m)}\Id_{E_m}(a)$ arising from the application of Lemma~\ref{lem:ApproxPointWithL2}.

For the $P^j_{1,2}(l)$ factor we can get bounds as we did in the proof of Lemma~\ref{lem:BoundTrilamBar}. For $Q^{j}_{3,4}(l;k)$ we need to do a bit more processing. For $j=1$,
\begin{multline}
    \abs*{\inner*{g_{3}}{\fOpconnf_k(l) g_{4}}} \leq \frac{1}{\Pcal(E_{3})\Pcal(E_{4})}\int_{E_{3}\times E_{4}}\abs*{\fconnf_k(l;a,b)}\Pcal^{\otimes 2}(\dd a, \dd b) \leq \InfNorm{\fOpconnf_k(l)} \\\leq C\left(1 - \SupSpec{\fOpconnf(k)}\right),
\end{multline}
where in the last inequality we used Lemma~\ref{thm:connfComparison}. For $j=2$, we use Cauchy-Schwarz to get
\begin{multline}
    \abs*{\inner*{g_{3}}{\fOpconnf_k(l)\fOpconnf(l) g_{4}}} \leq \inner*{g_3}{\fOpconnf_k(l)^2 g_3}^{\frac{1}{2}} \inner*{g_4}{\fOpconnf(l)^2 g_4}^{\frac{1}{2}}
    \leq \InfNorm{\fOpconnf_k(l)}\inner*{g_4}{\fOpconnf(l)^2 g_4}^{\frac{1}{2}} \\\leq C\left(1 - \SupSpec{\fOpconnf(k)}\right)\inner*{g_4}{\fOpconnf(l)^2 g_4}^{\frac{1}{2}}.
\end{multline}
An identical argument gives the same bound for $j=3$. For $j=4$, we can use the self-adjointness of $\fOpconnf(l)$ and $\fOptlam(l)$ and the Cauchy-Schwarz inequality to get
\begin{align}
    \abs*{\inner*{g_{3}}{\fOpconnf(l)\fOptlam(l)\fOpconnf_k(l) g_{4}}} &\leq \inner*{g_{3}}{\fOpconnf(l)\fOptlam(l)^2\fOpconnf(l)g_{3}}^{\frac{1}{2}} \inner*{g_{4}}{\fOpconnf_k(l)^2g_{4}}^{\frac{1}{2}} \nonumber\\
    &\leq f(\lambda)\fgmu(l)\inner*{g_{3}}{\fOpconnf(l)^2g_{3}}^{\frac{1}{2}}\InfNorm{\fOpconnf_k(l)}\nonumber \\
    &\leq Cf(\lambda)\fgmu(l)\left(1 - \SupSpec{\fOpconnf(k)}\right)\inner*{g_{3}}{\fOpconnf(l)^2g_{3}}^{\frac{1}{2}},
\end{align}
and similarly
\begin{multline}
    \abs*{\inner*{g_{3}}{\fOpconnf(l)\fOptklam(l)\fOpconnf(l) g_{4}}} \leq f(\lambda)\widehat{J}_{\mulam}\left(k,l\right)\inner*{g_3}{\fOpconnf(l)^2g_3}^\frac{1}{2}\inner*{g_4}{\fOpconnf(l)^2g_4}^\frac{1}{2} \\ \leq Cf(\lambda)\dispfgmu(l;k) \left(1 - \SupSpec{\fOpconnf(k)}\right)\inner*{g_3}{\fOpconnf(l)^2g_3}^\frac{1}{2}\inner*{g_4}{\fOpconnf(l)^2g_4}^\frac{1}{2},
\end{multline}
where we write for convenience
\begin{equation}
    \dispfgmu(l_3;k) = \fgmu(l_3)\fgmu(l_3-k) + \fgmu(l_3)\fgmu(l_3+k) + \fgmu(l_3-k)\fgmu(l_3+k).
\end{equation}
Note that all the bounds for the $Q^{j}_{3,4}(l)$ terms have a $\left(1 - \SupSpec{\fOpconnf(k)}\right)$ factor. So we only need to prove that the remaining $l$-integral of $\fgmu$ and $\fOpconnf$ factors is bounded.

As in our previous proofs, we partition $\Rd$ into the sets $A_n$ defined earlier in \eqref{eqn:DisplacedEpsilonBalls}. For the $A_1$, $A_2$, and $A_3$ parts we then use the bounds $C$ and $C^2$ for the various $\fOpconnf$ factors and $1/\left(C_2\abs*{l}^2\right)$ to bound the $\fgmu$ factors. By using the techniques used previously in this section, we can show that these contributions is bounded by some $\varepsilon$ and $d$ (only) dependent constant. Since we don't need to show decay, we don't need to worry about taking $\varepsilon$ to $0$.

For the $A_0$ contribution, we bound $\fgmu$ factors by $1/\left(C_2\varepsilon^2\right)$ and show that the integral of the $\fOpconnf$ factors over all of $\Rd$ is bounded. For $j_2\ne 1$ this proceeds similarly to the argument in the proof of Lemma~\ref{lem:BoundTrilamBar}. For $j_2=1$ we need to take a step back. These versions of \eqref{eqn:Wk_target_integral} arise from trying to bound one particular type of term, which we now treat separately. We bound
\begin{equation}
    \esssup_{\xbar\in\Rd,a_1,\ldots,a_4\in\Ecal}\int \tlam\left(\xbar-\ubar;a_1,a_2\right)\connf_k\left(\ubar;a_3,a_4\right)\dd \ubar \leq \esssup_{a_3,a_4\in\Ecal}\int \connf_k\left(\ubar;a_3,a_4\right)\dd \ubar \leq \InfNorm{\fOpconnf_k(0)}.
\end{equation}
Then Lemma~\ref{thm:connfComparison} gives the result.

\end{proof}

We are now only left with bounding $\HkBar$.
\MartiniBound*

\begin{proof}
We begin by defining a sightly different object to $\HkBar$. Given $\vec{a}=\left(a_1,\ldots,a_{16}\right)\in\Ecal^{16}$, define 
\begin{align}
    H'\left(\xbar_1,\xbar_2,\xbar_3;\vec{a};k\right) :=& \int\tlam(\sbar-\xbar_1;a_1,a_2)\tlam(\ubar;a_3,a_4) \tlam(\vbar-\sbar;a_5,a_6) \tlam(\vbar+\xbar_2-\tbar;a_7,a_8) \nonumber\\
    &\hspace{1cm}\times \tlam(\sbar-\wbar;a_9,a_{10})\tlam(\wbar-\ubar;a_{11},a_{12}) \tlam(\tbar-\wbar;a_{13},a_{14}) \nonumber\\
    &\hspace{2cm}\times\tklam(\tbar + \xbar_3-\ubar;a_{15},a_{16}) \dd\tbar\dd\wbar\dd\zbar\dd\ubar,
\end{align}
so we have
\begin{equation}
    \HkBar \leq f(\lambda)^5\esssup_{\xbar_1,\xbar_2\in\Rd,\vec{a}\in\Ecal^{16}}H'\left(\xbar_1,\xbar_2,\zerobar;\vec{a};k\right).
\end{equation}
We then proceed to bound $H'\left(\xbar_1,\xbar_2,\xbar_3;\vec{a};k\right)$ to get our result. This form is preferred because it more easily indicates what the Fourier arguments should be when we take a Fourier transform. A schematic Fourier diagram for this can be found in \cite[Figure~2]{HeyHofLasMat19}.

We then expand out the $\tlam$ and $\tklam$ terms. For each of the $\tlam$ terms we apply \eqref{eqn:tau-phi-Expansionfunction} and \eqref{eqn:tau-phi-Expansionfunction_2} to get
\begin{equation}
    \tlam(\xbar;a,b) \leq \connf(\xbar;a,b) + \lambda (\connf\connf)(\xbar;a,b)+ \lambda^2(\connf\tlam\connf)(\xbar;a,b).
\end{equation}
Similarly, with the additional use of the cosine-splitting lemma, we get
\begin{multline}
    \tklam(\xbar;a,b) \leq \connf_k(\xbar;a,b) + 2\lambda\left( (\connf_k\connf)(\xbar;a,b) +(\connf\connf_k)(\xbar;a,b)\right)+ 4\lambda^2(\connf\tlam\connf_k)(\xbar;a,b) \\ + 4\lambda^2(\connf\tklam\connf)(\xbar;a,b).
\end{multline}
By using Lemma~\ref{lem:ApproxPointWithL2}, the Fourier inversion theorem, and $\lambda\leq f(\lambda)$, we can realise that we only need to get appropriate bounds for the following objects. For $\vec{j}\in\left\{1,2,3\right\}^7\times\left\{1,2,3,4,5\right\}$, define
\begin{multline}
\label{eqn:Hk_target_integral}
    \mathcal{H}(k;\vec{j}) := \int P^{j_1}_{1,2}(l_1)P^{j_2}_{3,4}(l_1)P^{j_3}_{5,6}(l_2)P^{j_4}_{7,8}(l_2)P^{j_5}_{9,10}(l_1-l_2)P^{j_6}_{11,12}(l_1-l_3)P^{j_7}_{13,14}(l_2-l_3) \\ \times Q^{j_8}_{15,16}(l_3;k)\frac{\dd l_1 \dd l_2 \dd l_3}{\left(2\pi\right)^{3d}},
\end{multline}
where
\begin{align}
    P^{j}_{m,n}(l) &:= \begin{cases}
    \abs*{\inner*{g_m}{\fOpconnf(l) g_n}} &:j=1\\
    \abs*{\inner*{g_m}{\fOpconnf(l)^2g_n}} &:j=2\\
    \abs*{\inner*{g_m}{\fOpconnf(l)\fOptlam(l)\fOpconnf(l)g_n}} &: j=3,
    \end{cases} \qquad\text{for }m,n\in\left\{1,2,\ldots,14\right\},\\
    Q^{j}_{15,16}(l;k) &:= \begin{cases}
    \abs*{\inner*{g_{15}}{\fOpconnf_k(l) g_{16}}} &:j=1\\
    \abs*{\inner*{g_{15}}{\fOpconnf_k(l)\fOpconnf(l) g_{16}}} &:j=2\\
    \abs*{\inner*{g_{15}}{\fOpconnf(l)\fOpconnf_k(l) g_{16}}} &:j=3\\
    \abs*{\inner*{g_{15}}{\fOpconnf(l)\fOptlam(l)\fOpconnf_k(l) g_{16}}} &:j=4\\
    \abs*{\inner*{g_{15}}{\fOpconnf(l)\fOptklam(l)\fOpconnf(l) g_{16}}} &:j=5,
    \end{cases}
\end{align}
where $\left\{g_m\right\}_{m=1}^{16}$ are the functions of the form $g_m(a) = \frac{1}{\Pcal(E_m)}\Id_{E_m}(a)$ arising from the application of Lemma~\ref{lem:ApproxPointWithL2}. If we can bound these $\mathcal{H}(k;\vec{j})$ appropriately, then we can combine them to get the required bound.

We now recall pointwise bounds for the $P^{j}_{m,n}(l)$ and $Q^{j}_{15,16}(l)$ terms. In the proof of Lemma~\ref{lem:W_kBound} we derived
\begin{align}
    \abs*{\inner*{g_m}{\fOpconnf(l)g_n}} &\leq \InfNorm{\fOpconnf(0)}\\
    \abs*{\inner*{g_m}{\fOpconnf(l)^2g_n}} &\leq \inner*{g_m}{\fOpconnf(l)^2g_m}^{\frac{1}{2}}\inner*{g_n}{\fOpconnf(l)^2g_n}^{\frac{1}{2}}\\
    \abs*{\inner*{g_m}{\fOpconnf(l)\fOptlam(l)\fOpconnf(l)g_n}} &\leq f(\lambda)\fgmu(l)\inner*{g_m}{\fOpconnf(l)^2g_m}^{\frac{1}{2}}\inner*{g_n}{\fOpconnf(l)^2g_n}^{\frac{1}{2}},
\end{align}
and 
\begin{align}
    \abs*{\inner*{g_{15}}{\fOpconnf_k(l) g_{16}}} &\leq \InfNorm{\fOpconnf_k(l)} \leq C\left(1 - \SupSpec{\fOpconnf(k)}\right)\\
    \abs*{\inner*{g_{15}}{\fOpconnf_k(l)\fOpconnf(l) g_{16}}} &\leq C\left(1 - \SupSpec{\fOpconnf(k)}\right) \inner*{g_{16}}{\fOpconnf(l)^2g_{16}}^{\frac{1}{2}}\\
    \abs*{\inner*{g_{15}}{\fOpconnf_k(l)\fOpconnf(l) g_{16}}} &\leq C\left(1 - \SupSpec{\fOpconnf(k)}\right) \inner*{g_{15}}{\fOpconnf(l)^2g_{15}}^{\frac{1}{2}}\\
    \abs*{\inner*{g_{15}}{\fOpconnf(l)\fOptlam(l)\fOpconnf_k(l) g_{16}}} &\leq Cf(\lambda)\left(1 - \SupSpec{\fOpconnf(k)}\right)\fgmu(l)\inner*{g_{15}}{\fOpconnf(l)^2g_{15}}\\
    \abs*{\inner*{g_{15}}{\fOpconnf(l)\fOptklam(l)\fOpconnf(l) g_{16}}} &\leq Cf(\lambda)\left(1 - \SupSpec{\fOpconnf(k)}\right)\dispfgmu(l;k)\inner*{g_{15}}{\fOpconnf(l)^2g_{15}}^\frac{1}{2}\inner*{g_{16}}{\fOpconnf(l)^2g_{16}}^\frac{1}{2}
\end{align}
where
\begin{equation}
    \dispfgmu(l;k) = \fgmu(l)\fgmu(l-k) + \fgmu(l)\fgmu(l+k) + \fgmu(l-k)\fgmu(l+k).
\end{equation}
Also recall that we can bound $\abs*{\inner*{g_m}{\fOpconnf(l)g_n}}\leq C$ and $\inner*{g_m}{\fOpconnf(l)^2g_m}\leq C^2$ if required. Note that all the bounds for the $Q^{j}_{15,16}(l)$ terms have a $\left(1 - \SupSpec{\fOpconnf(k)}\right)$ factor. We therefore only need to show that the remaining integrals of $\fOpconnf$ and $\fgmu$ terms can all be bounded by some constant multiple of $\beta^2$. Also note that the bounds for $Q^{1}_{15,16}(l)$, $Q^{2}_{15,16}(l)$, and $Q^{3}_{15,16}(l)$ only differ by a uniform constant (after we bound $\inner*{g_m}{\fOpconnf(l)^2g_m}\leq C^2$), and so once we have dealt with one of these we have dealt with the other two.

Let us first consider $\vec{j}= \left(1,1,1,1,1,1,1,1\right)$. After applying the uniform pointwise bound for $Q^{1}_{15,16}(l)$, we use the Cauchy-Schwarz inequality to bound
\begin{align}
    &\int \abs*{\inner*{g_1}{\fOpconnf(l_1)g_2}}\abs*{\inner*{g_3}{\fOpconnf(l_1)g_4}}\abs*{\inner*{g_5}{\fOpconnf(l_2)g_6}}\abs*{\inner*{g_7}{\fOpconnf(l_2)g_8}}\abs*{\inner*{g_9}{\fOpconnf(l_1-l_2)g_{10}}} \nonumber\\
    &\hspace{7cm}\times\abs*{\inner*{g_{11}}{\fOpconnf(l_1-l_3)g_{12}}}\abs*{\inner*{g_{13}}{\fOpconnf(l_2-l_3)g_{14}}}\frac{\dd l_1 \dd l_2 \dd l_3}{\left(2\pi\right)^{3d}} \nonumber\\
    &\qquad \leq \left(\int \inner*{g_1}{\fOpconnf(l_1)g_2}^2\inner*{g_3}{\fOpconnf(l_1)g_4}^2\inner*{g_5}{\fOpconnf(l_2)g_6}^2\inner*{g_{13}}{\fOpconnf(l_2-l_3)g_{14}}^2\frac{\dd l_1 \dd l_2 \dd l_3}{\left(2\pi\right)^{3d}}\right)^{\frac{1}{2}}\nonumber \\
    &\hspace{2cm}\times \left(\int \inner*{g_7}{\fOpconnf(l_2)g_8}^2\inner*{g_9}{\fOpconnf(l_1-l_2)g_{10}}^2\inner*{g_{11}}{\fOpconnf(l_1-l_3)g_{12}}^2\frac{\dd l_1 \dd l_2 \dd l_3}{\left(2\pi\right)^{3d}}\right)^{\frac{1}{2}}. \label{eqn:coreIntegral}
\end{align}
To deal with these parentheses, we use a volume-preserving change of variables and factorise each into three integrals. For example,
\begin{multline}
    \int \inner*{g_1}{\fOpconnf(l_1)g_2}^2\inner*{g_3}{\fOpconnf(l_1)g_4}^2\inner*{g_5}{\fOpconnf(l_2)g_6}^2\inner*{g_{13}}{\fOpconnf(l_2-l_3)g_{14}}^2\frac{\dd l_1 \dd l_2 \dd l_3}{\left(2\pi\right)^{3d}} \\
    = \left(\int \inner*{g_1}{\fOpconnf(l_1)g_2}^2\inner*{g_3}{\fOpconnf(l_1)g_4}^2\frac{\dd l_1}{\left(2\pi\right)^{d}}\right)\left(\int \inner*{g_5}{\fOpconnf(l_2)g_6}^2\frac{\dd l_2}{\left(2\pi\right)^{d}}\right)\left(\int \inner*{g_{13}}{\fOpconnf(l_3')g_{14}}^2\frac{\dd l_3'}{\left(2\pi\right)^{d}}\right).
\end{multline}
There are now two types of integral we need to bound:
\begin{align}
    &\int \inner*{g_m}{\fOpconnf(l)g_n}^2 \frac{\dd l}{\left(2\pi\right)^d}\nonumber \\
    & \hspace{2cm} = \frac{1}{\Pcal\left(E_m\right)^2\Pcal\left(E_n\right)^2}\int \left(\int_{E_m\times E_n \times E_m \times E_n} \fconnf(l;a,b)\fconnf(l;c,d) \Pcal^{\otimes 4}\left(\dd a,\dd b,\dd c,\dd d\right)\right) \frac{\dd l}{\left(2\pi\right)^d} \nonumber\\
    &  \hspace{2cm} = \frac{1}{\Pcal\left(E_m\right)^2\Pcal\left(E_n\right)^2}\int_{E_m\times E_n \times E_m \times E_n} \left(\connf(\cdot;a,b)\star \connf(\cdot;c,d)\right)\left(\zerobar\right)\Pcal^{\otimes 4}\left(\dd a,\dd b,\dd c,\dd d\right)\nonumber \\
    &  \hspace{2cm} \leq \esssup_{a,b,c,d\in\Ecal} \left(\connf(\cdot;a,b)\star \connf(\cdot;c,d)\right)\left(\zerobar\right) \nonumber\\
    &  \hspace{2cm} \leq \esssup_{a,b,c,d\in\Ecal} \int\connf(\xbar;a,b)\connf(-\xbar;c,d)\dd \xbar \leq \esssup_{a,b\in\Ecal} \int \connf(\xbar;a,b) \dd \xbar \leq C, \label{eqn:connf1times2} \\
    &\int \inner*{g_1}{\fOpconnf(l)g_2}^2 \inner*{g_3}{\fOpconnf(l)g_4}^2  \frac{\dd l}{\left(2\pi\right)^d} \nonumber\\
    & \hspace{2cm} \leq \esssup_{a_1,\ldots,a_8\in\Ecal}\left(\connf(\cdot;a_1,a_2)\star \connf(\cdot;a_3,a_4) \star \connf(\cdot;a_5,a_6) \star \connf(\cdot;a_7,a_8)\right)\left(\zerobar\right) \leq g(d). \label{eqn:connf1times4}
\end{align}
In \eqref{eqn:connf1times2} we used $\connf\left(-\xbar;c,d\right)\in\left[0,1\right]$ and \ref{Assump:2ndMoment}. In \eqref{eqn:connf1times4} we bounded the convolution of four adjacency functions using \ref{Assump:BallDecay}. The result of this is that we can bound \eqref{eqn:coreIntegral} by $C^{\frac{5}{2}}g(d)^{\frac{1}{2}}$. Since $g(d)\leq \beta^4$, this satisfies the bound we require. As noted above, this argument also proves the required bound for $\vec{j}= \left(1,1,1,1,1,1,1,2\right)$ and $\vec{j}= \left(1,1,1,1,1,1,1,3\right)$.

Now suppose that we change some subset of $\left\{j_1,\ldots,j_7\right\}$ from having value $1$ to having value $2$. The argument proceeds similarly, except that after we have factorised our integrals we now want to also bound integrals of the following types:
\begin{align}
    &\int  \inner*{g_1}{\fOpconnf(l)^2g_1} \inner*{g_2}{\fOpconnf(l)^2g_2}\frac{\dd l}{\left(2\pi\right)^{d}}\nonumber\\
    & \hspace{2cm}= \frac{1}{\Pcal\left(E_1\right)^2\Pcal\left(E_2\right)^2}\int\left(\int_{E_1\times E_1\times E_2\times E_2}\fconnf^{(2)}(l;a,b)\fconnf^{(2)}(l;c,d)\Pcal^{\otimes 4}(\dd a,\dd b, \dd c, \dd d)\right)\frac{\dd l}{\left(2\pi\right)^{d}} \nonumber\\
    & \hspace{2cm}= \frac{1}{\Pcal\left(E_1\right)^2\Pcal\left(E_2\right)^2}\int_{E_1\times E_1\times E_2\times E_2}\left(\connf^{(2)}(\cdot;a,b)\star\connf^{(2)}(\cdot;c,d)\right)(\zerobar)\Pcal^{\otimes 4}(\dd a,\dd b, \dd c, \dd d) \nonumber\\
    & \hspace{2cm}\leq \esssup_{a,b,c,d\in\Ecal}\left(\connf^{(2)}(\cdot;a,b)\star\connf^{(2)}(\cdot;c,d)\right)\left(\zerobar\right) \leq g(d), \\
    &\int  \inner*{g_1}{\fOpconnf(l)^2g_1} \inner*{g_2}{\fOpconnf(l)^2g_2}\inner*{g_3}{\fOpconnf(l)g_4}^2\frac{\dd l}{\left(2\pi\right)^{d}}\nonumber\\
    &  \hspace{2cm} \leq \esssup_{a_1,\ldots,a_8\in\Ecal}\left(\connf^{(2)}(\cdot;a_1,a_2)\star \connf^{(2)}(\cdot;a_3,a_4) \star \connf(\cdot;a_5,a_6) \star \connf(\cdot;a_7,a_8)\right)\left(\zerobar\right) \nonumber \\
    &  \hspace{2cm} \leq \InfNorm{\fOpconnf(0)}^2\esssup_{\xbar\in\Rd,a_1,\ldots,a_4\in\Ecal}\left(\connf^{(2)}(\cdot;a_1,a_2)\star \connf^{(2)}(\cdot;a_3,a_4)\right)\left(\xbar\right)\leq C^2 g(d)\\
    &\int  \inner*{g_1}{\fOpconnf(l)^2g_1} \inner*{g_2}{\fOpconnf(l)^2g_2}\inner*{g_3}{\fOpconnf(l)^2g_3} \inner*{g_4}{\fOpconnf(l)^2g_4}\frac{\dd l}{\left(2\pi\right)^{d}}\nonumber\\
    &  \hspace{2cm} \leq \esssup_{a_1,\ldots,a_8\in\Ecal}\left(\connf^{(2)}(\cdot;a_1,a_2)\star \connf^{(2)}(\cdot;a_3,a_4) \star \connf^{(2)}(\cdot;a_5,a_6) \star \connf^{(2)}(\cdot;a_7,a_8)\right)\left(\zerobar\right) \nonumber \\
    &  \hspace{2cm} \leq \InfNorm{\fOpconnf(0)^2}^2\esssup_{\xbar\in\Rd,a_1,\ldots,a_4\in\Ecal}\left(\connf^{(2)}(\cdot;a_1,a_2)\star \connf^{(2)}(\cdot;a_3,a_4)\right)\left(\xbar\right)\leq C^4 g(d).
\end{align}
Up to a constant, these are all less than or equal to the bounds we used for the $\left\{j_1,\ldots,j_7\right\} = \left\{1,\ldots,1\right\}$ case, and therefore that case dominates. These other cases would behave as constant multiple of one of $g(d)$, $g(d)^{\frac{3}{2}}$, $g(d)^2$, $g(d)^{\frac{5}{2}}$, and $g(d)^3$.

We have now dealt with all the cases which have no factors of $\fgmu$. To demonstrate how we will deal with these, we first consider $\vec{j}=\left\{3,3,3,3,3,3,3,5\right\}$. For notational compactness, we define
\begin{equation}
    \left<\fOpconnf(l)^2\right>_{I} := \prod_{m\in I}\inner*{g_m}{\fOpconnf(l)^2g_m},
\end{equation}
where $I\subset \N$ and $l\in\Rd$. Since we have already extracted the factor of $\left(1-\SupSpec{\fOpconnf(k)}\right)$ via the pointwise bound on $Q^{5}_{15,16}(l;k)$, we are left trying to bound the following integral:
\begin{multline}
    \int \fgmu(l_1)^2\fgmu(l_2)^2\fgmu(l_1-l_2)\fgmu(l_1-l_3)\fgmu(l_2-l_3)\dispfgmu(l;k) \\ \qquad\times\left<\fOpconnf(l_1)^2\right>^{\frac{1}{2}}_{1,2,3,4} \left<\fOpconnf(l_2)^2\right>^{\frac{1}{2}}_{5,6,7,8} \left<\fOpconnf(l_1-l_2)^2\right>^{\frac{1}{2}}_{9,10} \left<\fOpconnf(l_1-l_3)^2\right>^{\frac{1}{2}}_{11,12}\left<\fOpconnf(l_2-l_3)^2\right>^{\frac{1}{2}}_{13,14} \frac{\dd l_1 \dd l_2 \dd l_3}{\left(2\pi\right)^{3d}}.
\end{multline}

To bound this integral we need to partition our space. Fix $\varepsilon>0$ and define
\begin{align*}
    B_1 &:= \left\{\abs*{l_1}< \varepsilon\right\}\\ 
    B_2 &:= \left\{\abs*{l_2}< \varepsilon\right\}\\ 
    B_3 &:= \left\{\abs*{l_1-l_3}< \varepsilon\right\}\\
    B_4 &:= \left\{\abs*{l_1-l_2}< \varepsilon\right\}\\
    B_5 &:= \left\{\abs*{l_2-l_3}< \varepsilon\right\}\\
    B_6 &:= \left\{\abs*{l_3}< \varepsilon\right\}\cup\left\{\abs*{l_3-k}< \varepsilon\right\}\cup\left\{\abs*{l_3+k}< \varepsilon\right\},
\end{align*}
and for $n=0,1,\ldots,6$ define
\begin{equation}
    A_n := \left\{(l_1,l_2,l_3)\in\left(\R^{d}\right)^3 \colon \#\left\{m\colon (l_1,l_2,l_3)\in B_m \right\}=n\right\}.
\end{equation}
In words, $A_n$ is the set of points within $\varepsilon$ (in Euclidean distance) of precisely $n$ of the hyperplanes on which the $\fgmu$ factors diverge. If we are working with a factor of $\fgmu$ on the set $B_m$ associated with its relevant hyperplane, we use the bound $\fgmu(l)\leq 1/\left(C_2\abs*{l}^2\right)$, and if we are working off this set we use the bound $\fgmu(l) \leq 1/\left(C_2\varepsilon^2\right)$.

When we integrate over $A_0$, we bound $\fgmu(l) \leq 1/\left(C_2\varepsilon^2\right)$ for all such factors and recycle the bound on the ``only $\fOpconnf$" integrals from above. The contribution from $A_0$ is therefore bounded by some constant multiple of 
\begin{equation}
    \frac{1}{\varepsilon^{18}}g(d)^3.
\end{equation}

For $A_1$, first suppose we are considering $A_1\cap B_m$ for $m\ne 6$. First we bound $\fgmu(l) \leq 1/\left(C_2\varepsilon^2\right)$ and $\dispfgmu(l;k) \leq 3/\left(C_2\varepsilon^2\right)$ for the ``irrelevant" directions, and use the pointwise bound for $\left<\fOpconnf(l)^2\right>_{m,n}$ in the ``relevant" direction. Then we can use the Cauchy-Schwarz and factorise argument we used above and recycle the integral bounds for the $\fOpconnf$ factors as we did above. However we are careful about the way we split apart terms in the Cauchy-Schwarz step. We split the $\fgmu$ terms across both factors and arrange the remaining $\fOpconnf$ terms so that each factor has ``directions" spanning $\left(\Rd\right)^3$. To demonstrate we consider $A_1\cap B_1$:
\begin{align}
    &\int_{A_1\cap B_1} \fgmu(l_1)^2 \left<\fOpconnf(l_2)^2\right>^{\frac{1}{2}}_{5,6,7,8} \left<\fOpconnf(l_1-l_3)^2\right>^{\frac{1}{2}}_{11,12} \left<\fOpconnf(l_2-l_3)^2\right>^{\frac{1}{2}}_{13,14} \frac{\dd l_1 \dd l_2 \dd l_3}{\left(2\pi\right)^{3d}} \nonumber\\
    & \hspace{2cm} \leq C^2\left(\int_{A_1\cap B_1}\fgmu(l_1)^2\left<\fOpconnf(l_2)^2\right>_{5,6}\left<\fOpconnf(l_1-l_3)^2\right>_{11,12}\frac{\dd l_1 \dd l_2 \dd l_3}{\left(2\pi\right)^{3d}}\right)^{\frac{1}{2}}\nonumber \\
    & \hspace{5cm}\times \left(\int_{A_1\cap B_1}\fgmu(l_1)^2 \left<\fOpconnf(l_2)^2\right>_{7,8} \left<\fOpconnf(l_2-l_3)^2\right>_{13,14}\frac{\dd l_1 \dd l_2 \dd l_3}{\left(2\pi\right)^{3d}}\right)^{\frac{1}{2}}.
\end{align}
After a change of variables, factorising and using our integral bounds of $\fOpconnf$ terms from above, we are only left needing to bound
\begin{equation}
    \int_{\abs*{l}<\varepsilon} \fgmu(l)^2 \frac{\dd l}{\left(2\pi\right)^{d}} \leq \frac{1}{C^2_2} \int_{\abs*{l}<\varepsilon} \frac{1}{\abs*{l}^4} \frac{\dd l}{\left(2\pi\right)^{d}} = \frac{1}{C^2_2}\frac{\mathfrak{S}_{d-1}}{\left(2\pi\right)^d}\frac{\varepsilon^{d-4}}{d-4}.
\end{equation}
Therefore the contribution from $A_1\cap B_m$ for $m\ne 6$ is bounded by some constant multiple of $\frac{\mathfrak{S}_{d-1}}{\left(2\pi\right)^d}\frac{\varepsilon^{d-18}}{d-4}g(d)^2$. For $A_1\cap B_6$ we need to be slightly more careful because for $\abs*{k}<2\varepsilon$ the constituent parts of $B_6$ overlap - this is the same considerations we had to account for in the proof of Lemma~\ref{lem:PTkTPBound}. We can overcome it in the same way by using H\"older's inequality (actually just Cauchy-Schwarz in this case) and then re-centring the integrals. The net result is that the bound of a constant multiple of $\frac{\mathfrak{S}_{d-1}}{\left(2\pi\right)^d}\frac{\varepsilon^{d-18}}{d-4}g(d)^2$ also applies to $A_1\cap B_6$. In summary, the contribution from $A_1$ is bounded by some constant multiple of 
\begin{equation}
    \frac{\mathfrak{S}_{d-1}}{\left(2\pi\right)^d}\frac{\varepsilon^{d-18}}{d-4}g(d)^2.
\end{equation}

The $A_2$ case proceeds similarly. After partitioning the set according to which are the overlapping hyperplanes, we take the relevant $\fgmu$ factors from these directions and one $\fOpconnf$ factor from a linearly independent direction, and pointwise bound all the other $\fgmu$ and $\fOpconnf$ factors. The contribution from $A_2$ is then bounded by some constant multiple of 
\begin{equation}
    \left(\frac{\mathfrak{S}_{d-1}}{\left(2\pi\right)^d}\right)^2\frac{\varepsilon^{2d-18}}{\left(d-4\right)^2}g(d),
\end{equation}
arising from the $A_1\cap B_1\cap B_2$, $A_1\cap B_1\cap B_6$, and $A_1\cap B_2\cap B_6$ integrals.

The $A_3$ case is more complicated because there are two qualitatively different sub-cases to consider. First suppose that the three normal vectors to the overlapping hyperplanes are linearly independent. Then we take the relevant $\fgmu$ factors from these directions, pointwise bound the other $\fgmu$ factors and all the $\fOpconnf$ factors. The contribution from this sub-case of $A_3$ is then bounded by some constant multiple of $\left(\frac{\mathfrak{S}_{d-1}}{\left(2\pi\right)^d}\right)^3\frac{\varepsilon^{3d-18}}{\left(d-4\right)^3}$, arising from the $A_1\cap B_1\cap B_2\cap B_6$ integral.

On the other hand, if the three normal vectors to the overlapping hyperplanes are linearly \emph{dependent}, then we need to retain a $\fOpconnf$ from the remaining linearly independent direction. To demonstrate, let us consider the $A_3\cap B_1 \cap B_2 \cap B_4$ integral. After removing unnecessary terms via uniform bounds, 
\begin{align}
    &\int_{A_3\cap B_1 \cap B_2 \cap B_4}\fgmu(l_1)^2\fgmu(l_2)^2\fgmu(l_1-l_2) \left<\fOpconnf(l_2-l_3)^2\right>^{\frac{1}{2}}_{13,14} \left<\fOpconnf(l_1-l_3)^2\right>^{\frac{1}{2}}_{11,12}\frac{\dd l_1 \dd l_2 \dd l_3}{\left(2\pi\right)^{3d}}\nonumber\\
    & \qquad \leq \left(\int_{A_3\cap B_1 \cap B_2 \cap B_4}\fgmu(l_1)^2\fgmu(l_2)^2\fgmu(l_1-l_2) \left<\fOpconnf(l_2-l_3)^2\right>_{13,14} \frac{\dd l_1 \dd l_2 \dd l_3}{\left(2\pi\right)^{3d}}\right)^\frac{1}{2}\nonumber \\
    &\qquad \qquad\times \left(\int_{A_3\cap B_1 \cap B_2 \cap B_4}\fgmu(l_1)^2\fgmu(l_2)^2\fgmu(l_1-l_2) \left<\fOpconnf(l_1-l_3)^2\right>_{11,12} \frac{\dd l_1 \dd l_2 \dd l_3}{\left(2\pi\right)^{3d}}\right)^\frac{1}{2}.
\end{align}
Applying a change of variables, factorising the integral, and applying H\"older's inequality allows us to bound these parentheses:
\begin{align}
    & \int_{A_3\cap B_1 \cap B_2 \cap B_4}\fgmu(l_1)^2\fgmu(l_2)^2\fgmu(l_1-l_2) \left<\fOpconnf(l_2-l_3)^2\right>_{13,14} \frac{\dd l_1 \dd l_2 \dd l_3}{\left(2\pi\right)^{3d}}\nonumber\\
    &\qquad \leq \left(\int_{\left\{\abs*{l_1}<\varepsilon\right\}\cap \left\{\abs*{l_2}<\varepsilon\right\}\cap \left\{\abs*{l_1-l_2}<\varepsilon\right\}}\fgmu(l_1)^2\fgmu(l_2)^2\fgmu(l_1-l_2) \frac{\dd l_1 \dd l_2}{\left(2\pi\right)^{2d}}\right)\left(\int \left<\fOpconnf(l)^2\right>_{13,14}\frac{\dd l}{\left(2\pi\right)^{d}}\right)\nonumber\\
    & \qquad \leq \frac{1}{C^5_2}\left(\int_{\left\{\abs*{l_1}<\varepsilon\right\}\cap \left\{\abs*{l_2}<\varepsilon\right\}\cap \left\{\abs*{l_1-l_2}<\varepsilon\right\}}\frac{1}{\abs*{l_1}^4}\frac{1}{\abs*{l_2}^4}\frac{1}{\abs*{l_1-l_2}^2} \frac{\dd l_1 \dd l_2}{\left(2\pi\right)^{2d}}\right) g(d)\nonumber\\
    & \qquad \leq \frac{1}{C^5_2}\left(\int_{\left\{\abs*{l_1}<\varepsilon\right\}\cap \left\{\abs*{l_2}<\varepsilon\right\}}\frac{1}{\abs*{l_1}^6}\frac{1}{\abs*{l_2}^3} \frac{\dd l_1 \dd l_2}{\left(2\pi\right)^{2d}}\right)^{\frac{2}{3}}\left(\int_{\left\{\abs*{l_2}<\varepsilon\right\}\cap \left\{\abs*{l_1-l_2}<\varepsilon\right\}}\frac{1}{\abs*{l_2}^6}\frac{1}{\abs*{l_1-l_2}^6} \frac{\dd l_1 \dd l_2}{\left(2\pi\right)^{2d}}\right)^{\frac{1}{3}} g(d)\nonumber \\
    & \qquad = \frac{1}{C^5_2}\left(\frac{\mathfrak{S}_{d-1}}{d-6}\frac{\varepsilon^{d-6}}{\left(2\pi\right)^d}\right)^{\frac{4}{3}}\left(\frac{\mathfrak{S}_{d-1}}{d-3}\frac{\varepsilon^{d-3}}{\left(2\pi\right)^d}\right)^{\frac{2}{3}}g(d).
\end{align}
For all of these linearly dependent cases, we can bound the integral by some constant multiple of $\left(\frac{\mathfrak{S}_{d-1}}{\left(2\pi\right)^d}\right)^2\frac{\varepsilon^{2d-18}}{\left(d-6\right)^2}g(d)$. In summary, the contribution from $A_3$ is bounded by some constant multiple of 
\begin{equation}
    \left(\frac{\mathfrak{S}_{d-1}}{\left(2\pi\right)^d}\right)^3\frac{\varepsilon^{3d-18}}{\left(d-4\right)^3} + \left(\frac{\mathfrak{S}_{d-1}}{\left(2\pi\right)^d}\right)^2\frac{\varepsilon^{2d-18}}{\left(d-6\right)^2}g(d).
\end{equation}

For $A_4$, $A_5$, and $A_6$ we don't have the two linearly dependent/independent cases because any four of the directions spans $\left(\Rd\right)^3$. We therefore don't need factors of $\fOpconnf$ to control the unbounded directions. Instead, we need to be careful that when we split our integral at the Cauchy-Schwarz step the resulting parentheses have sufficiently few factors of $\fgmu$ that their bounds will be finite for $d>6$. The way we perform this split will be different for the different parts of $\dispfgmu(l_3;k)$. Let us define 
\begin{equation}
    \widehat{D}(l;k) := \fgmu(l+k) + \fgmu(l-k).
\end{equation}

We first consider $A_6$. By writing $\dispfgmu(l;k) = \fgmu(l)\widehat{D}(l;k) + \fgmu(l+k)\fgmu(l-k)$, we have two integrals we want to bound. First
\begin{align}
    &\int_{A_6}\fgmu(l_1)^2\fgmu(l_2)^2\fgmu(l_1-l_2)\fgmu(l_1-l_3)\fgmu(l_2-l_3)\fgmu(l_3)\widehat{D}(l_3;k) \frac{\dd l_1 \dd l_2 \dd l_3}{\left(2\pi\right)^{3d}}\nonumber \\
    &\qquad\leq \left(\int_{A_6}\left[\fgmu(l_1)^3\right]\left[\fgmu(l_2)\fgmu(l_1-l_2)^2\right]\left[\fgmu(l_3)\widehat{D}(l_3;k)^2\right] \frac{\dd l_1 \dd l_2 \dd l_3}{\left(2\pi\right)^{3d}}\right)^{\frac{1}{2}}\nonumber \\
    &\qquad\qquad\times\left(\int_{A_6}\left[\fgmu(l_2)^3\right]\left[\fgmu(l_3)\fgmu(l_2-l_3)^2\right]\left[\fgmu(l_1)\fgmu(l_1-l_3)^2\right] \frac{\dd l_1 \dd l_2 \dd l_3}{\left(2\pi\right)^{3d}}\right)^{\frac{1}{2}}. \label{eqn:A_6integral}
\end{align}
We then perform the $l_i$ integrals in specific orders. For the first factor we integrate over $l_3$ (which factorises out already), then we integrate over $l_2$ for fixed $l_1$, and finally integrate over $l_1$. For the second factor we integrate over $l_1$ for fixed $l_2$ and $l_3$, then we integrate over $l_3$ for fixed $l_2$, and finally we integrate over $l_2$. There are therefore three forms of integral we need to bound:
\begin{align}
    \int_{\abs*{l}<\varepsilon}\fgmu(l)^3\frac{\dd l}{\left(2\pi\right)^{d}} &\leq \frac{1}{C^3_2}\int_{\abs*{l}<\varepsilon}\frac{1}{\abs*{l}^6}\frac{\dd l}{\left(2\pi\right)^{d}} = \frac{1}{C^3_2}\left(\frac{\mathfrak{S}_{d-1}}{\left(2\pi\right)^d}\right)\frac{\varepsilon^{d-6}}{d-6}\\
    \int_{\left\{\abs*{l}<\varepsilon\right\}\cap \left\{\abs*{l-l'}<\varepsilon\right\}}\fgmu(l)\fgmu(l-l')^2\frac{\dd l}{\left(2\pi\right)^{d}} &\leq \frac{1}{C^3_2}\int_{\left\{\abs*{l}<\varepsilon\right\}\cap \left\{\abs*{l-l'}<\varepsilon\right\}}\frac{1}{\abs*{l}^2}\frac{1}{\abs*{l-l'}^4}\frac{\dd l}{\left(2\pi\right)^{d}}\nonumber\\
    &\leq \frac{1}{C^3_2}\left(\int_{\abs*{l}<\varepsilon}\frac{1}{\abs*{l}^6}\frac{\dd l}{\left(2\pi\right)^{d}}\right)^\frac{1}{3}\left(\int_{\abs*{l-l'}<\varepsilon}\frac{1}{\abs*{l-l'}^6}\frac{\dd l}{\left(2\pi\right)^{d}}\right)^\frac{2}{3}\nonumber\\
    & = \frac{1}{C^3_2}\left(\frac{\mathfrak{S}_{d-1}}{\left(2\pi\right)^d}\right)\frac{\varepsilon^{d-6}}{d-6}\\
    \int_{\left\{\abs*{l}<\varepsilon\right\}\cup \left\{\abs*{l-k}<\varepsilon\right\}\cup \left\{\abs*{l+k}<\varepsilon\right\}}\fgmu(l)\widehat{D}(l;k)^2\frac{\dd l}{\left(2\pi\right)^{d}} &\leq \frac{C'}{C^3_2}\left(\frac{\mathfrak{S}_{d-1}}{\left(2\pi\right)^d}\right)\left(\frac{1}{d-6} + \frac{1}{d-4} + \frac{1}{d-2} + \frac{1}{d}\right)\varepsilon^{d-6}.
\end{align}
In this last inequality, $C'>0$ is some uniform constant. The last inequality is derived by applying H\"older's inequality in much the same way as for the second inequality, taking care to see when the three sets that are integrated over are overlapping. The calculation uses H\"older's inequality and a partition of the space, and is similar to the calculation performed in the proof of Lemma~\ref{lem:PTkTPBound}. The second integral we want to bound for $A_6$ is 
\begin{align}
    &\int_{A_6}\fgmu(l_1)^2\fgmu(l_2)^2\fgmu(l_1-l_2)\fgmu(l_1-l_3)\fgmu(l_2-l_3)\fgmu(l_3+k)\fgmu(l_3-k) \frac{\dd l_1 \dd l_2 \dd l_3}{\left(2\pi\right)^{3d}}\nonumber \\
    &\qquad\leq \left(\int_{A_6}\left[\fgmu(l_3+k)\fgmu(l_3-k)^2\right]\left[\fgmu(l_2)\fgmu(l_2-l_3)^2\right]\left[\fgmu(l_1)^3\right] \frac{\dd l_1 \dd l_2 \dd l_3}{\left(2\pi\right)^{3d}}\right)^{\frac{1}{2}}\nonumber \\
    &\qquad\qquad\times\left(\int_{A_6}\left[\fgmu(l_2)^3\right]\left[\fgmu(l_1)\fgmu(l_1-l_2)^2\right]\left[\fgmu(l_3-l_1)^2\fgmu(l_3+k)\right] \frac{\dd l_1 \dd l_2 \dd l_3}{\left(2\pi\right)^{3d}}\right)^{\frac{1}{2}}. \label{eqn:A_6integral2}
\end{align}
We then perform the $l_i$ integrals in specific orders. For the first factor we integrate over $l_1$ (which factorises out already), then we integrate over $l_2$ for fixed $l_3$, and finally integrate over $l_3$. For the second factor we integrate over $l_3$ for fixed $l_1$ and $l_2$, then we integrate over $l_1$ for fixed $l_2$, and finally we integrate over $l_2$. The new integrals we need to bound are:
\begin{multline}
    \int_{\left\{\abs*{l}<\varepsilon\right\}\cup \left\{\abs*{l-k}<\varepsilon\right\}\cup \left\{\abs*{l+k}<\varepsilon\right\}}\fgmu(l+k)\fgmu(l-k)^2\frac{\dd l}{\left(2\pi\right)^{d}} \\ \leq \frac{C'}{C^3_2}\left(\frac{\mathfrak{S}_{d-1}}{\left(2\pi\right)^d}\right)\left(\frac{1}{d-6} + \frac{1}{d-4} + \frac{1}{d-2} + \frac{1}{d}\right)\varepsilon^{d-6}
\end{multline}
\begin{multline}
    \int_{\left(\left\{\abs*{l}<\varepsilon\right\}\cup \left\{\abs*{l-k}<\varepsilon\right\}\cup \left\{\abs*{l+k}<\varepsilon\right\}\right)\cap\left\{\abs*{l-l'}<\varepsilon\right\}}\fgmu(l+k)\fgmu(l-l')^2\frac{\dd l}{\left(2\pi\right)^{d}} \\ \leq \frac{C'}{C^3_2}\left(\frac{\mathfrak{S}_{d-1}}{\left(2\pi\right)^d}\right)\left(\frac{1}{d-6} + \frac{1}{d-4}\right)\varepsilon^{d-6}.
\end{multline}
The calculation of these two bounds uses H\"older's inequality and a partition of the space, and is similar to the calculation performed in the proof of Lemma~\ref{lem:PTkTPBound}. The result of these bounds is that the integral in \eqref{eqn:A_6integral} is bounded by some constant multiple of
\begin{equation}
    \frac{1}{\left(d-6\right)^3}\left(\frac{\mathfrak{S}_{d-1}}{\left(2\pi\right)^d}\right)^3\varepsilon^{3d-18}.
\end{equation}

For $A_4$ and $A_5$ we can repeat the argument for $A_6$, but note that in some places the bound $1/\left(C_2\abs*{l}^2\right)$ will be replaced by $1/\left(C_2\varepsilon^2\right)$. The net result of this is that the factor $\left(d-6\right)^3$ can be replaced by $\left(d-6\right)^{\alpha_0}\left(d-4\right)^{\alpha_1}\left(d-2\right)^{\alpha_2} d^{\alpha_3}$ for some $\alpha_0,\alpha_1,\alpha_2,\alpha_3\leq 0$ such that $\alpha_0+\alpha_1+\alpha_2+\alpha_3 =3$. The net result is that the contribution from both $A_4$ and $A_5$ can also be bounded by some constant multiple of
\begin{equation}
    \frac{1}{\left(d-6\right)^3}\left(\frac{\mathfrak{S}_{d-1}}{\left(2\pi\right)^d}\right)^3\varepsilon^{3d-18}.
\end{equation}

Above we have established bounds for the various parts of the integral when $\vec{j}=\left(3,3,3,3,3,3,3,5\right)$. Having $j_8=4$ can also be dealt with by very similar arguments to that outlined above - the $l_3$ integrals are in fact simpler. The above bounds still hold with the $\varepsilon^{-18}$ becoming a $\varepsilon^{-16}$. Since $\varepsilon<1$, this produces a larger bound. Recall that up to a constant factor, having $j_m=1$ produces a larger bound than $j_m=2$ for $m=1,\ldots,7$. If we replace $j_m=3$ with $j_m=1$, then after ignoring constant values we lose a factor of $\varepsilon^{-2}$, but we may also lose a factor of $g(d)^\frac{1}{2}$ unless that was for $m=1,2$ (corresponding to $l_1$ direction) or $m=3,4$ (corresponding to $l_2$ direction). Note that doing for both directions will lose a $g(d)^\frac{1}{2}$ factor.

If $\frac{1}{d}\mathfrak{S}_{d-1}\varepsilon^d= O\left(g(d)^{\frac{1}{2}}\right)$ for all fixed $\varepsilon>0$, then it is clear that the dominating bound is $g(d)^\frac{1}{2}$. This arises, for example from the bound for $\vec{j}=\left(1,1,1,1,1,1,1,1\right)$. The situation is more complicated if $g(d)^{\frac{1}{2}}\ll\frac{1}{d}\mathfrak{S}_{d-1}\varepsilon^d$ for fixed $\varepsilon$, because we can choose to have $\varepsilon=\varepsilon(d)$ and take it to $0$ as $d\to\infty$. It is easy to see that all the $A_1$ components (and $A_2,\ldots,A_6$ components) for each $\vec{j}$ can be bounded by some constant multiple of $\frac{1}{d}\frac{\mathfrak{S}_{d-1}}{\left(2\pi\right)^d}\varepsilon^d$ (provided $d>6$). We therefore only need to compare this to the contributions from $A_0$.

The bound we get from the above arguments for each $A_0$ contribution are of the form $\varepsilon^{-2m}g(d)^\frac{n}{2}$, where $m$ is the number of factors of $\fgmu$ and $n$ is related to the multiplicity of $\Opconnf$ terms in each Fourier direction. To find the terms with the largest bound, we want to have as many factors of $\fgmu$ as possible without producing extra factors of $g(d)^\frac{1}{2}$. Since $g(d)^{\frac{1}{2}}\ll\frac{1}{d}\mathfrak{S}_{d-1}\varepsilon^d$ for fixed $\varepsilon$, increasing $n$ here instantly produces a smaller bound. By changing $j_8$ from $1$ to $4$ or $5$, or any other element from $1$ to $3$, we do gain a factor of $\fgmu$ but we also increase the number of $\Opconnf$ factors and therefore in most places we acquire extra factors of $g(d)^{\frac{1}{2}}$. The exception to this is for $j_1$ and $j_2$, which correspond to the $l_1$ directions, since they already produce a factor of $g(d)^{\frac{1}{2}}$. The $A_0$ component of the $\vec{j}=\left(3,3,1,1,1,1,1,1\right)$ term can be bounded by constant multiple of $\varepsilon^{-4}g(d)^{\frac{1}{2}}$. We then optimise our choice of $\varepsilon$ to have our two dominant bounds be of the same order. This produces $\varepsilon(d)^2 = g(d)^{\frac{1}{d}}\frac{2\pi d}{\e}\left(1+o(1)\right)$, and therefore our overall bound is given by
\begin{equation}
    \varepsilon^{-4}g(d)^\frac{1}{2} = g(d)^{\frac{1}{2}-\frac{2}{d}}\left(\frac{\e}{2\pi}\right)^2\frac{1}{d^2}\left(1+o(1)\right).
\end{equation}
Since $g(d)^{\frac{1}{2}-\frac{2}{d}} \leq g(d)^{\frac{1}{2}-\frac{3}{d}}$ for sufficiently large $d$, this bound is of order $\beta^2$ and our bound is proven.
\end{proof}

\end{appendix}

\paragraph{Acknowledgements.} This work is supported by  
\textit{Deutsche Forschungsgemeinschaft} (project number 443880457) through priority program ``Random Geometric Systems'' (SPP 2265). 
The authors thank the \emph{Centre de recherches math\'ematiques} Montreal for hospitality during a research visit in spring 2022 through the Simons-CRM scholar-in-residence program.

\bibliography{bibliography}{}
\bibliographystyle{alpha}

\end{document}